\documentclass[10.5pt]{article}

\usepackage{fullpage, amsmath}
\usepackage{amsthm,amsfonts,url,amssymb}
\usepackage{mathrsfs}
\usepackage{amssymb,amscd}
\usepackage{color}
\usepackage{chemarrow}
\usepackage{pictexwd,dcpic}
\usepackage[textwidth=18cm,centering]{geometry}
\usepackage{extarrows}
\usepackage[colorlinks,linkcolor=red,anchorcolor=green,citecolor=blue]{hyperref}
\usepackage{enumitem}
\usepackage{lipsum}
\usepackage[all,cmtip]{xy}
\usepackage{leftidx}
\usepackage[version=3]{mhchem}
\usepackage{makeidx}
\usepackage{bm} 
\usepackage{stmaryrd}

\makeindex
\newtheorem{thm}{Theorem}[section]
\newtheorem{cor}[thm]{Corollary}
\newtheorem{lem}[thm]{Lemma}
\newtheorem{pro}[thm]{Proposition}
\newtheorem{dfn}[thm]{Definition}
\newtheorem{hypothesis}[thm]{Hypothesis}
\newtheorem{rmk}[thm]{Remark}

\newtheorem{conjecture}[thm]{Conjecture}

\newcommand{\hooklongrightarrow}{\lhook\joinrel\longrightarrow}

\DeclareMathOperator{\End}{End}

\DeclareMathOperator{\Frob}{Frob}
\DeclareMathOperator{\Gal}{Gal}
\DeclareMathOperator{\GL}{GL}

\DeclareMathOperator{\Hom}{Hom}

\DeclareMathOperator{\Spm}{Sp}
\DeclareMathOperator{\Spec}{Spec}
\DeclareMathOperator{\Spf}{Spf}

\DeclareMathOperator{\Supp}{Supp}
\DeclareMathOperator{\Tor}{Tor}

\newcommand{\ccyc}{{\epsilon}}

\DeclareMathOperator{\coker}{coker}
\newcommand{\cris}{{\rm cris}}

\newcommand{\dR}{{\rm dR}}

\renewcommand{\mod}{{\rm\,mod\,}}

\newcommand{\red}{{\rm red}}

\DeclareMathOperator{\reg}{reg}
\newcommand{\rig}{{\rm rig}}

\newcommand{\gal}{{\rm Gal}}

\newcommand{\fil}{{\rm Fil}}

\newcommand{\ind}{{\rm Ind}}
\newcommand{\homo}{{\rm Hom}}
\newcommand{\EndO}{{\rm End}}
\newcommand{\ext}{{\rm Ext}}
\newcommand{\GLN}{{\rm GL}}

\newcommand{\ana}{{\rm an}}
\newcommand{\st}{{\rm St}}

\newcommand{\op}{{\overline{{\mathbf{P}}}}}

\newcommand{\ob}{{\overline{{\mathbf{B}}}}}
\newcommand{\on}{{\overline{{\mathbf{N}}}}}

\newcommand{\unr}{{\rm unr}}

\newcommand{\hH}{{\mathrm{H}}}

\newcommand{\ra}{{\longrightarrow}}
\DeclareMathOperator{\val}{\mathrm val}
\newcommand{\Modo}{{\textrm{mod}}}

\newcommand{\spf}{{\rm Spf}}

\newcommand{\df}{{\mathrm{DF}}}
\newcommand{\wdre}{{\textbf{r}}}

\newcommand{\rec}{{\rm rec}_L}
\newcommand{\lalg}{{\rm lalg}}

\newcommand{\cyc}{{\rm cyc}}
\newcommand{\Art}{{\rm Art}}

\newcommand{\Ker}{{\rm Ker}}

\newcommand{\wt}{{\rm wt}}

\newcommand{\pr}{{\rm pr}}
\newcommand{\dr}{{\rm dR}}
\newcommand{\alge}{{\rm alg}}

\newcommand{\tee}{{{\otimes}_{\cR_{E,L}}}}
\newcommand{\te}{{\otimes_E}}

\newcommand{\undelram}{{\delta^0}}

\newcommand{\gr}{{\rm gr}}
\newcommand{\hpi}{{{\mathbf{h}}}}
\newcommand{\Dpik}{{\mathbf{D}}}
\newcommand{\pdr}{{\rm pdR}}

\newcommand{\sbanpik}{{\big(\mathrm{Spec}\hspace{2pt}\FZ_{\Omega_{r}}^{\otimes k}\big)^{\mathrm{rig}}}}

\newcommand{\rigch}{{\mathcal{Z}_{\bL_{r,\emptyset},\mathcal{O}_L}}}
\newcommand{\rigchl}{{\mathcal{Z}_{\bL_{r,\emptyset},L}}}

\newcommand{\univ}{{\rm univ}}
\newcommand{\loc}{{\rm loc}}

\newcommand{\pst}{{\rm pst}}
\newcommand{\ver}{{\rm ver}}
\newcommand{\gen}{{\rm gen}}

\newcommand{\bmdel}{{ \bm{\delta}_\bh}}
\newcommand{\ul}{\underline}
\newcommand{\omepik}{\Omega_{r}^{\oplus k}}

\newcommand{\defvarring}{{R_{\overline{r}}^\Box}}

\newcommand{\defvar}{{ X^\Box_{\omepik,\mathbf{{h}}}(\overline{r})}}

\newcommand{\complocaldefvarrhox}{{ \widehat{X^\Box_{\omepik,\mathbf{{h}}}}(\overline{r})_{x}}}

\usepackage{amsfonts}
\newcommand{\BOne} {{\mathchoice{\hbox{\rm1\kern-2.7pt l\kern.9pt}}
		{\hbox{\rm1\kern-2.7pt l\kern.9pt}}
		{\hbox{\scriptsize\rm1\kern-2.3pt l\kern.4pt}}
		{\hbox{\scriptsize\rm1\kern-2.4pt l\kern.5pt}}}}

\newcommand{\BA}{{\mathbb{A}}}

\newcommand{\BC}{{\mathbb{C}}}

\newcommand{\BG}{{\mathbb{G}}}

\newcommand{\BQ}{{\mathbb{Q}}}

\newcommand{\BU}{{\mathbb{U}}}

\newcommand{\BW}{{\mathbb{W}}}

\newcommand{\BZ}{{\mathbb{Z}}}

\newcommand{\bA}{{\mathbf{A}}}
\newcommand{\bB}{{\mathbf{B}}}

\newcommand{\bF}{{\mathbf{F}}}
\newcommand{\bG}{{\mathbf{G}}}
\newcommand{\bH}{{\mathbf{H}}}
\newcommand{\bI}{{\mathbf{I}}}

\newcommand{\bL}{{\mathbf{L}}}

\newcommand{\bN}{{\mathbf{N}}}

\newcommand{\bP}{{\mathbf{P}}}
\newcommand{\bQ}{{\mathbf{Q}}}
\newcommand{\bR}{{\mathbf{R}}}

\newcommand{\bT}{{\mathbf{T}}}
\newcommand{\bU}{{\mathbf{U}}}

\newcommand{\bW}{{\mathbf{W}}}

\newcommand{\bZ}{{\mathbf{Z}}}

\newcommand{\bx}{{\mathbf{x}}}

\newcommand{\bh}{{\mathbf{h}}}
\newcommand{\bk}{{\mathbf{k}}}

\newcommand{\cN}{\mathcal N}
\newcommand{\cL}{\mathcal L}
\newcommand{\co}{\mathcal O}

\newcommand{\cB}{\mathcal B}
\newcommand{\cR}{\mathcal R}
\newcommand{\cH}{\mathcal H}
\newcommand{\cC}{\mathcal C}
\newcommand{\cS}{\mathcal S}
\newcommand{\cD}{\mathcal D}

\newcommand{\cT}{\mathcal T}
\newcommand{\cM}{\mathcal M}
\newcommand{\cF}{\mathcal F}
\newcommand{\cV}{\mathcal V}

\newcommand{\cG}{\mathcal G}

\newcommand{\cO}{\mathcal O}
\newcommand{\cP}{\mathcal P}

\newcommand{\cZ}{\mathcal Z}

\newcommand{\cX}{\mathcal X}

\newcommand{\FB}{{\mathfrak{B}}}
\newcommand{\FC}{{\mathfrak{C}}}

\newcommand{\FT}{{\mathfrak{T}}}
\newcommand{\FU}{{\mathfrak{U}}}

\newcommand{\FX}{{\mathfrak{X}}}
\newcommand{\FY}{{\mathfrak{Y}}}
\newcommand{\FZ}{{\mathfrak{Z}}}

\newcommand{\fa}{{\mathfrak{a}}}
\newcommand{\fb}{{\mathfrak{b}}}

\newcommand{\fg}{{\mathfrak{g}}}

\newcommand{\fj}{{\mathfrak{j}}}

\newcommand{\fl}{{\mathfrak{l}}}
\newcommand{\fm}{{\mathfrak{m}}}
\newcommand{\fn}{{\mathfrak{n}}}

\newcommand{\fp}{{\mathfrak{p}}}

\newcommand{\ft}{{\mathfrak{t}}}
\newcommand{\fu}{{\mathfrak{u}}}

\newcommand{\fz}{{\mathfrak{z}}}

\newcommand{\sD}{\mathscr D}

\newcommand{\sW}{\mathscr W}
\newcommand{\sL}{\mathscr L}

\begin{document}	
	
\title{\textbf{\textsc{Companion points and locally analytic socle conjecture for Steinberg case}}}

\author{Yiqin He
\thanks{Morningside Center of Mathematics, Chinese Academy of Science,\;No. 55, Zhongguancun East Road, Haidian District, Beijing 100190, P.R.China,\;E-mail address:\texttt{\;heyiqin@amss.ac.cn}
}}

\date{}
\maketitle

\begin{abstract}
In this paper,\;we will modify the Breuil-Hellmann-Schraen's (more generally,\;resp.,\;Breuil-Ding's) local model for the trianguline variety (resp.,\;Bernstein paraboline variety) to certain semistable (resp., potentially semistable) non-crystalline point with regular Hodge-Tate weights.\;Then we deduce several local-global compatibility results,\;including a classicality result,\;and the existence of expected companion points on the (definite) eigenvariety and locally analytic socle conjecture for such semistable non-crystalline Galois representations,\;under certain hypothesis on trianguline variety and the usual Taylor-Wiles assumptions.\;Moreover,\;we also discuss slightly the coherent sheaves obtained by patching argument and the coherent sheaves which are constructed from local models and Bezrukavnikov functor,\;under the route of the recently work of Hellmann-Hernandez-Schraen.\;
\end{abstract}

{\hypersetup{linkcolor=black}
\tableofcontents}

\numberwithin{equation}{section}
	
\numberwithin{thm}{section}

\setlength{\baselineskip}{15pt}

\section{Introduction}

Let $p$ be a prime number and $n\geq 2$ an integer.\;The aim of this paper is to prove several results in the problem of companion form or locally analytic socle conjecture when the $p$-adic Galois representation is semistable non-crystalline with regular Hodge-Tate weights and full monodromy rank (that we call Steinberg case,\;which is an extreme case in semistable case).\;

The socle phenomenon and the Breuil’s locally analytic socle conjecture in \cite{breuil2016socle} and \cite{breuil2015II} are some local-global compatibility results in the locally analytic aspect of the $p$-adic local Langlands program.\;This conjecture is closely related to the existence of companion points,\;see \cite{HN17companion} for a statement of the conjecture on the existence of all companion forms for finite slope overconvergent $p$-adic automorphic forms of $\GLN_n$ (in the language of determining the set of companion points on the eigenvariety that are associated with the same $p$-adic Galois representation but with possibly different weights).\;These problems were discussed by Breuil-Hellmann-Schraen in \cite{breuil2019local} (resp.,\;Z. Wu in \cite{wu2021local} and \cite{wu2023localcompan2},\;resp., Breuil-Ding in \cite{Ding2021}) when the $p$-adic Galois representation is \textit{generic} crystalline  with regular Hodge-Tate weights (resp.,\;\textit{generic} crystalline with non-regular Hodge-Tate weights,\;resp.,\;certain \textit{generic} potentially crystalline  with regular Hodge-Tate weights).\;

The Steinberg case is previously explored for $\GLN_2(L)$ case in \cite{2015Ding}.\;In this paper,\;we talk about the local model for the trianguline variety (resp.,\;paraboline deformation ring) at certain semistable non-crystalline (or say Steinberg) point (note that the Steinberg case is \textit{not generic}),\;then we discuss the existence of companion points and companion constituents on the eigenvariety (resp.,\;Bernstein eigenvariety).\;Furthermore,\;we also discuss the coherent sheaves obtained from patching functors and the coherent sheaves constructed from local models and the Bezrukavnikov functor,\;by following the route of \cite{HHS}.\;The relationship of such two coherent sheaves is predicted by the categorical $p$-adic Langlands correspondence.\;

%\subsection{Companion points and companion constituents and main results}

Before stating our main results,\;we briefly give the global setup of the paper (in the setting of definite unitary groups as Breuil).\;Let $F^+$ be a totally real field and $F$ be a quadratic totally imaginary extension of $F^+$.\;Let $\bG_U$ be a unitary group attached to the quadratic extension $F/F^+$ such that $\bG_U\times_{F^+}F\cong \GLN_n$ and $\bG_U(F^+\otimes_{\bQ}\bR)$ is compact.\;Let $S_p$ be the set of places of $F^+$ above $p$,\;and assume that each place in  $S_p$ is split in $F$.\;Fix a place $\fp$ of $F^+$ above $p$ and a place $\widetilde{\fp}$ of $F$ over $\fp$.\;Let $E$ be a sufficiently large finite extension of $\bQ_p$.\;Let $k_E$ be the reside field of $E$.\;From now on,\;we write $L=F_{\widetilde{\fp}}$ for simplicity.\;

Let $U^p=\prod_{v\nmid p}\bG_U(F_v^+)$ (resp.,\;$U_p^\fp=\prod_{v|p,p\neq \fp}U_v$) be a compact open subgroup of $\prod_{v|p,p\neq \fp}\bG_U(F_v^+)$ (resp.,\;$\prod_{v\nmid p}\bG_U(F_v^+)$).\;These give a prime-to-$\fp$ level $U^{\fp}:=U^pU_p^\fp\subseteq \bG_U(\bA_{F^+}^{\infty,\fp})$.\;Let $\widehat{S}_{\xi,\tau}(U^{\fp},E)$ be the space of certain $p$-adic automorphic forms (roughly speaking,\;the space of $p$-adic algebraic automorphic forms over coefficient field $E$ of tame level $U^p$,\;of fixed type $\sigma_0$ (a smooth representation of $\GLN_n(\cO_L)$) at the place $S_p\setminus\{\fp\}$,\;full level at $p$,\;and whose weight is $0$ at places above $\fp$,\;and given by some fixed weight at each of the places in $S_p\setminus\{\fp\}$).\;This space is an unitary Banach space representation of $\bG_U(F_{\widetilde{\fp}})=\GLN_n(L)$ (so that its locally $\bQ_p$-analytic vectors $\widehat{S}_{\xi,\tau}(U^{\fp},E)^{\ana}$ forms an admissible locally $\bQ_p$-analytic representation of $\GLN_n(L)$).\;This space is also equipped with a faithful action of a certain commutative global Hecke algebra $\bT^{S_p,\univ}$ over $\cO_E$ which is generated by some prime-to-$p$ Hecke operator.\;

Let $\overline{{\rho}}:\gal_F\rightarrow \GLN_n(k_E)$ be a $\Modo\;p$ irreducible representation.\;We can associate to $\overline{\rho}$ a maximal ideal $\fm_{\overline{\rho}}$ of $\bT^{S_p,\univ}$.\;Let $\widehat{S}_{\xi,\tau}(U^{\fp},E)^{\ana}_{\overline{\rho}}$ be the localization of $\widehat{S}_{\xi,\tau}(U^{\fp},E)^{\ana}$ at $\fm_{\overline{\rho}}$ (with respect to the $\bT^{S_p,\univ}$-structure).\;There is a rigid analytic variety $Y(U^{\fp},\overline{\rho})$ over $E$ (called the Hecke eigenvariety) that parametrizes the systems of Hecke eigenvalues of finite slope in the representation  $\widehat{S}_{\xi,\tau}(U^{\fp},E)^{\ana}_{\overline{\rho}}$ (more generally,\;Bernstein Hecke eigenvariety \cite{Ding2021},\;which parametrizes the systems of certain Hecke eigenvalues which are not of finite slope).\;

Suppose that $E$ is sufficiently large,\;i.e.\;$\Sigma_L:=\{\sigma: L\hookrightarrow \overline{\BQ}_p\} =\{\sigma: L\hookrightarrow E\}$.\;Let $d_L:=|\Sigma_L|$.\;Put $q_L:=p^{f_L}$,\;where $f_L$ denotes the unramified degree of $L$ over $\bQ_p$.\;Let $\GLN_n$ be the general linear group over $L$.\;Let $\bB$ (resp.,\;$\bT$) be the Borel subgroup of upper triangular matrices (resp.,\;the diagonal torus).\;Let $\fg$ (resp.,\;$\ft\subseteq\fb$) be the $E$-Lie algebra of $\GLN_n$ (resp.,\;$\bT\subseteq\bB$).\;Let $\ob$ be the parabolic subgroup of $\GLN_{n}$ opposite to $\bB$.\;Let $\bG_{L}:=(\mathrm{Res}_{L/\bQ_p}\GLN_n)\times_{\bQ_p}E$ (resp.,\;$\bT_{L}:=(\mathrm{Res}_{L/\bQ_p}\bT)\times_{\bQ_p}E\subseteq \bB_{L}:=(\mathrm{Res}_{L/\bQ_p}\bB)\times_{\bQ_p}E$.\;Let $\fg_L$ (resp.,\;$\ft_L\subseteq\fb_L$) be the $E$-Lie algebra of $\bG_{L}$ (resp.,\;$\bT_{L}\subseteq\bB_{L}$).\;We have $\bG_{L}\cong \prod_{\tau\in \Sigma_L}\GLN_n$ and $\fg_L\cong \prod_{\tau\in \Sigma_L}\fg$,\;etc.\;Let $\sW_{n,L}\cong S_n^{d_L}$ be the Weyl group of $\bG_{L}$.\;We put $G=\GLN_{n}(L)$.\;Let $w_0$ (resp.,\;$\underline{w}_0:=(w_0)_{\tau\in \Sigma_L}$) be the longest element in $\sW_{n}$ (resp.,\;$\sW_{n,L}$).\;

A point $y\in Y(U^{\fp},\overline{\rho})$ can be uniquely described by a pair $(\rho_y,\underline{\delta})$,\;where $\rho$ is a Galois deformation of $\overline{\rho}$ on $E$ and $\underline{\delta}=\delta_1\otimes\cdots\otimes\delta_n$ is a locally $\bQ_p$-analytic character of $L^n=\bT(L)$.\;We are interested in point $y=(\rho_y,\underline{\delta})$ that are semistable non-crystalline with full monodromy rank  (we say that $y$ is of \textit{Steinberg type}),\;this means that the local Galois representation  $\rho_L:=\rho_y|_{\gal_{F_{\widetilde{\fp}}}}=\rho_y|_{\gal_L}$ is  semistable non-crystalline,\;and the monodromy operator $N$ on $D_{\mathrm{st}}(\rho_L)$ satisfies $N^{n-1}=0$ (thus the associated smooth representation of $G$ is the  Steinberg representation via the classical local Langlands correspondence).\;

Under this assumption,\;we can associate to $y$ two permutations $w_y=(w_{y,\tau}),w_{\cF}=(w_{\cF,\tau})\in\sW_{n,L}$.\;The first one measuring the relative positions of the weights $\wt(\delta_i)$ of $\delta_i$ with the dominant order,\;and the second one measuring the relation position of two flags coming from the $p$-adic Hodge theory (i.e.,\;the unique $(\varphi,N)$-stable flag and the Hodge filtration,\;see Section \ref{Omegafil}).\;More precisely,\;let $\bh:=(\hpi_{\tau,1}>\hpi_{\tau,2}>\cdots>\hpi_{\tau,n} )_{\tau\in \Sigma_L}$ be the Hodge-Tate weights of $\rho_{L}$ (we fix such $\bh$ throughout this paper).\;For $1\leq i\leq n$,\;put $\hpi_{i}=(\hpi_{\tau,i})_{\tau\in \Sigma_L}$.\;Then $w_{y}$ and $w_{\cF}$ are given as follows.\;
\begin{itemize}
	\item Note that $D_{\rig}(\rho_{L})$ is trianguline.\;Assume that $D_{\rig}(\rho_L)$ admits a triangulation $\cF$ with parameters $\unr({\alpha})z^{w_{\cF}\underline{w}_0(\bh_1)},\;\cdots\,\;\unr({\alpha q_L^{i-1}})z^{w_{\cF}\underline{w}_0(\bh_i)},\;\cdots\,\;,\unr({\alpha q_L^{n-1}})z^{w_{\cF}\underline{w}_0(\bh_n)}$ for some $\alpha\in E$,\;where  $w_{\cF}\underline{w}_0(\bh_i)=(\bh_{\tau,(w_{\cF,\tau}{w}_0)^{-1}(i)})_{\tau\in\Sigma_L,1\leq i\leq n}$.\;
	\item $\wt_{\tau}(\delta_i)=\bh_{\tau,\;w_{y,\tau}^{-1}(i)}$ for $1\leq i\leq n$ and $\tau\in\Sigma_L$.\;
\end{itemize}
If $w=1$,\;we say $y$ is strictly dominant.\;Recall that $y'=(\rho_y,\underline{\delta'})\in Y(U^{\fp},\overline{\rho})$ is called a companion point of $y$ if $\underline{\delta}^{-1}\underline{\delta'}$ is a $\bQ_p$-algebraic character.\;

We state our main result on companion points and companion constituents under the assumption $L\neq \bQ_p$.\;We need the following so-called ``Talyor-Wiles hypothesis''.\;

\begin{hypothesis}\label{TWhypo}\hspace{20pt}
\begin{itemize}
	\item[(1)] $p>2$;
	\item[(2)] the field $F$ is unramified over $F^+$,\;$F$ does not contain a non-trivial root $\sqrt[p]{1}$ of $1$ and $G$ is quasi-split at all finite places of $F^+$;
	\item[(3)] $U_{v}$ is hyperspecial when the finite place $v$ of $F^+$ is inert in $F$;
	\item[(4)] $\overline{\rho}$ is absolutely irreducible and $\overline{\rho}(\gal_{F(\sqrt[p]{1})})$ is adequate.\;
\end{itemize}
\end{hypothesis}

We first have the following classicality result.\;let $\fm_{\rho}\subset \bT^{S_p,\univ}_{\overline{\rho}}[1/p]$ be the maximal ideal associated to $\rho$ and $\widehat{S}_{\xi,\tau}(U^{\fp},E)^{\ana}_{\overline{\rho}}[\fm_{{\rho}}]$ be the subspace of $\widehat{S}_{\xi,\tau}(U^{\fp},E)^{\ana}_{\overline{\rho}}$ annihilated by $\fm_{\rho}$.\;
\begin{thm}\label{introClassicality}(Classicality,\;See Theorem \ref{Classicality}) Assume Hypothesis \ref{TWhypo} and Hypothesis \ref{introhypo-1}.\;If the Galois representation $\rho:\gal(\overline{F}/F)\rightarrow\GLN_n(E)$ comes from a Steinberg type point $y=(\rho,\underline{\delta})\in Y(U^{\fp},\overline{\rho})$,\;then $\widehat{S}_{\xi,\tau}(U^{\fp},E)^{\lalg}_{\overline{\rho}}[\fm_{{\rho}}]\neq 0$,\;i.e.,\;$\rho$ is associated to a classical automorphic representation of $\bG_U(\BA_{F^+}^{\infty})$.\;
\end{thm}

It is conjectured in \cite[Conjecture 1.2.5]{HN17companion} and \cite[Conjecure 6.5]{breuil2015II} (note that Breuil only states such conjecture for potentially (generic) crystalline case) that the companion points of $y$ are parametrized by $w'\in \sW_{n,L}$ such that $w'\underline{w}_0\geq w_{\cF}$ (the usual Bruhat order in $\sW_{n,L}$).\;We write $y_{w'}$ for the conjectural companion point associated to $w'$.\;

\begin{thm}\label{introcompanglobal}(Theorem \ref{globalcompanion}) Assume Hypothesis \ref{TWhypo} and Hypothesis \ref{hy114} below (certain hypothesis on trianguline variety).\;If  $\rho:\gal(\overline{F}/F)\rightarrow\GLN_n(E)$ comes from a Steinberg type strictly dominant point $y=(\rho,\underline{\delta})\in Y(U^{\fp},\overline{\rho})$ (and thus  $\widehat{S}_{\xi,\tau}(U^{\fp},E)^{\lalg}_{\overline{\rho}}[\fm_{{\rho}}]\neq 0$ by Theorem \ref{introClassicality}),\;then all companion points of $y$ are $y_{w'}$ for $w'\underline{w}_0\geq w_{\cF}$.\;
\end{thm}

%In particular,\;if $\widehat{S}_{\xi,\tau}(U^{\fp},E)^{\lalg}_{\overline{\rho}}[\fm_{{\rho}}]\neq 0$ (i.e.,\;$\rho$ is associated to a classical automorphic representation of $\bG_U(\bA_{F^+})$),\;and $\rho_L$ is semistable non-crystalline with full monodromy rank,\;then such $\rho$ is a  typical example in Theorem \ref{introcompanglobal}.\;

The existence of companion points is a weaker version of the so-called locally analytic socle conjecture.\;Put ${\bm\lambda}_\bh:=(\hpi_{\tau,i}+i-1)_{\tau\in \Sigma_L,1\leq i\leq n}$,\;which is a dominant weight of $\bG_{L}$ with  respect to $\bB_{L}$.\;For $a\in E$,\;denote by $\unr(a)$ the unramified character of $L^\times$ sending uniformizers to $a$.\;By the Orlik-Strauch construction \cite[Theorem]{orlik2015jordan},\;we consider the locally $\bQ_p$-analytic representations 
$\bI_{w\underline{w}_0}:=\cF^G_{\ob(L)}\big(\overline{L}(-w\underline{w}_0\cdot{\lambda}_{\bh}),\unr(\beta)\big)$ for any $w\in\sW_{n,L}$ and $\beta:=\alpha q_L^{\frac{n-1}{2}}$.\;Recall that $\bI_{w\underline{w}_0}$ admits a unique quotient $C(w\underline{w}_0)$,\;which is a locally $\bQ_p$-analytic irreducible admissible representation.\;

%,\;where $1_{\bT_n(L)}$ is the trivial representation of $\bT_n(L)$ over $E$,\;$\beta:=\alpha q_L^{\frac{n-1}{2}}$ and $1_{\bT_n(L)}(\beta):=1_{\bT_n(L)}\otimes_E\unr(\beta)\circ\det$.\;As usual,\;we denote by $w\underline{w}_0\cdot{\lambda}_{\bh}$ the dot action on the weight ${\lambda}_{\bh}$.\;

%Let $\ob_n\subseteq\op_{w\underline{w}_0}$ be the maximal parabolic subgroup of $\GLN_n$ for the $\overline{L}(-w\underline{w}_0\cdot{\lambda}_{\bh})$,\;and let $\bL_{w\underline{w}_0}$ be the unique Levi subgroup of $\op_{w\underline{w}_0}$ containing $\bT_n$.\;Then the irreducible components of $\cF^G_{\ob_n(L)}\big(\overline{L}(-w\underline{w}_0\cdot{\lambda}_{\bh}),1_{\bT_n(L)}(\beta)\big)$ are given by $\cF^G_{\op_{w\underline{w}_0}(L)}\big(\overline{L}(-w\underline{w}_0\cdot{\lambda}_{\bh}),W(\beta)\big)$,\;where $W$ is taken over irreducible components of the smooth parabolic induction $\big(\ind^{\bL_{w\underline{w}_0}(L)}_{\ob_n(L)\cap \bL_{w\underline{w}_0}(L)}1_{\bT_n(L)}\big)^{\infty}$ and $W(\beta):=W\otimes_E\unr(\beta)\circ\det$.\;Let $\st^{\infty}_{\bL_{w\underline{w}_0}(L)}$ be the unique quotient of $\big(\ind^{\bL_{w\underline{w}_0}(L)}_{\ob_n(L)\cap \bL_{w\underline{w}_0}(L)}1_{\bT_n(L)}\big)^{\infty}$.\;We put 
%\[C(w\underline{w}_0):=\cF^G_{\op_{w\underline{w}_0}(L)}\big(\overline{L}(-w\underline{w}_0\cdot{\lambda}_{\bh}),\st^{\infty}_{\bL_{w\underline{w}_0}(L)}(\beta)\big).\;\]

\begin{thm}\label{introlocallysocle}(Theorem \ref{equicompanpointconsti},\;Theorem \ref{globalcompanion},\;``a special case of locally analytic socle conjecture") Assume Hypothesis \ref{TWhypo} and Hypothesis \ref{hy114}.\;Then  $C(w\underline{w}_0)$ is a subrepresentation of $\widehat{S}_{\xi,\tau}(U^{\fp},E)^{\ana}_{\overline{\rho}}[\fm_{{\rho}}]$ if and only if $w\geq w_{\cF}$.\;
\end{thm}

For $n=2$,\;Theorem \ref{introlocallysocle} was firstly proved by Ding \cite{2015Ding} with the condition $\lg(w\underline{w}_0)=1$.\;For $n>2$,\;such results are previously not known (to the author's knowledge).\;For $n=2$ and $\lg(w\underline{w}_0)>1$,\;we also give an alternative proof of Theorem \ref{introlocallysocle} (for $\GLN_2(L)$ case,\;without Hypothesis \ref{hy114}) in Appendix \ref{appGL2(L)},\;by combining the arguments in \cite{2015Ding} with \cite{CompanionpointforGLN2L},\;see Theorem \ref{socleGLN2}.\;The basic strategy is by computing directly the (partial de Rham) cohomology of $(\varphi,\Gamma)$-modules (without using local models) and using some stratifications of trianguline variety and patched eigenvariety to compare different complete local rings.\;

On the other hand,\;if $L=\bQ_p$,\;\cite[Theroem 5.6.5]{MSW} gives an explicit local model for Steinberg case (a variation of Breuil-Hellmann-Schraen's local model),\;and show that $\widehat{\cX}^{\flat}_{L,x_{\pdr}}$ is normal and Cohen-Macaulay at point $x_{\pdr}$,\;by blowing up the schemes given by Grothendieck's simultaneous resolution.\;The author use a calculation on the explicit basis of cohomology of $(\varphi,\Gamma)$-modules of rank $1$ (done by Colmez) to study the universal cocycle and the universal derivation.\;But the  campanion points and socle conjecture are not discussed in this work.\;We will treat this problem in the forthcoming work under the results in \cite{MSW}.\;

In the case of $n>2$ and $L\neq \bQ_p$,\;we prove the main results by modifying the methods in \cite{breuil2019local} and \cite{wu2021local} to our Steinberg case.\;The method in \cite{breuil2019local} was replacing the Hecke eigenvariety  $Y(U^{\fp},\overline{\rho})$ by the patched eigenvariety $X_{\fp}(\overline{\rho})$ in \cite{breuil2017smoothness} (constructed from the patching module \cite{PATCHING2016}).\;Then the local geometry of the  patched eigenvariety at generic crystalline points can be reflected by the corresponding local geometry of the so-called triaguline variety $X_{\mathrm{tri}}(\overline{r})$,\;where $\overline{r}:=\overline{\rho}|_{\gal_{F_{\widetilde{\fp}}}}=\overline{\rho}|_{\gal_L}$.\;The triaguline variety parameterizes local trianguline Galois representations.\;The new ingredient is that Breuil-Hellmann-Schraen find  local models of the trianguline variety at the  \textit{generic crystalline and regular points},\;by using some varieties studied in geometry representation theory.\;In precise,\;the formal completion of triaguline variety at certain  generic crystalline points can be reflected,\;up to formally smooth morphisms,\;from some algebraic varieties which are related to the Springer resolution.\;The \textit{generic} assumption is essential for the proof of the formally smoothness.\;Furthermore,\;Zhixiang Wu explores the non-regular cases in \cite{wu2021local} and \cite{wu2023localcompan2}.\;

To understanding the local geometry of triaguline variety at our Steinberg type
point for $L\neq \bQ_p$ case,\;we also need a variation of Breuil-Hellmann-Schraen's local model (the main difference is that the parameters of its triangulation are \textit{non-generic},\;so the previous local model map (i.e.,\;(\ref{keymorphismformally})) are not necessary formally smooth).\;We now explain our results explicitly.\;Indeed,\;the following discussion on local models are proved in more general situation,\;i.e.,\;for certain potentially semistable non-crystalline  Galois representation $\rho_L$ such that $D_{\rig}(\rho_L)$ admits the so-called critical special $\omepik$-filtration (see \cite{Ding2021} and Section \ref{Omegafil} for more precise statement,\;which can be viewed as a paraboline analogue of triangulation).\;For simplicity,\;we restrict to the trianguline case in  introduction.\;

Let $\overline{r}:=\overline{{\rho}}|_{\gal_L}:\gal_L\rightarrow \GLN_n(k_E)$.\;We denote by $R_{\overline{r}}^{\square}$ the maximal reduced and $p$-torsion free quotient of the universal $\co_E$-lifting ring of $\overline{r}$.\;Let $\widehat{T}_L$ be the rigid space over $E$ parametrizing continuous characters of $\bT(L)$.\;The triaguline variety $X_{\mathrm{tri}}(\overline{r})$ is a closed subspace of $\mathfrak{X}_{\overline{r}}^\Box\times\widehat{T}_L$,\;where 
$\mathfrak{X}_{\overline{r}}^\Box=(\spf\;R_{\overline{r}}^{\square})^{\rig}$.\;For $w'\in \sW_{n,L}$,\;we define
\[\delta_{w'}:=(\unr({\alpha})z^{w'(\bh_1)},\;\cdots\,\;\unr({\alpha q_L^{i-1}})z^{w'(\bh_i)},\;\cdots\,\;,\unr({\alpha q_L^{n-1}})z^{w'(\bh_n)}),\;w'(\bh_i)=(\bh_{\tau,w_{\tau}^{-1}(i)})_{\tau\in\Sigma_L}\]
For $w'=(w'_{\tau})\in \sW_{n,L}$,\;we have $x_{w'}=(\rho_L,\delta_{w'})\subseteq \mathfrak{X}_{\overline{r}}^\Box\times\widehat{T}_L$.\;We write $x:=x_{1}$ (the so-called strictly dominant point).\;The assumption on $y\in Y(U^{\fp},\overline{\rho})$ implies $x_{w_y}:=(\rho_L,\delta_{w_y})\in X_{\mathrm{tri}}(\overline{r})$.\;Let $\widehat{X_{\mathrm{tri}}(\overline{r})}_{x_{w_y}}$ be the completion of $X_{\mathrm{tri}}(\overline{r})$ at point $x_{w_y}$.\;

Let 
$\widetilde{\fg}:=\{(g\bB,\psi)\in \GLN_n/\bB\times\fg|\; \mathrm{Ad}(g^{-1})\psi\in \fb\}$,\;and $\widetilde{\fg}_L:=\{(g\bB_{L},\psi)\in \bG_{L}/\bB_{L}\times\fg_L|\; \mathrm{Ad}(g^{-1})\psi\in \fb_L\}$.\;The projection $\widetilde{\fg}\rightarrow \fg$ and $\widetilde{\fg}_L\rightarrow \fg_L$ are the so-called Grothendieck's simultaneous resolution of singularities.\;Let $X_L:=\widetilde{\fg}_L\times_{\fg_L}\widetilde{\fg}_L$ (resp.,\;$X:=\widetilde{\fg}\times_{\fg}\widetilde{\fg}$) be the scheme defined in \cite[(2.3)]{breuil2019local},\;which is equidimensional of dimension $d_L\dim G$ (resp.,\;$\dim G$).\;The irreducible components of $X_L$ are parameterized by $\{X_{L,w'}\}_{w'\in \sW_{n,L}}$.\;We have decompositions $\widetilde{\fg}_L=\prod_{\tau\in \Sigma_L}\widetilde{\fg}_\tau$ and $X_L=\prod_{\tau\in \Sigma_L}X_\tau$ by $\Sigma_L$-components with $\widetilde{\fg}_\tau\cong \widetilde{\fg}$ and $X_\tau\cong X$.\;

By the assumption on $x_w$ and the theory of almost de Rham representations (recall the period ring $B_{\pdr}$ and the associated functor $D_{\pdr}(-)$),\;the finite free  $L\otimes_{\bQ_p}E$-module $D_{\pdr}(\rho_L)$ of rank $n$ is equipped with a nilpotent endomorphism $N$ and two flags $\cD_{\bullet}$ (comes from the triangulation $\cF$) and $\fil_{\bullet}^H$ (comes from the Hodge-filtration),\;so that we can define a point $x_{\pdr}:=(\cD_{\bullet},\fil_{\bullet}^H,N)\in X_L(E)$ by choosing some basis of $D_{\pdr}(\rho_L)$.\;By using the theory of \cite[Section 3]{breuil2019local},\;there exists a natural morphism (the so-called local model map)
\begin{equation}\label{keymorphismformally}
	\Upsilon:\widehat{X_{\mathrm{tri}}(\overline{r})}_{x_{w_y}}\rightarrow \widehat{X}_{L,x_{\pdr}}\times_{\widehat{\ft_L}}\widehat{T}_{L,\delta_{w_y}}:=\widehat{\cX}_{L,x_{\pdr}}\cong \widehat{X}_{L,x_{\pdr}}\times \widehat{\ft},
\end{equation}
where $\widehat{X}_{L,x_{\pdr}}$ is the completion of $X_L$ at point $x_{\pdr}$ and $\widehat{\ft}$ is the completion of $\ft$ at $0$.\;

Let $\pr_1:\widehat{\cX}_{L,x_{\pdr}}\rightarrow \widehat{X}_{L,x_{\pdr}}$ be the natural formally smooth projection.\;For generic crystalline case studied in \cite{breuil2019local},\;the composition $\pr_1\circ \Upsilon$ (the previous local model map in \cite[Theorem 3.4.4]{breuil2019local}) is formally smooth \cite[Theorem 1.6]{breuil2019local}.\;But in our case,\;it is not true that $\pr_1\circ \Upsilon$ is formally smooth.\;In Section \ref{sectionlocalmodelmainresult},\;we show that:

\begin{pro} We construct a formal scheme $\widehat{\cX}^{\flat}_{L,x_{\pdr}}$ (a variation of $\widehat{\cX}_{L,x_{\pdr}}$ which pro-represents certain groupoid) such that:
\begin{itemize}
	\item[(a)] there is a natural morphism $\iota:\widehat{\cX}^{\flat}_{L,x_{\pdr}}\rightarrow \widehat{\cX}_{L,x_{\pdr}}$;
	\item[(b)] the natural morphism $\Upsilon: \widehat{X_{\mathrm{tri}}(\overline{r})}_{x_{w_y}}\rightarrow \widehat{\cX}_{L,x_{\pdr}}$ factors through $\widehat{\cX}^{\flat}_{L,x_{\pdr}}\rightarrow \widehat{\cX}_{L,x_{\pdr}}$;
	\item[(c)] $\widehat{X_{\mathrm{tri}}(\overline{r})}_{x_{w_y}}$ is formally smooth over $\widehat{\cX}^{\flat}_{L,x_{\pdr}}$.\;
\end{itemize}
%.\;Then $\Upsilon$ factors through $\widehat{\cX}^{\flat}_{L,w,x_{\pdr}}$.\;
\end{pro}
%We mention that $\widehat{\cX}^{\flat}_{L,x_{\pdr}}$ is obtained from $\widehat{X}_{L,x_{\pdr}}$ by ``cutting out some (not necessary algebraic) equations" (this is not a right description,\;only a rough feeling).\;Let $J$ be a subset of $\Sigma_L$.\;

We give the reader a comparison of $\widehat{\cX}^{\flat}_{L,x_{\pdr}}$ and the original formal completion $\widehat{\cX}_{L,x_{\pdr}}$ (equivalently,\;$\widehat{X}_{L,x_{\pdr}}$,\;up to formally smooth morphism).\;By definition,\;the point $x_{\pdr}\in\widehat{X}_{L,x_{\pdr}}$ splits into the product of its $\Sigma_L$-component,\;i.e.\;$x_{\pdr}=(x_{\pdr,\tau})_{\tau\in \Sigma_L}\in \prod_{\tau\in \Sigma_L}\widehat{X}_{\tau,x_{\pdr,\tau}}$.\;For $J\subseteq \Sigma_L$,\;put $x_{\pdr,J}=(x_{\pdr,\tau})_{\tau\in J}$ and   $\widehat{X}_{J,x_{\pdr,J}}:=\prod_{\tau\in J}{\widehat{X}}_{\tau,x_{\pdr,\tau}}$.\;
%We can consider $\widehat{\cX}_{J,x_{\pdr,J}}$ in a similar way.\;

%For any subset $J\subsetneqq \Sigma_L$,\;the following proposition compares the $J$-components of $\widehat{\cX}^{\flat}_{L,x_{\pdr}}$ and $\widehat{X}_{L,x_{\pdr}}$:
\begin{pro}\label{LNEQOpCOMPOAR} (Proposition \ref{Jversionwdfj}) For any $J\subsetneqq\Sigma_L$,\;the composition $\iota_J:\widehat{\cX}^{\flat}_{L,x_{\pdr}}\rightarrow \widehat{\cX}_{L,x_{\pdr}}\twoheadrightarrow \widehat{X}_{J,x_{\pdr,J}}$) is formally smooth.
\end{pro}
\begin{rmk}Our observation in  Proposition \ref{LNEQOpCOMPOAR} shows that the local model is not so complicate.\;In Section \ref{geooflocalmodel},\;we want to  describe $\widehat{\cX}^{\flat}_{L,x_{\pdr}}$
through some formal power series,\;which mix variables coming from different $\Sigma_L$-components of $\widehat{X}_{L,x_{\pdr}}$ and $\widehat{\ft}_{L}$-component of $\widehat{\cX}_{L,x_{\pdr}}$).\;This is why our local models are established though $\widehat{\cX}_{L,x_{\pdr}}$,\;which is different from the
previous generic crystalline case (the local models are established only though $\widehat{X}_{L,x_{\pdr}}$).\;
\end{rmk}
%\begin{rmk}We can use this geometric property to compare certain cycles on the eigenvariety.\;See the argument below (\ref{comparecycles}).\;This proposition is efficient when $L\neq \bQ_p$.\;\end{rmk}

%The natural morphism $\iota:\widehat{\cX}^{\flat}_{L,x_{\pdr}}\rightarrow \widehat{X}_{L,x_{\pdr}}$ may not be a closed immersion.\;

For any $w'\in \sW_{n,L}$ such that $x_{\pdr}\in X_{L,w'}(E)$,\;we put $\widehat{\cX}^{\flat}_{L,w',x_{\pdr}}:=\widehat{\cX}^{\flat}_{L,x_{\pdr}}\times_{\widehat{X}_{L,x_{\pdr}}}\widehat{X}_{L,w',x_{\pdr}}$.\;We finally have:
\begin{thm}(Proposition \ref{propertyofxrhombullet},\;Theorem \ref{localgeomertyonspecial}) Keep the above notation and situation.\;
\begin{itemize}
	\item[(1)] We have $w_y\underline{w}_0\geq w_{\cF}$.
	\item[(2)] There exists a formal scheme $X^{\Box,w\underline{w}_0}_{\rho_L,\cM_{\bullet}}$ over $E$ such that $\big(X^{\Box,w\underline{w}_0}_{\rho_L,\cM_{\bullet}}\big)^{\mathrm{red}}$ (the associated reduced formal scheme) is formally smooth  over $\widehat{\cX}^{\flat}_{L,w\underline{w}_0,x_{\pdr}}$ and formally smooth of dimension $n^2d_L$ over $\widehat{X_{\mathrm{tri}}(\overline{r})}_{x_w}$:
	\begin{equation}
	\xymatrix{
		& \big(X^{\Box,w_y\underline{w}_0}_{\rho_L,\cM_{\bullet}}\big)^{\mathrm{red}}
		\ar[dl] \ar[dr] &    \\
		\widehat{X_{\mathrm{tri}}(\overline{r})}_{x_{w_y}}	&  & \widehat{\cX}^{\flat}_{L,w_y\underline{w}_0,x_{\pdr}} .}
		\end{equation} 
\item[(c)] $\widehat{\cX}^{\flat}_{L,x_{\pdr}}$ is unibranch at point $x_{\pdr}$,\;and $X_{\mathrm{tri}}(\overline{r})$ is irreducible and Cohen-Macaulay at point $x_{w_y}$.\;
\end{itemize}
\end{thm}
Let $X_{\rho_L}$ be the formal scheme associated to the universal deformation ring of $\rho_L$.\;We consider a subscheme $X_{\rho_L,\cM_{\bullet}}:=\Spec R_{\rho_L,\cM_{\bullet}}$ of $X_{\rho_L}$.\;The irreducible components of $X_{\rho_L,\cM_{\bullet}}$ are indexed by $X^w_{\rho_L,\cM_{\bullet}}$ for $w\in\sW_{n,L}$ such that $w\underline{w}_0\geq w_{\cF}$.\;Then $X^{\Box,w}_{\rho_L,\cM_{\bullet}}$ is a ``framed" version of $X^w_{\rho_L,\cM_{\bullet}}$.\;

We have two morphisms  $\kappa_{1}:X_{L}\rightarrow \ft_{L}$ (resp.,\;$\kappa_{2}:X_{L}\rightarrow \ft_{L}$) by sending
\[(g_1\bB,g_2\bB,\psi)\mapsto\overline{\mathrm{Ad}(g_1^{-1})\psi},\;\text{resp.,\;}(g_1\bB,g_2\bB,\psi)\mapsto\overline{\mathrm{Ad}(g_2^{-1})\psi}\]
where  $\overline{\mathrm{Ad}(g_i^{-1})\psi}$ is the image of $\mathrm{Ad}(g_i^{-1})\psi\in \fb_L$ via $\fb_L\twoheadrightarrow\ft_L$.\;Let $\kappa_{w,i}$ be the restriction of $\kappa_{i}$ on $X_{w}$.\;They induce natural morphisms $\kappa_{i}:\widehat{\cX}^{\flat}_{L,\widehat{y}}\rightarrow \widehat\ft_{L}$ (resp.,\;$\kappa_{w,i}:\widehat{\cX}^{\flat}_{L,w,\widehat{y}} \rightarrow  \widehat\ft_{L}$) on our local models,\;where we view $\widehat{\cX}^{\flat}_{L,\widehat{y}}$,\; $\widehat{\cX}^{\flat}_{L,w,\widehat{y}}$ and $\widehat\ft_{L}$ as schemes.\;We suspect that

\begin{conjecture}\label{introflatnessconj}(Conjecture \ref{flatnessXconj})
	The morphism $\kappa_{i}:\widehat{\cX}^{\flat}_{L,\widehat{y}}\rightarrow \widehat\ft_{L}$ (resp.,\;$\kappa_{w,i}:\widehat{\cX}^{\flat}_{L,w,\widehat{y}} \rightarrow  \widehat\ft_{L}$) is  flat.\;
\end{conjecture}
In Proposition \ref{flatnessX},\;we use miracle flatness to shows that:
\begin{pro}\label{introflatness}
	The natural morphism $\kappa_{\underline{w}_0,i}:\widehat{\cX}^{\flat}_{r,\underline{w}_0,\widehat{y}} \rightarrow  \widehat\ft_{L}$ is  flat.\;
\end{pro}

Moreover,\;the projection $\pr_1:\widehat{\cX}_{L,x_{\pdr}}\cong \widehat{\ft}\times\widehat{X}_{L,x_{\pdr}} \rightarrow \widehat{\ft}$ gives projections $\kappa: X_{\rho_L,\cM_{\bullet}}\rightarrow \widehat\ft$ and  $\kappa:\widehat{\cX}^{\flat}_{L,\widehat{y}}\rightarrow \widehat\ft$,\;they are not flat.\;Let $\Spec \overline{R}_{\rho_L,\cM_{\bullet}}$ (resp.,\;$\Spec \overline{\overline{R}}_{\rho_L,\cM_{\bullet}}$) be the fiber of $\Spec {R}_{\rho_L,\cM_{\bullet}}\rightarrow \widehat\ft_{L}$ (resp.,\;$\kappa_i\times\kappa:\Spec {R}_{\rho_L,\cM_{\bullet}}\rightarrow \widehat\ft_{L}\times \widehat\ft$) over $0$ (resp.,\;$(0,0)$).\;In Proposition \ref{localmodelforover2},\;we give the local model of $\Spec \overline{\overline{R}}_{\rho_L,\cM_{\bullet}}$,\;which has a simple description,\;this ring will be used in the formulation of locally analytic ``Breuil-Mezard type" conjecture.\;

\begin{rmk}
The above discussion on local model is proved in more general situation,\;i.e.,\;for certain potentially semistable non-crystalline  Galois representation $\rho_L$ such that $D_{\rig}(\rho_L)$ admits the so-called critical special $\omepik$-filtration.\;
\end{rmk}

%More general,\;in Section \ref{stacky},\;we discuss the ``stacky" local model for the rigid analytic stacks of $(\varphi,\Gamma)$-modules,\;by summarizing the arguments in \cite[Section 5]{emerton2023introduction},\;\cite{wu2024translation} and \cite{MSW}.\;

%\begin{pro}\label{introconjonlci}
%	Keep the above situation.\;The rigid analytic space $X_{\mathrm{tri}}(\overline{r})$ is normal,\;locally complete intersection and Cohen-Macaulay at $x$.\;
%\end{pro}
%\label{conjonlci}Using this local model,\;we can give a bound for the dimension of the tangent space of $X_{\mathrm{tri}}(\overline{r})$ at $x_w$.\;

%\begin{rmk}
%We suspect our method are also suitable for the potentially semistable Galois representation $\rho_L$ which has a more general $\Omega$-filtration (with some mild assumptions on its parameters).\;By the socle conjecture becomes more complicated,\;we decide not to go further here.\;
%We suspect that our methods can be extended directly to any potentially semistable Galois representation $\rho_L$ which admits a general $\Omega$-filtration with arbitrarily parameters (so that critical special $\omepik$-filtration is an extreme case).\;
%\end{rmk}

We now explain the proof of the existences of (local and global) companion points and companion constituents on eigenvariety.\;We first describe the local companion points of $x:=x_1\in X_{\mathrm{tri}}(\overline{r})$.\;We need a hypothesis on local companion points of $x$.\;
\begin{hypothesis}\label{introhypo-1} 
The point $x_{w_{\cF}w_0}$ lies in $X_{\mathrm{tri}}(\overline{r})$.\;
\end{hypothesis}

Denoted by $\FX_{\overline{r}}^{\Box,\bh-\mathrm{st}}\subseteq \FX_{\overline{r}}$  the closed analytic subspace associated to (framed)  semistable deformations of $\overline{r}$ with Hodge-Tate weights $\bh$,\;and let $\FX_{\overline{r},\cP_{\min}}^{\Box,\bh-\mathrm{st}}$ be the locally closed subspace of $\FX_{\overline{r}}^{\Box,\bh-\mathrm{st}}$ consisting potentially semistable deformations with full monodromy rank (see Section \ref{semistableder} for this notation).\;For $\rho'_L\in \FX_{\overline{r},\cP_{\min}}^{\Box,\bh-\mathrm{st}}$,\;there exists a unique $a_{\rho'_L}\in k(\rho_L)$ such that ${a_{\rho'_L}},\cdots,{a_{\rho'_L}q_L^{i-1}},\cdots,{a_{\rho'_L}q_L^{n-1}}$ are $\varphi^{f_L}$-eigenvalues of the $(\varphi,N)$-module $D_{\mathrm{st}}(\rho'_L)$.\;The unique $(\varphi,N)$-stable complete flag on $D_{\mathrm{st}}(\rho'_{L})$ determine a element $w_{\rho'_{L}}\in\sW_{n,L}$  measuring the relation position of this complete flag and Hodge filtration.\;Put $\delta_{\rho'_L,w}=(\unr(a_{\rho'_L})z^{w(\bh_1)},\cdots,\unr({a_{\rho'_L} q_L^{i-1}})z^{w(\bh_i)},\cdots,\unr({a'_{\rho_L} q_L^{n-1}})z^{w(\bh_n)})$. %Consider the following morphisms of rigid spaces over $E$:
%\begin{equation}
%	\begin{aligned}
%		\iota_{\bh,w}:\FX_{\overline{r},\cP_{\min}}^{\Box,\bh-\mathrm{st}}\rightarrow \mathfrak{X}_{\overline{\rho}_{\fp}}^\Box\times\widehat{T}_L,\;
%		\rho'_L\mapsto (\rho'_L,\delta_{\rho'_L,w}).\;
%	\end{aligned}
%\end{equation}
We make the following stronger Hypothesis (which is also predicted by \cite[Conjecture 5.3.13]{emerton2023introduction}).\;Hypothesis \ref{introhypo-1} is contained in Hypothesis \ref{hy114}.\;
\begin{hypothesis}\label{hy114} (Hypothesis \ref{appenhypothesis}) For any $\rho'_L\in \FX_{\overline{r},\cP_{\min}}^{\Box,\bh-\mathrm{st}}$,\;$(\rho'_L,\delta_{\rho'_L,w_{\rho'_{L}}w_0})\in X_{\mathrm{tri}}(\overline{r})$.\;
	%We have$\iota_{\bh}\big(\FX_{\overline{r},\cP_{\min}}^{\Box,\bh-\mathrm{st}}\big)\subseteq X_{\mathrm{tri}}(\overline{r})\subseteq \mathfrak{X}_{\overline{\rho}_{\fp}}^\Box\times\widehat{T}_L$.\;
\end{hypothesis}

\begin{rmk}
The above hypothesis are natural.\;In generic crystalline case (see \cite{breuil2019local}),\;the generic crystalline point $(\rho',\delta)$ such that $\delta$ gives the right parameters of triangulation on $D_{\rig}(\rho')$ lies in $U_{\mathrm{tri}}(\overline{r})\subseteq X_{\mathrm{tri}}(\overline{r})$ automatically.\;But in semistable case,\;the author does not know whether it is true in general.\;It is also predicted by \cite[Conjecture 5.3.13]{emerton2023introduction}.\;See Remark \ref{hyoprecatesexplian} for more precise statements.\;The Hypothesis  \ref{hy114} are true by replacing $X_{\mathrm{tri}}(\overline{r})$ with a larger space $X'_{\mathrm{tri}}(\overline{r})$,\;see \cite{MSW}.\;
\end{rmk}

Under Hypothesis \ref{hy114},\;we prove the following result on the local companion points of $x$ and $x_{w_{\cF}w_0}$.\;Recall that the local companion points of $x$ are those $x'=(\rho_L,{\delta'})\in X_{\mathrm{tri}}(\overline{r})$ such that ${\delta'}{\delta}^{-1}$ is a $\bQ_p$-algebraic character.\;
\begin{thm}\label{localcompanion}(Proposition \ref{prolocalcomapn1}) Assume  Hypothesis \ref{hy114},\;then $x_{w}\in X_{\mathrm{tri}}(\overline{r})$ if and only if $w\underline{w}_0\geq w_{\cF}$.\;
\end{thm}
This theorem is proved by some Zariski-closure argument on semistable deformation rings and a study of the relation between semistable deformation spaces and trianguline variety (they are considered firstly for the crystalline case in  \cite[Section 2.2]{breuil2017smoothness},\;see Proposition \ref{semistableZarisikeargsecond} and Proposition \ref{pminloucsofsemidefringscP} for more detail).\;They are also used in the proof of Theorem \ref{globalcomppoitnproof}.\;

In $\GLN_2(L)$ case,\;the main theorem in Appendix \ref{appGL2(L)} (i.e.,\;if $\rho_L$ comes from some global setup) also implies the following result (without Hypothesis \ref{introhypo-1} and Hypothesis \ref{hy114}).\;
\begin{thm}\label{socleGLN2}(Restrict the above situation and notation to $\GLN_2(L)$ case) If $y=(\rho,\underline{\delta})\in Y(U^{\fp},\overline{\rho})$ is  strictly dominant and of Steinberg type.\;Suppose $\rho_{y,\widetilde{v}}$ is generic for $v\in \Sigma(U^p)\backslash S_p$,\;where $\Sigma(U^p)$ consists of some ``bad'' places.\;Then all companion points of $y$ are $y_{w'}$ for $w'\underline{w}_0\geq w_{\cF}$.\;In particular,\;$x_{w}\in X_{\mathrm{tri}}(\overline{r})$ if and only if $w\underline{w}_0\geq w_{\cF}$.\;Moreover,\;$X_{\mathrm{tri}}(\overline{r})$ is smooth at each point $x_{w}$.\;
\end{thm}

We now move to global setup.\;Under the ``Talyor-Wiles hypothesis'',\;we get a continuous Banach representation $\Pi_\infty$ of $G$,\;which is equipped with a continuous action of certain patched Galois deformation ring $R_\infty$ commuting with the  $G$-action.\;See Section \ref{preforpatching} for a brief summary.\;The proofs of  Theorem  \ref{introcompanglobal} and Theorem \ref{introlocallysocle} are related to the existence of companion constituents $\bI_{w\underline{w}_0}$.\;We prove the following result.\;

\begin{thm}\label{globalcomppoitnproof} (Proposition \ref{equicompanpointconsti} and Theorem \ref{globalcompanion})
	Assume Hypothesis \ref{TWhypo} and Hypothesis \ref{hy114}.\;Then  $\homo_G\Big(\bI_{w\underline{w}_0},\widehat{S}_{\xi,\tau}(U^{\fp},E)^{\ana}_{\overline{\rho}}[\fm_{r_y}]\Big)\neq 0$ if and only if $w\geq w_{\cF}$.\;In particular,\;$y_{w}\in X_{\fp}(\overline{\rho})$ if and only if $w\underline{w}_0\geq w_{\cF}$.\;
\end{thm}
Now Theorem \ref{introlocallysocle} is a direct consequence of Theorem \ref{globalcomppoitnproof}.\;Indeed,\;for any irreducible component $W\ncong C(w\underline{w}_0)$ of $\bI_{w\underline{w}_0}$,\;we have $\homo_{G}\Big(W,\widehat{S}_{\xi,\tau}(U^{\fp},E)^{\ana}_{\overline{\rho}}[\fm_{r_y}]\Big)=0$.\;We sketch the proof of Theorem \ref{globalcomppoitnproof} by using the coherent sheaves obtained from patching functors (resp.,\;constructed from local models).\;

The first cycles are obtained by patching functor.\;Firstly,\;we use \cite[Section 5]{breuil2019local} and \cite[Section 6.1]{HHS} to construct sheaves $\cM_{\infty}(M)$ over ${{\FX}}_{\infty,y}=\Spec R_{\rho_L,\cM_{\bullet}}[[x_1,\cdots,x_g]]$ for some integer $g$ given in the patching argument,\;for (parabolic) verma module (and its dual) $M=L(w\underline{w}_0\cdot \lambda_{\bh}),M_I(w\underline{w}_0\cdot \lambda_{\bh})^{\ast}$ and the $J$-deformed Verma modules $\widetilde{M}_I^J(w\underline{w}_0\cdot \lambda_{\bh})^{\ast}$ (and its subquotient) for $J\subseteq \Sigma_L$ and $\ast\{\emptyset,\vee\}$ (discussed in \cite[Section 2.3]{HHS} and Section \ref{deformedverma}).\;For the companion constituent $\bI_{w\underline{w}_0}$,\;we have the cycle $[\cM_{\infty}(L(w\underline{w}_0\cdot \lambda_{\bh}))]\neq 0$  if and only if $\homo_G\Big(\bI_{w\underline{w}_0},\Pi_\infty^{R_\infty-\ana}[\fm_{\rho}^\infty]\Big)\neq 0$ (equivalently,\;$\homo_G\Big(\bI_{w\underline{w}_0},\widehat{S}_{\xi,\tau}(U^{\fp},E)^{\ana}_{\overline{\rho}}[\fm_{r_y}]\Big)\neq 0$),\;$[\cM_{\infty}(M(w\underline{w}_0\cdot \lambda_{\bh}))]\neq 0\Leftrightarrow[\cM_{\infty}(\widetilde{M}_I^J(w\underline{w}_0\cdot \lambda_{\bh}))]\neq 0$ if and only if $y_{w}\in X_{\fp}(\overline{\rho})$.\;

%Firstly, we can replace  $\widehat{S}_{\xi,\tau}(U^{\fp},E)^{\ana}_{\overline{\rho}}[\fm_{{\rho}}]$ in Theorem \ref{introlocallysocle} by $\Pi_\infty^{R_\infty-\ana}[\fm_{\rho}^\infty]$ equivalently.\;

%or equivalently,\;
%\[\homo_{G}\Big(\bI_{w\underline{w}_0},\Pi_\infty^{R_\infty-\ana}[\fm_{\rho}^\infty]\Big)\neq 0,\]
%where
%$\Pi_\infty^{R_\infty-\ana}$ denotes the locally $R_\infty$-analytic vectors in $\Pi_\infty$(see \cite[Section 3.1]{breuil2017interpretation}).\;

The second cycles are obtained from local models.\;Our new ingredients are combining two types partially de-Rham Galois cycles with respect to the parabolic group (see \cite[Section 3.6]{wu2021local}) (resp.,\;the $\Sigma_L$-components (see \cite[Section 5.2]{CompanionpointforGLN2L}),\;see Section \ref{galoiscycles} for more details.\;For $J\subseteq \Sigma_L$,\;we let $Z^J_{L}$ be the fiber of $X_{L}$ at $0$ via the natural projection $X_{L}\rightarrow\ft_{L}\rightarrow \ft_J$ ($\ft_J$ is the $J$-component of $\ft_{L}$).\;One can show that $Z_{L}^J$ is equidimensional with reduced irreducible components given by $\{Z_{w'}:=X_{L,w'}\cap Z^J_{L}\}_{w'\in \sW_{n,L}}$.\;We have $x_{\pdr}\in Z^J_{L}$.\;Put $\widehat{Z}^{J,\flat}_{w',x_{\pdr}}:=\widehat{Z}^{J}_{w',x_{\pdr}}\times_{\widehat{X}_{L,w',x_{\pdr}}}\widehat{\cX}^{\flat}_{L,w',x_{\pdr}}$ if $x_{\pdr}\in Z_{w'}$.\;Pulling back each $\widehat{Z}^{J,\flat}_{w',x_{\pdr}}$  via the natural morphism $X_{\fp}(\overline{\rho})\rightarrow X_{\mathrm{tri}}(\overline{r})$ (see (\ref{injpatchtotri1}))  defines a cycle $\FZ^{J,\flat}_{w'}$ on $X_{\fp}(\overline{\rho})$.\;If $J\neq \Sigma_L$,\;$\FZ^{J,\flat}_{w'}$ is irreducible by Proposition \ref{LNEQOpCOMPOAR}.\;If $J=\Sigma_L$,\;we rewrite $\FZ^{\flat}_{w'}:=\FZ^{J,\flat}_{w'}$,\;which is equi-dimensional (not necessary irreducible in general) if Conjecture \ref{introflatnessconj} holds.\;

We use the strategy in the proof of \cite[Proposition 4.7,\;Proposition 4.9,\;Theorem 4.10,\;Theorem 4.12]{wu2021local} to compare the cycles $[\cL(w\underline{w}_0\cdot \lambda_{\bh})]$ with cycles $\FZ^{\flat}_{w'}$ and then prove the main results.\;The new ingredients in his proof (compare with \cite[Theorem 5.3.3]{breuil2019local}) are results relating the partially de Rham properties of Galois representations (the de Rhamness of graded pieces along the paraboline filtrations of the associated $(\varphi,\Gamma)$-modules over $\cR_{E,L}$) and the relevant properties of cycles on the generalized Steinberg varieties.\;

By descending induction and similar discussion (more precisely,\;some Zariski-closure argument on semistable deformation rings)  in the proof of Theorem \ref{localcompanion} (an easy modification of \cite[Proposition 4.9,\;Theorem 4.10]{wu2021local} or the Step $8$ and Step $9$ in the proof of \cite[Theorem 5.3.3]{breuil2019local}),\;one is reduced to showing the following statement (see Proposition \ref{prolocalcomapn}),\;which is the key step in the proof of Theorem \ref{globalcomppoitnproof}.\;
\begin{itemize}
	\item if $y_{w\underline{w}_0}\in X_{\fp}(\overline{\rho})$ for all $w> w_{\cF}$,\;then $y_{w_{\cF}\underline{w}_0}\in X_{\fp}(\overline{\rho})$.\;
\end{itemize}
%It is proved by matching the cycles $[\cL(w\underline{w}_0\cdot \lambda_{\bh})]$ and $\FZ^{\flat}_{w\underline{w}_0}$ near $y_{w_{\cF}\underline{w}_0}$ on the eigenvariety.\;
We sketch the proof roughly by using our partial cycle $\FZ^{J,\flat}_{w'}$ and the strategy in the proof of \cite[Proposition 4.7]{wu2021local}.\;We have $\FZ^{\flat}_{w\underline{w}_0}\neq 0$ (equivalently,\;$\FZ^{J,\flat}_{w\underline{w}_0}\neq 0$) for all $w> w_{\cF}$.\;There exists a simple reflection $s_{\alpha,\tau}$ of $\bG_{L}$ and a parabolic subgroup $\bB_{\tau}\subset \bP_{\tau}$ of $\bG$ such that $s_{\alpha,\tau} w_{\cF}\underline{w}_0(\bh)$ is strictly $\bP_{\tau}$-dominant and $w_{\cF}\underline{w}_0(\bh)$ is \textit{not} strictly $\bP_{\tau}$-dominant.\;The assumption implies that $y_{s_{\alpha,\tau}w_{\cF}\underline{w}_0}\in X_{\fp}(\overline{\rho})$.\;Put
\begin{equation}
	\begin{aligned}
	&[\cM_{\infty}(s_{\alpha,\tau}w_{\cF}\underline{w}_0,\Sigma_L\backslash\tau)]:=	\Big[\cM_{\infty}(\widetilde{M}^{\Sigma_L\backslash\tau}(w_{\cF,\Sigma_L\backslash\tau}\underline{w}_{0,\Sigma_L\backslash\tau}\cdot \lambda_{\bh})\otimes_EL(s_{\alpha,\tau}w_{\cF,\tau}{w}_0\cdot \lambda_{\bh}))\Big]\\
	&[\cM_{\infty}(w_{\cF}\underline{w}_0,\Sigma_L\backslash\tau)]:= \Big[\cM_{\infty}(\widetilde{M}^{\Sigma_L\backslash\tau}(w_{\cF,\Sigma_L\backslash\tau}\underline{w}_{0,\Sigma_L\backslash\tau}\cdot \lambda_{\bh})\otimes_EL(w_{\cF,\tau}{w}_0\cdot \lambda_{\bh}))\Big].
	\end{aligned}
\end{equation}
Then we have an equality of the underlying closed subspaces of cycles:
\begin{equation}\label{comparecycles}
	\begin{aligned}
		[\cM_{\infty}(s_{\alpha,\tau}w_{\cF}\underline{w}_0,\Sigma_L\backslash\tau)]\cup [\cM_{\infty}(w_{\cF}\underline{w}_0,\Sigma_L\backslash\tau)]=\FZ^{\{\tau\},\flat}_{s_{\alpha,\tau} w_{\cF}\underline{w}_0}\cup \FZ^{\{\tau\},\flat}_{w_{\cF}\underline{w}_0}.
	\end{aligned}
	\end{equation}
Then we show that $[\cM_{\infty}(s_{\alpha,\tau}w_{\cF}\underline{w}_0,\Sigma_L\backslash\tau)]$ is $\bP_{\tau}$-partially de Rham (in the terminology of \cite[Section 3.6]{wu2021local}),\;while the cycle $\FZ^{\Sigma_L\backslash\tau\flat}_{w_{\cF}\underline{w}_0}$ is not fully contained in the $\bP_{\tau}$-partially de Rham locus.\;Hence $\FZ^{\Sigma_L\backslash\tau,\flat}_{w_{\cF}\underline{w}_0}\not\subseteq [\cM_{\infty}(s_{\alpha,\tau}w_{\cF}\underline{w}_0,\Sigma_L\backslash\tau)]$ and then $[\cM_{\infty}(w_{\cF}\underline{w}_0,\Sigma_L\backslash\tau)]\neq 0$.\;This shows that $y_{w_{\cF}\underline{w}_0}\in X_{\fp}(\overline{\rho})$.\;

We end the introduction with a discussion on coherent sheaves $\cM_{\infty}(M)$ over ${{\FX}}_{\infty,y}$ (we follow the route  of the recently work of Hellmann-Hernandez-Schraen \cite[Section 7.2]{HHS}).\;As in \cite[Section 7.2]{HHS},\;we also consider the ``local" functor  $\cB_y$
associates,\;to an object $M$ in BGG category $\cO_{\chi_{\lambda_{\bh}}}$,\;a coherent sheaf $\cB_y(M)$ on ${{\FX}}_{\infty,y}$.\;The functor $\cB_y$ is induced by the Bezrukavnikov's functor.\;The  functor $\cB_y$ is not exact in our case,\;but the derived terms can be described through $\iota^{\flat}:\widehat{\cX}^{\flat}_{L,\widehat{y}}\rightarrow \widehat{X}_{L,\widehat{y}}$ in diagram (\ref{Bezrukavnikov1}).\;We discuss the schematic supports (together with the multiplicities) of $\cM_{\infty,y}(\widetilde{M}^{\Sigma_L\backslash J}_{{I}}(w\cdot {\bm\lambda}_\bh))$ and $\cM_{\infty,y}(\widetilde{M}^{\Sigma_L\backslash J}_{{I}}(w\cdot {\bm\lambda}_\bh)^{\vee})$ when $J\subsetneq \Sigma_L$ in Proposition \ref{schemesupport}-Proposition \ref{caseforw0coherentsheaf} and Proposition \ref{determinemulti}.\;The coherent sheaf $\cB_y(M)$ should be viewed as the local analogue of $\cM_{\infty}(M)$ (predicted by the categorical $p$-adic Langlands correspondence).\;The  relationship between $\cB_y(M)$ and $\cM_{\infty}(M)$ are given in Section \ref{Analysissheafcycles} and Proposition \ref{determinemulti},\;see Remark \ref{stackynotforfixsmooth} and Remark \ref{stackyforfixsmooth} for general argument,\;under the picture of categorical $p$-adic Langlands program.\;In the future,\;we want to give an explicit computation on $\GLN_3$-case,\;which may give more evidence for the categorical $p$-adic Langlands program.\;

%Note that the non-critical points in $\FX_{\overline{\rho}_{\fp},\cP_{\min}}^{\Box,\bh-\mathrm{st}}$ form a Zariski-open and Zariski-dense set,\;so it is enough to check it on this subset (so this is different to the crystalline case since $\iota_{\bh}(\widetilde{V}_{\overline{\rho}_{\fp}}^{\Box,\bh-\mathrm{cr}})$ lies in $U_{\mathrm{tri}}(\overline{\rho}_{\fp})$ automatically).\;See Appendix \ref{Appendixsemiastable} for more discussions.\;{\color{red}{(......)}}

%\begin{rmk}We end the introduction with a remark on the local-global compatibility results.\;The next goal is to explore the $p$-adic local-global compatibility results in critical semistable non-crystalline case (in particular,\;Steinberg case).\;%More precisely,\;if $\rho_L$ admits a critical $\omepik$-filtration,\;then we can attach to $\rho_L$ the partial parabolic Fontaine-Mazur simple $\sL$-invariants $\sL(\rho_L)$ (modifying the method in \cite{He20222}).\;If moreover $\rho_L$ comes from a patched automorphic representation of  $\bG_U(\BA_{F^+})$,\;then there exists an explicit locally $\bQ_p$-analytic representation $\Pi(\rho_L)$,\;which determines $\rho_L$,\;can be embedded into the associated Hecke-isotypic subspace of the Banach spaces of (patched) $p$-adic automophic forms on $\bG_U(\BA_{F^+})$.\;Such results were first given by \cite{2015Ding} on $\GLN_2(L)$-Steinberg case.\;We decide not to go further here.\;
%\end{rmk}

\section*{Acknowledgment}
The author thank Yiwen Ding,\;Zicheng Qian and Zhixiang Wu for discussions or answers to questions.\;I very much like to thank Yiwen Ding and Zhixiang Wu for pointing out some mistakes in local models and cycles in an earlier versions of the paper.\;

\section{Preliminaries}

\subsection{General notation}

\label{sec: lgln-LB-not}
\noindent Let $L$ (resp. $E$) be a finite extension of $\bQ_p$ with $\co_L$ (resp. $\co_E$) as its ring of integers and $\varpi_L$ (resp. $\varpi_E$) a uniformizer. Suppose $E$ is sufficiently large containing all the embeddings of $L$ in $\overline{\BQ}_p$. Put
\begin{equation*}
	\Sigma_L:=\{\sigma: L\hookrightarrow \overline{\BQ}_p\} =\{\sigma: L\hookrightarrow E\}.
\end{equation*}
Let $\val_L(\cdot)$ (resp. $\val_p$) be the $p$-adic valuation on $\overline{\bQ_p}$ normalized by sending uniformizers of $\co_L$ (resp.,\;$\BZ_p$) to $1$. Let $d_L:=[L:\bQ_p]=|\Sigma_L|$ and $q_L:=p^{f_L}=|\co_L/\varpi_L|$,\;where $f_L$ denotes the unramified degree of $L$ over $\bQ_p$.\;For a group $A$ and $a\in A$,\;denote by $\unr(\alpha)$ the unramified character of $L^\times$ sending uniformizers to $\alpha$.\;If $\bk:=(\bk_{\tau})_{\tau\in \Sigma_L}\in \BZ^{\Sigma_L}$,\;we denote $z^{\bk}:=\prod_{\tau\in  \Sigma_L}\tau(z)^{\bk_{\tau}}$.\;For a character of $\cO_L^{\times}$,\;denoted by $\chi_{\varpi_L}$ the character of $L^{\times}$ such that $\chi_{\varpi_L}|_{\cO_L^{\times}}=\chi$ and $\chi_{\varpi_L}(\varpi_L)=1$.\;

%Recall that the Lubin-Tate character $\chi_{\mathrm{LT}}:\gal_L\rightarrow \cO_L^\times\hookrightarrow E^\times$ makes that $\chi_{\mathrm{LT}}\circ\rec:L^\times\rightarrow \cO_L^\times$ satisfies $\chi_{\mathrm{LT}}\circ\rec(\varpi_{L})=1$ and $\chi_{\mathrm{LT}}\circ\rec|_{\cO_L^\times}=\mathrm{1}_{\cO_L^\times}$.\;

%We write $\gal_{L'}=\gal(\overline{\bQ}_p/L')$,\;$\Gamma_{L'}=\gal(L'_{\infty}/L')$,\;and $\cH_{L'}=\gal(\overline{\bQ}_p/L'_{\infty})$ for any subfield $L'\subset \overline{\bQ}_p$.\;We omit the subscript ${L'}$ if $L'=L$.\;

Let $\cR_{L}:=\bB_{\rig,L}^{\dagger}$ be the Robba ring.\;Let $A$ (resp.,\;$X$) be an $\bQ_p$-affinoid algebra (resp.\;a rigid analytic space), and let $\cR_{A,L}:=\cR_{L}\widehat{\otimes}_{\bQ_p} A$ (resp.,\;$\cR_{X,L}$) for the Robba ring associated to $L$ with $A$-coefficient (resp.,\;with $\cO_X$-coefficient).\;We write $\cR_{A,L}(\delta_A)$ for the $(\varphi,\Gamma)$-module of character type over $\cR_{A,L}$ associated to a continuous character $\delta_A:L^\times\rightarrow A^\times$.\;If $D$ is a $(\varphi,\Gamma)$-module over $\cR_{A,L}$,\;we denote $D(\delta_A):=D\otimes_{\cR_{A,L}} \cR_{A,L}(\delta_A)$ for simplicity.\;

Let $X$ be a scheme locally of finite type over $E$,\;or a locally noetherian formal scheme over $\cO_E$ whose reduction is locally of finite type over $k_E$.\;Let $X^{\mathrm{rig}}$ the associated rigid analytic space over $E$.\;If $x\in X$,\;denote by $\cO_{X,x}$ (resp.,\;$k(x)$) the  local ring (resp.,\;residue field) at $x$.\;Let $\widehat{\cO}_{X,x}$ be the $\fm_{\cO_{X,x}}$-adic completion of $\cO_{X,x}$,\;and $\widehat{X}_{x}:=\mathrm{Spf}\;\widehat{\cO}_{X,x}$.\;If $x$ is a closed point of $X$,\;then  $\widehat{\cO}_{X,x}$ is a noetherian complete local $k(x)$-algebra of residue field $k(x)$.\;
%zccLet $X$ be a scheme locally of finite type over $E$,\;or a rigid analytic space over $E$,\;we denoted by $X^{\mathrm{red}}$ the associated reduced Zariski-closed subspace.\;

Let $\GL_n$ be the general linear group over $L$.\;Let $\Delta_n$ be the set of simple roots of $\GL_n$ (with respect to the Borel subgroup $\bB$ of upper triangular matrices), and we identify the set $\Delta_n$ with $\{1,\cdots, n-1\}$ such that $i\in \{1,\cdots, n-1\}$ corresponds to the simple root $\alpha_{i}:\;(x_1,\cdots, x_n)\in \ft \mapsto x_{i}-x_{i+1}$, where $\ft$ denotes the $L$-Lie algebra of the torus $\bT$ of diagonal matrices.\;Let $\mathbf{P}_{I}$ be the parabolic subgroup of
$\GL_n$ containing $\bB$ such that $\Delta_n \backslash  I$ are precisely the simple roots of the unipotent radical $\mathbf{N}_{I}$ of $\mathbf{P}_{I}$.\;Denote by $\mathbf{L}_{I}$ the unique Levi subgroup of $\mathbf{P}_{I}$ containing $\bT$,\;and $I$ is equal to the set of simple roots of $\mathbf{L}_{I}$.\;In particular,\;we have $\mathbf{P}_{\Delta_n}=\GL_n$, $\mathbf{P}_{\emptyset}=\mathbf{B}$.\;Let $\overline{\mathbf{P}}_{I}$ be the parabolic subgroup opposite to $\mathbf{P}_{I}$.\;Let $\mathbf{N}_{I}$ (resp.\;$\overline{\mathbf{N}}_{I}$) be the nilpotent radical of $\mathbf{P}_{I}$ (resp.\;$\overline{\mathbf{P}}_{I}$).\;We have Levi decompositions $\mathbf{P}_{I}=\mathbf{L}_{I}\mathbf{N}_{I}$ (resp.\;$\overline{\mathbf{P}}_{I}=\mathbf{L}_{I}\overline{\mathbf{N}}_{I}$).\;Let $\bZ_n$ (resp.,\;$\bZ_I$) be the center of $\GL_n$ (resp.,\;$\mathbf{L}_{I}$).\;Let $\fg$, $\fp_{I}$, $\fl_{I}$,\;$\fn_I$,\;$\fz_I$ and $\ft$ be the $L$-Lie algebras of $\GL_n$,\;$\mathbf{P}_{I}$, $\mathbf{L}_{I}$,\;$\bN_I$,\;$\bZ_I$ and $\bT$ respectively.\;Put $\bU:=\mathbf{N}_{\emptyset}$ (resp.\;$\overline{\bN}:=\overline{\mathbf{N}}_{\emptyset}$) and $\fu:=\fn_{\emptyset}$.\;We put $G:=\GL_n(L)$.\;

For a Lie algebra $\fg$ over $L$,\;and $\sigma\in \Sigma_L$,\;let $\fg_{\sigma}:=\fg\otimes_{L,\sigma} E$ (which is  a Lie algebra over $E$).\;For $J\subseteq \Sigma_L$,\;let $\fg_{r,J}:=\prod_{\sigma\in J} \fg_{\sigma}$.\;In particular,\;we have $\fg_{\Sigma_L}\cong \fg\otimes_{\bQ_p} E$.\;For any algebraic group $\mathbf{H}$ over $L$,\;let $\mathrm{Res}_{L/\bQ_p}\mathbf{H}$ be the scalar restriction of $\mathbf{H}$ from $L$ to $\bQ_p$.\;We write $\mathbf{H}_{L}=(\mathrm{Res}_{L/\bQ_p}\mathbf{H})\times_{\bQ_p}E=\prod_{\sigma\in \Sigma_L}\bH_\tau$ that is isomorphic to $\bH^{d_L}$ as a algebraic group over $E$.\;

Let $m\in\BZ_{\geq 1}$,\;and $\pi$ be an irreducible smooth admissible representation of $\GLN_m(L)$,\;let $\rec(\pi)$ be the $m$-dimensional absolutely irreducible $F$-semi-simple Weil-Deligne representation of the Weil group $W_L$ via the normalized classical local Langlands correspondence (normalized in \cite{scholze2013local}).\;We normalize the reciprocity isomorphism $\rec:L^\times\rightarrow W_L^{\mathrm{ab}}$ of local class theory such that the uniformizer $\varpi_{L}$ is mapped to a geometric Frobenius morphism,\;where $W_L^{\mathrm{ab}}$ is the abelization of the Weil group $W_L\subset \gal_L$.\;

Let ${\ccyc}:\gal_L\rightarrow \bZ_p^\times$ be the $p$-adic cyclotomic character (i.e.,\;the character defined by the formula $g(\epsilon_n)=\epsilon_n^{{\ccyc}(g)}$ for any $n\geq 1$ and $g\in \gal_L$).\;Then we have ${\ccyc}\circ\rec=\unr(q_L^{-1})\prod_{\tau\in \Sigma_L}\tau:L^\times\rightarrow E^\times$ by local class theory.\;We define the Hodge-Tate weights of a de Rham representation as the opposite of the gaps of the filtration on the covariant de Rham functor,\;so that the Hodge-Tate weights of ${\ccyc}$ is $1$.\;A character $\delta:L^\times \rightarrow E^\times$ is called \emph{special} if $\delta:=\unr(q_L^{-1})z^{\bk}=\ccyc z^{\bk-1}$ for some $\bk:=(\bk_{\tau})_{\tau\in \Sigma_L}\in \BZ^{\Sigma_L}$ and $\bk-1:=(\bk_{\tau}-1)_{\tau\in \Sigma_L}\in \BZ^{\Sigma_L}$,\;where $z^{\bk}$ is the character $z\mapsto \prod_{\tau\in  \Sigma_L}\tau(z)^{\bk_{\tau}}\in E$.\;

Let $A$ be an affinoid $E$-algebra.\;A locally $\bQ$-analytic character $\delta:L^\times\rightarrow A^\times$ induces a $\bQ_p$-linear map $L\rightarrow A$,\;$x\mapsto \frac{d}{dt}\delta(\exp(tx))|_{t=0}$ and hence it induces an $E$-linear map $L\otimes E=\prod_{\tau\in \Sigma_L}E\rightarrow A$.\;There exist $\wt(\delta):=(\wt_{\tau}(\delta))_{\tau\in \Sigma_L}$ such that the latter map is given by $(a_{\tau})_{\tau\in \Sigma_L}\mapsto\sum_{\tau\in \Sigma_L}a_{\tau}\wt_{\tau}(\delta)$.\;We call $\wt(\delta)$ the weight of $\delta$.\;

Let $\ul{\lambda}:=(\lambda_{1,\sigma}, \cdots, \lambda_{n,\sigma})_{\sigma\in \Sigma_L}$ be a weight of $\ft_{\Sigma_L}$.\;For $\underline{I}=\prod_{\sigma\in \Sigma_L}I_{\tau}\subseteq \Delta_n^{\Sigma_L}$,\;let $\bP_{\underline{I}}:=\prod_{\sigma\in \Sigma_L}\bP_{{I}_\tau}$ be the parabolic subgroup of $\GLN_{n,L}$ associated to the $\underline{I}$.\;We call that $\ul{\lambda}$ is $\underline{I}$-dominant (resp.,\;strictly $\underline{I}$-dominant)  if $\lambda_{i,\sigma}\geq \lambda_{i+1,\sigma}$ (resp.,\;$\lambda_{i,\sigma}> \lambda_{i+1,\sigma}$)   for all $i\in I_{\tau}$ and $\sigma\in \Sigma_L$.\;In particular,\;if $\underline{I}_L:=\prod_{\sigma\in \Sigma_L}I$ for the same $I\subseteq \Delta_{n}$,\;we denote by $X_{I}^+$ the set of $\underline{I}_L$-dominant integral weights  of $\ft_{\Sigma_L}$ .\;For $\ul{\lambda}\in X_{I}^+$, there exists a unique irreducible algebraic representation, denoted by $L(\ul{\lambda})_{I}$, of $(\bL_{I})_{L}$ with highest weight $\ul{\lambda}$ with respect to $(\bL_{I})_{L}\cap \bB_{L}$,\;so that $\overline{L}(-\ul{\lambda})_{I}:=(L(\ul{\lambda})_{I})^\vee$ is the irreducible algebraic representation of $(\bL_{I})_{L}$ with highest weight $-\ul{\lambda}$ with respect to $(\bL_{I})_{L}\cap \ob_{L}$.\;Denote $\chi_{\ul{\lambda}}:=L(\ul{\lambda})_{\emptyset}$.\;If $\ul{\lambda}\in X_{\Delta_n}^+$, let $L(\ul{\lambda}):=L(\ul{\lambda})_{\Delta_n}$.\;A $\bQ_p$-algebraic representation of $G$ over $E$ is the induced action of $G\subset \GL_{n,L}(E)$ on an algebraic representation of $\GL_{n,L}$.\;By abuse of notation we will use the same notation to denote $\bQ_p$-algebraic representations induced from an algebraic representation of $\GL_{n,L}$.\;Let $\ul{\lambda}\in X_I^+$ be an integral weight,\;denote by $M_I(\ul{\lambda}):=\text{U}(\fg_{\Sigma_L})\otimes_{\text{U}(\fp_{I,L})} L(\ul{\lambda})_{I}$ (resp.\;$\overline{M}_I(\ul{\lambda}):=\text{U}(\fg_{\Sigma_L})\otimes_{\text{U}(\overline{\fp}_{I,L})} L(\ul{\lambda})_{I}$),\;the corresponding Verma module with respect to $\fp_{I,L}$ (resp.\;$\overline{\fb}_{\Sigma_L}$).\;Let $L(\ul{\lambda})$ (resp. $\overline{L}(\ul{\lambda})$) be the unique simple quotient of $M(\ul{\lambda}):=M_{\emptyset}(\ul{\lambda})$ (resp. of $\overline{M}_{\emptyset}(\ul{\lambda})$).\;

%Actually, when $\ul{\lambda}\in X_{\Delta_n}^+$ (i.e.\;$-\ul{\lambda}\in X_{\Delta_n}^-$),\;$L(\ul{\lambda})$ is finite dimensional and isomorphic to the algebraic representation $L(\ul{\lambda})$ introduced above (hence there is no conflict of notation).\;We have $\overline{L}(-\ul{\lambda})\cong L(\ul{\lambda})^{\vee}$.\;

%In general,\;for any subset $I$ of $\Delta_n$,\;and $\ul{\lambda}\in X_{I}^+$,\;we define the generalized parabolic Verma module $M_{I}(\ul{\lambda}):=\text{U}(\fg_{\Sigma_L})\otimes_{\text{U}(\fp_{I,L})} L(\ul{\lambda})_{I}$ (resp. $\overline{M}_{I}(-\ul{\lambda}):=\text{U}(\fg_{\Sigma_L})\otimes_{\text{U}(\overline{\fp}_{I,L})} \overline{L}(-\ul{\lambda})_{I})$ with respect to $\fp_{I,L}$ (resp.\;$\overline{\fp}_{I,L}$).
%see \cite[Chapter.\;9]{humphreysBGG} for precise definition.\;

Denote by $\sW_n$ ($\cong S_n$) the Weyl group of $\GL_n$, and denote by $s_{i}$ the simple reflection corresponding to $i\in \Delta_n$.\;For any $I\subset \Delta_n$,\;define $\sW_{I}$ to be the subgroup of $\sW_{n}$ generated by simple reflections $s_{i}$ with $i\in I$ (so that $\sW_I$ is the Weyl group of $\bL_I$).\;For $w\in \sW_n$,\;we  identity $w$ with the permutation matrix corresponding to it.\;Let $I,J$ be subsets of $\Delta_n$,\;recall that $\sW_{I}\backslash \sW_n/\sW_{J}$ has a canonical set of representatives,\;which we will denote by $\sW^{I,J}_n$ (resp.,\;$\sW^{I,J}_{n,\max}$),\;given by taking in each double coset the elements of minimal (resp.,\;maximal) length.\;The Weyl group of $\GL_{n,L}$ is $\sW_{n,L}:=\Pi_{\sigma\in \Sigma_L}\sW_{n,\sigma}\cong S_n^{d_L}$,\;where $\sW_{n,\sigma}\cong \sW_n$ be the $\sigma$-th factor of $\sW_{n,L}$.\;Let $I,J$ be subsets of $\Delta_n$,\;and let  $\sW^{I,J}_{n}$ (resp.,\;$\sW^{I,J}_{n,\max}$)  be the set of minimal (resp.,\;maximal) length representatives in $\sW_{n}$  in  $\sW_{I}\backslash \sW_{n}/\sW_{J}$.\;For $w\in \sW_{n,L}$,\;denoted by $w^{\min}\in\sW^{I,J}_{n}$ (resp,\;$w^{\max}\in\sW^{I,J}_{n,\max}$) the corresponding representative of $\sW_{n}$.\;More general,\;for subset $S\subseteq \Sigma_L$ and  $\underline{I}=\prod_{\sigma\in S}I_{\tau}\subseteq \Delta_n^{S}$ and $\underline{J}=\prod_{\sigma\in S}J_{\tau}\subseteq \Delta_n^{S}$,\;we put $\sW_{\underline{I},S}:=\prod_{\sigma\in S} \sW_{I_\sigma}$,\;$\sW^{\underline{I},\underline{J}}_{n,S}:=\prod_{\sigma\in S}\sW^{I_{\sigma},J_{\sigma}}_{n}$  and $\sW^{\underline{I},\underline{J}}_{n,S,\max}:=\prod_{\sigma\in S}\sW^{I_{\sigma},J_{\sigma}}_{n,\max}$.\;In particular,\;if $\underline{I}=\underline{I}_S$ and $\underline{J}=\underline{J}_S$ for $I,J\subseteq\Delta_n$,\;we write $\sW_{{I},S}:=\sW_{\underline{I}_S,S}$ $\sW^{{I},{J}}_{n,S}:=\sW^{\underline{I}_S,\underline{J}_S}_{n,S}$ and  $\sW^{{I},{J}}_{n,S,\max}:=\sW^{\underline{I}_S,\underline{J}_S}_{n,S,\max}$ for simplicity.\;If $S=\Sigma_L$,\;we replace the subscript $S$ by $L$.\;Let $w_0$ (resp.,\;$\underline{w}_0=(w_0)_{\tau\in\Sigma_L}$) be the longest elements in $\sW_{n,L}$ (resp.,\;$\sW_{n,L}$).\;Let $w_{I,0}$ (resp.,\;$\underline{w}_{\underline{I},0}=(w_{I_{\sigma},0})_{\sigma\in\Sigma_L}$) be the longest elements in $\sW_{I}$ (resp.,\;$\sW_{\underline{I},L}$).\;
%For $i\in \Delta_n$ and $\sigma\in \Sigma_L$,\;let $s_{i,\sigma}\in \sW_{n,\sigma}$ be the simple reflection corresponding to $i\in \Delta_n$.\;For any $I\subset \Delta_n$ ,\;define $\sW_{I,\sigma}$ to be the $\sigma$-copy of $\sW_{I}$,\;i.e.,\;to be the subgroup of $\sW_{n,\sigma}$ generated by simple reflections $s_{i,\sigma}$ with $i\in I$.\;For $S\subseteq \Sigma_L$ and  $I\subset {\Delta_n}$, we put $\sW_{I,S}:=\prod_{\sigma\in S} \sW_{I,\sigma}$.\;If $\ul{i}_{S}:=(i_{\sigma})_{\sigma\in S}\in {\Delta_n}^{|S|}$,\;denote by $s_{\ul{i}_S}:=\prod_{\sigma\in S}s_{i_\sigma,\sigma}$.\;

%When $J=\emptyset$,\;we put $^{I}\sW_{n}=[\sW_{I}\backslash\sW_{n}/\sW_\emptyset]^{\min}$.\;For $S\subseteq \Sigma_L$ and  $I\subset {\Delta_n}$,\;we put $^{I}\sW_{n,S}:=\prod_{\sigma\in S}{^{I}\sW_{n,\sigma}}$.\;

%If $H$ is a closed subgroup of $G$ and $(\rho,V)$ is an abstract representation of $H$ on $E$-vector space,\;we put $H^w:=w^{-1}Hw$ (resp.,\;${}^wH:=wHw^{-1}$),\;and $\rho^w(h):=\rho(whw^{-1})$.\;Then $(\rho^w,V)$ forms an abstract representation of $H^w$ on the same $E$-vector space $V$.;Moreover,\;we have relative Bruhat decompositions\[G=\coprod_{w\in [\sW_{r,I}\backslash \sW_n/\sW_{r,J}]}\bP_{{I}}(L)w\bP_{{J}}(L)=\coprod_{w\in [\sW_{r,I}\backslash \sW_n/\sW_{r,J}]}\op_{{I}}(L)w\op_{{J}}(L).\]

If $V$ is a continuous representation of $G$ over $E$,\;we denote by $V^{\bQ_p-\ana}$ its locally $\bQ_p$-analytic vectors.\;If $V$ is
locally $\bQ_p$-analytic representations of $G$,\;we denote by $V^{\mathrm{sm}}$ (resp.\;$V^{\mathrm{lalg}}$) the smooth (resp,\;locally $\bQ_p$-algebraic) subrepresentation of $V$ consists of its smooth (locally $\bQ_p$-algebraic) vectors (see \cite{schneider2002banach} and \cite{Emerton2007summary} for details).\;Let $\pi_P$ be a continuous representation of $P$ over $E$ (resp.,\;locally $\bQ_p$-analytic representations of $P$ on a locally convex $E$-vector space of compact type,\;resp.,\;smooth representations of $P$ over $E$),\;we denote by
\begin{equation}\label{smoothadj2}
	\begin{aligned}
		&(\mathrm{Ind}_{P}^{G}\pi_P)^{\cC^0}:=\{f:G\rightarrow \pi_P \text{\;continuous},\;f(pg)=pf(g)\},\\
		&\text{resp.,\;}(\mathrm{Ind}_{P}^{G}\pi_P)^{\bQ_p-\ana}:=\{f:G\rightarrow \pi_P \text{\;locally $\bQ_p$-analytic representations},\;f(pg)=pf(g)\},\\
		&\text{resp.,\;}i^G_P\pi_P:=(\mathrm{Ind}_{P}^{G}\pi_P)^{\infty}=\{f:G\rightarrow \pi_P \text{\;smooth},\;f(pg)=pf(g)\}
	\end{aligned}
\end{equation}
the continuous parabolic induction (resp.,\;the locally $\bQ_p$-analytic parabolic induction,\;resp., the un-normalized smooth parabolic induction) of $G$.\;It becomes a continuous representation (resp.,\;locally $\bQ_p$-analytic  representation) of $G$ over $E$ (resp.,\;on a locally convex $E$-vector space of compact type,\;resp.,\;smooth representations of $G$ over $E$) by endowing the left action of $G$ by right translation on functions: $(gf)(g')=f(g'g)$.\;

\subsection{\texorpdfstring{$(\varphi,\Gamma)$-module  over $\cR_{E,L}$ with special $\omepik$-filtration}{Lg}}\label{Omegafil}

\noindent Let $k,r$ be two integers such that $n=kr$.\;We put ${\Delta_n(k)}=\{r,2r,\cdots,(k-1)r\}\subseteq \Delta_{n}$ and  $\Delta_n^k=\Delta_{n}\backslash{\Delta_n(k)}$.\;For a subset $I\subset {\Delta_n(k)}$,\;we put $\bL_{r,I}:=\bL_{\Delta_n^k\cup I}$,\;$\bP_{r,I}:=\bP_{\Delta_n^k\cup I},\;\op_{r,I}:=\op_{\Delta_n^k\cup I}$,\;$\bN_{r,I}:=\bN_{\Delta_n^k\cup I},\;\on_{r,I}:=\on_{\Delta_n^k\cup I}$ and $\bZ_{r,I}:=\bZ_{\Delta_n^k\cup I}$.\;For example,\;we have
\[\bL_{r,\emptyset}:=\left(\begin{array}{cccc}
	\GLN_r & 0 & \cdots & 0 \\
	0 & \GLN_r & \cdots & 0 \\
	\vdots & \vdots & \ddots & 0 \\
	0 & 0 & 0 & \GLN_r \\
\end{array}\right)\subseteq \op_{r,\emptyset}:=\left(\begin{array}{cccc}
	\GLN_r & 0 & \cdots & 0 \\
	\ast & \GLN_r & \cdots & 0 \\
	\vdots & \vdots & \ddots & 0 \\
	\ast & \ast & \cdots & \GLN_r \\
\end{array}\right)\]
The parabolic subgroups of $\GL_n$ containing the parabolic subgroup $\op_{r,\emptyset}$ are given by $\{\op_{r,I}\}_{I\subseteq \Delta_n(k)}$.\;Let $\fl_{r,I}$,\;$\fp_{r,I}$, $\fn_{r,I}$, $\overline{\fp}_{r,I}$ $\overline{\fn}_{r,I}$,\;$\fz_{r,I}$ be the $L$-Lie algebras of $\bL_{r,I}$,\;$\bP_{r,I}$,\;$\bN_{r,I}$,$\op_{r,I}$,\;$\on_{r,I}$ and $\bZ_{r,I}$ respectively.\;For $\ul{\lambda}\in X_{\Delta_n^k\cup I}^+$,\;we put $L(\ul{\lambda})_{r,I}:=L(\ul{\lambda})_{\Delta_n^k\cup I},\;\overline{L}(-\ul{\lambda})_{r,I}:=\overline{L}(-\ul{\lambda})_{\Delta_n^k\cup I}$.\;Similarly,\;for  $I\subseteq \Delta_n(k)$,\;we have notation $\sW_{\Delta_n^k\cup I,L}$,\;$\sW^{\Delta_n^k,\emptyset}_{n,L,\max}$,\;etc.\;

%,\;denote by $^{\;\;r,I}_{\max}\sW_{S}:=[\backslash\sW_{n,S}]^{\max}$ (resp.,\;$^{\;\;r,I}_{\min}\sW_{S}:=[\sW_{\Delta_n^k\cup I,S}\backslash\sW_{n,S}]^{\min}$), and $\sW^{r,I}_{\max,S}:=[\sW_{n,S}/\sW_{\Delta_n^k\cup I,S}]^{\max}$ (resp.,\;$\sW^{r,I}_{\min,S}:=[\sW_{n,S}/\sW_{\Delta_n^k\cup I,S}]^{\min}$).\;Note that $w\in\;^{\;\;r,I}_{\min}\sW_{\Sigma_L}$ if and only if $w\underline{w}_0\in\;^{\;\;r,I}_{\max}\sW_{\Sigma_L}$.

%${M}_{r,I}(\ul{\lambda}):={M}_{\Delta_n^k\cup I}(\ul{\lambda})$ and $\overline{M}_{r,I}(-\ul{\lambda}):=\overline{{M}}_{\Delta_n^k\cup I}(-\ul{\lambda})$.

In the sequel,\;we fix a cuspidal Bernstein component $\Omega_r$ of $\GLN_r(L)$ and an irreducible smooth cuspidal representation $\pi_{0}\in \Omega_r$ over $E$ of type $ \Omega_r$.\;We put
\[\omepik=\prod_{i=1}^k\Omega_{i},\;\Omega_{i}=\Omega_{r}\]
which is a cuspidal Bernstein component of $\bL_{r,\emptyset}(L)$.\;Let $\FZ_{\Omega_r}$ be the (rational) Bernstein centre of $\Omega_r$ over $E$.\;Let $\FZ_{\omepik}\cong\otimes_{i=1}^k\FZ_{\Omega_{i}}=\FZ_{\Omega_{r}}^{\otimes k}$ be the  (rational)
Bernstein centre of $\omepik$ over $E$ (see \cite[Section 3.13]{PATCHING2016}).\;

Let $\mathbf{WD}_{L'/L,E}$ be the category of representations $(r,N,V)$ of $W_L$,\;on an $E$-vector space $V$ of finite dimension such that $r$ is unramified when restricted to the $W_{L'}$.\;Let $\mathbf{DF}_{L'/L,E}$ be the category of Deligne-Fontaine modules,\;i.e.,\;the category of quadruples $(\varphi,N,\gal(L'/L),D)$ where $D$ is an $L'_0\otimes_{\bQ_p}E$-module free of finite rank,\;which is endowed with a Frobenius $\varphi:D\rightarrow D$ (resp.,\;an $L'_0\otimes_{\bQ_p}E$-linear endomorphism $N:D\rightarrow D$) such that $N\varphi=p\varphi N$ and an action of $\gal(L'/L)$ commuting with $\varphi$ and $N$ such that $g((l\otimes e)d)=(g(l)\otimes e)d$ for $g\in \gal(L'/L),\;l\in L'_0,\;e\in E,\;d\in D$.\;Then the  Fontaine's theory asserts that there is a functor $\mathrm{WD}:\mathbf{WD}_{L'/L,E}\rightarrow \mathbf{DF}_{L'/L,E}$ gives an equivalence of categories (\cite[Proposition 4.1]{breuil2007first}).\;

Let $\pi\in \Omega_r$ be any irreducible smooth cuspidal representation over $E$ of type $ \Omega_r$.\;Then $\pi\cong \pi_0\otimes_E\unr(\alpha_\pi )$ for some $\alpha_\pi \in E^\times$.\;The irreducible cuspidal representation $\pi$ corresponds an $E$-point $x_\pi$ of $\Spec\FZ_{\Omega_{r}}$,\;a $r$-dimensional absolutely irreducible $F$-semi-simple Weil-Deligne representation $\wdre_\pi:=\rec(\pi)$ of $W_L$ over $E$ via the normalized classical local Langlands correspondence (see \cite{scholze2013local}),\;a Deligne-Fontaine module $\df_\pi$ (by Fontaine's equivalence of categories as in \cite[Proposition 4.1]{breuil2007first}),\;and a $p$-adic differential equation ${\Delta_\pi}$ over $\cR_{E,L}$ (by Berger's theory \cite[Theorem A]{berger2008equations},\;a $(\varphi,\Gamma)$-module of rank $m$ over $\cR_{k(x),L}$ which is de Rham of constant Hodge-Tate weights $0$ such that $D_{\mathrm{pst}}(\Delta_x)$ is isomorphic to $\df_x$ by forgetting the Hodge filtration).\;Assume that $\wdre_\pi$ is unramified when restricted to $W_{L'}$ for some finite Galois extension $L'$ of $L$.\;Then the associated (absolutely) irreducible Deligne-Fontaine module $\df_\pi=(\varphi_\pi,N=0,\Gal(L'/L),\df_\pi)\in \mathrm{DF}_{L'/L,E}$,\;where 
$\varphi_\pi:\df_\pi\rightarrow \df_\pi$ is the Frobenius semi-linear operator  on $\df_\pi$.\;

%\subsection{\texorpdfstring{$(\varphi,\Gamma)$-module  over $\cR_{E,L}$ with special $\omepik$-filtration}{Lg}}\label{omegafil}

Keep the notation and terminology in \cite[Section 2.3,\;Section 4.1.2]{Ding2021}.\;We consider a special case of $\Omega$-filtration which associated to an (absolutely) indecomposable Weil-Deligne representation,\;that we call the special $\omepik$-filtration.\;

Let $\Dpik$ be a potentially semistable $(\varphi,\Gamma)$-module over $\cR_{E,L}$ of rank $n$.\;Let $L'$ be a finite Galois extension of $L$ such that $\Dpik|_{L'}$ is a semistable $(\varphi,\Gamma)$-module over $\cR_{E,L'}$ of rank $n$.\;We consider the Deligne-Fontaine module associated to $\Dpik$:
$$\df_{\Dpik}=(\varphi,N,\Gal(L'/L),D_{\pst}(\Dpik))$$
where $D_{\pst}(\Dpik)=D_{\mathrm{st}}^{L'}(\Dpik\otimes_{\cR_{E,L}}\cR_{E,L'})$ is a finite free $L'_0\otimes_{\bQ_p}E$-module of rank $n$,\;$L'_0$ being the maximal unramified subextension of $L'$ (over $\bQ_p$),\;where the $(\varphi,N)$-action on $D_{\pst}(\Dpik)$ is induced from the $(\varphi,N)$-action on $B_{\mathrm{st}}$,\;and where the 
$\Gal(L'/L)$-action on  $D_{\pst}(\Dpik)$ is the residual action of $\gal_L$.\;

We say that $\wdre_{\Dpik}$  admits an increasing $\omepik$-filtration $\cF$ if $\wdre_{\Dpik}$ admits an increasing filtration $\cF$ by Weil-Deligne subrepresentations:
\[\cF=\fil_{\bullet}^{\cF} \wdre_{\Dpik}: \ 0 =\fil_0^{\cF} \wdre_{\Dpik} \subsetneq \fil_1^{\cF} \wdre_{\Dpik} \subsetneq \cdots \subsetneq \fil_{k}^{\cF} \wdre_{\Dpik}=\wdre_{\Dpik},\]
such that $\gr^{\cF}_{i}\wdre_{\Dpik}\cong \wdre_\pi\otimes_E|\rec^{-1}|^{k-i}$ for all $1\leqslant i\leqslant k$ and some irreducible smooth cuspidal representation $\pi\in \Omega_r$ over $E$ of type $ \Omega_r$.\;We further assume that the monodromy operator $N$ sends $\gr^{\cF}_{i}\wdre_{\Dpik}$ to
$\gr^{\cF}_{i-1}\wdre_{\Dpik}$ via the identity map on $\wdre_\pi$ for $2\leq i\leq k$,\;and sends $\gr^{\cF}_1\wdre_{\Dpik}$ to zero.\;

By \cite[Proposition 4.1]{breuil2007first},\;the $\omepik$-filtration $\cF$ on $\wdre_{\Dpik}$ corresponds to an $\omepik$-filtration on the Deligne-Fontaine module$\df_{\Dpik}$ (still denoted by $\cF$)  by Deligne-Fontaine submodules
\[\cF=\fil_{\bullet}^{\cF}\df_{\Dpik}: \ 0 =\fil_0^{\cF} \df_{\Dpik} \subsetneq \fil_1^{\cF} \df_{\Dpik} \subsetneq \cdots \subsetneq \fil_{k}^{\cF} \df_{\Dpik}=\df_{\Dpik},\]
such that $\fil_{i}^{\cF} \df_{\Dpik}$ is associated to $\fil_{i}^{\cF} \wdre_{\Dpik}$ via 
\cite[Proposition 4.1]{breuil2007first}.\;Then $\gr^{\cF}_{i}\df_{\Dpik}\cong (\varphi_{\pi,i},N_{\gr^{\cF}_{i}\df_{\Dpik}}=0,\Gal(L'/L),\df_{\pi,i})$ for $1\leqslant i\leqslant k$,\;where $\df_{\pi,i}$ is isomorphic to $\df_\pi$ as a module,\;endowed with a Frobenius morphism $\varphi_{\pi,i}=p^{i-k}\varphi_\pi$ (i.e.,\;a twist of $\varphi_\pi$ by $p^{i-k}$).\;The monodromy operator $N$ is zero on  $(p^{1-k}\varphi_\pi,N=0,\Gal(L'/L),\df_\pi)$,\;and sending $(p^{i-k}\varphi_\pi,N=0,\Gal(L'/L),\df_\pi)$ to
$(p^{(i-1)-k}\varphi_\pi,N=0,\Gal(L'/L),\df_\pi)$ via the identity map on $\df_\pi$ for  $2\leq i\leq k$.\;
%As a $\phi$-module,\;we see that $\fil_{r,J}^{\cF} \df_{\Dpik}=\bigoplus_{i=1}^j(p^{i-k}\varphi_\pi,N=0,\Gal(L'/L),\df_\pi)$.\;

Let  ${\Delta_\Dpik}$  be the $p$-adic differential equation over $\cR_{E,L}$ associated to  $\df_{\Dpik}$.\;The  $\omepik$-filtration on $\df_{\Dpik}$ now induces an  $\omepik$-filtration $\fil_{\bullet}^{\cF}{\Delta_\Dpik}=\{\fil_{i}^{\cF}{\Delta_\Dpik}\}$ on ${\Delta_\Dpik}$ by saturated $(\varphi,\Gamma)$-submodules over $\cR_{E,L}$,\;such that $\fil_{i}^{\cF}{\Delta_\Dpik}$ is the $p$-adic differential equation over $\cR_{E,L}$ associated to $\fil_{i}^{\cF} \df_{\Dpik}$.\;In particular,\;we see that $\gr^{\cF}_{i}{\Delta_\Dpik}\cong {\Delta_\pi}\tee \cR_{E,L}(\unr(q_L^{i-k}))$ for $1\leqslant i\leqslant k$.\;Consider
\[\cM_\Dpik=\Dpik\Big[\frac{1}{t}\Big]\cong \Delta_\Dpik \Big[\frac{1}{t}\Big]\]
By inverting $t$,\;the filtration $\cF$ on $\Delta_\Dpik$ induces an increasing filtration $\cF:= \fil_{i}^{\cF}\cM_\Dpik:=\fil_{i}^{\cF}{\Delta_\Dpik}\big[\frac{1}{t}\big]$ on $\cM_\Dpik$ by $(\varphi,\Gamma)$-submodules over $\cR_{E,L}\big[\frac{1}{t}\big]$.\;Therefore,\;the filtration $\cF$ on $\cM_\Dpik=D\big[\frac{1}{t}\big]$ induces a filtration on $\Dpik$:
\[\cF=\fil_{\bullet}^{\cF}\Dpik: \ 0 =\fil_0^{\cF}\Dpik \subsetneq \fil_1^{\cF}\Dpik \subsetneq \cdots \subsetneq \fil_{k}^{\cF} \Dpik=\Dpik,\;\fil_{i}^{\cF}\Dpik=(\fil_{i}^{\cF}\cM_\Dpik)\cap \Dpik,\]
by saturated $(\varphi,\Gamma)$-submodules of $\Dpik$ over $\cR_{E,L}$.\;

Since $\Dpik$ is potentially semistable,\;it is de Rham.\;Hence we have $D_{\dr}(\Dpik)\cong (D_{\pst}(\Dpik)\otimes_{L_0'}L')^{\Gal(L'/L)},\;$ which  is a free $L\otimes_{\bQ_p}E$-module of rank $n$.\;The $\omepik$-filtration $\cF$  on $\df_{\Dpik}$
induces a $\omepik$-filtration $\cF$ on $D_{\dr}(\Dpik)$ by free $L\otimes_{\bQ_p}E$-submodules $\fil_{\bullet}^{\cF}D_{\dr}(\Dpik):=(\fil_{\bullet}^{\cF}\df_{\Dpik}\otimes_{L_0'}L')^{\Gal(L'/L)}$.\;The module $D_{\dr}(\Dpik)$ is equipped with a natural Hodge filtration.\;We assume that $D_{\dr}(\Dpik)$  has distinct Hodge-Tate weights $\bh:=(\hpi_{\tau,1}>\hpi_{\tau,2}>\cdots>\hpi_{\tau,n} )_{\tau\in \Sigma_L}$.\;Denote by $\hpi_{i}=(\hpi_{\tau,i})_{\tau\in \Sigma_L}$ for $1\leq i\leq n$.\;

Hence,\;for each $\tau\in \Sigma_L$,\;the natural Hodge filtration can be expressed by the following complete flag:
\[\fil_{\bullet}^{H}D_{\dr}(\Dpik)_\tau: \ 0 \subsetneq \fil_{-\hpi_{\tau,n}}^{H} D_{\dr}(\Dpik)_\tau \subsetneq \fil_{-\hpi_{\tau,n-1}}^{\cF}D_{\dr}(\Dpik)_\tau \subsetneq \cdots \subsetneq \fil_{{-\hpi_{\tau,1}}}^{H}D_{\dr}(\Dpik)_\tau=D_{\dr}(\Dpik)_\tau.\]
For each $\tau\in \Sigma_L$,\;we now fix a a basis of $D_{\dr}(\Dpik)_\tau$ over $E$.\;Then the Hodge filtration $\fil_{\bullet}^{H}$ (resp.,\;$\cF$)  corresponds to an $E$-point $(g_{2,\tau}\bB(E))_{\tau\in \Sigma_L}\in \GLN_{n,L}/\bB_{L}$ (resp.,\;$(g_{1,\tau}\bP_{r,\emptyset}(E))_{\tau\in \Sigma_L}\in \GLN_{n,L}/\bP_{r,\emptyset,L}$).\;For each $\tau\in \Sigma_L$,\;there exists thus a unique $w_{\cF,\tau}\in \sW^{\Delta_n^k,\emptyset}_{n,L,\max}$ such that 
\[(g_{1,\tau}\bP_{r,\emptyset}(E),g_{2,\tau}\bB(E))\in\GLN_n(E)(1,w_{\cF,\tau})(\bP_{r,\emptyset}\times\bB)(E)\subset (\GLN_n/\bP_{r,\emptyset}\times\GLN_n/\bB)(E).\]
We say that $\cF$ is \textit{non-critical} if $w_{\cF,\tau}={w}_0$ for all $\tau\in \Sigma_L$.

Now by Berger's equivalence of categories,\;we see that $\fil_{i}^{\cF}\Dpik$ corresponds to the filtered Delingen-Fontaine module $\fil_{i}^{\cF} \df_{\Dpik}$ equipped with the induced filtration from the Hodge filtration on $D_{\pst}(\Dpik)$.\;In this case,\;we see that the Hodge-Tate weights of $\fil_{i}^{\cF}\Dpik$ are given by $$\{\hpi_{\tau,(w_{\cF,\tau}{w}_0)^{-1}(1)},\hpi_{\tau,(w_{\cF,\tau}{w}_0)^{-1}(2)},\cdots,\hpi_{\tau,(w_{\cF,\tau}{w}_0)^{-1}(ir)}\}_{\tau\in \Sigma_L}.\;$$
This implies that the Hodge-Tate weights of $\gr_{i}^{\cF} \Dpik$ are $$\{\hpi_{\tau,(w_{\cF,\tau}{w}_0)^{-1}((i-1)r+1)},\hpi_{\tau,(w_{\cF,\tau}{w}_0)^{-1}((i-1)r+2)},\cdots,\hpi_{\tau,(w_{\cF,\tau}{w}_0)^{-1}(ir)}\}_{\tau\in \Sigma_L}.\;$$
We put $w_{\cF}(\bh)_{j}:=\{\hpi_{\tau,(w_{\cF,\tau}{w}_0)^{-1}(j)}\}_{\tau\in \Sigma_L}$.\;In this case,\;using Berger's equivalence of categories \cite[Theorem  A]{berger2008equations} and comparing the weight (or see \cite[(2.4)]{Ding2021}),\;we have an injection of $(\varphi,\Gamma)$-modules over $\cR_{E,L}$
\begin{equation}\label{Dpikinjection}
	\mathbf{I}_{i}:\gr_{i}^{\cF} \Dpik \hookrightarrow {\Delta_\pi}\tee \cR_{E,L}(\unr(q_L^{i-k}))\tee \cR_{E,L}(z^{w_{\cF}(\bh)_{jr}})=\gr^{\cF}_{i}{\Delta_\Dpik}\tee \cR_{E,L}(z^{w_{\cF}(\bh)_{jr}}),
\end{equation}
for $i=1,\cdots,k$.\;This implies that the $(\varphi,\Gamma)$-module $\Dpik$ admits an $\omepik$-filtration $\cF$.\;

Let $\rigchl$ (resp.,\;$\rigch$) be the rigid space over $E$ parametrizing continuous characters of $\bL_{r,\emptyset}(L)$ (resp.,\;$\bL_{r,\emptyset}(\cO_L)$).\;The parameters of $\cF$ in $\sbanpik\times\rigchl $ or $\sbanpik\times\rigch$ are given as follows.\;Recall that $\pi\cong \pi_0\otimes\unr(\alpha_\pi )$ for some $\alpha_\pi \in E^\times$.\;

%\begin{dfn}\label{weaklynoncritical}
%	The $(\varphi,\Gamma)$-module $\Dpik$ admits an $\omepik$-filtration $\cF$ with parameter
%\end{dfn}
%\begin{rmk}For convenience,\;we may use thees two kinds of parameters depending on the situation.\;The parameters of $\cF$ in $\sbanpik\times\rigchl$ or $\sbanpik\times\rigch$ are in general not unique,\;see \cite[Lemma 4.1.9]{Ding2021} for more precise statement.\;The above parameters seem to have the simplest form.\;\end{rmk}

\begin{dfn}\label{dfnnoncriticalspecial}\textbf{(Special $\omepik$-filtration)} 
Put $(\bx_0,\bmdel)\in \sbanpik\times\rigchl$ with
\begin{equation}
	\begin{aligned}
		&\bx_0=(\bx_{0,i}\cong x_{\pi_0})_{1\leq i\leq k},\;\bmdel:=(\bm{\delta}_{\bh,i}=\unr(\alpha_\pi  q_L^{i-k}){z^{w_{\cF}(\bh)_{ir}}})_{1\leq i\leq k},\;
	\end{aligned}
\end{equation}
or $(\widetilde{\bx}_{\pi,\bh},\widetilde{\bm{\delta}}_\bh)\in \sbanpik\times\rigch$ with
\begin{equation}
	\begin{aligned}
		&\widetilde{\bx}_{\pi,\bh}=(\widetilde\bx_{\pi,i})_{1\leq i\leq k},\;\pi_{\widetilde\bx_{\pi,\bh,i}}\cong \pi_0\te\unr(\alpha_\pi  q_L^{i-k}z^{w_{\cF}(\bh)_{jr}}(\varpi_L)),\;\widetilde{\bm{\delta}}_\bh=(\widetilde{\bm{\delta}}_{\bh,i}={z^{w_{\cF}(\bh)_{ir}}}|_{\co_L^\times})_{1\leq i\leq k}.
	\end{aligned}
\end{equation}	
We call an $\omepik$-filtration on $\Dpik$ is special with parameter $(\bx_0,\bmdel)\in \sbanpik\times\rigchl $,\;(resp.,\; with parameter $(\widetilde{\bx}_{\pi,\bh},\widetilde{\bm{\delta}}_\bh)\in \sbanpik\times\rigch $)\;if $\Dpik$ admits a $\omepik$-filtration with parameter $(\bx_0,\bmdel)\in \sbanpik\times\rigchl $,\;(resp.,\;with parameter $(\widetilde{\bx}_{\pi,\bh},\widetilde{\bm{\delta}}_\bh)\in \sbanpik\times\rigch $ ),\;and $\Dpik^{i+1}_{i}$ is non-split for each  $ir\in\Delta_n(k)$.\;We say a $p$-adic Galois representation $\rho_L:\gal_L\rightarrow \GLN_{n}(E)$ admits a special $\omepik$-filtration (or special triangulation) if $D_{\rig}(\rho_L)$ has this property.\; %(equivalently,\;$({\Delta_{\Dpik}})^{i+1}_{i}$ is potentially semistable and potentially noncrystalline for all $s\in\Delta_n(k)$,\;see the Lemma \ref{criofnoncrystalline}).\;
\end{dfn}

In particular,\;we can restrict them to classical trianguline case.\;

\begin{rmk} (Special triangulation) If $r=1$ (so $k=n$),\;we have $\bL_{1,\emptyset}=\bT$ and $\bP_{1,\emptyset}=\bB$.\;Let $\widehat{T}_L$ denote the character space of $\bT(L)$ over $E$,\;i.e.,\;the rigid space over $E$ parameterizing continuous character of $\bT(L)$.\;Via the isomorphism
	\[\iota_{\Omega_{1}^{\otimes n},\bh}: \big(\mathrm{Spec}\hspace{2pt}\mathfrak{Z}_{\Omega_{1}^{\otimes n}}\big)^{\mathrm{rig}}\times\mathcal{Z}_{\bL_{1,\emptyset},\mathcal{O}_L}\xrightarrow{\sim}\widehat{T}_L,(\underline{x},\undelram)\mapsto (\boxtimes_{i=1}^r\pi_{x_i}) \undelram z^{\bh}.\]
if $r=1$ and $\pi_0=v_1^{\frac{1-n}{2}}$,\;then the $\omepik$-filtration $\cF$ on $\Dpik$ becomes the so-called special triangulation with parameter $\bmdel:=(\bm{\delta}_{\bh,i}=\unr(\alpha q_L^{i-k}){z^{w_{\cF}(\bh)_{i}}})_{1\leq i\leq n}\in \widehat{T}_L$.\;
\end{rmk}

%Throughout Section 3,\;we fix such a $(\varphi,\Gamma)$-module $\Dpik$.\;For $1\leq i\leq j\leq k$,\;we put $\Dpik^{j}_{i}:=\fil_{r,J}^{\cF}\Dpik/\fil_{i-1}^{\cF}\Dpik$. We also use $\cF$ to denote the induced $\Omega_{[i,j]}$-filtration on $\Dpik^{j}_{i}$,\;where $\Omega_{[i,j]}:=\prod_{l=i}^j\Omega_{r,l}$.\;Then  $\Dpik^{j}_{i}$ admits an special $\Omega_{[i,j]}$-filtration with parameter $((\bx_{0,s})_{i\leq s\leq j},(\bm{\delta}_{\bh,s})_{i\leq s\leq j})$.\;

\begin{rmk}
Suppose that $\Dpik$ is of slope zero,\;i.e.,\;it comes from a $p$-adic Galois representation $\rho_L:\gal_L\rightarrow \GLN_n(E)$.\;If $n=2$,\;and $w_{\cF}\neq \underline{w}_0$,\;then $d_L>1$,\;then $w_{\cF}=(w_{\cF,\tau})_{\tau\in \Sigma_L}$ with $w_{\cF,\tau}=s_1$ (resp.,\;$w_{\cF,\tau}=1$) if and only if $\tau\in S$ (resp.,\;$\tau\not\in S$),\;where $S$ is a subset of $\Sigma_L$.\;If $L=\bQ_p$ and $n=3$,\;the possible choices of $w_{\cF}$ are only $\{s_1{w}_0,s_2w_0\}$.\;Although the non-critical case is the most common,\;it seems less difficult for an special $\omepik$-filtration (in particular,\;a special triangulation) to be  critical as $n\rightarrow\infty$.\;
\end{rmk}

\subsection{Some preliminaries on potentially semistable deformation ring}
\label{semistableder}

The proofs of main theorems need some Zariski-closure argument on semistable deformation rings and the relationship between semistable deformation spaces and trianguline variety.\;In this section,\;we make some preliminaries on potentially semistable deformation ring.\;

Keep the notation in Section \ref{Omegafil}.\;Let $\Omega_n$ be the Bernstein component of $\GLN_n(L)$ associated to the cuspidal Bernstein component $\omepik$ of $\bL_{r,\emptyset}$ and let $\tau:I_{L}\rightarrow \GLN_n(E)$ be the associated inertial type (see \cite[Section 3.2]{PATCHING2016}).\;Recall that $\tau_0|_{I_{L'}}$ is trivial.\;

Let $\overline{r}:\gal_L\rightarrow \GLN_n(k_E)$ be a residual representation of $\gal_L$.\;Let $R_{\overline{r}}^{\Box,\tau,\bh}$ be the unique reduced and $p$-torsion free quotient of the (framed) local deformation ring $R_{\overline{r}}^{\Box}$ corresponding to potentially semistable deformations of $\overline{r}$ with inertial type $\tau$ and  Hodge-Tate weights $\bh$ (in the sense of \cite{KisinSEMISTABLE}) over $\cO_E$.\;By \cite[Theorem 3.3.4]{KisinSEMISTABLE},\;the space $X_{\overline{r}}^{\Box,\tau,\bh}:=\Spec R_{\overline{r}}^{\Box,\tau,\bh}[1/p]$ is equi-dimensional of dimension $n^2+d_L\frac{n(n-1)}{2}$.\;We write $\FX_{\overline{r}}^{\Box,\tau,\bh}:=(\Spf R_{\overline{r}}^{\Box,\tau,\bh})^{\rig}$ for the  closed analytic subspace $\FX_{\overline{r}}^{\Box}:=(\Spf R_{\overline{r}}^{\Box})^{\rig}$.\;By \cite[Theorem 2.5.5]{KisinSEMISTABLE},\;there is a universal (coherent) filtered $(\varphi,N)$-module $(\cD,\Phi,\bN,\cF^{\bullet})$ that is locally free over $\FX_{\overline{r}}^{\Box,\tau,\bh}$ such that for all $y\in \FX_{\overline{r}}^{\Box,\tau,\bh}$,\;we have $(\cD,\Phi,\bN,\cF^{\bullet})\otimes_{\cO_{\FX_{\overline{r}}^{\Box,\tau,\bh}}}k(y)=(D_{\mathrm{st}}(\rho_y),\;\phi_y,N_y,\cF_y^{\bullet}).$

We introduce two stratifications on $\FX_{\overline{r}}^{\Box,\tau,\bh}$.\;For $w\in \sW_{n,L}$,\;we write $\FX_{\overline{r},w}^{\Box,\tau,\bh}\subset \FX_{\overline{r}}^{\Box,\tau,\bh}$ for the inverse image of the Bruhat cell 
$\big(\mathrm{Res}_{L/\bQ_p}(\bB w\bB/\bB)\big)^{\rig}$. Then $\FX_{\overline{r},w}^{\Box,\tau,\bh}$  is locally closed in $\FX_{\overline{r}}^{\Box,\tau,\bh}$ and the $\FX_{\overline{r},w}^{\Box,\tau,\bh}$ for $w\in  \sW_{n,L}$ set-theoretically cover $\FX_{\overline{r}}^{\Box,\tau,\bh}$.\;Let $\overline{\FX_{\overline{r},w}^{\Box,\tau,\bh}}$ be the Zariski-closure of $\FX_{\overline{r},w}^{\Box,\tau,\bh}$ in $\FX_{\overline{r}}^{\Box,\tau,\bh}$.\;

Another stratification of $\FX_{\overline{r}}^{\Box,\tau,\bh}$ is
 given by concerning the sharp of nilpotent operater.\;For any  $x\in X_{\overline{r}}^{\Box,\tau,\bh}$,\;the Jordan normal norm of the nilpotent operator $N_x$ is determined uniquely up to conjugacy by a partition $\cP_x$ of $n$.\;Define a partial order $\leq$ on partitions which is the reverse of so-called dominance order,\;i.e.,\;$(n_1,n_2,\cdots,n_t)\leq (n_1',n_2',\cdots,n_s')$ with $n_1\geq n_2\geq\cdots\geq n_t>0$ and $n_1'\geq n_2'\geq\cdots\geq n_s>0'$ if and only if $\sum_{i=1}^ln_i\leq \sum_{i=1}^ln_i'$.\;Let $\cP_{\max}$ (resp.,\;$\cP_{\min}$) be the maximal (resp.,\;minimal) partition under this dominance order.\;Let $\cN$ be the nilpotent cone of $\fg$.\;Then there is a bijection between the nilpotent orbits of $\cN$ and the partitions.\;For partition $\cP$,\;let $\cO_{\cP}\subset \cN$ be the nilpotent orbit  associated to the $\cP$.\;In particular,\;if $\cP=\cP_{\min}$,\;then it corresponds to the regular nilpotent orbit $\cO_{\reg}:=\cO_{\cP_{\min}}$ in $\fg$.\;Let $N_{\reg}$ be the nilpotent elements with all ones on the superdiagonal,\;then $\cO_{\reg}$ is the  orbit of $N_{\reg}$.\;

%For partition-valued function $\cP$,\;let $k_{\cP}$ be the number of the blocks of $\cP$.\;
For partition $\cP$,\;by \cite[Definition 4.3]{BSCONJ},\;there is a reduced,\;$p$-torsion free quotient $R_{\overline{r},\geq\cP}^{\Box,\tau,\bh}:=R_{\overline{r}}^{\Box,\tau,\bh}/I_{\cP}$ of $R_{\overline{r}}^{\Box,\tau,\bh}$ and $X_{\overline{r},\geq\cP}^{\Box,\tau,\bh}:=\Spec R_{\overline{r},\geq\cP}^{\Box,\tau,\bh}[1/p]$ such that $x\in X_{\overline{r},\cP}^{\Box,\tau,\bh}\hookrightarrow X_{\overline{r}}^{\Box,\tau,\bh}$ if and only if $\cP_x\geq \cP$.\;Let $X_{\overline{r},\cP}^{\Box,\tau,\bh}$ be the open subspace of $X_{\overline{r},\geq\cP}^{\Box,\tau,\bh}$ such that $x\in X_{\overline{r},\cP}^{\Box,\tau,\bh}\hookrightarrow X_{\overline{r},\geq\cP}^{\Box,\tau,\bh}$ if and only if $\cP_x=\cP$.\;In particular,\;$R_{\overline{r},\geq\cP}^{\Box,\tau,\bh}$ is the potentially crystalline deformation ring $R_{\overline{r},\cP_{\min}}^{\Box,\tau,\bh-\mathrm{pcr}}$ (resp.,\;potentially semistable deformation ring) if $\cP=\cP_{\max}$  (resp.,\;$\cP=\cP_{\min}$).\;

From now on,\;we restrict the discussion to the case that $L'=L$ and $\tau=1$ is trivial.\;We replace the superscript $``\Box,\tau,\bh"$ by $``\Box,\bh-\mathrm{st}"$ in above notation for simplicity.\;Let $\FY_{\overline{r}}^{\Box,\bh-\mathrm{st}}\rightarrow \FX_{\overline{r}}^{\Box,\bh-\mathrm{st}}$ be the $\big(\mathrm{Res}_{L/\bQ_p}\GLN_{n,L}\big)\times_{\bQ_p} E$-torsor of the trivialization of the underlying coherent sheaf of the universal filtered $(\varphi,N)$-module $(\cD,\Phi,\bN,\cF^{\bullet})$.\;Then sending a semistable deformation with a trivialization of $D_{\mathrm{st}}$ to its Frobenius $\Phi$,\;the monodromy operator $\bN$ and the Hodge filtration define a morphism:
\begin{equation}
	\begin{aligned}
		\widetilde{\gamma}:\FY_{\overline{r}}^{\Box,\bh-\mathrm{st}}\longrightarrow \big((\mathrm{Res}_{L'_0/\bQ_p}\GLN_{n,L'_0})\times_{\bQ_p} E\big)^{\rig}\times_{\Spm E}&\big((\mathrm{Res}_{L_0/\bQ_p}\mathrm{Mat}_{n,L'_0})\times_{\bQ_p} E\big)^{\rig}\\ &\times_{\Spm E} \big((\mathrm{Res}_{L/\bQ_p}\GLN_{n,L}/\mathrm{Res}_{L/\bQ_p}\bB)\times_{\bQ_p} E\big)^{\rig}.
	\end{aligned}
\end{equation}
Keep the notation in \cite{universalsemistable}.\;Let ${\sD}_{\phi,\mu}^{\mathrm{ad}}$ be the quotient stack of the adic space $\big(X_{\varphi,N}\big)^{\mathrm{ad}}$ associated with $X_{\varphi,N}$ by the action of $\big(\mathrm{Res}_{L_0/\bQ_p}\GLN_{n,L_0}\big)\times E$.\;Let  ${\sD}_{\phi,\mu}^{\mathrm{ad},\mathrm{adm}}$ be the open subspace of ${\sD}_{\phi,\mu}^{\mathrm{ad}}$ such that there is a universal representation of $\gal_{L}$ on a vector bundle $\cV$ on ${\sD}_{\phi,\mu}^{\mathrm{ad},\mathrm{adm}}$,\;and $\widetilde{\sD}_{\phi,\mu}^{\mathrm{ad},\mathrm{adm}}$ be the stack over ${\sD}_{\phi,\mu}^{\mathrm{ad},\mathrm{adm}}$ trivializing $\cV$.\;By \cite{universalsemistable},\;we get that $\FX_{\overline{r}}^{\Box,\tau,\bh}$ is isomorphic to an open subspace $\widetilde{\sD}_{\phi,\mu}^{\mathrm{ad},\mathrm{adm},+}(\overline{r})$ of $\widetilde{\sD}_{\phi,\mu}^{\mathrm{ad},\mathrm{adm}}$.\;This shows that $\widetilde{\gamma}$ is smooth.\;
%\begin{pro}\label{pminloucsofsemidefrings} We have 
%	\[h^{-1}\big(\overline{\big(\mathrm{Res}_{L/\bQ_p}(\bB w\bB/\bB)\big)^{\rig}_E}\big)=\overline{\FX_{\overline{r},w}^{\Box,\bh-\mathrm{st}}}=\coprod_{w'\leq w}\FX_{\overline{r},w'}^{\Box,\bh-\mathrm{st}}.\]
%	Therefore,\;$x\in \overline{\FX_{\overline{r},w}^{\Box,\bh-\mathrm{st}}}$ if and only if $w_x\leq w$.\;
%\end{pro}

For any partition function $\cP$,\;let $\FY_{\overline{r},\cP}^{\Box,\bh-\mathrm{st}}$ be the fiber of the nilpotent orbit $\cO_{\cP}\subset \cN$ via the morphism
$\FY_{\overline{r}}^{\Box,\bh-\mathrm{st}}\longrightarrow \big((\mathrm{Res}_{L_0/\bQ_p}\mathrm{Mat}_{n,L'_0})\times_{\bQ_p} E\big)^{\rig}$,\;which is locally closed in $\FY_{\overline{r},\cP}^{\Box,\bh-\mathrm{st}}$.\;Let 
$\overline{\FY_{\overline{r},\cP}^{\Box,\bh-\mathrm{st}}}$ be the closure of $\FY_{\overline{r},\cP}^{\Box,\bh-\mathrm{st}}$ in $\FY_{\overline{r}}^{\Box,\bh-\mathrm{st}}$.\;For $w\in \sW_{n,L}$,\;let $\FX_{\overline{r},\cP,w}^{\Box,\bh-\mathrm{st}}\subset \FX_{\overline{r},\cP}^{\Box,\bh-\mathrm{st}}$ for the inverse image of the Bruhat cell 
$\big(\mathrm{Res}_{L/\bQ_p}(\bB w\bB/\bB)\times_{\bQ_p} E\big)^{\rig}$. Then $\FX_{\overline{r},\cP,w}^{\Box,\bh-\mathrm{st}}$  is locally closed in $\FX_{\overline{r},\cP}^{\Box,\bh-\mathrm{st}}$ and the $\FX_{\overline{r},\cP,w}^{\Box,\bh-\mathrm{st}}$ for $w\in  \sW_{n,L}$ set-theoretically cover $\FX_{\overline{r},\cP}^{\Box,\bh-\mathrm{st}}$.

\begin{pro}\label{pminloucsofsemidefringscP} We have 
	\[h^{-1}\big(\overline{\big(\mathrm{Res}_{L/\bQ_p}(\bB w\bB/\bB)\times_{\bQ_p} E\big)^{\rig}}\big)=\overline{\FX_{\overline{r},\cP_{\min},w}^{\Box,\bh-\mathrm{st}}}=\coprod_{w'\leq w}\FX_{\overline{r},\cP_{\min},w'}^{\Box,\bh-\mathrm{st}},\]
where $\overline{\FX_{\overline{r},\cP_{\min},w}^{\Box,\bh-\mathrm{st}}}$ is the closure of $\FX_{\overline{r},\cP_{\min},w}^{\Box,\bh-\mathrm{st}}$ in $\FX_{\overline{r},\cP_{\min}}^{\Box,\bh-\mathrm{st}}$ (so $\overline{\FX_{\overline{r},\cP_{\min},w}^{\Box,\bh-\mathrm{st}}}$ is equal to the intersection of the closure of $\FX_{\overline{r},\cP_{\min},w}^{\Box,\bh-\mathrm{st}}$ in $\FX_{\overline{r}}^{\Box,\bh-\mathrm{st}}$ with $\FX_{\overline{r},\cP_{\min}}^{\Box,\bh-\mathrm{st}}$).\;
\end{pro}
\begin{proof}Since the nilpotent orbit $\cO_{\cP_{\min}}$,\;i.e.,\;the regular nilpotent orbit is  smooth and open dense in $\cN$,\;we get that $\FY_{\overline{r},\cP_{\min}}^{\Box,\bh-\mathrm{st}}\rightarrow \big((\mathrm{Res}_{L/\bQ_p}\GLN_{n,L}/\mathrm{Res}_{L/\bQ_p}\bB)\times_{\bQ_p} E\big)^{\rig}$ is smooth.\;By the similar argument as in the proof of \cite[Theorem 4.2.3]{breuil2019local} or \cite[Theorem 4.1]{wu2021local} (
	using the fact that smooth morphisms are open),\;we deduce the closure relations $\overline{\FX_{\overline{r},\cP_{\min},w}^{\Box,\bh-\mathrm{st}}}=\coprod_{w'\leq w}\FX_{\overline{r},\cP_{\min},w'}^{\Box,\bh-\mathrm{st}}$ when we descent along the map $\FY_{\overline{r},\cP_{\min}}^{\Box,\bh-\mathrm{st}}\rightarrow \FX_{\overline{r},\cP_{\min}}^{\Box,\bh-\mathrm{st}}$.\;
\end{proof}

\section{Local models on special $\omepik$-case}

By modifying the methods in \cite[Section 6]{Ding2021},\;we show that the local geometry of the Bernstein paraboline variety at our potentially semistable non-crysatalline points is closely related to the schemes studied in geometric representation theory.\;

\subsection{Preliminaries}\label{preonschemegeorepn}
We apply \cite[Section 5.1-Section 5.4]{Ding2021} to $\GLN_{n,L}$ and get the following schemes.\;We obtain the affine scheme $\fg_L$ associated to the Lie algebra $\fg_L$.\;In this paper,\;we usually add a subscript ``$L$" (resp.,\;$J$) when we consider $\Sigma_L$-components (resp.,\;$J$-components).\;
\begin{equation}
	\begin{aligned}
		&\widetilde{\fg}_{r,L}:=\prod_{\tau\in \Sigma_L}\widetilde{\fg}_{r,\tau},\;\widetilde{\fg}_{r,\tau}=\widetilde{\fg}_{r}=\{(g\bP_{r,\emptyset},\psi)\in \GLN_n/\bP_{r,\emptyset}\times\fg\;|\;\mathrm{Ad}(g^{-1})\psi\in \tau_{{r,\emptyset}}\},\;
	\end{aligned}
\end{equation}
where  $\tau_{{r,\emptyset}}$ is the full radical of $\fp_{r,\emptyset}$,\;i.e.,\;$\tau_{{r,\emptyset}}=\fn_{r,\emptyset}\rtimes \fz_{r,\emptyset}$.\;If $r=1$,\;we put $\widetilde{\fg}:=\widetilde{\fg}_{1}$ (note that $\bP_{1,\emptyset}=\bB$) and $\widetilde{\fg}_L:=\widetilde{\fg}_{1,L}$.\;We have natural morphisms $q_{\bP_{r,\emptyset}}:\widetilde{\fg}_{r}\rightarrow \fg$ (resp.\;$q_{\bB}:\widetilde{\fg}\rightarrow \fg$) given by $(g\bP_{r,\emptyset},\psi)\mapsto \psi$ (resp.,\;$(g\bB,\psi)\mapsto \psi$).\;We put
\begin{equation}
	\begin{aligned}
		&X_{r,L}:=\widetilde{\fg}_{r,L}\times_{\fg_{L}}\widetilde{\fg}_L\cong\prod_{\tau\in \Sigma_L}X_{r,\tau},X_{r,\tau}=X_{r}:=\widetilde{\fg}_{{r}}\times_{\fg}\widetilde{\fg}, \\&\widetilde{\fg}_{r}\times_{\fg}\widetilde{\fg}=\{(g_1\bB,g_2\bP_{r,\emptyset},\psi)\in \GLN_n/\bP_{r,\emptyset}\times\GLN_n/\bB\times\fg\;|\;\mathrm{Ad}(g_1^{-1})\psi\in \tau_{{r,\emptyset}},\mathrm{Ad}(g_2^{-1})\psi\in \bB\}.\;
	\end{aligned}
\end{equation}
By \cite[Corollary 5.2.2]{Ding2021},\;the scheme $X_{r}$ is equidimensional of dimension $n^2-\frac{(r+2)(r-1)k}{2}$ (in particular,\;of dimension $n^2$ if $r=1$).\;Let $\pi:X_{r}\hookrightarrow \GLN_n/\bP_{r,\emptyset}\times\GLN_n/\bB\times\fg\twoheadrightarrow \GLN_n/\bP_{r,\emptyset}\times\GLN_n/\bB$ be the composition.\;For $w\in \sW^{\Delta_n^k,\emptyset}_{n}$,\;let $X_{r,w}$ be the closed subscheme of $X_{r}$ defined as the reduced Zariski-closure of $\pi^{-1}(U_w)$,\;where $U_w:=\GLN_n(1,w)(\bP_{r,\emptyset}\times\bB)$ (note that $U_w$ and thus $X_{r,w}$ only depends on the coset $\sW_{\Delta_n^k}w$).\;The irreducible components of $X_{r}$ are indexed by $\{X_{r,w}\}_{w\in \sW^{\Delta_n^k,\emptyset}_{n}}$.\;For $w=(w_\tau)_{\tau\in \Sigma_L}\in \sW_{n,L}$,\;let $X_{r,w}:=\prod_{\tau\in \Sigma_L}X_{r,w_\tau}$,\;this is an irreducible component of $X_{r,L}$ which only depends on the coset $\sW_{\Delta_n^k,L}w$.\;

We also need the closed subscheme $q_{\bP_{r,\emptyset}}^{-1}(\tau_{{r,\emptyset}})$ of $\widetilde{\fg}_{{r}}$.\;Let $U_w:=\bB w\bP_{r,\emptyset}/\bP_{r,\emptyset}$ and  $V_w:=q_{\bP_{r,\emptyset}}^{-1}(\tau_{\bP_{r,\emptyset},\Sigma_L})\cap \pi^{-1}(U_w)$,\;where $\pi$ is the composition $\pi:\widetilde{\fg}_{{r}}\hookrightarrow \GLN_n/\bP_{r,\emptyset}\times\fg\rightarrow \GLN_n/\bP_{r,\emptyset}$.\;Then the projection $V_w\rightarrow U_w$ is a geometric vector bundle of dimension $\dim \tau_{{r,\emptyset}}-\lg(w^{\min})$ (the proof is analogous to the proof of \cite[Proposition 5.2.1]{Ding2021}).\;Let $q_{\bP_{r,\emptyset}}^{-1}(\tau_{{r,\emptyset}})_w$ be the closed subscheme of $q_{\bP_{r,\emptyset}}^{-1}(\tau_{{r,\emptyset}})$ defined as the reduced Zariski-closure of $V_w$ in $q_{\bP_{r,\emptyset}}^{-1}(\tau_{{r,\emptyset}})$.\;Then we see that the $q_{\bP_{r,\emptyset}}^{-1}(\tau_{{r,\emptyset}})$ it is equidimensional of  dimension $k+\frac{n(n-r)}{2}$ and  the irreducible components are given by $q_{\bP_{r,\emptyset}}^{-1}(\tau_{{r,\emptyset}})_w$ for $w\in \sW^{\emptyset,\Delta_n^k}_{n}$.\;

 Denote by  $\kappa_{\bP_{r,\emptyset}}:X_{r,L}\rightarrow \fz_{r,\emptyset,L}$ (resp.,\;$\kappa_{\bB}:X_{r,L}\rightarrow \ft_{L}$) the morphism (recall \cite[(5.6)]{Ding2021})
\[(g_1\bP_{r,\emptyset},g_2\bB,\psi)\mapsto\overline{\mathrm{Ad}(g_1^{-1})\psi},\;\text{resp.,\;}(g_1\bP_{r,\emptyset},g_2\bB,\psi)\mapsto\overline{\mathrm{Ad}(g_2^{-1})\psi}\]
where $\overline{\mathrm{Ad}(g_1^{-1})\psi}$ is the image of $\mathrm{Ad}(g_1^{-1})\psi\in \tau_{{r,\emptyset}}$ via $\tau_{{r,\emptyset}}\twoheadrightarrow\fz_{r,\emptyset,L}$ and $\overline{\mathrm{Ad}(g_2^{-1})\psi}$ is the image of $\mathrm{Ad}(g_2^{-1})\psi\in \fb_L$ via $\fb_L\twoheadrightarrow\ft_L$.\;For $w\in \sW_{n,L}$,\;let  $\kappa_{\bP_{r,\emptyset},w},\kappa_{\bB,w}$ be the restriction of  $\kappa_{\bP_{r,\emptyset}},\kappa_{\bB}$ at  $X_{r,w}$.\;

Consider the affine $E$-scheme $\cT_{r,L}:=\fz_{r,\emptyset,L}\times_{\ft_{L}/\sW_{n,L}}\ft_{L}$.\;By \cite[Lemma 5.2.6]{Ding2021},\;the irreducible components of $\cT_{r,L}$ are given by $\cT_{r,w}:=\{(\mathrm{Ad}(w^{-1})z,z):z\in \fz_{r,\emptyset,L}\}$ for  $w\in \sW^{\Delta_n^k,\emptyset}_{n,L}$.\;We have a natural morphism
$\kappa=(\kappa_{\bP_{r,\emptyset}},\kappa_{\bB}):X_{r,L}\rightarrow \cT_{r,L}$.\;By \cite[Lemma 5.2.6]{Ding2021},\;we note that $X_{r,w}$ is the unique irreducible component of $X_{r,L}$ such $\kappa(X_{r,w})=\cT_{r,w}$.\;

The generalized Steinberg variety is defined by $Z_{r,L}=\kappa^{-1}_{\bB}(0)^{\red}=\kappa^{-1}_{\bP_{r,\emptyset}}(0)^{\red}$.\;Let $\cN$ (resp.,\;$\cN_L=\prod_{\tau\in \Sigma_L}\cN_{\tau}$ with $\cN_{\tau}\cong \cN$) be the  nilpotent cone in $\fg$ (resp.,\;$\fg_L$).\;Put 
\begin{equation}
	\widetilde{\cN}_{r,L}:=\prod_{\tau\in \Sigma_L}\widetilde{\cN}_{r,\tau},\;\widetilde{\cN}_{r,\tau}=\widetilde{\cN}_{r}:=\{(g\bP_{r,\emptyset},\psi)\in \GLN_n/\bP_{r,\emptyset}\times\fg\;|\;\mathrm{Ad}(g^{-1})\psi\in {\fn_{r,\emptyset}}\}
\end{equation}
If $r=1$,\;we put $\widetilde{\cN}_{L}:=\widetilde{\cN}_{1,L}$.\;As in \cite[Section 5.4]{Ding2021},\;we have the so-called Springer resolution (resp. generalized Springer resolution) $\widetilde{\cN}_{L}\rightarrow \cN_L$ (resp.,\;$\widetilde{\cN}_{r,L}\rightarrow \cN_L$).\;Then we have
\begin{equation}
	Z_{r,L}:=\prod_{\tau\in \Sigma_L}Z_{r,\tau},\;Z_{r,\tau}=Z_{r}:=(\widetilde{\cN}_{r}\times_{\cN}\widetilde{\cN})^{\red}.\;
\end{equation}
We have $Z_{r,L}\cong (\widetilde{\cN}_{r,L}\times_{\cN_L}\widetilde{\cN}_{L})^{\red}$.\;Moreover,\;by \cite[Corollary 5.4.1]{Ding2021},\;the irreducible components of $Z_{r}$ are index by $\{Z_{r,w}\}_{w\in \sW^{\Delta_n^k,\emptyset}_{n}}$. For $w=(w_\tau)_{\tau\in \Sigma_L}\in \sW_{n,L}$,\;we denote $Z_{r,w}=\prod_{\tau\in \Sigma_L}Z_{r,w_\tau}\hookrightarrow Z_{r,L}$.\;

Moreover,\;we set $\overline{X}_{r,w}:=\kappa_{\bP_{r,\emptyset},w}^{-1}(0)$.\;By the argument before \cite[Lemma 5.4.4]{Ding2021},\;$\overline{X}_{r,w}$ is equidimensional of dimension $\dim Z_{r,L}$,\;and each irreducible components of $\overline{X}_{r,w}$ is $Z_{r,w'}$ for some $w'\in\sW_{n,L}$.\;

All above schemes lie in the following commutative diagram:
\begin{equation}\label{LINVcompatiblemap2}
	\xymatrix{ Z_{r,L} \ar[r] \ar[d] &  X_{r,L}=\widetilde{\fg}_{r,L}\times_{\fg_L}\widetilde{\fg}_{L} \ar[d]^{q_{\bB}}   \\
		\widetilde{\cN}_{r,L} \ar[d]  \ar[r]	&  \widetilde{\fg}_{r,L} \ar[d]^{q_{\bP_{r,\emptyset}}}\\
		\cN_L \ar[r]	&  \fg_L,}
\end{equation}
where $q_{\bB}:X_{r,L}\rightarrow \widetilde{\fg}_{r,L}$ is the base change of $q_{\bB}:\widetilde{\fg}_{L}\rightarrow {\fg}_{L}$ via the morphism $q_{\bP_{r,\emptyset}}:\widetilde{\fg}_{r,L}\rightarrow {\fg}_{L}$.\;

%In the sequel,\;for $\ast\in\{\emptyset,L\}$,\;we denote  $\widetilde{\fg}_{\bP_{r,\emptyset},\ast}$ (resp.,\;$X_{\bP_{r,\emptyset},\ast}$,\;resp.,\;$Z_{\bP_{r,\emptyset},\ast}$)  by $\widetilde{\fg}_{r,\ast}$ (resp.,\;$X_{r,\ast}$,\;resp.,\;$Z_{r,\ast}$) for simplicity.\;

%\begin{rmk}\label{YwDfn}
%By \cite[Remark 5.2.3(2)]{Ding2021},\;let ${q^{-1}_{\bP_{r,\emptyset}}}(\fb)^0_w$ be the preimage of $\bB w\bP_{r,\emptyset}/\bP_{r,\emptyset}$ via the composition ${q^{-1}_{\bP_{r,\emptyset}}}(\fb)\hookrightarrow \GLN_n/\bP_{r,\emptyset}\twoheadrightarrow \GLN_n/\bP_{r,\emptyset}$.\;Let $Y_{r,w}$ be the reduced Zariski-closure of ${q^{-1}_{\bP_{r,\emptyset}}}(\fb)^0_w$ in ${q^{-1}_{\bP_{r,\emptyset}}}(\fb)$.\;It is an irreducible component of ${q^{-1}_{\bP_{r,\emptyset}}}(\fb)$.\;We have an isomorphism $X_{r,w}\xrightarrow{\sim}\GLN_{n}\times^{\bB}Y_{r,w}$.\;
%\end{rmk}
%\begin{rmk}\label{verdeforscheme}

%Consider the scheme $Y'_{r}:=\fz_{r,\emptyset}\times q_{\bP_{r,\emptyset}}^{-1}(\tau_{{r,\emptyset}})$,\;it is equidimensional of  dimension $k+d_L\big(k+\frac{n(n-r)}{2}\big)$.\;
%\end{rmk}

For any $J\subseteq  \Sigma_L$ and $w\in \sW^{\Delta_n^k,\emptyset}_{n,L}$,\;we put $\widetilde{\fg}_{r,J}:=\prod_{\tau\in J}\widetilde{\fg}_{r,\tau}$,\;$X_{r,J}:=\prod_{\tau\in J}X_{r,\tau}$ and $Z_{r,J}:=\prod_{\tau\in J}Z_{r,\tau}$.\;We also denote  $X_{r,J,w}=\prod_{\tau\in J}X_{r,w_\tau}\hookrightarrow X_{r,J}$ and $Z_{r,J,w}=\prod_{\tau\in J}Z_{r,w_\tau}\hookrightarrow Z_{r,J}$.\;

In particular,\;if $r=1$,\;we omit the subscript $r$ in above notation.\;

\subsection{Recollection of some groupoids}\label{sectionlocalmodel}

Keep the situation in Section \ref{Omegafil}.\;Recall that we have fixed  a $(\varphi,\Gamma)$-module $\Dpik$ (resp.,\;$\cM_\Dpik:=\Dpik[1/t]$) over $\cR_{E,L}$ of rank $n$,\;which admits an $\omepik$-filtration $\cF$ with  parameter $({\bx}_\pi,\bm{\delta}_{\bh})\in \sbanpik\times\rigchl$ or with parameter $(\widetilde{\bx}_{\pi,\bh},\widetilde{\bm{\delta}}_\bh)\in \sbanpik\times\rigch $.\;Since our parameter $({\bx}_\pi,\bm{\delta}_{\bh})\in \sbanpik\times\rigchl$ is not generic in the sense of \cite[(6.5)]{Ding2021},\;the groupoids in \cite[Section 6]{Ding2021} have different behavior.\;We study them in this section.\;Keep the notation in \cite[Section 6]{Ding2021}.\;

\subsubsection{Almost de Rham $(\varphi,\Gamma)$-modules}
We recall a little about Fontaine's theory of almost de Rham representations.\;Let $B_{\pdr}^+:=B_{\dr}^+[\log t]$ and $B_{\pdr}:=B_{\pdr}^+[1/t]$.\;The  $\gal_L$-action on $B_{\dr}$ extend uniquely to an action of  $\gal_L$-action on $B_{\pdr}$ with $g(\log t)=\log t+\log(\ccyc(g))$.\;Let $v_{\pdr}$ denote the unique $B_{\dr}$-derivation of $B_{\pdr}$ such that $v_{\pdr}(\log t)=-1$.\;Note that $v_{\pdr}$ and $\gal_L$ commute and both preserve $B_{\pdr}^+$.\;

Let $\mathrm{Rep}_{B_{\dr}}(\gal_L)$ (resp.,\;$\mathrm{Rep}_{B^+_{\dr}}(\gal_L)$) be the category of (free of finite rank) $B_{\dr}$-representations of $\gal_L$ (resp., $B^+_{\dr}$-representations) of $\gal_L$.\;If $W\in \mathrm{Rep}_{B_{\dr}}(\gal_L)$,\;let $D_{\pdr}(W):=(B_{\pdr}\otimes_{B_{\dr}}W)^{\gal_L}$, which is a finite-dimensional $L$-vector space of dimension no more than $\dim_{B_{\dr}}W$.\;The $B_{\dr}$-representation $W$ is called \textit{almost de Rham} if $\dim_LD_{\pdr}(W)=\dim_{B_{\dr}}W$.\;The $B^+_{\dr}$-representation $W^+$ is called \textit{almost de Rham} if $W^+[1/t]$ is almost de Rham.\;Let $\mathrm{Rep}_{\pdr}(\gal_L)$ be the category of almost de Rham $B_{\dr}$-representations $W$ of $\gal_L$.\;

Let $A\in \Art_E$ be a local Artinian $E$-algebra with the maximal ideal $\fm_A$.\;Let $\mathrm{Rep}_{\pdr,A}(\gal_L)$ be the category of almost de Rham $B_{\dr}$-representations $W$ of $\gal_L$ together with a morphism of $\bQ_p$-algebras $A\rightarrow \EndO_{\mathrm{Rep}_{\pdr}(\gal_L)}(W)$ such that $W$ is finite free over $B_{\dr}\otimes_{\bQ_p}A$.\;Let $\mathrm{Rep}_{A\otimes_{\bQ_p}L}(\bG_a)$ be the category of pairs $(V_A,\nu_A)$ where $\nu_A$ is a nilpotent endomorphism of a finite free $A\otimes_{\bQ_p}L$-module $V_A$.\;Then the functor $D_{\pdr}$ induces an equivalence of categories between $\mathrm{Rep}_{\pdr,A}(\gal_L)$ and $\mathrm{Rep}_{A\otimes_{\bQ_p}L}(\bG_a)$ (see \cite[Lemma 3.1.4]{breuil2019local}).\;

\subsubsection{Groupoids}

We recall some groupoids over $\Art_E$ that introduced in \cite[Section 6]{Ding2021}.\;Recall that the $(\varphi,\Gamma)$-module $\cM_\Dpik:=\Dpik[1/t]$ over $\cR_{E,L}[1/t]$  admits a $\omepik$-filtration $\cM_\bullet=(\cM_i)_{1\leq i\leq k}$ with $\cM_i:=\fil_{i}^{\cF}{\Delta_\Dpik}\big[\frac{1}{t}\big]$ with  parameter $({\bx}_\pi,\bm{\delta}_{\bh})\in \sbanpik\times\rigchl$.\;

Let $\bW_{\Dpik}=W_{\dr}(\cM_\Dpik)$ (resp,.\;$\bW^+_{\Dpik}:=W_{\dr}^+(\Dpik)$)  be the $\bB_\dr\otimes_{\bQ_p}E$-representation (resp.,\;$\bB_\dr^+\otimes_{\bQ_p}E$-representation) of $\gal_L$ associated to $\cM_\Dpik$.\;Moreover,\;the $\omepik$-filtration $\cM_\bullet$ on $\cM_{\Dpik}$ induces a $\bP_{r,\emptyset}$-filtration $\bF_{\bullet}=(\bF_{i}):=(W_\dr(\cM_{i}))$ on $\bW_{\Dpik}$ with $\bB_\dr\otimes_{\bQ_p}E$-subrepresentations of  $\bW_{\Dpik}$.\;For $1\leq i\leq k$,\;we put $\gr_i{\bF_{\bullet}}:=\bF_{i}/\bF_{i-1}$,\;so $\gr_i{\bF_{\bullet}}\cong (B_{\dR}\otimes_{\bQ_p}E)^{\oplus r}$ as $\gal_L$-representation.\;We recall certain groupoids of deformations of $\bW_{\Dpik}$ and $\bW^+_{\Dpik}$.\;

Let $X_{\bW_\Dpik}$  be the groupoid over $\Art_E$ of deformations
of $\bW_\Dpik$,\;i.e.,\;consists of triples $(A,W_A,\iota_A)$ where $A\in \Art_E$ and $W_A\in \mathrm{Rep}_{\pdr,A}(\gal_L)$ and $\iota_A:W_A\otimes_AL\xrightarrow{\sim} \bW_{\Dpik}$.\;A morphism $(A,W_A,\iota_A)\rightarrow (B,W_B,\iota_B)$ in $X_{\bW_\Dpik}$ is a morphism $A\rightarrow B$ in  $\Art_E$ and isomorphism $W_A\otimes_A B\xrightarrow{\sim} W_B$ compatible with $\iota_A$ and $\iota_B$.\;Fix an isomorphism $\alpha:(L\otimes_{\bQ_p}E)^n\xrightarrow{\sim} D_{\mathrm{pdR}}(\bW_\Dpik)$.\;Let  $X^{\Box}_{\bW_\Dpik}$ be the groupoid over $\Art_E$ of 
framed deformations of $\bW_\Dpik$,\;i.e.,\;consists of triples $(A,W_A,\iota_A,\alpha_A)$ where $(A,W_A,\iota_A)\in X_{\bW_\Dpik}$ and $\alpha_A:(A\otimes_{\bQ_p}E)^n\xrightarrow{\sim} D_{\mathrm{pdR}}(\bW_A)$ such that $\alpha_A$ modulo $\fm_A$ coincides with $\alpha$.\;

Let $X_{\bW_\Dpik,\bF_{\bullet}}$ be the groupoid over $\Art_E$ of deformations of $\bW_\Dpik$ together with the filtration $\bF_{\bullet}$ (see \cite[Section 6.1]{Ding2021}),\;i.e.,\;consists of triples $(A,W_A,\bF_{A,\bullet},\iota_A)$ where $(A,W_A,\iota_A)\in X_{\bW_\Dpik}$ and $\bF_{A,\bullet}=(\bF_{A,i})_{0\leq i\leq k}$ is a $\bP_{r,\emptyset}$-filtration of $W_A$ in $\mathrm{Rep}_{\pdr,A}(\gal_L)$ such that $\bF_{A,0}=0$ and $\bF_{A,i}/\bF_{A,i-1}$ for $1\leq i\leq k$ are free of $r$ over $B_{\dR}\otimes_{\bQ_p}A$ and is isomorphic to $\bF_{i}/\bF_{i-1}\otimes_{B_{\dR}\otimes_{\bQ_p}E}\epsilon_{A,i}$ for some rank one $B_{\dR}\otimes_{\bQ_p}A$ representation $\epsilon_{A,i}$ and $\iota_A$ induces $\bF_{A,\bullet}\otimes_A E\xrightarrow{\sim}\bF_{\bullet}$.\;For $1\leq i\leq k$,\;we put $\gr_i{\bF_{A,\bullet}}:=\bF_{A,i}/\bF_{A,i-1}$.\;Moreover,\;we put $X^0_{\bW_\Dpik}:=X_{\bW_\Dpik,\bF_{\bullet}}$ if the filtration $\bF_{\bullet}=(0\subseteq \bW_\Dpik)$ (i.e.,\;the trivial one),\;which is subgroupoid of $X_{\bW_\Dpik}$.\;We put $X^\Box_{\bW_\Dpik,\bF_{\bullet}}:=X^\Box_{\bW_\Dpik,\bF_{\bullet}}\times_{X_{\bW_\Dpik}}X_{\bW_\Dpik}^\Box$ and $X^{0,\Box}_{\bW_\Dpik,\bF_{\bullet}}:=X^0_{\bW_\Dpik}\times_{X_{\bW_\Dpik}}X_{\bW_\Dpik}^\Box$.\;

Let $\bW_\Dpik^+$ be the $\bB_\dr^+\otimes_{\bQ_p}E$-representation $W_\dr^+(\Dpik)$ of $\gal_L$ (see \cite[Section 6.3]{Ding2021}).\;We define groupoid $X_{\bW_\Dpik^+}$
over $\Art_E$ exactly as we define $X_{\bW_\Dpik}$ by replacing $W$,\;$W_A$ in $X_{\bW_\Dpik}$ by $W_\Dpik^+$,\;$W_A^+$ with $W_A^+$ an almost de Rham $A\otimes_{\bQ_p}B_{\dR}^+$-representation of $\gal_L$.\;We have natural morphism $X_{\bW_\Dpik^+} \rightarrow X_{\bW_\Dpik}$ (by inverting $t$).\;We put $X_{\bW_\Dpik^+}^\Box:=X_{\bW_\Dpik^+}{\times}_{X_{\bW_\Dpik}}X_{\bW_\Dpik}^\Box$,\;$X_{\bW_\Dpik^+,\bF_{\bullet}}:=X_{\bW_\Dpik^+}\times_{X_{\bW_\Dpik}}X_{\bW_\Dpik,\bF_{\bullet}}$ and  $X^\Box_{\bW_\Dpik^+,\bF_{\bullet}}:=X^\Box_{\bW_\Dpik^+,\bF_{\bullet}}\times_{X_{\bW_\Dpik}}X_{\bW_\Dpik}^\Box$.\;

Now let $\cD_{\bullet}=(\cD_{i})_{1\leq i\leq n}:=(D_{\pdr}(\bF_{j}))_{1\leq i\leq n}$,\;which is a complete flag of $D_{\pdr}(\bW_\Dpik)$.\;On the other hand,\;the $B_{\dr}^+$-lattice $\bW^+_\Dpik$ induces another complete flag 
\[\fil_{\bW^+_\Dpik,\bullet}:=\Big(\fil_{\bW^+_\Dpik,i}(D_{\pdr}(\bW_\Dpik))\Big)_{1\leq i\leq n},\]
of $D_{\pdr}(\bW_\Dpik)$ (see \cite[(3.5)]{breuil2019local}),\;with 
\begin{equation}
	\begin{aligned}
		\fil_{\bW^+_\Dpik,i}(D_{\pdr}(\bW_\Dpik)):=\bigoplus_{\tau\in\Sigma_L}\fil^{-\bh_{\tau,n+1-i}}_{\bW^+_\Dpik}(D_{\pdr,\tau}(\bW_\Dpik)):=\bigoplus_{\tau\in\Sigma_L}(t^{\bh_{\tau,n+1-i}}\bW^+_\Dpik)_{\tau}^{\gal_L}
	\end{aligned}
\end{equation}
where $D_{\pdr,\tau}(\bW_\Dpik):=D_{\pdr}(\bW_\Dpik)\otimes_{L\otimes_{\bQ_p}E}(L\otimes_{L,\tau}E)$ and $(t^{\bh_{\tau,n+1-i}}\bW^+_\Dpik)_{\tau}:=t^{\bh_{\tau,n+1-i}}\big(\bW^+_\Dpik\otimes_{L\otimes_{\bQ_p}E}(L\otimes_{L,\tau}E)\big)$.\;

Let $y$  be the closed point of the $E$-scheme $X_{r,L}$  corresponding to the triple $(\alpha^{-1}(\cD_{\bullet}),\;\alpha^{-1}(\fil_{\bW^+_\Dpik,\bullet}),N_{\bW_\Dpik})$ (by assumption,\;$N_{\bW_\Dpik}=0$).\;In the sequel,\;we write $y_1:=\pr_1y=(\alpha^{-1}(\cD_{\bullet}),N_{\bW_\Dpik})\in \widetilde{\fg}_{r,L}$,\;$y_2:=\pr_2y=(\alpha^{-1}(\fil_{\bW^+_\Dpik,\bullet}),N_{\bW_\Dpik})\in \widetilde{\fg}_{L}$ and $z=(N_{\bW_\Dpik})\in {\fg}$.\;We have:
\begin{itemize}
	\item by \cite[Corollary  3.1.6]{breuil2019local} (resp.,\;\cite[Theorem  3.2.5]{breuil2019local}),\;the groupoid $X^\Box_{\bW_\Dpik}\cong|X_{\bW_\Dpik}^\Box|$ (resp.,\;$X_{\bW_\Dpik^+}^\Box\cong |X_{\bW_\Dpik^+}^\Box|$) is pro-representable,\;and the functor $|X^\Box_{\bW_\Dpik}|$ (resp.,\;$|X_{\bW_\Dpik^+}^\Box|$) is pro-representated by the formal scheme $\widehat{\widetilde{\fg}}_{L,z}$ (resp.,\;$\widehat{\widetilde{\fg}}_{L,y_2}$);
	\item by \cite[Proposition 6.1.2]{Ding2021} (resp.,\;\cite[Proposition 6.3.2 (1)]{Ding2021}),\;the groupoid $X^\Box_{\bW_\Dpik,\bF_{\bullet}}\cong|X_{\bW_\Dpik,\bF_{\bullet}}^\Box|$ (resp.,\;$X^\Box_{\bW_\Dpik^+,\bF_{\bullet}}\cong |X^\Box_{\bW_\Dpik^+,\bF_{\bullet}}|$) is pro-representable,\;and the functor $|X^\Box_{\bW_\Dpik,\bF_{\bullet}}|$ (resp.,\;$|X^\Box_{\bW_\Dpik^+,\bF_{\bullet}}|$) is pro-representated by the formal scheme $\widehat{\widetilde{\fg}}_{r,L,y_1}$ (resp.,\;$\widehat{X}_{r,L,y}$).\;
\end{itemize}
For $w\in \sW_n$,\;we define the groupoid $X^{\Box,w}_{\bW^+_\Dpik,\bF_{\bullet}}:=X^{\Box}_{\bW^+_\Dpik,\bF_{\bullet}}\times_{|X^{\Box}_{\bW^+_\Dpik,\bF_{\bullet}}|}\widehat{X}_{r,w,y}$.\;Then the groupoid $X^{\Box,w}_{\bW^+_\Dpik,\bF_{\bullet}}$ over $\Art_E$ is pro-representable.\;The functor $|X^{\Box,w}_{\bW^+_\Dpik,\bF_{\bullet}}|$  is pro-represented by the formal scheme $\widehat{X}_{r,w,y}$.\;Let $X^{w}_{\bW^+_\Dpik,\bF_{\bullet}}$ be the image of $X^{\Box,w}_{\bW^+_\Dpik,\bF_{\bullet}}$ by the forgetful morphism  $X^{\Box}_{\bW^+_\Dpik,\bF_{\bullet}}\rightarrow X_{\bW^+_\Dpik,\bF_{\bullet}}$.\;By \cite[Proposition 6.3.4]{Ding2021},\;the morphism of groupoids $X^{w}_{\bW^+_\Dpik,\bF_{\bullet}}\rightarrow X_{\bW^+_\Dpik,\bF_{\bullet}}$,\;$X^{\Box,w}_{\bW^+_\Dpik,\bF_{\bullet}}\rightarrow X^{\Box}_{\bW^+_\Dpik,\bF_{\bullet}}$ are relatively representable and are closed immersions.\;
%\begin{rmk}
%In this paper,\;we will use the versal deformation space $X^{\mathrm{ver}}_{\bW_\Dpik,\bF_{\bullet}}$ of $X_{\bW_\Dpik,\bF_{\bullet}}$ to establish our local models (see (\ref{dfnwdfver})).\;We also consider the $\mathrm{ver}$-version $X^{\mathrm{ver}}_{\bW_\Dpik^+,\bF_{\bullet}}$,\;$X^{\mathrm{ver},w}_{\bW_\Dpik^+,\bF_{\bullet}}$ and so on.\;
%\end{rmk}

We then list some groupoids of $\omepik$-deformations of $\cM_{\Dpik}$ (or $\Dpik$).\;

We say $\cM_A$ over $\cR_{A,L}$ is of type $\omepik$ if there exists a filtration $\cM_{A,\bullet}=(\cM_{A,i})_{0\leq i\leq k}$ by  $(\varphi,\Gamma)$-submodule of $\cM_A$ over $\cR_{A,L}[1/t]$ such that $\cM_{A,0}=0$ and $\cM_{A,i}/\cM_{A,i-1}\cong \Delta_{x_i}\otimes_{\cR_{E,L}}\cR_{A,L}(\delta_{A,i})[1/t]$ for some continuous character $\delta_{A,i}:L^{\times}\rightarrow A^{\times}$ and $x_i\in \Spec \FZ_{\Omega_{r}}$.\;Such a filtration is called a $\omepik$-filtration,\;and $((x_i)_{1\leq i\leq k},(\delta_{A,i})_{1\leq i\leq k})\in \sbanpik\times\rigchl$ is called a parameter of $\cM_{A,\bullet}$.\;

As in \cite[Section 6.2]{Ding2021},\;we define the groupoid $X_{\cM_\Dpik,\cM_{\bullet}}$ over $\Art_E$ of $\omepik$-filtration of $\cM_\Dpik$,\;i.e.,\;consist of triples $(A,\cM_A,\cM_{A,\bullet},j_A)$ where $\cM_A$ is a $(\varphi,\Gamma)$-module over $\cR_{A,L}[1/t]$ of type $\omepik$ ,\;and $\cM_{A,\bullet}$ is an $\omepik$-filtration of $\cM_A$  such that $j_A:\cM_A\otimes_AE\xrightarrow{\sim}\cM_\Dpik$ is compatible with the filtrations.\;By \cite[Lemma 6.2.2]{Ding2021},\;for $(A,\cM_A,\cM_{A,\bullet},j_A)\in X_{\cM_\Dpik,\cM_{\bullet}}$,\;there exist unique characters $\underline{\delta}_A=(\delta_{A,i})_{1\leq i\leq k}:L^{\times}\rightarrow A^{\times}$ such that $\delta_{A,i}\equiv \bm{\delta}_{\bh,i}\;\Modo\;\fm_A$ and $((x_i)_{1\leq i\leq k},(\delta_{A,i})_{1\leq i\leq k})\in \sbanpik\times\rigchl$ is a parameter of $\cM_{A,\bullet}$.\;

Let $X_{\cM_\Dpik}$ be the groupoid over $\Art_E$ by forgetting everywhere the $\omepik$-filtrations in $X_{\cM_\Dpik,\cM_{\bullet}}$.\;The functor $W_\dr$ defines a morphism of groupoids  $X_{\cM_\Dpik}\rightarrow X_{\bW_\Dpik}$ and $X_{\cM,\cM_{\bullet}}\rightarrow X_{\bW_\Dpik,\bF_{\bullet}}$.\;We put $X_{\cM_\Dpik}^\Box:=X_{\cM_\Dpik}\times_{X_{\bW_\Dpik}}X_{\bW_\Dpik}^\Box$ and $X_{\cM_\Dpik,\cM_{\bullet}}^\Box:=X_{\cM_\Dpik,\cM_{\bullet}}\times_{X_{\bW_\Dpik}}X_{\bW_\Dpik}^\Box$.\;

Let $X_{\Dpik}$  be the groupoid over $\Art_E$ of deformations of $\Dpik$ (see \cite[Section 6.3]{Ding2021}).\;We have natural morphisms $X_{\Dpik}\rightarrow X_{\bW_\Dpik^+}$ (induced by the functor $W_\dr^+(-)$),\;$X_{\Dpik}\rightarrow X_{\cM_\Dpik}$ (by inverting $t$).\;Note that we have a natural morphism $X_{\Dpik}\rightarrow X_{\bW_\Dpik^+}\times_{X_{\bW_\Dpik}}X_{\cM_\Dpik}$,\;which is an equivalence by \cite[Proposition 3.5.1]{breuil2019local}.\;We  put 
$X_{\Dpik}^\Box:=X_{\Dpik}\times_{X_{\bW_\Dpik}}X_{\bW_\Dpik}^\Box$, and $X_{\Dpik,\cM_{\bullet}}^\Box:=X_{\Dpik,\cM_{\bullet}}\times_{X_{\bW_\Dpik}}X_{\bW_\Dpik}^\star$.\;For $w\in \sW_n$,\;We put
\begin{equation}
	\begin{aligned}
		&X^{\star,w}_{\Dpik,\cM_{\bullet}}:=X^\Box_{\Dpik,\cM_{\bullet}}\times_{X^{\Box}_{\bW^+_\Dpik,\bF_{\bullet}}}X^{\Box,w}_{\bW^+_\Dpik,\bF_{\bullet}},X^{w}_{\Dpik,\cM_{\bullet}}:=X_{\Dpik,\cM_{\bullet}}\times_{X_{\bW^+_\Dpik,\bF_{\bullet}}}X^{w}_{\bW^+_\Dpik,\bF_{\bullet}}.\;
	\end{aligned}
\end{equation}
By \cite[Proposition 6.3.4]{Ding2021},\;the morphism of groupoids  $X^{w}_{\Dpik,\cM_{\bullet}}\rightarrow X_{\Dpik,\cM_{\bullet}}$ and $X^{\Box,w}_{\Dpik,\cM_{\bullet}}\rightarrow X^{\Box}_{\Dpik,\cM_{\bullet}}$ are relatively representable and are closed immersions.\;

Let $\rho_L:\gal_L\rightarrow \GLN_n(E)$ be a continuous group morphism and let $V(\rho_L)$ be the associated representation of $\rho_L$.\;Supppse $\Dpik=D_{\rig}(V(\rho_L))$.\;Recall that $X_{\rho_L}$ denotes the groupoid over $\Art_E$ of deformations of the group morphism $\rho_L$.\;Let $X_{V(\rho_L)}$ be the groupoid over $\Art_E$ of deformations of the representation $V(\rho_L)$.\;We can identity $X_{\rho_L}$ with the framed deformations of $V(\rho_L)$.\;Therefore the morphisms $X_{\rho_L}\rightarrow X_{V(\rho_L)}$ is relatively representable and formally smooth of relative dimension $n^2$.\;We also have an equivalence $X_{V(\rho_L)}\xrightarrow{\sim} X_{\Dpik}$ (induced by the $D_{\rig}(-)$-functor) and $X_{\rho_L} \xrightarrow{\sim} |X_{\rho_L}|$.\;We put \[X_{V(\rho_L),\cM_{\bullet}}:=X_{V(\rho_L)}\times_{X_{\Dpik}}X_{\Dpik,\cM_{\bullet}},\;X_{\rho_L,\cM_{\bullet}}:=X_{\rho_L}\times_{X_{\Dpik}}X_{\Dpik,\cM_{\bullet}}.\]
Then $X_{\rho_L,\cM_{\bullet}}\rightarrow X_{\rho_L}$  is a closed immersion by base change.\;For $w\in \sW_{n,L}$,\;we define 
\begin{equation}
	\begin{aligned}
		& X^{w}_{V(\rho_L),\cM_{\bullet}}:=X_{V(\rho_L)}\times_{X_{\Dpik}}X^{w}_{\Dpik,\cM_{\bullet}},\;
		X^{w}_{\rho_L,\cM_{\bullet}}:=X_{\rho_L}\times_{X_{\Dpik}}X^{w}_{\Dpik,\cM_{\bullet}},\;\hspace{100pt}\\
		%		& X^{\star,w}_{V(\rho_L),\cM_{\bullet}}:=X_{V(\rho_L)}\times_{X_{\Dpik}}X^{\star,w}_{\Dpik,\cM_{\bullet}}\cong X^{w}_{V(\rho_L),\cM_{\bullet}}\times_{X_{\cM_\Dpik,\cM_{\bullet}}}X^{\star}_{\cM,\cM_{\bullet}},\;\\
		%		& X^{\star,w}_{\rho_L,\cM_{\bullet}}:=X_{V(\rho_L)}\times_{X_{\Dpik}}X^{\star,w}_{\Dpik,\cM_{\bullet}}\cong X^{w}_{\rho_L,\cM_{\bullet}}\times_{X_{\cM_\Dpik,\cM_{\bullet}}}X^{\star}_{\cM,\cM_{\bullet}}.
	\end{aligned}
\end{equation}
and their $\Box$-versions $X^{\star,w}_{V(\rho_L),\cM_{\bullet}}$,\;$X^{\star,w}_{\rho_L,\cM_{\bullet}}$.\;In next section,\;in order to study groupoids $X^{\bullet,\star}_{\ast,\cM_{\bullet}}$ for $\ast\in \{\cM_{\Dpik},\Dpik,\rho_L\}$,\;$\bullet\in\{\Box,\emptyset\}$ and $\star\in\{w,\emptyset\}$,\;we introduce certain groupoid  $\widehat{X}^{(\varphi,\Gamma)}_{\bW_\Dpik,\bF_{\bullet}}$.\;

%Note that $X_{\Dpik,\cM_{\bullet}}^\Box:=X_{\Dpik}\times_{X_{\bW_\Dpik^+}}X_{\bW_\Dpik^+,\bF_{\bullet}}$ and $X_{\bW_\Dpik^+,\bF_{\bullet}}:=X_{\bW_\Dpik^+}\times_{X_{\bW_\Dpik}}X_{\bW_\Dpik,\bF_{\bullet}}^\Box$.\;

%We define  the groupoid $X'_{\cM,\cM_{\bullet}}$ over $\Art_E$ of objects $\{(\cM_A,j_A,\cM_{A,\bullet}')\}$,\;where $A\in \Art_E$,\;$\cM_A$ is a $(\varphi,\Gamma)$-module of rank $n$ over $\cR_{A,L}$ which admits an increasing   $\cM_{A,\bullet}'$ of $(\varphi,\Gamma)$-module over $\cR_{A,L}[\frac{1}{t}]$ on $\cM_A$,\;the isomorphism $j_A:\cM_A\otimes_AE\cong \cM_\Dpik$ satisfies that  $j_A:\cM_{A,i}\otimes_AE\cong \cM_{i}$ is an isomorphism.\;Let $\Delta$ be an irreducible $(\varphi,\Gamma)$-module of rank $r$ over $\cR_{E,L}$,\;de Rham of constant Hodge-Tate weight $0$.\;Assume that $M=\Delta\otimes_{\cR_{E,L}}\cR_{E,L}({\delta})[\frac{1}{t}]$  for some continuous character $\delta: L^\times\rightarrow E^\times$.\;We consider the functor $F_{M}^0$ over $\Art(E)$ which sends $A$ to the set of isomorphism classes $\{(M_A,\pi_A)\}/\sim$,\;such that $M_A\cong {\Delta}\otimes_{\cR_{E,L}}\cR_{A,L}({\delta}_{A})$ is a $(\phi,\Gamma)$-module  of rank $r$ with over $\cR_{A,L}$ with $\pi_A:M_A\otimes_AE\cong M$.\;By \cite[Lem.,\;6.2.2]{Ding2021},\;we see that there exists a unique character $\delta_A$ such that ${\delta}_{A}\equiv \delta (\mod \mathfrak{m}_A)$.\;By definition,\;we have $$|X_{\cM_\Dpik,\cM_{\bullet}}|\cong |X_{\cM_\Dpik,\cM_{\bullet}}'|\times_{\prod_{ir\in \Delta_n(k)}|X_{\gr^i\cM}|}\prod_{ir\in \Delta_n(k)}F_{\gr^i\cM}^0.\;$$

\subsection{Variation of local models and its geometry}\label{sectionlocalmodelmainresult}

This section aim to study $X^\Box_{\cM_\Dpik,\cM_{\bullet}}$ (for $\star\in\{\Box,\ver\}$) and find its local model.\;More precisely,\;recall that we have a morphism of groupoids over $\Art_E$:
\begin{equation}
	\omega_{\bm{\delta}_{\bh}}:X_{\cM_\Dpik,\cM_{\bullet}}\rightarrow \widehat{(\cZ_{\bL_{r,\emptyset},L})}_{\bm{\delta}_{\bh}}
\end{equation}
by setting the triple $(A,\cM_A,\cM_{A,\bullet},j_A)$ to the parameter of $\cM_{A,\bullet}$ given in \cite[Lemma 6.2.2]{Ding2021} (in precise,\;$(A,\cM_A,\cM_{A,\bullet},j_A)\in X_{\cM_\Dpik,\cM_{\bullet}}$,\;there exist unique characters $\underline{\delta}_A=(\delta_{A,i})_{1\leq i\leq k}:L^{\times}\rightarrow A^{\times}$ such that $\delta_{A,i}\equiv \bm{\delta}_{\bh,i}\;\Modo\;\fm_A$ and $((x_i)_{1\leq i\leq k},(\delta_{A,i})_{1\leq i\leq k})\in \sbanpik\times\rigchl$ is a parameter of $\cM_{A,\bullet}$).\;Let $\widehat\fz_{r,\emptyset,L}$ be the completion of $\fz_{r,\emptyset,L}$ at $0$.\;By the diagram in \cite[Proposition 6.2.3,\;Theorem\;6.2.6]{Ding2021},\;we have a morphism,\;i.e.,\;the so-called \textit{local model map} (induced by $\omega_{\bm{\delta}_{\bh}}$ and the natural morphism $X_{\cM_\Dpik,\cM_{\bullet}}\rightarrow X_{\bW_\Dpik,\bF_{\bullet}}$):
\begin{equation}\label{localmodelmap}
\Upsilon:X_{\cM_\Dpik,\cM_{\bullet}}\rightarrow \widehat{(\cZ_{\bL_{r,\emptyset},L})}_{\bm{\delta}_{\bh}}\times_{\widehat\fz_{r,\emptyset,L}} X_{\bW_\Dpik,\bF_{\bullet}}.\;
\end{equation}

To analysis this morphism $\Upsilon$,\;we begin with some computations on the cohomology of $(\varphi,\Gamma)$-modules over $\cR_{E,L}[1/t]$.\;

For $1\leq i,j\leq k$,\;denote by $\cM_{i,j}^0:=\EndO(\Delta_{\pi})\tee\cR_{E,L}(\bm{\delta}_{\bh,i}^{-1}\bm{\delta}_{\bh,j})$ and $\cM_{i,j}:=\cM_{i,j}^0[1/t]$.\;In particular,\;we have $\cM_{i,i-1}\cong \EndO(\Delta_{\pi})\tee \cR_{E,L}(\bm{\delta}_{\bh,i}^{-1}\bm{\delta}_{\bh,i-1})[1/t]=\EndO(\Delta_{\pi})\tee \cR_{E,L}(\ccyc)[1/t]$,\;which is not dependent on $2\leq i\leq k$.\;
\begin{lem}\label{dimlemmaCM}\hspace{20pt}\\
	(1) For any $j\neq i-1,i$,\;we have $\hH^0_{(\varphi,\Gamma)}(\cM_{i,j})=\hH^2_{(\varphi,\Gamma)}(\cM_{i,j})=0$,\;and $\dim_E\hH^1_{(\varphi,\Gamma)}(\cM_{i,j})=d_Lr^2$.\;\\
	(2) For $2\leq i\leq k$,\;we also have $\hH^0_{(\varphi,\Gamma)}(\cM_{i,i-1})=\hH^2_{(\varphi,\Gamma)}(\cM_{i,i-1})=0$,\;and $\dim_E\hH^1_{(\varphi,\Gamma)}(\cM_{i,j})=d_Lr^2$.\;
%	(3) For $1\leq i\leq k$,\;we have $\hH^0_{(\varphi,\Gamma)}(\cM_{i,i})\cong E$,\; $\hH^2_{(\varphi,\Gamma)}(\cM_{i,i})=0$,\;and $\dim_E\hH^1_{(\varphi,\Gamma)}(\cM_{i,i})=1+d_Lr^2$.\;
\end{lem}
\begin{proof}Part (1) follows from  \cite[Lemma 6.2.5]{Ding2021},\;since the parameter of $\cM_{i,j}$ is generic in the sense of \cite[(6.5)]{Ding2021}.\;We prove $(2)$.\;We deduce from the morphism $t^{-N}\cM_{i,i-1}^0\hookrightarrow t^{-N-1}\cM_{i,i-1}^0$ a long exact sequence
	\begin{equation}\label{ccccccc}
		\begin{aligned}
			0 \rightarrow\; &\hH^0_{(\varphi,\Gamma)}(t^{-N}\cM_{i,i-1}^0)  \rightarrow 
			\hH^0_{(\varphi,\Gamma)}(t^{-N-1}\cM_{i,i-1}^0)  \rightarrow 
			\hH^0_{(\varphi,\Gamma)}(t^{-N-1}\cM_{i,i-1}^0/t^{-N}\cM_{i,i-1}^0)  \rightarrow \\
			&\hH^1_{(\varphi,\Gamma)}(t^{-N}\cM_{i,i-1}^0)  \rightarrow 
			\hH^1_{(\varphi,\Gamma)}(t^{-N-1}\cM_{i,i-1}^0)  \rightarrow 
			\hH^1_{(\varphi,\Gamma)}(t^{-N-1}\cM_{i,i-1}^0/t^{-\sigma_N}\cM_{i,i-1}^0)  \rightarrow \\
			&\hH^2_{(\varphi,\Gamma)}(t^{-N}\cM_{i,i-1}^0)  \rightarrow 
			\hH^2_{(\varphi,\Gamma)}(t^{-N-1}\cM_{i,i-1}^0)  \rightarrow 
			\hH^2_{(\varphi,\Gamma)}(t^{-{N-1}}\cM_{i,i-1}^0/t^{-N}\cM_{i,i-1}^0)  \rightarrow  0.
		\end{aligned}
	\end{equation}
	By \cite[Theorem 4.7]{liu2007cohomology},\;we see that $\hH^2_{(\varphi,\Gamma)}(t^{-N-1}\cM_{i,i-1}^0/t^{-N}\cM_{i,i-1}^0)=0$ and 
	\[\dim_E\hH^0_{(\varphi,\Gamma)}(t^{-N-1}\cM_{i,i-1}^0/t^{-N}\cM_{i,i-1}^0)=\dim_E\hH^1_{(\varphi,\Gamma)}(t^{-N-1}\cM_{i,i-1}^0/t^{-N}\cM_{i,i-1}^0)<\infty.\;\]
	By \cite[Lemma 5.1.1]{breuil2020probleme},\;we have 
	\begin{equation}
		\begin{aligned}
			\hH^0_{(\varphi,\Gamma)}(t^{-N-1}&\cM_{i,i-1}^0/t^{-N}\cM_{i,i-1}^0)\cong 
			\hH^0(\gal_L,t^{-N-1}W_{\dr}^+(\cM_{i,i-1}^0)/t^{-\sigma_N}W_{\dr}^+(\cM_{i,i-1}^0))\\
			&\cong \hH^0\big(\gal_L,t^{-N-1}W_{\dr}^+(\cR_{E,L}(\bm{\delta}_{\bh,i}^{-1}\bm{\delta}_{\bh,i-1}))/t^{-N}W_{\dr}^+(\cR_{E,L}(\bm{\delta}_{\bh,i}^{-1}\bm{\delta}_{\bh,i-1}))\big)^{\oplus r^2}.\\
		\end{aligned}
	\end{equation}
	By \cite[Lemma 2.16]{nakamura2009classification},\;the latter is when $N$ is sufficiently large.\;Then by (\ref{ccccccc}) and \cite[(3.11)]{breuil2019local},\;we get that $\hH^j_{(\varphi,\Gamma)}(\cM_{i,j})=\hH^j_{(\varphi,\Gamma)}(t^{-N}\cM_{i,j}^0)$ for sufficiently large $N$.\;Since the parameter of $t^{-N}\cM_{i,j}^0$ is generic in the sense of \cite[(4.13)]{Ding2021} for sufficiently large $N$,\;Part $(2)$ follows from \cite[Lemma 4.1.12]{Ding2021}.\;
	%For (3),\;we apply the above discussion to $\cM_{i,i}^0$.\;We deduce that $\hH^s_{(\varphi,\Gamma)}(t^{-N}\cM_{i,i})=\hH^s_{(\varphi,\Gamma)}(t^{-{N+1}}\cM_{i,i}^0)$ for $N\geq 0$ and $s=0,1,2$ by a similar long exact sequence.\;Note that 
%	\begin{equation}
%		\hH^i_{(\varphi,\Gamma)}(t^{-{N}}\cM_{i,i}^0)=\left\{
%		\begin{array}{ll}
%			0, & \hbox{$i=2$;} \\
%			E, & \hbox{$i=0.$}
%		\end{array}
%		\right.
%	\end{equation}
%	for all $N\geq 0$.\;We deduce that $\hH^0_{(\varphi,\Gamma)}(\cM_{i,i})\cong E$ and $\hH^2_{(\varphi,\Gamma)}(\cM_{i,j})=0$.\;Then $\dim_E\hH^1_{(\varphi,\Gamma)}(\cM_{i,i})=1+d_Lr^2$.\;
\end{proof}

\begin{lem}\label{lemmaMM} The morphism $|X_{\cM_\Dpik,\cM_{\bullet}}|\rightarrow |X_{\cM_\Dpik}|$ is relatively representable,\;and $|X_{\cM_\Dpik,\cM_{\bullet}}|$ is a subfunctor of $|X_{\cM_\Dpik}|$.\;Therefore the morphism $X_{\cM_\Dpik,\cM_{\bullet}}\rightarrow X_{\cM_\Dpik}$ of groupoids is relatively representable and is a closed immersion.\;
\end{lem}
\begin{proof}We first to show $|X_{\cM_\Dpik,\cM_{\bullet}}|$ that is a subfunctor  $|X_{\cM_\Dpik}|$,\;i.e.,\;the $\omepik$-filtration  $\cM_{A,\bullet}$ deforming $\cM_{\bullet}$ on a deformation $\cM_A$ is unique.\;This follows from the similar argument in  \cite[Lemma 4.1.14]{Ding2021}.\;The proof proceeds by induction on the length of $\cM_{A,\bullet}$,\;we should show that $\cM_{A,\bullet}$ is an $\Omega$-filtration on $\cM_A$,\;then $\cM_{A,1}$ is uniquely determined as a $(\varphi,\Gamma)$-submodule of $\cM_{A}$.\;Now suppose that $\widetilde{\cM}_{A,1}$ is another $(\varphi,\Gamma)$-submodule of $\cM_{A}$ deforming $\cM_{1}$.\;Observe that $\widetilde{\cM}_{A,1}$ (resp.,\;$\cM_A/\cM_{A,1}$) is a successive extension of $\cM_{1}$ (resp.,\;$\cM/\cM_{1}$),\;then by Lemma \ref{dimlemmaCM},\;we deduce $\homo_{(\varphi,\Gamma)}(\widetilde{\cM}_{A,1},\cM_A/\cM_{A,1})=0$.\;Therefore,\;we see that $\widetilde{\cM}_{A,1}\subset \cM_{A,1}$.\;Then we see that $\widetilde{\cM}_{A,1}=\cM_{A,1}$ since they have the same rank.\;On the other hand,\;by the same argument as in \cite[Proposition 6.2.8]{Ding2021},\;we have an equivalence of groupoids over $\Art_E$:
	\[X_{\cM_\Dpik,\cM_{\bullet}}\xrightarrow{\sim}X_{\cM_\Dpik}\times_{ |X_{\cM_\Dpik}|}|X_{\cM_\Dpik,\cM_{\bullet}}|,\]
and $|X_{\cM_\Dpik,\cM_{\bullet}}|\hookrightarrow |X_{\cM_\Dpik}|$ is still relatively representable.\;The last statement is a direct consequence of the first assertion.\;
\end{proof}

Using the same argument as in the proof of \cite[Proposition 6.2.10]{Ding2021},\;we deduce:

\begin{pro}\label{proprexmm}
	The groupoid $X^{\star}_{\cM_\Dpik,\cM_{\bullet}}$ (for $\star\in\{\Box,\ver\}$) over $\Art_E$ is pro-representable.\;The functor $X^\Box_{\cM_\Dpik,\cM_{\bullet}}$ (resp.,\;$X^{\mathrm{ver}}_{\cM_\Dpik,\cM_{\bullet}}$) is pro-represented by a formally smooth noetherian complete local ring of residue field $E$ and dimension $d_L\big(n^2+k+\frac{n(n-r)}{2}\big)$ (resp.,\;$k+d_L\big(k+\frac{n(n-r)}{2}\big)$).\;
\end{pro}
\begin{proof}By the same argument as in \cite[Lemma 6.2.9]{Ding2021},\;we can show that $X_{\cM_\Dpik,\cM_{\bullet}}\rightarrow X_{\bW_\Dpik,\bF_{\bullet}}$ (and thus $X^\Box_{\cM_\Dpik,\cM_{\bullet}}\rightarrow X^\Box_{\bW_\Dpik,\bF_{\bullet}}$) is relatively representable.\;Since $X^\Box_{\bW_\Dpik,\bF_{\bullet}}$ is pro-representable,\;so the first statement follows.\;For the second assertion,\;we recall the groupoids $X_{\cM_\Dpik,\cM_{\bullet}}^{\mathrm{ver}}$ and $X_{\cM_\Dpik,\cM_{\bullet}}^{\mathrm{ver},\Box}$ defined in the proof of \cite[Proposition 6.2.10]{Ding2021}.\;Recall that $X_{\cM_\Dpik,\cM_{\bullet}}^{\mathrm{ver}}\cong |X_{\cM_\Dpik,\cM_{\bullet}}^{\mathrm{ver}}|$.\;We first show that $|X_{\cM_\Dpik,\cM_{\bullet}}^{\mathrm{ver}}|$ is pro-representable.\;It is clear that $|X_{\cM_{\Dpik,1},\cM_{\bullet}}^{\mathrm{ver}}|$ is pro-represented by $\widehat{\cO}_{\widehat{L^{\times}},\bm{\delta}_{\bh,1}}\cong E[[x_1,\cdots,x_{d_L+1}]]$.\;Now assume that $|X_{\cM_{\Dpik,i-1},\cM_{\bullet}}^{\mathrm{ver}}|$ is pro-represented by a formally smooth noetherian complete local ring $R_{i-1}$ of residue field $E$  and dimension $i-1+d_L(i-1+r^2\frac{i(i-1)}{2})$.\;Let $S_{i}$ denoted the completion of $R_{i-1}\otimes_E\widehat{\cO}_{\widehat{L^{\times}},\bm{\delta}_{\bh,i}}$ with respect to the maximal ideal generated by the maximal ideal of $R_{i-1}$ and the one of $\widehat{\cO}_{\widehat{L^{\times}},\bm{\delta}_{\bh,i}}$.\;For any morphism $S_{i}\rightarrow A$ with $A\in \Art_E$,\;let $\cM_{i-1,A}$ be the $(\varphi,\Gamma)$-module over $\cR_{A,L}[1/t]$ given by the pull-back along $R_{i-1}\rightarrow S_{i}\rightarrow A$ of the universal $(\varphi,\Gamma)$-module over $\cR_{R_{i-1},L}[1/t]$ and let $\delta_{A,i}$ be the character $L^{\times}\rightarrow \widehat{\cO}_{\widehat{L^{\times}},\bm{\delta}_{\bh,i}}\rightarrow S_{i}\rightarrow A$.\;Let 
\begin{equation}\label{dfnNi}
N_{i}:=\varprojlim_{S_{i}\rightarrow A}\ext^1_{(\varphi,\Gamma)}\Big(\Delta_\pi\otimes_{\cR_{A,L}}\cR_{A,L}(\delta_{A,i})\Big[\frac{1}{t}\Big],\cM_{i-1,A}\Big).
\end{equation}
By Lemma \ref{dimlemmaCM},\;we deduce that $N_{i}$ is a free $S_{i}$-module of rank $(i-1)d_Lr^2$ (use Lemma \ref{dimlemmaCM} (1) and a d\'{e}vissage).\;Then $[\cM_{\Dpik,i}]$ correspondences to a maximal ideal $\fm_{i}$ with residue field $E$ of the polynomial $S_i$-algebra $\mathrm{Symm}_{S_{i}}N_{i}^{\vee}$.\;Let $R_{i}$ be the completion of  $\mathrm{Symm}_{S_{i}}N_{i}^{\vee}$ at $\fm_{i}$.\;One can check that $X^{\mathrm{ver}}_{\cM_i,\cM_{\bullet}}$ is pro-represented by $R_i$.\;In particular,\;$X^{\mathrm{ver}}_{\cM_\Dpik,\cM_{\bullet}}$ is pro-represented by a formally smooth noetherian complete local ring of residue field $E$ and dimension $k+d_L\big(k+\frac{n(n-r)}{2}\big)$.\;By the same argument as in the last paragraph in the proof of \cite[Proposition 6.2.10]{Ding2021},\;we obtain the dimension of  $X^\Box_{\cM_\Dpik,\cM_{\bullet}}$.\;
\end{proof}

%let $\bW_\pi$ (resp.\;$\bW_\pi^0$) be the $E$-$B$-pair associated with the $\EndO(\Delta_{\pi})$ (resp.\;$\EndO(\Delta_{\pi})^0=\EndO(\Delta_{\pi})/\cR_{E,L}$),\;. We also put $\bW_\pi(\delta):=\bW_\pi\otimes_E B_E(\delta))$ for any continuous character $\delta:L^\times \rightarrow E^\times$.\;

Suppose $y_A:=(\cM_{\Dpik,A},\cM_{A,\bullet},j_A)$  is an object in $X_{\cM_\Dpik,\cM_{\bullet}}(A)$,\;and  $\underline{\delta}_A=(\delta_{A,i})_{1\leq i\leq k}:L^{\times}\rightarrow A^{\times}\in \widehat{(\cZ_{\bL_{r,\emptyset},L})}_{\bm{\delta}_{\bh}}(A)$ is a parameter of  $\cM_{A,\bullet}$ (note that $\delta_{A,i}:L^{\times}\rightarrow A^{\times}$ are continuous characters such that $\delta_{A,i}\equiv \bm{\delta}_{\bh,i}\;\Modo\;\fm_A$).\;Let $\cM_{A,i,j}=\EndO(\Delta_{\pi})\tee \cR_{E,L}({\delta}_{A,i}^{-1}{\delta}_{A,j})[1/t]$.\;For $J\subseteq\Sigma_L$,\;recall 
\[\hH^1_{g,J}(\cM_{A,i,j}):=\Ker\big[\hH^1_{(\varphi,\Gamma)}(\cM_{A,i,j})\rightarrow \hH^1(\gal_L,W_{\dr,J}(\cM_{A,i,j}))\big].\]
In particular,\;$\hH^1_{g,\Sigma_L}(\cM_{A,i,j})=\hH^1_{g}(\cM_{A,i,j})$.\;

\begin{lem}\label{dimkerCMWdr}Let $1\leq i,j\leq k$ and $J\subsetneqq\Sigma_L$.\;
		\begin{itemize}
			\item[(1)] If $j\neq i-1,i$,\;the natural morphism $\hH^1_{(\varphi,\Gamma)}(\cM_{i,j})\rightarrow \hH^1(\gal_L,W_{\dr}(\cM_{i,j}))$ is an isomorphism.\;
			\item[(2)]  If $j=i-1$,\;we have $\dim_E\hH^1_{g}(\cM_{i,j})=1$ and $\dim_E\hH^1_{g,J}(\cM_{i,j})=(d_L-|J|)r^2$.\;Moreover,\;the natural morphism $\hH^1_{(\varphi,\Gamma)}(\cM_{i,j})\rightarrow \hH^1(\gal_L,W_{\dr,J}(\cM_{i,j}))$ is a surjection.\;
%			\item[(3)] The natural morphism .\;
			\item[(3)] Then the natural morphism
			\begin{equation}\label{surtwomorphisms}
			\begin{aligned}
				&\hH^1_{(\varphi,\Gamma)}(\cM_{A,i,j})\rightarrow \hH^1(\gal_L,W_{\dr}(\cM_{A,i,j})),j\neq i-1,i\\
				&\hH^1_{(\varphi,\Gamma)}(\cM_{A,i,i-1})\rightarrow \hH^1(\gal_L,W_{\dr,J}(\cM_{A,i,i-1}))
			\end{aligned}
			\end{equation}
		are surjectvie.\;
		\end{itemize}	
	\end{lem}
	\begin{proof}When $N$ is sufficiently large,\;it is enough to study the map 
	$$\hH^1_{(\varphi,\Gamma)}(t^{-N}\cM_{i,j}^0)\rightarrow \hH^1(\gal_L,W_{\dr}(t^{-N}\cM_{i,j}^0))\rightarrow \hH^1(\gal_L,W_{\dr}(\cM_{i,j})).$$In this case,\;we have $t^{-N}\cM_{i,j}^0=\EndO(\Delta_\pi)\otimes_{\cR_{E,L}}\cR_{E,L}(\unr(q_L^{j-i})t^{\bk})$ for some $\bk\in \BZ^{d_L}_{<0}$.\;So the first one is \cite[Proposition 6.2.5 (2)]{Ding2021}.\;By \cite{nakamura2009classification},\;we translate to the language of $E$-$B$-pairs.\;Let ${\bW_{\pi,N}^{i,j}}$ be the $E$-$B$-pair associated to $t^{-N}\cM_{i,j}^0$,\;we have $\widetilde{H}^2_J(\gal_L,{\bW_{\pi,N}^{i,j}})=0$ for $J\subsetneqq\Sigma_L$ and 
$\dim_E\widetilde{H}^2_{\Sigma_L}(\gal_L,{\bW_{\pi,N}^{i,j}})=1$.\;Then $(2)$ follows from the dimension formula $\dim_E\hH^1_{g,J}(\cM_{i,j})=d_Lr^2+\dim_E\hH^0(\gal_L,\cM_{i,j})+\dim_E\widetilde{H}^2_J(\gal_L,{\bW_{\pi,N}^{i,j}})-\dim_E\hH^0(\gal_L,W_{\dR,J}(\cM_{i,j}))$ in \cite[Proposition A.3]{CompanionpointforGLN2L}.\;Finally,\;the surjectivity of the first (resp.,\;second) morphism in (\ref{surtwomorphisms}) follows from $(1)$ (resp.,\;$(2)$) together with the fact that the functor $W\mapsto \hH^1(\gal_L,W)$ on  $W\in \mathrm{Rep}_{\pdr,E}(\gal_L)$ is right exact.\;
		\end{proof}
\begin{rmk}
If $L\neq \bQ_p$ and $r=1$,\;Part $(2)$ implies that the natural inclusion $\hH^1_{g}(\cM_{i,j})\subseteq \hH^1_{g,\Sigma_L\backslash\tau}(\cM_{i,j})$ induces an isomorphism for any $\tau\in\Sigma_L$.\;
\end{rmk}
%	Then  $\hH^1_{g}(\cM_{i,j})\cong\hH^1_g(\bW_{\pi,N}^{i,j})$ in \cite[Definition 2.4]{nakamura2009classification}.\;It follows from \cite[Definition 2.11]{nakamura2009classification} that $\dim_E\hH^1_g(\bW_\pi^{i,j})=d_Lr^2+\dim_E\widetilde{\hH}^2_g(\bW_\pi^{i,j})-d_Lr^2=1$.\;In a similar way,\;we have $\dim_E\hH^1_{g,J}(\bW_\pi^{i,j})=(d_L-|J|)r^2$ if $J\subsetneqq\Sigma_L$.\;This proves $(2)$.\;

Lemma \ref{dimkerCMWdr} implies that the local model map $\Upsilon$ in (\ref{localmodelmap})  is not longer formally smooth in our case.\;To overcome this problem,\;we introduce certain subgroupoid $X^{(\varphi,\Gamma)}_{\bW_\Dpik,\bF_{\bullet}}$ of $\widehat{(\cZ_{\bL_{r,\emptyset},L})}_{\bm{\delta}_{\bh}}\times_{\widehat\fz_{r,\emptyset,L}} X_{\bW_\Dpik,\bF_{\bullet}}$ in Section \ref{maincontr}.\;We try to analysis the geometry via some power formall series in Section \ref{geooflocalmodel}.\;

%By Lemma \ref{dimkerCMWdr} (2),\;we see that the image of the natural morphism
%\[\hH^1_{(\varphi,\Gamma)}(\cM_{i,i-1})\rightarrow \hH^1(\gal_L,W_{\dr}(\cM_{i,i-1}))\]
%is a $d_Lr^2-1$-dimensional $E$-vector space.\;It is clear that this image is independent on the $2\leq i\leq k $.\;

\subsubsection{Main constructions}\label{maincontr}

Let $A\twoheadrightarrow B$ be a surjective map in $\Art_E$,\;and let $K=\ker(A\twoheadrightarrow B)$.\;Suppose $y_A:=(\cM_{\Dpik,A},\cM_{A,\bullet},j_A)$ (resp.,\;$y_B:=(\cM_{\Dpik,B},\cM_{B,\bullet},j_B)$) is an object in $X_{\cM_\Dpik,\cM_{\bullet}}(A)$ (resp.,\;$X_{\cM_\Dpik,\cM_{\bullet}}(B)$).\;Suppose that $x_A$ is isomorphic to $x_B$ when modulo $J$.\;

Let $\underline{\delta}_A=(\delta_{A,i})_{1\leq i\leq k}:L^{\times}\rightarrow A^{\times}$ (resp., $\underline{\delta}_B=(\delta_{B,i})_{1\leq i\leq k}:L^{\times}\rightarrow B^{\times}$) be the parameter of  $\cM_{A,\bullet}$ (resp., $B$) given by \cite[Lemma 6.2.2]{Ding2021}. For $2\leq i\leq k$,\;we see that the exact functor $W_{\dR}$ induces natural morphisms
\[j_{\underline{\delta}_{\ast},i}:\hH^1_{(\varphi,\Gamma)}({\cM}_{\ast,i,i-1})\rightarrow \hH^1(\gal_L,W_{\dr}({\cM}_{\ast,i,i-1})),\ast\in\{A,B\}.\;\]
(In particular,\;$\dim_E\mathrm{Im}j_{\bm{\delta}_{\bh},i}=d_Lr^2-1$ by Lemma \ref{dimkerCMWdr} (2)).\;By the proof of \cite[Theorem 3.4.4]{breuil2019local} and Lemma \ref{dimlemmaCM},\;we have 
\begin{equation}\label{functorialproperty}
	\begin{aligned}
		&\hH^1_{(\varphi,\Gamma)}({\cM}_{A,i,i-1})\otimes_AB\xrightarrow{\sim}\hH^1_{(\varphi,\Gamma)}({\cM}_{B,i,i-1}),\\
		&\hH^1(\gal_L,W_{\dr}({\cM}_{A,i,i-1}))\otimes_AB\xrightarrow{\sim}\hH^1(\gal_L,W_{\dr}({\cM}_{B,i,i-1})).
	\end{aligned}
\end{equation}
They fit into a commutative diagram:
\begin{equation}\label{imagAtoB}
	\xymatrix{
		\hH^1_{(\varphi,\Gamma)}({\cM}_{A,i,i-1})\ar[r]\ar[d]_{j_{\underline{\delta}_{A},i}} & \hH^1_{(\varphi,\Gamma)}({\cM}_{B,i,i-1}) \ar[d]_{j_{\underline{\delta}_{B},i}} \ar[r]  & 0\\
		\hH^1(\gal_L,W_{\dr}({\cM}_{A,i,i-1})) \ar[r]^{j_{\underline{\delta}_A,\underline{\delta}_B,i}} & \hH^1(\gal_L,W_{\dr}({\cM}_{B,i,i-1})) \ar[r] & 0,}
\end{equation}
where  the morphism $j_{\underline{\delta}_A,\underline{\delta}_B,i}$ is induced by modulo $K$.\;Furthermore,\;it is easy to see that $\ker j_{\underline{\delta}_A,\underline{\delta}_B,i}\cong J\hH^1(\gal_L,W_{\dr}({\cM}_{A,i,i-1}))$ (as a $A$-module) and
$j^{-1}_{\underline{\delta}_A,\underline{\delta}_B}(\mathrm{Im}j_{\underline{\delta}_{B},i})=\mathrm{Im}j_{\underline{\delta}_{A},i}+J\hH^1(\gal_L,W_{\dr}({\cM}_{A,i,i-1}))$ by diagram chasing.\;From this we deduce 
\begin{lem}\label{LEMimagAtoBexactseq}Keep the above situation and notation.\;We have a short exact sequence: 
	\begin{equation}\label{imagAtoBexactseq}
		\begin{aligned}
			0\rightarrow J\hH^1(\gal_L,W_{\dr}({\cM}_{A,i,i-1}))\cap\mathrm{Im}j_{\underline{\delta}_{A},i} \rightarrow\mathrm{Im}j_{\underline{\delta}_{A},i}\xrightarrow{j_{\underline{\delta}_A,\underline{\delta}_B},i}\mathrm{Im}j_{\underline{\delta}_{B},i}\rightarrow 0.
		\end{aligned}
	\end{equation}
\end{lem}
%\begin{rmk}Indeed,\;by \cite[Lemma 6.2.2]{Ding2021},\;for $\ast\in\{A,B\}$,\;there exist deformations $\underline{\delta}_{\ast}=(\delta_{\ast,i})_{1\leq i\leq k}:L^{\times}\rightarrow \ast^{\times}$ such that $\delta_{\ast,i}\equiv \bm{\delta}_{\bh,i} \Modo\;\fm_\ast$ and $\cM_{i,i-1,\ast}\cong \EndO(\Delta_{\pi})\tee \cR_{E,L}({\delta}_{\ast,i}^{-1}{\delta}_{\ast,i-1})[1/t]$.\;Then $j_{y_{\ast},i}$ only depend on the ${\delta}_{\ast,i}^{-1}{\delta}_{\ast,i-1}$.\;
	%Therefore,\;if we have indicated the ${\delta}_{\ast,i}^{-1}{\delta}_{\ast,i-1}$,\;we rewrite $j_{y_{\ast},i}$ by $j_{{\delta}_{\ast,i},{\delta}_{\ast,i-1}}$.\;
%\end{rmk}
%\begin{rmk}We first put $B=E$,\;then we see that $0\rightarrow \fm_A\hH^1_{(\varphi,\Gamma)}({\cM}_{A,i,i-1})\cap\mathrm{Im}j_A \rightarrow\mathrm{Im}j_A\rightarrow\mathrm{Im}j_E\rightarrow 0$.\;Put $A=B[\epsilon]/\epsilon^2$.\;Then the exact sequence $(\ref{imagAtoBexactseq})$ becomes since $\mathrm{Im}(a_{r,J})\cap\mathrm{Im}j_{B[\epsilon]/\epsilon^2}=\epsilon\hH^1_{(\varphi,\Gamma)}({\cM}_{B,i,i-1[\epsilon]/\epsilon^2})\cap \mathrm{Im}j_{B[\epsilon]/\epsilon^2}\cong \epsilon\hH^1_{(\varphi,\Gamma)}({\cM}_{B,i,i-1})\cap \mathrm{Im}j_{B[\epsilon]/\epsilon^2}$.\;Note that $\epsilon\mathrm{Im}J_{\bB}\subset \mathrm{Im}(a_{r,J})\cap\mathrm{Im}j_{B[\epsilon]/\epsilon^2}$.\;
%\end{rmk}
Denote by $\widehat{X}_{\bW_\Dpik,\bF_{\bullet}}:=\widehat{(\cZ_{\bL_{r,\emptyset},L})}_{\bm{\delta}_{\bh}}\times_{\widehat\fz_{r,\emptyset,L}} X_{\bW_\Dpik,\bF_{\bullet}}\cong {\widehat\fz_{r,\emptyset}}\times X_{\bW_\Dpik,\bF_{\bullet}} $ for simplicity (note that $\widehat{(\cZ_{\bL_{r,\emptyset},L})}_{\bm{\delta}_{\bh}}\cong \widehat\fz_{r,\emptyset}\times \widehat\fz_{r,\emptyset,L}$,\;similar to the argument after \cite[(3.17)]{breuil2019local}).\;We define a full subgroupoid $\widehat{X}^{(\varphi,\Gamma)}_{\bW_\Dpik,\bF_{\bullet}}$ of $\widehat{X}_{\bW_\Dpik,\bF_{\bullet}}$. Roughly speaking,\;$\widehat{X}^{(\varphi,\Gamma)}_{\bW_\Dpik,\bF_{\bullet}}$ is certain subgroupoid of $\widehat{X}_{\bW_\Dpik,\bF_{\bullet}}$ that describes  the ``image" of the local model map $\Upsilon$ in (\ref{localmodelmap}),\;i.e.,\;the deformations of $(\bW_\Dpik,\bF_{\bullet})$ that come from deformations of $(\cM_\Dpik,\cM_{\bullet})$.\;
\begin{dfn}
The objects of $\widehat{X}^{(\varphi,\Gamma)}_{\bW_\Dpik,\bF_{\bullet}}$ are the $5$-tuples $(A,\underline{\delta}_A,\bW_{\Dpik,A},\bF_{A,\bullet},\iota_{A})$
where
\begin{itemize}
	\item $A\in \Art_E$ and $\bW_{\Dpik,A}$ is a $B_{\dr}\otimes_{\bQ_p}A$-representation of $\gal_L$;
	\item $\underline{\delta}_A=(\delta_{A,i})_{1\leq i\leq k}:L^{\times}\rightarrow A^{\times}$ such that $\delta_{A,i}\equiv \bm{\delta}_{\bh,i} \Modo\;\fm_A$;
	\item $\bF_{A,\bullet}$ is a $\bP_{r,\emptyset}$-filtration on $\bW_{\Dpik,A}$ by $B_{\dr}\otimes_{\bQ_p}A$-subrepresentation of $\gal_L$ such that $\bF_{A,0}=0$ and $\bF_{A,i}/\bF_{A,i-1}$ ($1\leq i\leq k$) is free of rank $r$ over $B_{\dr}\otimes_{\bQ_p}A$ and isomorphic to  $\bF_{i}/\bF_{i-1}\otimes_{B_{\dr}\otimes_{\bQ_p}E}W_{\dR}(\delta_{A,i})$;
	\item $\iota_{A}:\bW_{\Dpik,A}\otimes_{A}E\xrightarrow{\sim}\bW_\Dpik$ is an isomorphism of $B_{\dr}\otimes_{\bQ_p}E$-representations of $\gal_L$ which induces isomorphism $\iota_{A}:\bF_{i,A}\otimes_{A}E\xrightarrow{\sim}\bF_{i}$ for $0\leq i\leq k$;
	\item Put $\cM_{\underline{\delta}_{A},i,i-1}:=\EndO(\Delta_{\pi})\tee \cR_{A,L}({\delta}_{A,i}^{-1}{\delta}_{A,i-1})[1/t]$.\;The natural morphism  
	\[j_{\underline{\delta}_{A},i}:\hH^1_{(\varphi,\Gamma)}(\cM_{\underline{\delta}_{A},i,i-1})\rightarrow \hH^1(\gal_L,W_{\dr}(\cM_{\underline{\delta}_{A},i,i-1}).\;\]
	Then
	$\bF_{A,i}/\bF_{A,i-2}\in  \mathrm{Im}(j_{\underline{\delta}_{A},i})$ for $2\leq i\leq k$,\;where we view $\bF_{A,i}/\bF_{A,i-2}$ as an element in the extension group $\hH^1(\gal_L,W_{\dr}(\cM_{\underline{\delta}_{A},i,i-1}))$.\;
\end{itemize}
A morphism $(A,\underline{\delta}_A,\bW_{\Dpik,A},\bF_{A,\bullet},\iota_{A})\rightarrow (A',\underline{\delta}_{A'},\bW_{\Dpik,A'},\bF_{A',\bullet},\iota_{A'})$ is a morphism $A\rightarrow A'$ in $\Art_E$ and an isomorphism $\bW_{\Dpik,A}\otimes_AA'\xrightarrow{\sim}\bW_{\Dpik,A'}$ of $B_{\dr}\otimes_{\bQ_p}A'$-representation of $\gal_L$ which are compatible with all above structures.\;
\end{dfn}

The following proposition is an analogue of \cite[Theorem\;6.2.6]{Ding2021}.\;

\begin{pro}\label{proformallysmooth}The morphism $X_{\cM_\Dpik,\cM_{\bullet}}\rightarrow \widehat{X}_{\bW_\Dpik,\bF_{\bullet}}$ of groupoids over $\Art_E$ factors through the full subgroupoid $\widehat{X}^{(\varphi,\Gamma)}_{\bW_\Dpik,\bF_{\bullet}}\hookrightarrow \widehat{X}_{\bW_\Dpik,\bF_{\bullet}}$.\;Moreover,\;the morphism $X_{\cM_\Dpik,\cM_{\bullet}}\rightarrow \widehat{X}^{(\varphi,\Gamma)}_{\bW_\Dpik,\bF_{\bullet}}$ of groupoids over $\Art_E$ is formally smooth.\;
\end{pro}
\begin{proof}The proof of \cite[Theorem 3.4.4]{breuil2019local} is also suitable for us.\;Suppose $y_A:=(\cM_{\Dpik,A},\cM_{A,\bullet},j_A)$ (resp.,\;$y_B:=(\cM_{\Dpik,B},\cM_{B,\bullet},j_B)$) is an object in $X_{\cM_\Dpik,\cM_{\bullet}}(A)$ (resp.,\;$X_{\cM_\Dpik,\cM_{\bullet}}(B)$).\;Suppose that $x_A$ is isomorphic to $x_B$ when modulo $K$.\;In our case,\;the sujectivity of \cite[(3.23)]{breuil2019local} is replaced by the 
sujectivity of	
\begin{equation}
	\begin{aligned}
		\hH^1_{(\varphi,\Gamma)}({\cM}_{A,i,i-1})\longrightarrow\;&\mathrm{Im} j_{\underline{\delta}_{A},i}\times_{\mathrm{Im} j_{\underline{\delta}_{B},i}}\hH^1_{(\varphi,\Gamma)}({\cM}_{B,i,i-1})\\
		&\cong \mathrm{Im} j_A\times_{\hH^1(\gal_L,W_{\dr}({\cM}_{B,i,i-1}))}\hH^1_{(\varphi,\Gamma)}({\cM}_{B,i,i-1}).
	\end{aligned}
\end{equation}
This is a direct consequence of \cite[Lemma 3.4.5]{breuil2019local},\;Lemma \ref{LEMimagAtoBexactseq},\;commutative diagram (\ref{imagAtoB}) and (\ref{imagAtoBexactseq}).\;The result follows.\;
%The first assertion is a direct consequence of Lemma \ref{gcohomolozero}.\;Note that 
%\[\widehat{\cT^n}\times_{\widehat{\ft}_n} X^{(\varphi,\Gamma)}_{\bW_\Dpik,\bF_{\bullet}}=\prod_{i=1}^n\widehat{\cT^1}\times_{\widehat{\ft}_1}X_{\bW_{\Dpik,i}}.\]
%The \cite[Theorem 3.4.4]{breuil2019local} shows that $X^0_{\cM_{\Dpik,i}}\rightarrow \widehat{\cT^1}\times_{\widehat{\ft}_1}X_{\bW_{\Dpik,i}}$ is formally smooth.\;On the other hand,\;note that $X_{\cM_\Dpik,\cM_{\bullet}}\rightarrow \prod_{i=1}^nX^0_{\cM_{\Dpik,i}}$ is also formally smooth,\;this shows that $X_{\cM_\Dpik,\cM_{\bullet}}\rightarrow \widehat{\cT^n}\times_{\widehat{\ft}} X^T_{\bW_\Dpik,\bF_{\bullet}}$ is formally smooth.\;Identifying $X^T_{\bW_\Dpik,\bF_{\bullet}}$ with $X^{(\varphi,\Gamma)}_{\bW_\Dpik,\bF_{\bullet}}$ gets the second assertion.\;
\end{proof}

%Let ${X}_{\bW_\Dpik,\bF_{\bullet}}^{(\varphi,\Gamma)}$ be the essential image of $\widehat{X}^{(\varphi,\Gamma)}_{\bW_\Dpik,\bF_{\bullet}}$ via the forget morphism $\widehat{X}_{\bW_\Dpik,\bF_{\bullet}}\rightarrow X_{\bW_\Dpik,\bF_{\bullet}}$.\;
%Similar to \cite[Corollary 6.2.7]{Ding2021},\;we have
%\begin{cor}\label{corformallysmooth}
%	The morphism $X_{\cM_\Dpik,\cM_{\bullet}}\rightarrow X^{(\varphi,\Gamma)}_{\bW_\Dpik,\bF_{\bullet}}$,\;$X^\Box_{\cM_\Dpik,\cM_{\bullet}}\rightarrow X^{(\varphi,\Gamma),\Box}_{\bW_\Dpik,\bF_{\bullet}}$ of groupoids over $\Art_E$ are formally smooth.\;
%\end{cor}
%\begin{rmk}We also consider a closed subgroupoid $X^{\dR}_{\bW_\Dpik,\bF_{\bullet}}$ of $X_{\bW_\Dpik,\bF_{\bullet}}$ consists of de Rham deformations.\;Put $X^{\dR}_{\cM_\Dpik,\cM_{\bullet}}:=X_{\cM_\Dpik,\cM_{\bullet}}\times_{X_{\bW_\Dpik,\bF_{\bullet}}} X^{\dR}_{\bW_\Dpik,\bF_{\bullet}}$.\;\end{rmk}

%\begin{cor}
%	The morphism $X_{\cM_\Dpik,\cM_{\bullet}}\rightarrow X_{\bW_\Dpik,\bF_{\bullet}}$,\;$X^\Box_{\cM_\Dpik,\cM_{\bullet}}\rightarrow 
%tfX^{\Box}_{\bW_\Dpik,\bF_{\bullet}}$ of groupoids over $\Art_E$ are formally smooth.\;\end{cor}$\widehat{X}^{(\varphi,\Gamma)}_{\bW_\Dpik,\bF_{\bullet}}$ and ${X}^{(\varphi,\Gamma)}_{\bW_\Dpik,\bF_{\bullet}}$ are formally smooth over $\Art_E$.\;

\begin{lem}\label{formallysmoothphigammaprepare} $\widehat{X}^{(\varphi,\Gamma)}_{\bW_\Dpik,\bF_{\bullet}}$ is formally smooth over $\Art_E$.\;
\end{lem}
%Therefore,\;$\widehat{X}^{(\varphi,\Gamma)}_{\bW_\Dpik,\bF_{\bullet}}$ is pro-representable.\;
\begin{proof}Let $A\rightarrow A/I$ be a surjection in $\Art_E$ with $I^2=0$.\;We show that $|\widehat{X}^{(\varphi,\Gamma)}_{\bW_\Dpik,\bF_{\bullet}}|(A)\rightarrow |\widehat{X}^{(\varphi,\Gamma)}_{\bW_\Dpik,\bF_{\bullet}}|(A/I)$ is surjective.\;Let $(\underline{\delta}_{A/I},\bW_{\Dpik,{A/I}},\bF_{A/I,\bullet},\iota_{A/I})\in |\widehat{X}^{(\varphi,\Gamma)}_{\bW_\Dpik,\bF_{\bullet}}|(A/I)$.\;Since $\widehat{(\cZ_{\bL_{r,\emptyset},L})}_{\bm{\delta}_{\bh}}$ is formally smooth,\;we can choose a lifting $\underline{\delta}_{A}:L^{\times}\rightarrow A^\times$,\;which is a continuous character such that $\underline{\delta}_{A}\equiv \underline{\delta}_{A/I}\;\mathrm{mod}\;I$.\;On the other hand,\;the diagram  (\ref{imagAtoB}) and the exact sequence (\ref{imagAtoBexactseq}) show that the natural morphisms $\mathrm{Im}(j_{\underline{\delta}_{A},i})\rightarrow \mathrm{Im}(j_{\underline{\delta}_{A/I},i})$ is surjective for each $1\leq i\leq k-1$.\;By induction on $1\leq i\leq k-1$,\;we can construct a object $\bW_{\Dpik,{A}}$ which is a successive extension of $\bF_{A,i+1}/\bF_{A,i}\otimes_{B_{\dr}\otimes_{\bQ}E}W_{\dr}({\delta}_{A,i})$ such that $\bW_{\Dpik,{A}} \equiv \bW_{\Dpik,{A/I}}\;\mathrm{mod}\;I$.\;The natural filtration $\bF_{A,\bullet}$ on $\bW_{\Dpik,{A}}$ also gives $\bF_{A/I,\bullet}$ when modulo $I$.\;This completes the proof.\;
\end{proof}

Choose $(\bW_{\Dpik,A},\bF_{A,\bullet},\iota_{A})\in X_{\bW_\Dpik,\bF_{\bullet}}(A)$.\;We have $\bW_{\Dpik,A}\cong\oplus_{\tau\in \Sigma_L}\bW_{\Dpik,A,\tau}$ with respect to the isomorphism $B_{\dr}\otimes_{\bQ_p}A\cong \prod_{\tau\in \Sigma_L}B_{\dr,\tau}$ ($B_{\dr,\tau}:=B_{\dr}\otimes_{L,\tau}A$).\;It equipped with an $L\otimes_{\bQ_p}E$-linear operator $N_A=\oplus_{\tau\in \Sigma_L}N_{A,\tau}$.\;For $\tau\in \Sigma_L$,\;$N_{A,\tau}=0$ if and only if $\bW_{\Dpik,A}$ is $\tau$-de Rham.\;

For any $J\subseteq\Sigma_L$,\;we put $\bW_{\Dpik,A,J}\cong\oplus_{\tau\in J}\bW_{\Dpik,A,\tau}$.\;Therefore,\;we let
$X_{\bW_\Dpik,\bF_{\bullet},J}$ (resp.,\;$X^{\Box}_{\bW_\Dpik,\bF_{\bullet},J}$) be the $J$-component of $X_{\bW_\Dpik,\bF_{\bullet}}$ (resp.,\;$X^{\Box}_{\bW_\Dpik,\bF_{\bullet}}$), i.e.,\;for $A\in\Art_E$,\;we define its $A$-points by 
\[X_{\bW_\Dpik,\bF_{\bullet},J}(A)=\left\{(A,\bW_{\Dpik,A,J},\bF_{A,\bullet,J},\iota_{A,J})\right\},\;\big(\text{resp.,\;}X^{\Box}_{\bW_\Dpik,\bF_{\bullet},J}(A)=\left\{(A,\bW_{\Dpik,A,J},\bF_{A,\bullet,J},\iota_{A,J},\alpha_{A,J})\right\}\big).\]
We can define $X_{\bW_\Dpik^+,\bF_{\bullet},J}$ and $X^{\Box}_{\bW_\Dpik,\bF_{\bullet},J}$ in a similar way.\;Note that $L\otimes_{\bQ_p}A\cong \prod_{\tau\in \Sigma_L}A$,\;we also have a natural decomposition $D_{\pdr}(\bW_{\Dpik,A})\cong \prod_{\tau\in \Sigma_L} D_{\pdr,\tau}(\bW_{\Dpik,A})$ (note that $D_{\pdr,\tau}(\bW_{\Dpik,A})\cong D_{\pdr}(\bW_{\Dpik,A,\tau})$ by writing $B_{\dr}\otimes_{\bQ_p}A\cong B_{\dr}\otimes_{L}(L\otimes_{\bQ_p} A)$).\;For any $J\subseteq\Sigma_L$,\;we put 
$D_{\pdr,J}(\bW_{\Dpik,A})\cong \prod_{\tau\in J} D_{\pdr,\tau}(\bW_{\Dpik,A})$.\;
Hence the point $y_{1,J}=(y_{1,\tau})_{\tau\in J}$ (resp.,\;$y_{J}=(y_{\tau})_{\tau\in J}$) lies in $\widetilde{\fg}_{r,J}$ (resp.,\;$X_{r,J}$).\;It is easy to see that the groupoid $X^\Box_{\bW_\Dpik^+,\bF_{\bullet},J}$ (resp.,\;$X^\Box_{\bW_\Dpik,\bF_{\bullet},J}$) over $\Art_E$ is pro-representable,\;and the functor $|X^\Box_{\bW_\Dpik^+,\bF_{\bullet},J}|$ (resp.,\;$|X^\Box_{\bW_\Dpik,\bF_{\bullet},J}|$) is pro-representated by the formal scheme $\widehat{X}_{r,J,y_J}$ (resp.,\;$\widehat{\widetilde{\fg}}_{r,J,y_{1,J}}$).\;

%let $\cZ_{\bL_{r,\emptyset},J}$ be the rigid closed subspace of $\cZ_{\bL_{r,\emptyset},L}$ parameterizing locally $J$-analytic characters of $\bL_{r,\emptyset}$.\;
Put scheme $Y_{r,L}:={\fz_{r,\emptyset}}\times {\widetilde{\fg}}_{r,L}$.\;For $J\subseteq \Sigma_L$,\;put $Y_{r,J}:=\fz_{r,\emptyset}\times{\widetilde{\fg}}_{r,J}$.\;Let $\widehat{y}_1=(\underline{0},y_1)\in Y_{r,L}$,\;and $\widehat{y}_{1,J}$ the corresponding  point in $Y_{r,J}$.\;We define   $\widehat{X}_{\bW_\Dpik,\bF_{\bullet},J}:=\widehat\fz_{r,\emptyset}\times X^\Box_{\bW_\Dpik,\bF_{\bullet},J}$.\;Then the functor $|X^\Box_{\bW_\Dpik,\bF_{\bullet},J}|$ is pro-representated by the formal scheme $\widehat{Y}_{r,J,\;\widehat{y}_{1,J}}$.\;

\begin{pro}\label{Jversionwdfj}For any $J\subsetneqq\Sigma_L$,\;the composition ${X}_{\cM_\Dpik,\cM_{\bullet}}\rightarrow \widehat{X}_{\bW_\Dpik,\bF_{\bullet}}\twoheadrightarrow \widehat{X}_{\bW_\Dpik,\bF_{\bullet},J}$ and the composition   $\iota_{J}:\widehat{X}^{(\varphi,\Gamma)}_{\bW_\Dpik,\bF_{\bullet}}\rightarrow \widehat{X}_{\bW_\Dpik,\bF_{\bullet}}\twoheadrightarrow \widehat{X}_{\bW_\Dpik,\bF_{\bullet},J}$ of groupoids over $\Art_E$ are formally smooth.\;In particular,\;${X}_{\cM_\Dpik,\cM_{\bullet}}\rightarrow  {X}_{\bW_\Dpik,\bF_{\bullet},J}$ and $\iota_{J}:\widehat{X}^{(\varphi,\Gamma)}_{\bW_\Dpik,\bF_{\bullet}}\rightarrow {X}_{\bW_\Dpik,\bF_{\bullet},J}$ are also formally smooth.\;
\end{pro}
\begin{proof}We follows the route of \cite[Theorem 3.4.4]{breuil2019local}.\;Let $A\twoheadrightarrow B$ be a surjective map in $\Art_E$.\;Note that for $i\neq j$ we have commutative diagram (as in (\ref{imagAtoB})):
	\begin{equation}
		\xymatrix{
			\hH^1_{(\varphi,\Gamma)}({\cM}_{A,i,j})\ar[r]\ar[d]_{j_{y_{A},i}} & \hH^1_{(\varphi,\Gamma)}({\cM}_{B,i,j}) \ar[d] \ar[r]  & 0\\
			\hH^1(\gal_L,W_{\dr,J}({\cM}_{A,i,j})) \ar[r] & \hH^1(\gal_L,W_{\dr,J}({\cM}_{B,i,j})) \ar[r] & 0,}
	\end{equation}
	By replacing the \cite[Lemma 3.4.3]{breuil2019local} with our Lemma \ref{dimkerCMWdr} $(3)$,\;and using  \cite[Lemma 3.4.5]{breuil2019local},\;we get the sujectivity of the following map (similar to the sujectivity of \cite[(3.23)]{breuil2019local}):
	\begin{equation}
		\begin{aligned}
			\hH^1_{(\varphi,\Gamma)}({\cM}_{A,i,j})\longrightarrow\; \hH^1(\gal_L,W_{\dr,J}({\cM}_{A,i,j}))\times_{\hH^1(\gal_L,W_{\dr,J}({\cM}_{B,i,j}))}\hH^1_{(\varphi,\Gamma)}({\cM}_{B,i,j}).
		\end{aligned}
	\end{equation}
	and thus the surjectivity of 
	\begin{equation}
		\begin{aligned}
			\hH^1&_{(\varphi,\Gamma)}({\cM}_{A,i-1}\tee\Delta_{\pi}^{\vee}\tee \cR_{E,L}({\delta}_{A,i}^{-1})[1/t])\longrightarrow\\ &\hH^1(\gal_L,W_{\dr,J}({\cM}_{A,i-1}\tee\Delta_{\pi}^{\vee}\tee \cR_{E,L}({\delta}_{A,i}^{-1})[1/t]))\\
			&\times_{\hH^1(\gal_L,W_{\dr,J}({\cM}_{B,i-1}\tee\Delta_{\pi}^{\vee}\tee \cR_{E,L}({\delta}_{B,i}^{-1})[1/t]))}\hH^1_{(\varphi,\Gamma)}({\cM}_{B,i-1}\tee\Delta_{\pi}^{\vee}\tee \cR_{E,L}({\delta}_{B,i}^{-1})[1/t]).
		\end{aligned}
	\end{equation}
	By the same argument as in the proof of \cite[Theorem 3.4.4]{breuil2019local},\;we get that $X_{\cM_\Dpik,\cM_{\bullet}}\rightarrow \widehat{X}_{\bW_\Dpik,\bF_{\bullet},J}$ is formally smooth.\;Note that this morphism factors through $X_{\cM_\Dpik,\cM_{\bullet}}\rightarrow\widehat{X}^{(\varphi,\Gamma)}_{\bW_\Dpik,\bF_{\bullet}}$,\;therefore $\iota_{J}:\widehat{X}^{(\varphi,\Gamma)}_{\bW_\Dpik,\bF_{\bullet}}\rightarrow  \widehat{X}_{\bW_\Dpik,\bF_{\bullet},J}$ is formally smooth (since any object in $\widehat{X}^{(\varphi,\Gamma)}_{\bW_\Dpik,\bF_{\bullet}}$ has a preimage in $X_{\cM_\Dpik,\cM_{\bullet}}$ by definition of $\widehat{X}^{(\varphi,\Gamma)}_{\bW_\Dpik,\bF_{\bullet}}$).\;
	%The first assertion is a direct consequence of Lemma \ref{gcohomolozero}.\;Note that 
	%\[\widehat{\cT^n}\times_{\widehat{\ft}_n} X^{(\varphi,\Gamma)}_{\bW_\Dpik,\bF_{\bullet}}=\prod_{i=1}^n\widehat{\cT^1}\times_{\widehat{\ft}_1}X_{\bW_{\Dpik,i}}.\]
	%The \cite[Theorem 3.4.4]{breuil2019local} shows that $X^0_{\cM_{\Dpik,i}}\rightarrow \widehat{\cT^1}\times_{\widehat{\ft}_1}X_{\bW_{\Dpik,i}}$ is formally smooth.\;On the other hand,\;note that $X_{\cM_\Dpik,\cM_{\bullet}}\rightarrow \prod_{i=1}^nX^0_{\cM_{\Dpik,i}}$ is also formally smooth,\;this shows that $X_{\cM_\Dpik,\cM_{\bullet}}\rightarrow \widehat{\cT^n}\times_{\widehat{\ft}} X^T_{\bW_\Dpik,\bF_{\bullet}}$ is formally smooth.\;Identifying $X^T_{\bW_\Dpik,\bF_{\bullet}}$ with $X^{(\varphi,\Gamma)}_{\bW_\Dpik,\bF_{\bullet}}$ gets the second assertion.\;
\end{proof}

\begin{cor}
For $W\in\{X,\widehat{X}\}$,\;the natural morphisms ${X}^{\Box}_{\cM_\Dpik,\cM_{\bullet}}\rightarrow  W_{\bW_\Dpik,\bF_{\bullet},J}$ and  $\iota_{J}:\widehat{X}^{(\varphi,\Gamma),\Box}_{\bW_\Dpik,\bF_{\bullet}}\rightarrow  W_{\bW_\Dpik,\bF_{\bullet},J}$ of groupoids over $\Art_E$ are formally smooth.\;
\end{cor}

%and ${X}^\Box_{\cM_\Dpik,\cM_{\bullet}}\rightarrow \widehat{X}^{\Box}_{\bW_\Dpik,\bF_{\bullet},J}$

\subsubsection{Representablity}

For $1\leq i\leq k$,\;we fix isomorphisms $\alpha_i:L\otimes_{\bQ_p}E\xrightarrow{\sim} D_{\mathrm{pdR}}(\gr_i\bF_{\bullet})$ (so we fix an isomorphism $\beta:(L\otimes_{\bQ_p}E)^n\xrightarrow{\sim} D_{\mathrm{pdR}}(\bW_\Dpik)=\oplus_{i=1}^kD_{\mathrm{pdR}}(\gr_i\bF_{\bullet})$).\;Consider the groupoids 
\begin{equation}\label{dfnwdfver}
	W_{\bW_\Dpik,\bF_{\bullet}}^{\mathrm{ver}}:=X_{\bW_\Dpik,\bF_{\bullet}}\times_{\prod_{i=1}^kX_{\gr_i\bF_{\bullet}}}\prod_{i=1}^kX^{0,\Box}_{\gr_i\bF_{\bullet}},\;W\in\{X,\widehat{X}\}.\;
\end{equation}
(where the $\Box$ in $X^{0,\Box}_{\gr_i\bF_{\bullet}}$ are given by the tuples $\underline{\alpha}:=(\alpha_i)_{1\leq i\leq k}$) and $W_{\bW_\Dpik,\bF_{\bullet}}^{\Box}$ for $W\in\{X,\widehat{X}\}$ (with respect to the isomorphism $\beta:(L\otimes_{\bQ_p}E)^n\xrightarrow{\sim} D_{\mathrm{pdR}}(\bW_\Dpik)$).\;For $\star\in\{\Box,\emptyset\}$,\;the functor  $D_{\pdr}(-)$ induce the following natural morphisms:
\[X^{\star}_{\cM_\Dpik,\cM_{\bullet}}\rightarrow \widehat{X}_{\bW_\Dpik,\bF_{\bullet}}^{\ast}\rightarrow X_{\bW_\Dpik,\bF_{\bullet}}^{\ast}.\]
This section aims to prove the pro-representablity of $\widehat{X}_{\bW_\Dpik,\bF_{\bullet}}^{(\varphi,\Gamma),\star}$ for $\star\in\{\Box,\ver\}$.\;Note that a object in $X^{\mathrm{ver}}_{\bW_\Dpik,\bF_{\bullet}}(A)$ is a $5$-tuples $(A,W_A,\bF_{A,\bullet},\iota_{A},\underline{\alpha}_A)$,\;by imposing the additional data $\underline{\alpha}_A=(\alpha_{A,i})$,\;where $\alpha_{A,i}:(L\otimes_{\bQ_p}A)^r\xrightarrow{\sim}D_{\pdr}(\gr_i{\bF_{A,\bullet}}W_A)$.\;

%We have a closed immersion $\widehat{X}_{\bW_\Dpik,\bF_{\bullet}}^{\mathrm{ver}}\rightarrow \widehat{X}_{\bW_\Dpik,\bF_{\bullet}}^{\Box}$ (and thus $\widehat{X}_{\bW_\Dpik,\bF_{\bullet}}^{(\varphi,\Gamma),\mathrm{ver}}\rightarrow \widehat{X}_{\bW_\Dpik,\bF_{\bullet}}^{(\varphi,\Gamma),\Box}$) of  groupoids  over $\Art_E$.\;Put $\widehat{X}_{\bW_\Dpik,\bF_{\bullet}}^{(\varphi,\Gamma),\mathrm{ver},\Box}:=\widehat{X}_{\bW_\Dpik,\bF_{\bullet}}^{(\varphi,\Gamma),\mathrm{ver}}\times_{X_{\bW_\Dpik,\bF_{\bullet}}}X^{\Box}_{\bW_\Dpik,\bF_{\bullet}}$.\;Then we have a diagram
%\begin{equation}
%	\xymatrix{
%		& \widehat{X}_{\bW_\Dpik,\bF_{\bullet}}^{(\varphi,\Gamma),\mathrm{ver},\Box}
%		\ar[dl]_{\beta^{\Box}} \ar[dr]^{\beta^{\mathrm{ver}}} &    \\
%		\widehat{X}_{\bW_\Dpik,\bF_{\bullet}}^{(\varphi,\Gamma),\mathrm{ver}}\ar[rr] \ar[dr]_{\alpha^{\mathrm{ver}}}	&  & \widehat{X}_{\bW_\Dpik,\bF_{\bullet}}^{(\varphi,\Gamma),\Box} \ar[dl]_{\alpha^{\Box}}\\ 	& \widehat{X}^{(\varphi,\Gamma)}_{\bW_\Dpik,\bF_{\bullet}} &  .}
%\end{equation} 
%such that the the morphisms $\beta^{\Box}$,\;$\beta^{\mathrm{ver}}$ and $\alpha^{\Box}$,\;$\alpha^{\mathrm{ver}}$  are formally smooth.\; 

%Consider the scheme $Y_{r,L}:={\fz_{r,\emptyset}} \times q_{\bP_{r,\emptyset}}^{-1}(\tau_{\bP_{r,\emptyset},\Sigma_L})$,\;which is equidimensional of dimension $k+d_L\big(k+\frac{n(n-r)}{2}\big)$.\;Then $y_1=(\beta^{-1}(\cD_{\bullet}),N_{\bW_\Dpik})\in q_{\bP_{r,\emptyset}}^{-1}(\tau_{\bP_{r,\emptyset},\Sigma_L})$ and
%$\widehat{y}_1=(\underline{0},\beta^{-1}(\cD_{\bullet}),N_{\bW_\Dpik})\in Y_{r,L}$.\;

\begin{pro}\label{relativeverbox}
 We have a natural  morphism $\widehat{X}_{\bW_\Dpik,\bF_{\bullet}}^{\Box}\rightarrow \widehat{X}_{\bW_\Dpik,\bF_{\bullet}}^{\mathrm{ver}}$
of groupoids  over $\Art_E$ and  an equivalence of groupoids $X_{\bW_\Dpik,\bF_{\bullet}}^{(\varphi,\Gamma),\Box}\cong X_{\bW_\Dpik,\bF_{\bullet}}^{\Box}\times_{X_{\bW_\Dpik,\bF_{\bullet}}^{\mathrm{ver}}}X^{(\varphi,\Gamma),\mathrm{ver}}_{\bW_\Dpik,\bF_{\bullet}}$ over $\Art_E$.\;Moreover,\;$\widehat{X}_{\bW_\Dpik,\bF_{\bullet}}^{\Box}\rightarrow \widehat{X}_{\bW_\Dpik,\bF_{\bullet}}^{\mathrm{ver}}$ is relatively pro-representable.\;
\end{pro}
\begin{proof}
Choose $(A,W_A,\bF_{A,\bullet},\iota_A,\alpha_A)\in X^{\Box}_{\bW_\Dpik,\bF_{\bullet}}(A)$.\;By the proof of \cite[Lemma 6.1.1]{Ding2021},\;the induced action of $\nu_{W_A}$ on $W_{\dr}(\bF_{A,i})/W_{\dr}(\bF_{A,i-1})$ gives the scalars $(\nu_{\epsilon_{A,1}},\cdots,\;\nu_{\epsilon_{A,k}})\in \prod_{i=1}^kX^{0,\Box}_{\gr_i\bF_{\bullet}}$.\;This gives a natural  morphism $\widehat{X}_{\bW_\Dpik,\bF_{\bullet}}^{\Box}\rightarrow \widehat{X}_{\bW_\Dpik,\bF_{\bullet}}^{\mathrm{ver}}$.\;Fix $\eta_A:=(A,W_A,\bF_{A,\bullet},\iota_{A},\underline{\alpha}_A)\in X^{\mathrm{ver}}_{\bW_\Dpik,\bF_{\bullet}}(A)$,\;where $\underline{\alpha}_A=(\alpha_{A,i})$ with $\alpha_{A,i}:(L\otimes_{\bQ_p}A)^r\xrightarrow{\sim}D_{\pdr}(\gr_i{\bF_{A,\bullet}})$.\;Denoted by $\widetilde{\eta_A}$ the groupoid over $\Art_E$ it represents.\;Then for each $A$-algebra $A'$ in $\Art_E$,\;the  $A'$-points of groupoid $(\widetilde{\eta_A}\times_{X_{\bW_\Dpik,\bF_{\bullet}}^{\mathrm{ver}}}X_{\bW_\Dpik,\bF_{\bullet}}^{\Box})(A')$ consists of objects $(A',W_{A'},\bF_{A',\bullet},\iota_{A'},\underline{\alpha}_{A'},\beta_{A'},\alpha_{A'})$,\;where $\beta_{A'}:(W_A\otimes_AA',\bF_{A,\bullet}\otimes_AA',\iota_{A}\otimes 1,\underline{\alpha}_A\otimes 1)\rightarrow (A',W_{A'},\bF_{A',\bullet},\iota_{A'},\underline{\alpha}_{A'})$ is a morphism in $X_{\bW_\Dpik,\bF_{\bullet}}^{\mathrm{ver}}$ and $\alpha_{A'}:(L\otimes_{\bQ_p}A')^n\xrightarrow{\sim}D_{\pdr}(W_{A'})$ such that $\underline{\alpha}_{A'}$ and $\alpha_{A'}$ are compatible (so that we can drop the data $\underline{\alpha}_{A'}$ since it is already determined by $\alpha_{A'}$).\;Recall that we have the natural morphism $\widetilde{\fg}_{r,L}\rightarrow \fz_{r,\emptyset,L},(g\bP_{r,\emptyset}\psi)\mapsto \overline{\mathrm{Ad}(g^{-1})\psi}$.\;Then the tuple $\underline{\alpha}_A$ gives a point $z_A$ in $\fz_{r,\emptyset,L}$.\;Let $(\widehat{\cO}_{{\widetilde{\fg}}_{r,L},y_1})_{z_A}$ be the fiber of $\widehat{\cO}_{{\widetilde{\fg}}_{r,L},y_1}$ over $z_A$.\;Then the functor $|(\widetilde{\eta_A}\times_{X_{\bW_\Dpik,\bF_{\bullet}}^{\mathrm{ver}}}X_{\bW_\Dpik,\bF_{\bullet}}^{\Box})|$ is pro-represented by the completion of $\widehat{\cO}_{{\widetilde{\fg}}_{r,L},y_1}$ via the closed subspace  $(\widehat{\cO}_{{\widetilde{\fg}}_{r,L},y_1})_{z_A}$.\;This completes the proof.\;\end{proof}

We first deduce that
\begin{pro}\label{constrXLbox}
	The groupoids $X_{\bW_\Dpik,\bF_{\bullet}}^{\mathrm{ver}}$ and $\widehat{X}_{\bW_\Dpik,\bF_{\bullet}}^{\mathrm{ver}}$  over $\Art_E$ are pro-representable.\;The functor $|X_{\bW_\Dpik,\bF_{\bullet}}^{\mathrm{ver}}|$ (resp.,\;$\widehat{X}_{\bW_\Dpik,\bF_{\bullet}}^{\mathrm{ver}}$) is pro-represented by
a formally smooth noetherian complete local ring $R_{\bW_\Dpik,\bF_{\bullet}}^{\mathrm{ver}}$ of residue field $E$.\;
% and dimension $d_L(k+\frac{n(n-r)}{2})$ (resp.,\;$k+d_L\big(k+\frac{n(n-r)}{2}\big)$).\;
\end{pro}
\begin{proof}It is clear that $|X_{\bF_{1},\bF_{\bullet}}^{\mathrm{ver}}|\cong |X_{\bF_{i}}^{\Box}|$  is pro-represented by a formally smooth noetherian complete local ring $U_1$ of residue field $E$ and dimension $d_L$.\;Denoted by $V_{i}$ the formal scheme pro-represents the functor $|X_{\gr_i\bF_{\bullet}}^{0,\Box}|$ for $1\leq i\leq k$.\;Assume that $|X_{\bF_{i-1},\bF_{\bullet}}^{\mathrm{ver}}|$ is pro-represented by $U_{i-1}$,\;where $U_{i-1}$ is a formally smooth noetherian complete local ring of residue field $E$.\;Let $T_{i}$ (resp.,\;$T'_i$) denoted the completion of $U_{i-1}\otimes_EV_{i}$ (resp.,\;$U_{i-1}\otimes_E\widehat{\cO}_{\widehat{L^{\times}},\bm{\delta}_{\bh,i}}$) with respect to the maximal ideal generated by the maximal ideal of $U_{i-1}$ and the one of $V_{i}$.\;For any morphism $T_{i}\rightarrow A$ with $A\in \Art_E$,\;let $\bF_{A,i-1}$ be the $B_{\dr}\otimes_{\bQ_p}A$-representation given by the pull-back along $U_{i-1}\rightarrow T_{i}\rightarrow A$ of the universal $B_{\dr}\otimes_{\bQ_p}E$-representation-module over $B_{\dr}\otimes_{\bQ_p}U_{i-1}$ and let $\gr_i\bF_{A,\bullet}$ be the $B_{\dr}\otimes_{\bQ_p}A$-representation given by the pull-back along $V_{i}\rightarrow T_{i}\rightarrow A$ of the universal $B_{\dr}\otimes_{\bQ_p}E$-representation-module over $B_{\dr}\otimes_{\bQ_p}V_{i}$.\;Let 
	\begin{equation}\label{dfnWi}
		\begin{aligned}
			&W_{i}:=\varprojlim_{T_{i}\rightarrow A}\ext^1_{
				\gal_L}\Big(\gr_i\bF_{A,\bullet},\bF_{A,i-1}\Big),\\
			&{W}'_{i}:=\varprojlim_{T'_{i}\rightarrow A}\ext^1_{
				\gal_L}\Big(W_{\dR}\big(\Delta_\pi\otimes_{\cR_{E,L}}\cR_{A,L}(\delta_{A,i})\Big[\frac{1}{t}\Big]\big),\bF_{A,i-1}\Big)
		\end{aligned}
\end{equation}	By definition,\;$[\bF_{i}]$ correspondences to a maximal ideal $\fm_{i}$ with residue field $E$ of the gradded commutative symmetric $T_{i}$-algebra $\mathrm{Symm}_{T_{i}}W_{i}$ and $\mathrm{Symm}_{T'_{i}}W'_{i}$.\;Let $U_{i}$ (resp.,\;$U_i'$) be the completion of  $\mathrm{Symm}_{T_{i}}W_{i}$ (resp.,\;$\mathrm{Symm}_{T'_{i}}W'_{i}$) at $\fm_{i}$.\;Then $|X_{\bF_{i},\bF_{\bullet}}^{\mathrm{ver}}|$ (resp.,\;$|\widehat{X}_{\bF_{i},\bF_{\bullet}}^{\mathrm{ver}}|$) is pro-represented by $U_{i}$ (resp.,\;$U_i'$).\;On the other hand,\;it is clear that $X_{\bW_\Dpik,\bF_{\bullet}}\rightarrow\prod_{i=1}^kX_{\gr_i\bF_{\bullet}}$ is formally smooth.\;Since $\prod_{i=1}^kX^{0,\Box}_{\gr_i\bF_{\bullet}}$ is formally smooth,\;we get that $X_{\bW_\Dpik,\bF_{\bullet}}^{\mathrm{ver}}$ is formally smooth.\;This shows that $|X_{\bW_\Dpik,\bF_{\bullet}}^{\mathrm{ver}}|$ is pro-represented by a formally smooth noetherian complete local ring of residue field $E$.\;
%It remains to compute the dimension.\;Note the relative dimension $X_{\bW_\Dpik,\bF_{\bullet}}\rightarrow{\prod_{i=1}^kX_{\gr_i\bF_{\bullet}}}$ is equal to $d_L\frac{n(n-r)}{2}$,\;we see that $|X_{\bW_\Dpik,\bF_{\bullet}}^{\mathrm{ver}}|$ (resp.,\;$\widehat{X}_{\bW_\Dpik,\bF_{\bullet}}^{\mathrm{ver}}$) has dimension $d_L(k+\frac{n(n-r)}{2})$ (resp.,\;$k+d_L\big(k+\frac{n(n-r)}{2}\big)$).\;
\end{proof}
%It suffices to calculate $\dim_E|X_{\bW_\Dpik,\bF_{\bullet}}^{\mathrm{ver}}|(E[\epsilon]/\epsilon^2)$.\;Note that $\dim_E|X_{\bW_\Dpik,\bF_{\bullet}}^{\mathrm{ver}}|(E[\epsilon]/\epsilon^2)=d_L(k+n(n-r))-d_Ln^2+kd_Lr^2=kd_L$ (by definition (\ref{dfnwdfver})).\;
%Note that the action of $S_{i}$ factors through $\wt-\wt(\bm{\delta}_{\bh,i}):S_{i}\rightarrow \cO(\widehat{\ft}_{1,i})\cong\cO(\BG_{a,i})$.\;%We next give more precise structure of $X_{\bW_\Dpik,\bF_{\bullet}}^{(\varphi,\Gamma),\Box}$.\;
%By definition,\;the $W_{\dr}$-functor induces a morphism $X_{\cM_\Dpik,\cM_{\bullet}}^{\mathrm{ver}}\rightarrow X_{\bW_\Dpik,\bF_{\bullet}}^{\mathrm{ver}}$
%of  groupoids  over $\Art_E$.\;We put 
%\begin{equation}
%	\begin{aligned}
%		&\widehat{X}^{(\varphi,\Gamma),\mathrm{ver}}_{\bW_\Dpik,\bF_{\bullet}}:=\widehat{X}^{(\varphi,\Gamma)}_{\bW_\Dpik,\bF_{\bullet}}\times_{X_{\bW_\Dpik,\bF_{\bullet}}}X^{\mathrm{ver}}_{\bW_\Dpik,\bF_{\bullet}},\;\widehat{X}^{(\varphi,\Gamma),\Box}_{\bW_\Dpik,\bF_{\bullet}}:=\widehat{X}_{\bW_\Dpik,\bF_{\bullet}}^{(\varphi,\Gamma)}\times_{X_{\bW_\Dpik,\bF_{\bullet}}}X^{\Box}_{\bW_\Dpik,\bF_{\bullet}}\\
%	\end{aligned}
%\end{equation}

We next show the pro-representablity of $\widehat{X}_{\bW_\Dpik,\bF_{\bullet}}^{(\varphi,\Gamma),\mathrm{ver}}$.\;
\begin{pro}\label{constrXLboxphigamma}The groupoid $\widehat{X}^{(\varphi,\Gamma),\mathrm{ver}}_{\bW_\Dpik,\bF_{\bullet}}$ is pro-representable.\;The functor $|\widehat{X}^{(\varphi,\Gamma),\mathrm{ver}}_{\bW_\Dpik,\bF_{\bullet}}|$ is pro-represented by a formally smooth noetherian complete local ring $R_{\bW_\Dpik,\bF_{\bullet}}^{\mathrm{ver},{\flat}}$ of residue field $E$.\;Moreover,\;we have a natural ring homomorphism $\iota^{\flat}:R_{\bW_\Dpik,\bF_{\bullet}}^{\mathrm{ver}}\rightarrow R_{\bW_\Dpik,\bF_{\bullet}}^{\mathrm{ver},{\flat}}$.\;
\end{pro}
\begin{proof}Keep the notation in the proof of Proposition \ref{constrXLbox}.\;The exact functor $W_{\dr}$ induces a homomorphism of $S_i$-modules $W_{\dr}:N_{i}\rightarrow W_{i}'$ (see (\ref{dfnNi}) and (\ref{dfnWi})).\;Let $W_{i}^{(\varphi,\Gamma)}:=W_{\dr}(N_{i})$,\;it admits a natural $S_{i}$-module structure (and thus have a $T_{i}'$-module structure).\;It is clear that $W_{i}^{(\varphi,\Gamma)}$ is a finitely generated module over $S_i$ or $T_i'$.\;Then $[\bF_{i}]$ corresponds to a maximal ideal $\fm_{i}$ with residue field $E$ of the gradded commutative $T_{i}$-algebra $\mathrm{Symm}_{T_{i}}(W_{i}^{(\varphi,\Gamma)})$.\;We prove this proposition by induction on $k$.\;It is clear that $|\widehat{X}_{\bF_{1},\bF_{\bullet}}^{(\varphi,\Gamma),\mathrm{ver}}|=|\widehat{X}_{\bF_{1},\bF_{\bullet}}^{\mathrm{ver}}|$  is pro-represented $\widehat{\cO}_{\widehat{L^{\times}},\bm{\delta}_{\bh,1}}\cong E[[x_1,\cdots,x_{d_L+1}]]$.\;If $\widehat{X}^{(\varphi,\Gamma),\mathrm{ver}}_{\bF_{i-1},\bF_{\bullet}}$ is already defined and is pro-represented by $U_{i}^{(\varphi,\Gamma),\flat}$. Let $T_{i}^{(\varphi,\Gamma),\flat}$ denoted the completion of $U^{(\varphi,\Gamma),\flat}_{i-1}\otimes_E\widehat{\cO}_{\widehat{L^{\times}},\bm{\delta}_{\bh,i}}$ with respect to the maximal ideal generated by the maximal ideal of $U^{(\varphi,\Gamma),\flat}_{i-1}$ and the one of $\widehat{\cO}_{\widehat{L^{\times}},\bm{\delta}_{\bh,i}}$ (note that we have a natural homomorphism $T_i'\rightarrow T_{i}^{(\varphi,\Gamma),\flat}$
).\;We see that $[\bF_{i}]$ correspondences to a maximal ideal $\fm_{i}$ with residue field $E$ of the gradded commutative $T_{i}^{(\varphi,\Gamma),\flat}$-algebra $\mathrm{Symm}_{T_{i}}(W^{(\varphi,\Gamma),\flat}_{i})$ (by assumption,\;we see that the $T_i'$-module structure $\mathrm{Symm}_{T_{i}}(W^{(\varphi,\Gamma),\flat}_{i})$ factors through the natural homomorphism $T_i'\rightarrow T_{i}^{(\varphi,\Gamma),\flat}$,\;so $\mathrm{Symm}_{T_{i}'}(W^{(\varphi,\Gamma),\flat}_{i})$ has a $T_{i}^{(\varphi,\Gamma),\flat}$-module structure).\;We let $U_{i}^{(\varphi,\Gamma),\flat}$ be the completion of  $\mathrm{Symm}_{T_{i}'}(W_{i}^{(\varphi,\Gamma)})$ at $\fm_{i}$.\;This pro-represents $\widehat{X}^{(\varphi,\Gamma),\mathrm{ver}}_{\bF_{i},\bF_{\bullet}}$.\;This completes the induction.\;On the other hand,\;it is clear that $\widehat{X}^{(\varphi,\Gamma)}_{\bW_\Dpik,\bF_{\bullet}}\rightarrow\prod_{i=1}^kX_{\gr_i{\bF_{\bullet}}}$ is formally smooth (similar to the proof of Lemma \ref{formallysmoothphigammaprepare}),\;we see that the functor $|\widehat{X}^{(\varphi,\Gamma),\mathrm{ver}}_{\bW_\Dpik,\bF_{\bullet}}|$ is pro-represented by a formally smooth noetherian complete local ring $R_{\bW_\Dpik,\bF_{\bullet}}^{\mathrm{ver},{\flat}}$.\;
%by the formally smoothness of $X_{\bW_\Dpik,\bF_{\bullet}}^{\mathrm{ver}}$ (see Proposition \ref{constrXLbox}).\;
%The last assertion follows from the fact that $  N_{i}/N_{i,g}\xrightarrow{\sim} W_{i}^{(\varphi,\Gamma)}$ for each $2\leq i\leq k$.\;By Lemma \ref{formallysmoothphigammaprepare},\;we see the formally smoothness of this ring.\;
\end{proof}

\begin{rmk}\label{Jversionwdfjrmk}
The above proposition is equivalent to the following facts via the module $W_{i}^{(\varphi,\Gamma)}$.\;Note that 
\begin{equation}
	N_{i,g}:=\varprojlim_{S_{i}\rightarrow A}\ext^1_{(\varphi,\Gamma),g}\Big(\Delta_\pi\otimes_{\cR_{A,L}}\cR_{A,L}(\delta_{A,i})\Big[\frac{1}{t}\Big],\cM_{i-1,A}\Big).
\end{equation}
is the kernel of the map $W_{\dr}:N_{i}\rightarrow W_{i}'$.\;Recall the definition of $S_i$ (resp,\;$T_i'$) in the proof of Proposition \ref{proprexmm} (resp.,\;Proposition \ref{constrXLbox}).\;For any $J\subsetneqq\Sigma_L$,\;the formally smoothness of $\widehat{X}^{(\varphi,\Gamma),\mathrm{ver}}_{\bW_\Dpik,\bF_{\bullet}}\rightarrow X_{\bW_\Dpik,\bF_{\bullet},J}^{\mathrm{ver}}$ has the following explanation.\;We put 
\begin{equation}
	W'_{i,J}:=\varprojlim_{T_{i}'\rightarrow A}\ext^1_{
		\gal_L}\Big(W_{\dR,J}\big(\Delta_\pi\otimes_{\cR_{A,L}}\cR_{A,L}(\delta_{A,i})\Big[\frac{1}{t}\Big]\big),\bF_{A,i-1,J}\Big)
\end{equation}
Note that the $W_{\dr,J}$-functor induces a map $N_i\rightarrow W'_{i,J}$ (recall (\ref{dfnNi})) of $S_i$-modules with kernel $N_{i,g,J}$.\;By Lemma \ref{dimkerCMWdr} (3),\;we see that the map $N_i\rightarrow W'_{i,J}$ is surjective.\;Therefore,\;we obtain a short exact sequence 
$0\rightarrow N_{i,g,J} \rightarrow N_{i}\rightarrow W'_{i,J} \rightarrow 0$ of $S_i$-modules,\;where  
\begin{equation}
	N_{i,g,J}:=\varprojlim_{S_{i}\rightarrow A}\ext^1_{g,J}\Big(\Delta_\pi\otimes_{\cR_{A,L}}\cR_{A,L}(\delta_{A,i})\Big[\frac{1}{t}\Big],\cM_{i-1,A}\Big).\;
\end{equation}
Via the natural projection $W'_{i}\rightarrow W'_{i,J}$,\;the image of $W_{i}^{(\varphi,\Gamma)}=W_{\dr}(N_i)$  is equal to $W'_{i,J}$.\;
\end{rmk}

%There exist $\{a_\tau\}_{\tau\in\Sigma_L}$ ($a_\tau\in E$) such that\[R_{i}\cong \sum_{\tau\in\Sigma_L}a_\tau e_{\tau}+E[[x_1,\cdots,x_{d_L},y]]^{+}\otimes_{\bQ_p}L\big/\big({\log(1+x_1)},\;\cdots,{\log(1+x_{d_L})}\big),\]where $E[[x_1,\cdots,x_{d_L},y]]^{+}:=(x_1,\cdots,x_{d_L},y)E[[x_1,\cdots,x_{d_L},y]]$ is the maximal ideal of $E[[x_1,\cdots,x_{d_L},y]]$.\;

%By \cite[Corollary 3.5.8,\;Corollary 3.5.9,\;Corollary 3.5.11]{breuil2019local},\;the groupoid $X^\Box_{\Dpik,\cM_{\bullet}}$ over $\Art_E$ is pro-representable.\;The functor $|X^\Box_{\Dpik,\cM_{\bullet}}|$ is pro-represented by a formal scheme which is formally smooth of relation dimension $d_L\frac{n(n+r)}{2}$ over $\widehat{X}_{y}$.\;

\begin{pro}
The groupoid $\widehat{X}^{(\varphi,\Gamma),\Box}_{\bW_\Dpik,\bF_{\bullet}}$ over $\Art_E$ is pro-representable.\;The functor $|\widehat{X}^{(\varphi,\Gamma),\Box}_{\bW_\Dpik,\bF_{\bullet}}|$ is pro-represented by a formally smooth noetherian complete local ring $\widehat{\cO}^{\flat}_{Y_{r,L},\widehat{y}_1}$ of residue field $E$.\;Moreover,\;we have a natural homomorphism $\iota^{\flat}:\widehat{\cO}_{Y_{r,L},\widehat{y}_1}\rightarrow \widehat{\cO}^{\flat}_{Y_{r,L},\widehat{y}_1}$ and $\widehat{\cO}^{\flat}_{Y_{r,L},\widehat{y}_1}$ is a formal power series over the ring $\widehat{\cO}_{Y_{r,J},\widehat{y}_{1,J}}$ for any $J\subsetneq \Sigma_L$.\;
\end{pro}
\begin{proof}The first assertion follows form Proposition \ref{relativeverbox} and  Proposition \ref{constrXLboxphigamma}.\;The second assertion follows from  Proposition \ref{constrXLboxphigamma} and Proposition \ref{Jversionwdfj}.\;
\end{proof}

%except the dimension.\;This follows by using the following formally smooth morphisms step by step 
%\[X^{(\varphi,\Gamma)}_{\bF_{i},\bF_{\bullet}}\rightarrow X^{(\varphi,\Gamma)}_{\bF_{i-1},\bF_{\bullet}}\times_{X^{(\varphi,\Gamma)}_{\gr_{i-1}\bF_{\bullet}}}X^{(\varphi,\Gamma)}_{\bF_{i}/\bF_{i-2},\bF_{\bullet}}\]
%for $2\leq i\leq n$ and Corollary \ref{2dimXMFCOMPUTATION}.

In Section \ref{geooflocalmodel} below,\;we want to explore the {geometry of local models},\;i.e.\;$\widehat{X}^{(\varphi,\Gamma),\mathrm{ver}}_{\bW_\Dpik,\bF_{\bullet}}\rightarrow {\widehat\fz_{r,\emptyset}}\times X_{\bW_\Dpik,\bF_{\bullet}}$ through formal power series.\;These power formall series should mix ${X}_{\bW_\Dpik,\bF_{\bullet}}$-component (and its $\Sigma_L$-components) and ${\widehat\fz_{r,\emptyset}}$-component (and its $\Sigma_L$-components).\;This phenomena is different from the previous generic (potentially) crystalline case,\;where the local models are established only though ${X}_{\bW_\Dpik,\bF_{\bullet}}$.\;

\subsubsection{Description of local models through formal power series}\label{geooflocalmodel}

By (\ref{functorialproperty}),\;we can choose the cocycle of ${\cM}_{A,i,i-1}$ and $W_{\dr}({\cM}_{A,i,i-1})$ in a functorial way,\;so that we can choose the universal cocycle of the universal $(\varphi,\Gamma)$-modules,\;this allows us to study the module  $W_{i}^{(\varphi,\Gamma)}$ (see the proof of Proposition \ref{constrXLboxphigamma}).\;If we can find a basis of $(\varphi,\Gamma)$-modules and the associated cohomology groups functorial in $A\in\Art_E$,\;then the universal cocycle and differential operator $\partial$ can be written as matrix forms (we explain them below).\;This strategy is achieved in \cite{MSW} for $L=\bQ_p$ case.\;This section is not necessary for our application,\;we just try our best to explain the structure of $\widehat{X}^{(\varphi,\Gamma),\mathrm{ver}}_{\bW_\Dpik,\bF_{\bullet}}$.\;
%,\;which are given by a basis the free $\widehat{\cO}_{\widehat{L^{\times}},\bm{\delta}_{1}}\otimes_E\widehat{\cO}_{\widehat{L^{\times}},\bm{\delta}_{2}}$-module $$N_2=\varprojlim\limits_{\widehat{\cO}_{\widehat{L^{\times}},\bm{\delta}_{i}}\rightarrow A}\ext^1_{(\varphi,\Gamma)}\Big(\Delta_\pi\otimes_{\cR_{A,L}}\cR_{A,L}(\delta_{A,2})\Big[\frac{1}{t}\Big],\Delta_\pi\otimes_{\cR_{A,L}}\cR_{A,L}(\delta_{A,1})\Big[\frac{1}{t}\Big]\Big)$$ of rank $d_L$.\;
We first assume that $k=2$.\;For simplicity,\;we assume $r=1$.\;Let $\bm{\delta}_{i}=\unr(\alpha q_L^{i-1})$ for 
$i=1,2$.\;By the proof of Proposition \ref{proprexmm},\;${X}_{\cM_\Dpik,\cM_{\bullet}}^{\mathrm{ver}}$ is pro-represented by the ring
\[R_2\cong E[[Y_1,\{X_{1,\tau}\}_{\tau\in \Sigma_L },Y_2,\{X_{2,\sigma}\}_{\sigma\in \Sigma_L},\{U_{v}\}_{v\in \Sigma_L}]],\]
the $Y_i,\{X_{i,\tau}\}_{\tau\in \Sigma_L }$ correspond to the coordinates of $\widehat{\cO}_{\widehat{L^{\times}},\bm{\delta}_{i}}$ for $i=1,2$ and $\{U_{v}\}_{1\leq v\leq d_L}$ correspond to the extension parameters.\;Recall the argument after \cite[Theorem 5.3.26]{emerton2023introduction}.\;Let $\cM^{\univ}$ be the universal $(\varphi,\Gamma)$-module over $R_2$.\;Write $\widehat{\cM}^{\univ}=\cM^{\univ}\otimes_{\cR_{E,L}[1/t]}L_{\infty}((t))$ for the scalar extension.\;Then the derivation of the $\Gamma$-action at $1$ which gives a derivation
\[\partial^{\univ}_{\infty}:\widehat{\cM}^{\univ}\rightarrow \widehat{\cM}^{\univ}\]
above the fixed derivation $t\frac{d}{dt}$ on $L_{\infty}((t))$.\;We can define the sub-$L_{\infty}$ vector space $D_{\pdr,\infty}(\widehat{\cM}^{\univ})$ of $\widehat{\cM}^{\univ}$ such that the canonical map
\[D_{\pdr,\infty}(\widehat{\cM}^{\univ})\otimes_{L_{\infty}}L_{\infty}((t))\rightarrow \widehat{\cM}^{\univ}\]
is an isomorphism.\;Let $\bW_{\Dpik}^{\univ}=W_{\dr}(\cM^{\univ})$ be the corresponding universal $B_{\dr}\otimes_{\bQ_p}R_2$-representation of $\gal_L$.\;Then $\partial^{\univ}_{\infty}$ is compatible with the nilpotent operator $\partial^{\univ}$ on $D_{\pdr}(\bW_{\Dpik}^{\univ})$ when modulo $t$ and descent to $L$.\;

%Recall that $\Delta_{\Dpik}$ is the $p$-adic differential equation over $\cR_{E,L}$ associated to $\Dpik$,\;it is and extension of $\Delta_{\pi_0}\otimes_{\cR_{E,L}}\cR_{E,L}(\bm{\delta}_{2})$ by $\Delta_{\pi_0}\otimes_{\cR_{E,L}}\cR_{E,L}(\bm{\delta}_{1})$.\;

We remark that the arguments  in \cite[Sections 5.1-5.4]{MSW} are also suitable for general $L$.\;Similar to the argument before \cite[Lemma 5.4.4,\;Proposition 5.6.2]{MSW},\;by choosing a basis of the $(\varphi,\Gamma)$-modules and the associated cohomology groups (functorial in $A\in\Art_E$ and $X_{\cM_\Dpik,\cM_{\bullet}}^{\mathrm{ver}}(A)$),\;we see that $\cM^{\univ}$ corresponds to a universal cocycle $c_{12}\in R_2$ (and thus get a  matrix $H\in \mathrm{Mat}_{2\times 2}(R_2)$ of the $\partial^{\univ}_{\infty}$,\;see the result in \cite[Lemma 5.4.4]{MSW}).\;Since the universal cocycle $c_{12}$ depends only on the quotient $\bm{\delta}_{1}{\bm{\delta}_{2}}^{-1}$,\;we get $c_{12}=F(Y,\{X_{\tau}\},\{U_{v}\})\in E[[Y,\{X_{\tau}\}_{\tau\in \Sigma_L},\{U_{v}\}_{v\in \Sigma_L}]]$,\;where $Y=Y_1-Y_2$ and $X_{\tau}=X_{1,\tau}-X_{2,\tau}$.\;Moreover,\;the formal power series $F(Y,\{X_{\tau}\},\{U_{v}\})$ is symmetric with respect to the subscripts $\tau\in \Sigma_L$.\;The above discussion on universal cocycles can be generated to arbitrarily $k$ and $r=1$ easily.\;

%,\;which has the form (by definition,\;the extended $\Gamma$-operation is given by the cocycle $c_{12}$):
%\begin{equation}
%	\Bigg(\begin{array}{cc}
%		\sum_{\tau\in\Sigma_L}X_{1,\tau}	& F(Y,\{X_{\tau}\},\{U_{v}\}) \\
%		0	& \sum_{\sigma\in\Sigma_L}X_{2,\sigma}
%	\end{array}\Bigg)
%\end{equation}
%for some $F\in R_2$.\;

\begin{rmk}\label{explianlocalmodelanddefor}
	If $L=\bQ_p$,\;based on the Colmez's  computations on the cohomology of $(\varphi,\Gamma)$-modules,\;\cite[Proposition 5.6.2]{MSW} shows that the universal cocyle $F(Y,X,U)\cong (YU)$.\;The author does not know how to explain $F(Y,\{X_{\tau}\},\{U_{v}\})$ if $L\neq \bQ_p$.\;
	
	%We can explain this fact as follows.\;Note that $W_2'=E[[Y]]$ only depend on the deformations of $\varphi$ which coincides with the discussions on the first paragraph in the proof of \cite[Proposition]{MSW} (not true for $L\neq\bQ_p$).\;So the cocvele $c_{12}$ vanishes in the closed subspace $V(Y)$-locus and then $F(Y,X,U)$must has the $Y$-factor.\;At last,\;we see that $L\neq\bQ_p$ has a quite different behavior.\;
	%,\;the author show that the universal cocycle does not depend on the $\Gamma$-action but only on the $\varphi$-action.\;
\end{rmk}

Consider the ring $R_2':=\widehat{\cO}_{\widehat{L^{\times}},\bm{\delta}_{1}{\bm{\delta}_{2}}^{-1}}\cong E[[Y,\{X_{\tau}\}_{\tau\in \Sigma_L}]]$.\;For any $R_2'\rightarrow A$,\;by \cite[Lemma 3.3.6,\;Lemma 3.3.7]{breuil2019local},\;we see that (induced by the $D_{\pdr}(-)$-functor)
\begin{equation}
	\begin{aligned}
		&\;\hH^1\Big(\gal_L,W_{\dr}(\delta_{A})\Big)\cong \coker( \wt(\delta_{A})-\wt(\bm{\delta}_{1}{\bm{\delta}_{2}}^{-1}):A\otimes_{\bQ_p}L\rightarrow A\otimes_{\bQ_p}L).\;
	\end{aligned}	
\end{equation}
Let $F_{\tau}(X_{\tau})\in E[[Y,\{X_{\tau}\}_{\tau\in \Sigma_L}]]$ be the element which corresponds to the universal $\tau$-weight.\;Similar to the proof of \cite[Lemma 5.6.1]{MSW},\;$F_{\tau}(X_{\tau})\in E[[X_{\tau}]]$ is a uniformizer.\;

\begin{lem}\label{descriptionW2} As a $R_2'$-module,\;we have an isomorphism induced by the $D_{\pdr}(-)$-functor:
\begin{equation}
	\begin{aligned}
		W_{2}':&=\varprojlim_{R'_{2}\rightarrow A}\hH^1\Big(\gal_L,W_{\dr}(\delta_{A})\Big)\\&\cong \prod_{\tau\in \Sigma_L} E[[Y_\tau,\{X_{\sigma,\tau}\}_{\sigma\in \Sigma_L\backslash\tau}]]=R_2'\otimes_{\bQ_p}L/(\{X_{\tau,\tau}\}_{\tau\in\Sigma_L}),\;
	\end{aligned}
\end{equation}
where we add the subscript $\tau$ to $Y,X_{\sigma}$ to indicate that they lies in $\tau$-component.\;In particular,\;$W_{2,\tau}'\cong E[[Y_\tau,\{X_{\sigma,\tau}\}_{\sigma\in \Sigma_L\backslash\tau}]]$.\;Therefore,\;$W_{2,\tau}'$ vanishs in the closed subspace $V(Y,\{X_{\sigma}\}_{\sigma\in \Sigma_L\backslash\{\tau\}})$ of $R_2'$.\;
%\item[(2)] $W'_{2}/W_{2}^{(\varphi,\Gamma)}=(\{G_i\}_{1\leq i\leq t})$ for some power series $G_i:=G_i(Y_\tau,\{X_{\sigma,\tau}\}_{\sigma\neq \tau})$ with the following forms: 
%\[G(Y_\tau,\{X_{\sigma,\tau}\}_{\sigma\neq \tau})=\sum_{(p_{\sigma,\tau},q_\tau)\in S}\prod_{\tau\neq \sigma}X_{\sigma,\tau}^{p_{\sigma,\tau}}Y_\tau^{q_\tau}U_\tau,\]
%where $S:=\{(p_{\sigma,\tau},q_\tau)\in\ \BZ^2:\sigma\neq \tau\in\Sigma_L,\;\text{for any\;} \tau,\;\exists\;\sigma\neq \tau,\;\text{such that\;} (p_{\sigma,\tau},q_\tau)\neq (0,0)\}$.\;
	%Therefore,\;as a $R_2$-module,\;$W_{i}^{(\varphi,\Gamma)}$ isomorphic to $FR_2$.\;
	
\end{lem}
\begin{proof}
	We have the following short exact sequence:
	\[0\rightarrow \varprojlim_{R'_{2}\rightarrow A} A\otimes_{\bQ_p}L\xrightarrow{v_{R'_{2}}}\varprojlim_{R'_{2}\rightarrow A}A\otimes_{\bQ_p}L\rightarrow  W'_{2}\rightarrow 0,\]
	where $v_{R'_{2}}:=\varprojlim_{R'_{2}\rightarrow A} \wt(\delta_{A})=\big(F_{\tau}(X_{\tau})\big)_{\tau\in \Sigma_L}$.\;Hence,\;$W'_{2}\cong E[[Y,\{X_{\tau}\}_{\tau\in \Sigma_L}]]\otimes_{\bQ_p}L/\big(F_{\tau}(X_{\tau})\big)_{\tau\in \Sigma_L}\cong \prod_{\tau\in\Sigma_L}E[[Y_\tau,\{X_{\sigma,\tau}\}_{\sigma\in \Sigma_L\backslash\tau}]]$.\;This completes the proof.\;
%By the argument in Remark \ref{Jversionwdfjrmk},\;the $R_2'$-submodule $W_{2}^{(\varphi,\Gamma)}$ of $W_{2}$ satisfies that $W_{2,\tau}^{(\varphi,\Gamma)}\cong E[[Y_\tau,\{X_{\sigma,\tau}\}_{\sigma\in \Sigma_L\backslash\tau}]]$.\;Then the quotient module $W'_{2}/W_{2}^{(\varphi,\Gamma)}=(\{G_i\}_{1\leq i\leq t})$ which satisfies the conditions $W'_{2,J}/W_{2,J}^{(\varphi,\Gamma)}=0$  are clearly generated by the finitely many polynomials $G(Y_\tau,\{X_{\sigma,\tau}\}_{\sigma\neq \tau})=\sum_{(p_{\sigma,\tau},q_\tau)\in S}\prod_{\tau\neq \sigma}X_{\sigma,\tau}^{p_{\sigma,\tau}}Y_\tau^{q_\tau}U_\tau$.\;
\end{proof}
Recall that \[N_{2,\ast}:=\varprojlim_{R'_{2}\rightarrow A}\hH^1\Big(\gal_L,R_{A,L}(\delta_{A})\Big)\] for $\ast\in \{\emptyset,g\}$.\;Note that $N_{2}$ is a free $R_2'$-module of rank $d_L$.\;Recall that $W_{2}^{(\varphi,\Gamma)}\cong N_{2}/N_{2,g}$ as $R_2'$-module and it is a submodule of $W_2'$.\;Let $\big\langle\big(F_i(Y_\tau,X_{\sigma,\tau})\big)_{\tau\in\Sigma_L}\big\rangle_{i\in I}$ be the generated elements of $W_{2}^{(\varphi,\Gamma)}$ (as a $R_2'$-submodule of $W_2'$).\;Then Remark \ref{Jversionwdfjrmk} translates into
\begin{pro}Assume $L\neq \bQ_p$ and $J\subsetneq \Sigma_L$,\;then $\big\langle\big(F_i(Y_\tau,X_{\sigma,\tau})\big)_{\tau\in J}\big\rangle_{i\in I}$ is equal to the unit ideal of $\prod_{\tau\in J}E[[Y_\tau,\{X_{\sigma,\tau}\}_{\sigma\in \Sigma_L\backslash\tau}]]$.\;
	\end{pro}

\subsection{The case of $(\varphi,\Gamma)$-modules and Galois representations}\label{localmodelforgalodef}

In this section,\;we assume $L\neq \bQ_p$.\;Consider the scheme $\cX_{r,L}:=Y_{r,L}\times_{\widetilde{\fg}_{r,L}}X_{L}$,\;and $\widehat{y}:=(\widehat{y}_1,y_2)\in \cX_{r,L}$ (for $J\subseteq \Sigma_L$,\;we have the corresponding $J$-component $\cX_{r,J}$ and the point $\widehat{y}_J$).\;Define formal scheme $\widehat{\cX}^{\flat}_{r,L,\widehat{y}}:=\widehat{Y}_{r,L,\widehat{y_1}}^{\flat}\times_{\widehat{\widetilde{\fg}}_{r,L,y_1}}\widehat{X}_{r,L,y}$.\;We have natural morphism $\iota^{\flat}:\widehat{\cX}^{\flat}_{r,L,\widehat{y}}\rightarrow \widehat{\cX}_{r,L,\widehat{y}}$.\;For $w\in \sW_{n,L}$.\;let $\cX_{r,w}:=Y_{r,L}\times_{{{\fg}}_{r,L}}X_{r,w}=\cX_{r,L}\times_{X_{r,L}}X_{r,w}$.\;Consider the formal completion $\widehat{X}_{r,w,y}$,\;it is empty if $y\not\in {X}_{r,w}(E)$.\;We put $\widehat{\cX}^{\flat}_{r,w,y}=\widehat{\cX}^\flat_{r,L,\widehat{y}}\times_{\widehat{X}_{r,L,y}}\widehat{X}_{r,w,y}$.\;Denoted by $\widehat{\cO}^{\flat}_{\cX_{r,L},\widehat{y}}$ (resp.,\;$\widehat{\cO}^{\flat}_{\cX_{r,w},\widehat{y}}$) the ring pro-represent
$|\widehat{\cX}^{\flat}_{r,L,\widehat{y}}|$ (resp.,\;$|\widehat{\cX}^{\flat}_{r,w,\widehat{y}}|$).\;

Put $\widehat{X}^{(\varphi,\Gamma),\Box}_{\bW^+_\Dpik,\bF_{\bullet}}:={X}^{\Box}_{\bW^+_\Dpik,\bF_{\bullet}}\times_{\widehat{X}_{\bW_\Dpik,\bF_{\bullet}}}\widehat{X}^{(\varphi,\Gamma)}_{\bW_\Dpik,\bF_{\bullet}}$ and $\widehat{X}^{(\varphi,\Gamma),\Box,w}_{\bW^+_\Dpik,\bF_{\bullet}}:=\widehat{X}^{(\varphi,\Gamma),\Box}_{\bW^+_\Dpik,\bF_{\bullet}}\times_{|X^{(\varphi,\Gamma),\Box}_{\bW^+_\Dpik,\bF_{\bullet}}|}\widehat{X}^\flat_{w,y}$  for $w\in \sW_{n,L}$.\;Similar to \cite[Proposition 6.3.2,\;Corollary 6.3.3,\;Proposition 6.3.4]{Ding2021},\;we see that

%Then $\widehat{\cO}^{\flat}_{X_{r,w},\widehat{y}}\cong \widehat{\cO}_{X_{w},y}\widehat{\otimes}_{\widehat{\cO}_{{\widetilde{\fg}}_{r,L},y_1}}\widehat{\cO}^{\flat}_{{Y}_{r,L},\widehat{y_1}}$ (the complete tensor product).\;The point $\widehat{y}$ corresponds to the maximal ideal $\fm_{\widehat{y}}^{\flat}$,\;i.e.,\;the image of $\fm_y\widehat{\cO}_{X_{w},y}\otimes \widehat{\cO}^{\flat}_{{Y}_{r,L},\widehat{y_1}}+\widehat{\cO}_{X_{w},y}\otimes\fm_{\widehat{y_1}}\widehat{\cO}^{\flat}_{{Y}_{r,L},\widehat{y_1}}$ in $\widehat{\cO}_{X_{w},y}\widehat{\otimes}_{\widehat{\cO}_{{\widetilde{\fg}}_{r,L},y_1}}\widehat{\cO}^{\flat}_{{Y}_{r,L},\widehat{y_1}}$.\;

\begin{pro}We have the following facts.\;
	\begin{itemize}
		\item[(a)] $\widehat{\cO}^{\flat}_{\cX_{r,L},\widehat{y}}$ is a noetherian complete local ring of residue field $E$ and has a finite number of irreducible components.\;Moreover,\;we have a natural homomorphism $\iota^{\flat}:\widehat{\cO}_{\cX_{r,L},\widehat{y}}\rightarrow \widehat{\cO}^{\flat}_{\cX_{r,L},\widehat{y}}$ and $\widehat{\cO}^{\flat}_{\cX_{r,L},\widehat{y}}$ is a formal power series over the ring $\widehat{\cO}_{\cX_{r,J},\widehat{y}_J}$ for any $J\subsetneq \Sigma_L$.\;
		\item[(b)]The groupoid $\widehat{X}^{(\varphi,\Gamma),\Box}_{\bW^+_\Dpik,\bF_{\bullet}}$ (resp.,\;$\widehat{X}^{(\varphi,\Gamma),\Box,w}_{\bW^+_\Dpik,\bF_{\bullet}}$) over $\Art_E$ is pro-representable and the functor $|X^{(\varphi,\Gamma),\Box}_{\bW^+_\Dpik,\bF_{\bullet}}|$ (resp.,$|\widehat{X}^{(\varphi,\Gamma),\Box,w}_{\bW^+_\Dpik,\bF_{\bullet}}|$) is pro-represented by  $\widehat{\cX}^{\flat}_{r,L,\widehat{y}}$ (resp.,\;$\widehat{\cX}^{\flat}_{r,w,\widehat{y}}$).\;
		\item[(c)] The morphism of groupoids $\widehat{X}^{(\varphi,\Gamma),w}_{\bW^+_\Dpik,\bF_{\bullet}}\rightarrow \widehat{X}^{(\varphi,\Gamma)}_{\bW^+_\Dpik,\bF_{\bullet}}$,\;$\widehat{X}^{(\varphi,\Gamma),\Box,w}_{\bW^+_\Dpik,\bF_{\bullet}}\rightarrow \widehat{X}^{(\varphi,\Gamma),\Box}_{\bW^+_\Dpik,\bF_{\bullet}}$ are relatively representable and are closed immersions.\;
	\end{itemize}
\end{pro}

For $w\in \sW_{n,L}$,\;let $\widehat{X}^{(\varphi,\Gamma),w}_{\bW^+_\Dpik,\bF_{\bullet}}$ be the image of $\widehat{X}^{(\varphi,\Gamma),\Box,w}_{\bW^+_\Dpik,\bF_{\bullet}}$ through the forgetful morphism  $\widehat{X}^{(\varphi,\Gamma),\Box}_{\bW^+_\Dpik,\bF_{\bullet}}\rightarrow \widehat{X}^{(\varphi,\Gamma)}_{\bW^+_\Dpik,\bF_{\bullet}}$.\;It is clear that
\begin{equation}
\begin{aligned}
	&X^{w}_{\Dpik,\cM_{\bullet}}\cong X_{\Dpik,\cM_{\bullet}}\times_{\widehat{X}^{(\varphi,\Gamma)}_{\bW^+_\Dpik,\bF_{\bullet}}}\widehat{X}^{(\varphi,\Gamma),w}_{\bW^+_\Dpik,\bF_{\bullet}},\;X^{\Box,w}_{\Dpik,\cM_{\bullet}}\cong X^\Box_{\Dpik,\cM_{\bullet}}\times_{\widehat{X}^{(\varphi,\Gamma),\Box}_{\bW^+_\Dpik,\bF_{\bullet}}}\widehat{X}^{(\varphi,\Gamma),\Box,w}_{\bW^+_\Dpik,\bF_{\bullet}}.
\end{aligned}
\end{equation}
%\cong X^\Box_{\Dpik,\cM_{\bullet}}\times_{|X^{(\varphi,\Gamma),\Box}_{\bW^+_\Dpik,\bF_{\bullet}}|}\widehat{X}^\flat_{r,w,y}

%Put $X_{\Dpik,\cM_{\bullet}}^{(\varphi,\Gamma)}:=X_{\Dpik,\cM_{\bullet}}\times_{X_{\bW_\Dpik^+,\bF_{\bullet}}}X_{\bW_\Dpik^+,\bF_{\bullet}}^{(\varphi,\Gamma)}$ and $X^{(\varphi,\Gamma),\Box}_{\Dpik,\cM_{\bullet}}:=X^{\Box}_{\Dpik,\cM_{\bullet}}\times_{X^{\Box}_{\bW_\Dpik^+,\bF_{\bullet}}}X^{(\varphi,\Gamma),\Box}_{\bW_\Dpik^+,\bF_{\bullet}}$.\;

Let $l_{\Dpik}:=k+d_L\big(k+\frac{n(n-r)}{2}\big)-\dim_E\widehat{X}_{\bW_\Dpik,\bF_{\bullet}}^{(\varphi,\Gamma),\Box}(E[\epsilon]/\epsilon^2)$.\;Similar to \cite[Corollary 6.3.5]{Ding2021},\;we have
\begin{pro}For $w\in \sW_{n,L}$,\;the groupoid $X^\Box_{\Dpik,\cM_{\bullet}}$ (resp.,\;$X^{(\varphi,\Gamma),\Box,w}_{\Dpik,\cM_{\bullet}}$) over $\Art_E$ is pro-representable. The functor $|X^\Box_{\Dpik,\cM_{\bullet}}|$ is pro-represented by a formal scheme which is formally smooth of relative dimension $l_{\Dpik}$ over $\widehat{\cX}^\flat_{r,L,\widehat{y}}$ (resp.,\;$\widehat{\cX}^\flat_{r,w,\widehat{y}}$).\;
\end{pro}
\begin{proof}By Proposition \ref{proprexmm} and Corollary \ref{proformallysmooth} and base change.\;
%For part $(c)$,\;we first note that the groupoid $X^\Box_{\Dpik}$ is flat over $X^\Box_{\bW}$,\;therefore $X^\Box_{\Dpik}\times_{X^\Box_{\bW}}X^\Box_{\bW_\Dpik,\bF_{\bullet}}$ is flat over $X^\Box_{\bW_\Dpik,\bF_{\bullet}}$.\;
\end{proof}

%We follow the route of the proof in \cite[Corollary 6.3.5]{Ding2021}.\;We only need to note that the functor $|X^\Box_{\Dpik,\cM_{\bullet}}|$ is formally smooth over $X_{\bW_\Dpik^+,\bF_{\bullet}}^{(\varphi,\Gamma)}$ of relative dimension $d_L\big(n^2+k+\frac{n(n-r)}{2}\big)-\dim_EX_{\bW_\Dpik,\bF_{\bullet}}^{(\varphi,\Gamma)}(E[\epsilon]/\epsilon^2)$.\;

As in Definition \ref{dfnnoncriticalspecial},\;in the sequel,\;we fix a $p$-adic potentially semistable non-crystalline Galois representation $\rho_L:\gal_L\rightarrow \GLN_{n}(E)$ which admits a special $\omepik$-filtration with parameter $(\bx_0,\bmdel)\in \sbanpik\times\rigchl $,\;(resp.,\; with parameter $(\widetilde{\bx}_{\pi,\bh},\widetilde{\bm{\delta}}_\bh)\in \sbanpik\times\rigch $).\;

Let $w_{\cF}\in \sW^{\Delta_n^k,\emptyset}_{n,L,\max}$ measuring the relative position of the two flags $(\alpha^{-1}(\cD_{\bullet}),\alpha^{-1}(\fil_{\bW^+_\Dpik,\bullet}))$,\;i.e.,\;it lies in the $\GLN_{n,L}$-orbit of $(1,w_{\cF})$ in $ \GLN_{n,L}/\bP_{r,\emptyset,L}\times\GLN_{n,L}/\bB_{L}$.\;We put $\cS(y):=\{w\in\sW_n:y\in {X}_{r,w}(E)\}=\{w\in\sW_n:\widehat{X}_{r,w,y}\neq 0\}=\{w\in\sW_n:X^w_{\bW^+_\Dpik,\bF_{\bullet}}\neq 0\}$.\;

The map
$\kappa:X_{r,L}\rightarrow \cT_{r,L}$ induces a morphism  $\widehat{X}_{r,L,y}\rightarrow \widehat{\cT}_{r,L,(0,0)}$,\;thus the pullback $\kappa:\widehat{\cX}_{r,L,\widehat{y}}\rightarrow \cT_{r,L,(0,0)}$.\;Denoted by $\Theta$ the composition:
\[X^{\Box}_{\rho_L,\cM_{\bullet}}\rightarrow X^{\Box}_{\Dpik,\cM_{\bullet}}\rightarrow \widehat{X}^{\Box}_{\bW^+_\Dpik,\bF_{\bullet}}\xrightarrow{\sim}\widehat{\cX}_{r,L,\widehat{y}}\rightarrow\cT_{r,L,(0,0)}\]
which factors through the morphism $X^{\Box}_{\rho_L,\cM_{\bullet}}\rightarrow X_{\rho_L,\cM_{\bullet}}$ still denoted by  $\Theta:X_{\rho_L,\cM_{\bullet}}\rightarrow \cT_{r,L,(0,0)}$.\;

The main proposition of this section is given as follows.\;

\begin{pro}\label{propertyofxrhombullet} Assume $L\neq \bQ_p$.\;We have the following facts.\;
	\begin{itemize}
		\item[(a)] The groupoid $X_{\rho_L,\cM_{\bullet}}$ (resp.\,$X^{w}_{\rho_L,\cM_{\bullet}}$ for $w\in \cS(y)$) over $\Art_E$ is pro-representable.\;The functor $|X_{\rho_L,\cM_{\bullet}}|$ (resp.,\;$|X^{w}_{\rho_L,\cM_{\bullet}}|$) is pro-represented by a ring $R_{\rho_L,\cM_{\bullet}}$ (resp.,\;$R^{w}_{\rho_L,\cM_{\bullet}}$) of residue field $E$ and dimension $n^2+d_L\big(k+\frac{n(n-1)}{2}\big)$.\;
		\item[(b)] For $w\in \cS(y)$,\;the ring $R^{w}_{\rho_L,\cM_{\bullet}}$ is  irreducible and Cohen-Macaulay.\;The ring $R_{\rho_L,\cM_{\bullet}}$ is equidimensional and $R^{w}_{\rho_L,\cM_{\bullet}}\cong R_{\rho_L,\cM_{\bullet}}/\fp_w$ for a minimal prime ideal $\fp_w$ of $R_{\rho_L,\cM_{\bullet}}$.\;The map $w\mapsto \fp_w$ is a bijection between $\cS(y)$ and the set of minimal prime ideals of $R_{\rho_L,\cM_{\bullet}}$.\;Moreover,\;the dimension of $X^{w}_{\rho_L,\cM_{\bullet}}(E[\epsilon]/\epsilon^2)$ is equal to
		$n^2-n^2d_L+l_{\Dpik}+\dim_E\widehat{\cX}^{\flat}_{w,y}(E[\epsilon]/\epsilon^2)$.\;
		\item[(c)] The morphism $|X^{w}_{\rho_L,\cM_{\bullet}}|\hookrightarrow |X_{\rho_L,\cM_{\bullet}}|\xrightarrow{\Theta} \cT_{r,L,(0,0)}$ of groupoids over $\Art_E$ factors through $\cT_{r,w,(0,0)}\hookrightarrow\widehat{\cT}_{r,(0,0)}$ if and only if $\sW_{\Delta_n^k,L}w'=\sW_{\Delta_n^k,L}w$.\;This implies that if $X^{w}_{\rho_L,\cM_{\bullet}}\neq 0$,\;then $ w\underline{w}_0\geq w_{\cF}$.\; 
	\end{itemize}
\end{pro}
\begin{proof}Note that $X_{\rho_L}\rightarrow X_{\Dpik}$ is relatively representable and formally smooth of relative dimension $n^2$,\;so is the morphism $X^{\Box}_{\rho_L,\cM_{\bullet}}\rightarrow X^\Box_{\Dpik,\cM_{\bullet}}$.\;Then we get that
\begin{equation}
	\begin{aligned}
		\dim_ER_{\rho_L,\cM_{\bullet}}&\;=n^2+d_L\big(n^2+k+\frac{n(n-r)}{2}\big)-\dim_E\widehat{X}_{\bW_\Dpik,\bF_{\bullet}}^{(\varphi,\Gamma)}(E[\epsilon]/\epsilon^2)+\dim_E\widehat{X}_{\bW_\Dpik^+,\bF_{\bullet}}^{(\varphi,\Gamma)}(E[\epsilon]/\epsilon^2)-n^2d_L\\
		&\;=n^2+d_L\big(k+\frac{n(n-1)}{2}\big).
	\end{aligned}
\end{equation}	
We prove $(b)$ by applying the same argument in the proof of \cite[Theorem 6.4.1,\;Proposition 6.4.3]{Ding2021} and using our Proposition \ref{Jversionwdfj}.\;In precise,\;for each $\tau\in\Sigma_L$,\;we see that $X^{\Box,w}_{\rho_L,\cM_{\bullet}}$ is formally smooth over $X^{\Box,w}_{\bW_\Dpik,\bF_{\bullet},\{\tau\}}$.\;Recall that  $\Spec \widehat{\cO}_{X_{r,w_{\tau}},{y}_{\tau}}$ is irreducible,\;we get that $\Spec R^{\Box,w}_{\rho_L,\cM_{\bullet}}$ is irreducible too,\;where $R^{\Box,w}_{\rho_L,\cM_{\bullet}}$ is the ring pro-represents the functor
$|X^{\Box,w}_{\rho_L,\cM_{\bullet}}|$.\;On the other hand,\;the relative dimension of $X^{\Box,w}_{\rho_L,\cM_{\bullet}}$  over $X^{\Box,w}_{\bW_\Dpik,\bF_{\bullet},\{\tau\}}$ is equal to the relative dimension of $X^{\Box}_{\rho_L,\cM_{\bullet}}$  over $X^{\Box}_{\bW_\Dpik,\bF_{\bullet},\{\tau\}}$.\;Since $\Spec \widehat{\cO}_{X_{r,L},y}$ is equidimensional,\;we deduce that $R_{\rho_L,\cM_{\bullet}}$ is equidimensional too.\;Part $(c)$ also follows easily from \cite[Lemma 5.2.7]{Ding2021} and the following argument.\;Let $\cP$ be a geometric property of $X_{r}$ (resp.,\;$X_{r,w}$ for $w\in\sW_n$) that stable under a formally smooth morphism.\;Then $\widehat{\cX}^\flat_{r,L,\widehat{y}}$ (resp.,\;$\widehat{\cX}^\flat_{r,w,\widehat{y}}$,\;where we view them as schemes) inherit the geometric property $\cP$ since $\widehat{\cX}^\flat_{r,L,\widehat{y}}\rightarrow\Spec \widehat{\cO}_{X_{r,w_{\tau}},{y}_{\tau}}$ (resp.,\;$\widehat{\cX}^\flat_{r,w,\widehat{y}}\rightarrow\Spec \widehat{\cO}_{X_{r,w_{\tau}},{y}_{\tau}}$) is formally smooth.\;
\end{proof}
%\begin{rmk}Let $\widehat{\eta}^{\flat}_{r,w}$ be the generic point of $X^{w}_{\rho_L,\cM_{\bullet}}$.\;Then $\kappa|_{\widehat{X}^\flat_{r,L,y}}(\widehat{\eta}^{\flat}_{w})=\widehat{\eta}_{\cT_w}$.\;
%\end{rmk}
\begin{rmk}\label{propertyofxrhombulletrmkconj}
Since $X^{\Box,w}_{\rho_L,\cM_{\bullet}}$ is also formally smooth over $\widehat{X}_{\bW_\Dpik^+,\bF_{\bullet}}^{(\varphi,\Gamma),\Box,w}$, by the proof we see that 
$\widehat{X}_{\bW_\Dpik^+,\bF_{\bullet}}^{(\varphi,\Gamma),\Box,w}\cong |\widehat{X}_{\bW_\Dpik^+,\bF_{\bullet}}^{(\varphi,\Gamma),\Box,w}|$ is irreducible,\;i.e.,\;the scheme $\Spec \widehat{\cO}^{\flat}_{\cX_{r,w},\widehat{y}}$ is irreducible.\;Therefore,\;$\Spec \widehat{\cO}^{\flat}_{\cX_{r,L},y}$ is equidimensional and there is a bijection between irreducible components of  $\Spec \widehat{\cO}^{\flat}_{\cX_{r,L},y}$ and $\Spec \widehat{\cO}_{\cX_{r,L},y}$.\;We suspect that $\Spec \widehat{\cO}^{\flat}_{\cX_{r,L},y}$ and $\Spec \widehat{\cO}_{\cX_{r,L},y}$ have the same dimension,\;and the natural morphism $\iota^{\flat}:\Spec \widehat{\cO}^{\flat}_{\cX_{r,L},y}\rightarrow \Spec \widehat{\cO}_{\cX_{r,L},y}$ maps the generic point of $\Spec \widehat{\cO}^{\flat}_{\cX_{r,L},y}$ to the generic point of $\Spec \widehat{\cO}_{\cX_{r,L},y}$ (in particular,\;$\iota^{\flat}$ is dominant).\;
\end{rmk}
\begin{rmk}\label{alterlneaQp}
If $L=\bQ_p$,\;the same results as in Proposition \ref{propertyofxrhombullet} and the conjecture in \ref{propertyofxrhombulletrmkconj}  can be deduced from \cite[Theroem 5.6.5]{MSW}.\;The results in \cite[Theroem 5.6.5]{MSW} indicate an explicit geometric description of $\widehat{\cX}^{\flat}_{y}$ and  $\widehat{\cX}^{\flat}_{w,y}$. Moreover, the morphism $\iota^{\flat}:\Spec \widehat{\cO}^{\flat}_{\cX_{r,L},y}\rightarrow \Spec \widehat{\cO}_{\cX_{r,L},y}$ should induced by a composition of blow-ups and open immersions of some algebraic schemes.\;The author use a calculation on the explicit basis of $\hH^1_{(\varphi,\Gamma)}(\cR_{E,\bQ_p}(\delta))$ done by Colmez to study the universal cocycle and the universal derivation.\;The author do not know how to generate these methods to $L\neq \bQ_p$-case.\;
%See Section \ref{proofsLQp} for more details.\;
\end{rmk}

\begin{rmk}\label{generaltoOmegafil} Our method are suitable for any potentially semistable Galois representation $\rho_L$ which admits a general $\Omega$-filtration with arbitrarily parameters (with some mild regularity assumptions).\;
\end{rmk}

\subsubsection{Flatness}

Assume $L\neq \bQ_p$.\;Write $\kappa_1:=\kappa_{\bP_{r,\emptyset}}$ and $\kappa_2:=\kappa_{\bB}$ for simplicity.\;This section aims to prove the flatness of the natural morphism $\kappa_{i}:\Spec\widehat{\cO}^{\flat}_{{\cX}^{\flat}_{r,L},\widehat{y}} \rightarrow  \Spec\widehat{\cO}_{\ft_{L},0}$ and $\kappa_{w,i}:\Spec\widehat{\cO}^{\flat}_{{\cX}^{\flat}_{r,{w}},\widehat{y}} \rightarrow  \Spec\widehat{\cO}_{\ft_{L},0}$ (such results are motivated by the previous local models on generic (potentially) crystalline case and the $\bQ_p$-semistable case in  \cite[Proposition 2.3.6]{MSW} (we remark that such results are closely related to the flatness of the weight map of trianguline variety at Steinberg point).\;The main tool is  miracle flatness \cite[Tag 00R3]{Stack}.\;

\begin{pro}\label{proprexmm}
	The natural morphism $X^\Box_{\cM_\Dpik,\cM_{\bullet}}\rightarrow \widehat\fz_{r,\emptyset,L}$ is flat,\;and thus $\Spec\widehat{\cO}^{\flat}_{Y_{r,L},\widehat{y}_1}\rightarrow \Spec \widehat{\cO}_{\fz_{r,\emptyset,L},0} $ is flat.\;
\end{pro}
\begin{proof}
	
	Firstly,\;we prove that $X^{\mathrm{ver}}_{\cM_\Dpik,\cM_{\bullet}}\rightarrow \widehat\fz_{r,\emptyset,L}$ is flat.\;Let $R^{\mathrm{ver}}_{\cM_\Dpik,\cM_{\bullet}}$ be the formally smooth noetherian complete local ring pro-represents the $X^{\Box}_{\cM_\Dpik,\cM_{\bullet}}$.\;Let $\fm_{\widehat\fz_{r,\emptyset,L}}$ be the maximal ideal of the complete local ring $\widehat{\cO}_{\fz_{r,\emptyset,L},0}$.\;Then the ring $R^{\Box}_{\cM_\Dpik,\cM_{\bullet}}/\fm_{\widehat\fz_{r,\emptyset,L}}$ pro-represents the subgroupoid $X^{\mathrm{ver}}_{\cM_\Dpik,\cM_{\bullet},0}$ of
	$X^{\mathrm{ver}}_{\cM_\Dpik,\cM_{\bullet}}$ consists of object $(A,\cM_A,\cM_{A,\bullet},j_A,\underline{\beta}_A)\in X^{\mathrm{ver}}_{\cM_\Dpik,\cM_{\bullet}}(A)$ such that  $\wt_{\tau}(\delta_{A,i})=\wt_{\tau}(\bm{\delta}_{\bh,i})$ .\;We show that $|X_{\cM_\Dpik,\cM_{\bullet},0}^{\mathrm{ver}}|$ is pro-representable by a formally smooth noetherian complete local ring with dimension $k+d_L\frac{n(n-r)}{2}$.\;It is clear that $|X_{\cM_{\Dpik,1},\cM_{\bullet},0}^{\mathrm{ver}}|$ is pro-represented by $E[[x]]$.\;Now assume that $|X_{\cM_{\Dpik,i-1},\cM_{\bullet}}^{\mathrm{ver}}|$ is pro-represented by a formally smooth noetherian complete local ring $R_{i-1}$ of residue field $E$  and dimension $i-1+d_Lr^2\frac{i(i-1)}{2}$.\;Let $S_{i}$ denote the completion of $R_{i-1}\otimes_EE[[x]]$ with respect to the maximal ideal generated by the maximal ideal of $R_{i-1}$ and the one of $E[[x]]$.\;For any morphism $S_{i}\rightarrow A$ with $A\in \Art_E$,\;let $\cM_{i-1,A}$ be the $(\varphi,\Gamma)$-module over $\cR_{A,L}[1/t]$ given by the pull-back along $R_{i-1}\rightarrow S_{i}\rightarrow A$ of the universal $(\varphi,\Gamma)$-module over $\cR_{R_{i-1},L}[1/t]$ and let $\delta_{A,i}$ be the character $L^{\times}\rightarrow E[[x]]\rightarrow S_{i}\rightarrow A$.\;Put $N_{i}:=\varprojlim_{S_{i}\rightarrow A}\ext^1_{(\varphi,\Gamma)}\Big(\Delta_\pi\otimes_{\cR_{A,L}}\cR_{A,L}(\bm{\delta}_{\bh,i}\delta_{A,i})\Big[\frac{1}{t}\Big],\cM_{i-1,A}\Big).$
	Then  $N_{i}$ is a free $S_{i}$-module of rank $(i-1)d_Lr^2$  by Lemma \ref{dimlemmaCM}.\;Then $[\cM_{\Dpik,i}]$ correspondences to a maximal ideal $\fm_{i}$ with residue field $E$ of the polynomial $S_i$-algebra $\mathrm{Symm}_{S_{i}}N_{i}^{\vee}$.\;Let $R_{i}$ be the completion of  $\mathrm{Symm}_{S_{i}}N_{i}^{\vee}$ at $\fm_{i}$.\;One can check that $X^{\Box}_{\cM_i,\cM_{\bullet},0}$ is pro-represented by $R_i$.\;In particular,\;$X^{\mathrm{ver}}_{\cM_\Dpik,\cM_{\bullet},0}$ is pro-represented by a formally smooth noetherian complete local ring of residue field $E$ and dimension $k+d_L\frac{n(n-r)}{2}$.\;Therefore,\;$R^{\mathrm{ver}}_{\cM_\Dpik,\cM_{\bullet}}/\fm_{\widehat\fz_{r,\emptyset,L}}$ has co-dimension $d_Lk=\dim \widehat{\cO}_{\fz_{r,\emptyset,L},0}$ in $R^{\mathrm{ver}}_{\cM_\Dpik,\cM_{\bullet}}$.\;By the miracle flatness,\;we see that $X^{\mathrm{ver}}_{\cM_\Dpik,\cM_{\bullet}}\rightarrow \widehat\fz_{r,\emptyset,L}$ is flat.\;By the proof of \cite[Lemma 2.3.1]{breuil2019local} and \cite[Lemma 2.3.2]{breuil2019local},\;we see that the morphism	$\widetilde{\fg}_{L}\rightarrow \ft_{L},(g\bB_{L},\psi)\mapsto \overline{\mathrm{Ad}(g^{-1})\psi}$ is flat.\;By the proof of Proposition \ref{relativeverbox},\;the natural  morphism $\widehat{X}_{\bW_\Dpik,\bF_{\bullet}}^{\Box}\rightarrow \widehat{X}_{\bW_\Dpik,\bF_{\bullet}}^{\mathrm{ver}}$ is flat.\;This deduces that $X^\Box_{\cM_\Dpik,\cM_{\bullet}}\rightarrow \widehat\fz_{r,\emptyset,L}$ is also flat.\;
\end{proof}

\begin{conjecture}\label{flatnessXconj}Assume $r=1$.\;The natural morphism $X^{\Box}_{\rho_L,\cM_{\bullet}}\rightarrow \widehat\ft_{L}$ (resp.,\;$X^{\Box,w}_{\rho_L,\cM_{\bullet}}\rightarrow \widehat\ft_{L}$) is flat.\;Thus $\kappa_{i}:\Spec\widehat{\cO}^{\flat}_{{\cX}^{\flat}_{L},\widehat{y}} \rightarrow  \Spec\widehat{\cO}_{\ft_{L},0}$ (resp.,\;$\kappa_{w,i}:\Spec\widehat{\cO}^{\flat}_{{\cX}^{\flat}_{{w}},\widehat{y}} \rightarrow  \Spec\widehat{\cO}_{\ft_{L},0}$) is  flat.\;
\end{conjecture}
%\begin{proof}
%Choose $\tau\in \Sigma_L$.\;Then $X^{\Box,w}_{\rho_L,\cM_{\bullet}}$ is formally smooth over $X^{\Box,w}_{\bW_\Dpik,\bF_{\bullet},\{\tau\}}$,\;we get that $\kappa_{w,i,\tau}:\Spec\widehat{\cO}^{\flat}_{{\cX}^{\flat}_{{w}},\widehat{y}} \rightarrow  \Spec \widehat{\cO}_{\ft_{\tau},0}$ is flat.\;This flat morphism factors through $\kappa_{w,i}:\Spec\widehat{\cO}^{\flat}_{{\cX}^{\flat}_{{w}},\widehat{y}} \rightarrow  \Spec\widehat{\cO}_{\ft_{L},0}$.\;Note that the natural projection  $\Spec\widehat{\cO}_{\ft_{L},0}\rightarrow \Spec \widehat{\cO}_{\ft_{\tau},0}$ is a formally smooth morphism,\;

%\end{proof}

\begin{pro}\label{flatnessX}
	Assume that $r=1$,\;the natural morphism $X^{\Box,\underline{w}_0}_{\rho_L,\cM_{\bullet}}\rightarrow \widehat\ft_{L}$ is flat,\;equivalently ,\;$\kappa_{\underline{w}_0,i}:\Spec \widehat{\cO}^{\flat}_{{\cX}^{\flat}_{\underline{w}_0},\widehat{y}} \rightarrow  \Spec\widehat{\cO}_{\ft_{L},0}$ is  flat.\;
\end{pro}
\begin{proof}
	Note that $X^{\underline{w}_0}_{\rho_L,\cM_{\bullet}}$ is Cohen-Macaulay by Proposition \ref{propertyofxrhombullet}.\;Every deformation in $R^{\underline{w}_0}_{\rho_L,\cM_{\bullet}}/\fm_{\widehat\ft_{L}}$ is semistable by \cite[Proposition 2.3.4]{AST2009324R10}.\;But each semistable  deformation $\rho_A$ of $\rho_L$  gives a natural triangulation $\cM_{\bullet,A}$.\;Indeed,\;we can use the arguments before \cite[Theorem 3.3.4]{bergertri},\;let $\rho_A\in \widehat{\FX}_{\overline{r},\underline{w}_0,\rho_L}^{\bh-\mathrm{st}}(A)$ for $A\in\Art_E$,\;and if $M$ is a $(\varphi,N)$-stable subspace of $D_{\mathrm{st}}(\rho_A)$,\;then $(\cR_{A,L}[\log(X),1/t]\otimes_AM)^{N=0}\cap D_{\rig}(\rho_A)$ (resp.,\;$(\cR_{A,L}[\log(X),1/t]\otimes_AM)^{N=0}$) is a sub-$(\varphi,\Gamma)$-module of
	$D_{\rig}(\rho_A)$ (resp.,\;$D_{\rig}(\rho_A)[1/t]$) which has the same rank as $M$.\;We obtain the desired triangulation by using this argument step by step.\;Thus we can identity $R^{\underline{w}_0}_{\rho_L,\cM_{\bullet}}/\fm_{\widehat\ft_{L}}$  with $\widehat{\cO}_{{\FX_{\overline{r},\underline{w}_0}^{\bh-\mathrm{st}}},\rho_L}$, which has dimension $n^2+d_L\frac{n(n-1)}{2}$ (and thus has co-dimension $d_Ln$ in $R^{\underline{w}_0}
	_{\rho_L,\cM_{\bullet}}$).\;By the miracle flatness,\;we finally deduce that $X^{\underline{w}_0}_{\rho_L,\cM_{\bullet}}\rightarrow \widehat\ft_{L}$ is flat and thus $\kappa_{\underline{w}_0,i}:\Spec \widehat{\cO}^{\flat}_{{\cX}^{\flat}_{\underline{w}_0},\widehat{y}} \rightarrow  \Spec\widehat{\cO}_{\ft_{L},0}$ is flat.\;
\end{proof}

%\begin{proof}
%	Suppose that $w\cdot \underline{0}$ is $\underline{I}$-dominant.\;Then by  \cite[Proposition 2.3.4]{AST2009324R10},\;every deformation in $R^{w}_{\rho_L,\cM_{\bullet}}/\fm_{\widehat\ft_{L}}$ is $I$-de Rham,\;so we have $\Spec R^{w}_{\rho_L,\cM_{\bullet}}/\fm_{\widehat\ft_{L}}$ is contained in $\Spec R^{\underline{I}}_{\rho_L,\cM_{\bullet}}\cap \Spec R^{w}_{\rho_L,\cM_{\bullet}}$.\;
%\end{proof}

\begin{rmk}\label{flatnessXconjBQ_p}
	For $L=\bQ_p$ case,\;Conjecture \ref{flatnessXconj} are obtained in \cite[Proposition 2.3.6]{MSW} if we consider the local models given in  \cite[Theorem 5.6.5]{MSW}.\;%See Section \ref{proofsLQp} for more details.\;
\end{rmk}

\begin{rmk}\label{resonfornonflat}
	We also have projections $\Spec\widehat{\cO}^{\flat}_{\cX_{r,L},\widehat{y}}\rightarrow \Spec \widehat{\cO}_{\fz_{r,\emptyset},0}$ and  $\Spec\widehat{\cO}^{\flat}_{\cX_{r,w},\widehat{y}}\rightarrow \Spec \widehat{\cO}_{\fz_{r,\emptyset},0}$,\;which are induced by the composition $\Spec\widehat{\cO}^{\flat}_{\cX_{r,L},\widehat{y}}\rightarrow \Spec\widehat{\cO}^{\flat}_{Y_{r,L},\widehat{y}_1}\xrightarrow{\pr_1 }\fz_{r,\emptyset}$,\;where $\pr_1:Y_{r}=\fz_{r,\emptyset}\times {\widetilde{\fg}}_{r,L}\rightarrow \fz_{r,\emptyset}$.\;They are not necessary flat.\; 
\end{rmk}

\section{Local applications}

We drive several local consequences of the results of local models:\;further properties of Bernstein paraboline varieties (in particular,\;trianguline variety),\;existence of local companion points,\;and a locally analytic ``Breuil-M\'{e}zard type" statement for Steinberg case.\;

\subsection{Local geometry of Bernstein paraboline varieties at special point}

We now recall the definition of Bernstein paraboline varieties \cite[Section 4.2]{Ding2021}.\;Let $\overline{r}:\gal_L\rightarrow\GLN_{n}(k_E)$ be a continuous group morphism.\;The Bernstein paraboline variety $\defvar$ of type {$(\omepik,\bh)$} is a subspace of $\mathfrak{X}_{\overline{r}}^\Box\times \sbanpik\times\rigch$.\;It contains a subspace $U^\Box_{\omepik,\mathbf{{h}}}(\overline{r})$  consists of the point $(\rho,\underline{x},\undelram)$ such that
\begin{itemize}\label{dfnvardef}
	\item[(1)] $(\underline{x},\undelram)\in \big(\sbanpik\times\rigch\big)^{\gen}$ (the set of generic points in $\sbanpik\times\rigch$,\;see \cite[Section 4.2]{Ding2021}),
	\item[(2)] $D_{\rig}(\rho)$ admits an $\omepik$-filtration $\cF=\fil_{\bullet}^{\cF} D_{\rig}(\rho)$ such that
	\begin{equation}\label{paradefiinj}
		\begin{aligned}
			\gr_{i}^{\cF}D_{\rig}(\rho)\otimes_{\cR_{k(x),L}}\cR_{k(x),L}((\delta_{i}^0)^{-1}_{\varpi_L})\hookrightarrow \Delta_{x_{i}}\otimes_{\cR_{k(x),L}}\cR_{k(x),L}(z^{\bh_{ir}})
		\end{aligned}
	\end{equation}
	and the image has Hodge-Tate weights $(\bh_{(i-1)r+1},\cdots,\bh_{ir})$.\;
%By using Berger's equivalence of categories \cite[Theorem  A]{berger2008equations} and comparing the Hodge-Tate weights,\;we see that (\ref{paradefiinj}) is equivalent to 
%	\begin{equation}\label{dfnvardef1}
%		\Delta_{x_{i}}\otimes_{\cR_{k(x),L}}\cR_{k(x),L}(z^{\bh_{(i-1)r+1}})  \hookrightarrow\gr_{i}^{\cF}D_{\rig}(\rho)\otimes_{\cR_{k(x),L}}\cR_{k(x),L}((\delta_{i}^0)^{-1}_{\varpi_L}).
%	\end{equation}
\end{itemize}
Then $\defvar$ is the Zariski-closure of $U^\Box_{\omepik,\mathbf{{h}}}(\overline{r})$ in $\mathfrak{X}_{\overline{r}}^\Box\times \sbanpik\times\rigch$.\;By \cite[Theorem\;4.2.5,\;Corollary\;4.2.5]{Ding2021},\;we have:
\begin{pro}\label{propertyparavar}\hspace{20pt}
	\begin{itemize}
		\item[(1)] The rigid space $\defvar$ is equidimensional of dimension $n^2+\left(\frac{n(n-1)}{2}+k\right)d_L$.\;
		\item[(2)] The set $U^\Box_{\omepik,\mathbf{{h}}}(\overline{r})$ is Zariski-open and  Zariski-dense in $\defvar$,\;and is smooth over $E$.\;
%		\item[(3)] The rigid space  $U^\Box_{\omepik,\mathbf{{h}}}(\overline{r})$ is ,\;and the morphism $${\omega}|_{U^\Box_{\omepik,\mathbf{{h}}}(\overline{r})}:U^\Box_{\omepik,\mathbf{{h}}}(\overline{r})\rightarrow \sbanpik\times\rigch$$ is smooth.\;
		\item[(3)] Let $x=(\rho_x,\underline{\pi}_x,\undelram)\in \defvar$,\;then $D_{\rig}(\rho_x)$ admits an $\omepik$-filtration $\cF=\{\fil_{i}^\cF D_{\rig}(\rho_x)\}$ such that,\;for all $1=1,\cdots,s$,\;
		\[\gr_{i}^{\cF}D_{\rig}(\rho_x)\otimes_{\cR_{k(x),L}}\cR_{k(x),L}((\delta_{i}^0)^{-1}_{\varpi_L})\Big[\frac{1}{t}\Big]=\Delta_{x_{i}}\Big[\frac{1}{t}\Big].\]
	\end{itemize}
\end{pro}
%\begin{rmk}In general,\;$(\underline{x},((\delta_{i}^0)_{\varpi_L}z^{\bh_{ir}}))$ is not a right parameter  of the $\omepik$-filtration $\cF$ in (3).\;\end{rmk}
\begin{rmk}\label{concidetotrivar} (Comparison with trianguline variety) If $r=1$ (so $k=n$),\;we have $\bL_{1,\emptyset}=\bT$ and $\bP_{1,\emptyset}=\bB$.\;Let $\widehat{T}_L$ denote the character space of $\bT(L)$ over $E$,\;i.e.,\;the rigid space over $E$ parameterizing continuous character of $\bT(L)$.\;Then \cite[Remark 4.2.4]{Ding2021} show that $\defvar$ coincides with trianguline variety $X^{\Box}_{\mathrm{tri}}(\overline{r})$ of \cite[Section 2.2]{breuil2017interpretation},\;by using the isomorphism
\[\iota_{\Omega_{1}^{\otimes n},\bh}: \big(\mathrm{Spec}\hspace{2pt}\mathfrak{Z}_{\Omega_{1}^{\otimes n}}\big)^{\mathrm{rig}}\times\mathcal{Z}_{\bL_{1,\emptyset},\mathcal{O}_L}\xrightarrow{\sim}\widehat{T}_L,(\underline{\pi}_x,\undelram)\mapsto (\boxtimes_{i=1}^r\pi_{x_i}) \undelram z^{\bh}.\]
As in \cite[Section 3.7]{breuil2019local},\;we will drop the $\Box$ in $X^{\Box}_{\mathrm{tri}}(\overline{r})$ in Section \ref{conjsection} in order to avoid any confusion with the other kind of framing used in local models.\;
\end{rmk}

Let $\rho_L$ be a Galois representation which admits a special $\omepik$-filtration (i.e.,\;We keep the situation in Proposition \ref{propertyofxrhombullet}).\;Suppose that $x=(\rho_L,\underline{\pi}_x,\undelram)$ appears on the $\defvar$.\;We have a natural morphism of formal schemes (recall that $(\widehat{\mathfrak{X}_{\overline{r}}^\Box})_{\rho_L}$ is equivalent to $X_{\rho_L}\cong |X_{\rho_L}|$) $$\complocaldefvarrhox \rightarrow (\widehat{\mathfrak{X}_{\overline{r}}^\Box})_{\rho_L}\cong X_{\rho_L}.$$
By the argument before \cite[Proposition 6.4.6]{Ding2021},\;there exists $w_x=(w_{x,\tau})_{\tau\in \Sigma_L}\in\sW^{\Delta_n^k,\emptyset}_{n,L}$ such that,\;for $1\leq j\leq n$,\;$\bh_{\tau,w_{x,\tau}^{-1}(j)}=\wt_\tau(\chi_i)+\bh_{\tau,j}$ where $i$ is the integer such that $(i-1)r<j\leq ir$.\;

The proof of \cite[Proposition 6.4.5,\;Proposition 6.4.6,\;Corollary 6.4.7,\;Corollary 6.4.8]{Ding2021} are also suitable for our case.\;By an easy variation of the above proofs,\;we deduce from Proposition \ref{propertyofxrhombullet}:
\begin{pro}\label{localgeomertyonspecial}Assume $L\neq \bQ_p$.\;We have
\begin{itemize}
	\item[(a)] The canonical morphism $\complocaldefvarrhox \rightarrow  X_{\rho_L}$ factors through a morphism
	$\complocaldefvarrhox \rightarrow X_{\rho_L,\cM_{\bullet}}.$
	\item[(b)] The morphism $\complocaldefvarrhox \rightarrow  X_{\rho_L}$ and $\complocaldefvarrhox \rightarrow  X_{\rho_L,\cM_{\bullet}}$ are closed immmersions.\;
	\item[(c)] Denote by $\Theta_x:\complocaldefvarrhox \rightarrow \widehat{\cT}_{r,L,(0,0)}$ the composition of fomall schemes:
	\[\complocaldefvarrhox\hookrightarrow X_{\rho_L,\cM_{\bullet}}\xrightarrow{\Theta}\widehat{\cT}_{r,L,(0,0)}.\]
	Then morphism $\Theta_x$ factors through $\widehat{\cT}_{r,w_x\underline{w}_0,(0,0)}$.\;
	\item[(d)] The closed immersion $\complocaldefvarrhox \rightarrow  X_{\rho_L,\cM_{\bullet}}$ in $(2)$ factors through an isomorphism of groupoids over $\Art_E$:
	\[\complocaldefvarrhox \xrightarrow{\sim} X^{w_{x}\underline{w}_0}_{\rho_L,\cM_{\bullet}}.\]
	Then $w_x\underline{w}_0\geq w_{\cF}$.\;In particular,\;$\defvar$ is irreducible
 and  Cohen-Macaulay at the point $x$.\;
\end{itemize}
\end{pro}

\begin{rmk} In the case $L=\bQ_p$,\;see \cite[Proposition 6.2.7,\;Corollary 6.2.9,Theorem 6.2.10]{MSW} for related results.\;The author also shows that trianguline variety is normal and Cohen-Macaulay at the semistable point $x$.\;%See Section \ref{proofsLQp} for more details.\;
\end{rmk}

\subsection{Local companion points}

In the remainder of this section,\;we restrict to the case $r=1$.\;By Remark \ref{concidetotrivar},\;the corresponding point of $(\rho_L,\underline{\pi}_x,\undelram)\in X^\Box_{\Omega_{[1,n]},\mathbf{{h}}}(\overline{r})$ via the isomorphism $X^\Box_{\Omega_{[1,n]},\mathbf{{h}}}(\overline{r})\xrightarrow{\sim}X^{\Box}_{\mathrm{tri}}(\overline{r})$ is $x=(\rho_L,\chi)\in X^{\Box}_{\mathrm{tri}}(\overline{r})$ with $\chi:=\iota_{\Omega_{[1,n]},\bh}(\underline{\pi}_x,\undelram)\in \widehat{T}_L$.\;

Recall that $\rho_L$ is a semistable non-crystalline $p$-adic Galois representation with full monodromy rank,\;i.e., the monodromy operator $N$ on $D_{\mathrm{st}}(\rho_{L})$ satisfies $N^{n-1}\neq 0$.\;Let $\bh:=(\hpi_{\tau,1}>\hpi_{\tau,2}>\cdots>\hpi_{\tau,n} )_{\tau\in \Sigma_L}$ be the Hodge-Tate weights of $\rho_{L}$.\;There is a unique ${\alpha}\in E^\times$ such that ${\alpha}$, ${\alpha q_L^{1}}, \cdots ,{\alpha q_L^{n-1}}$ are  $\varphi^{f_L}$-eigenvalues of $D_{\mathrm{st}}(\rho_{L})$.\;Put $\unr(\alpha)_n:=(\unr({\alpha}),\unr({\alpha q_L^{1}}),\;\cdots,\;\unr({\alpha q_L^{n-1}}))$.\;Then $D_{\rig}(\rho_{L})$ admits a triangulation $\cF$ with parameters $(\unr(\alpha)_n)\cdot z^{w_{\cF}\underline{w}_0(\bh)}$ for $w_{\cF}\in \sW^{\emptyset,\emptyset}_{n,L,\max}\cong \sW_{n,L}$.\;Note that $\rho_L$ is non-critical if $w_{\cF}=w_0$.\;

Denote by $\hpi_{i}=(\hpi_{\tau,i})_{\tau\in \Sigma_L}$ for $1\leq i\leq n$.\;For $w\in\sW_{n,L}$,\;we put $x_{w\underline{w}_0}:=(\rho_{L},(\unr(\alpha)_n)\cdot z^{w\underline{w}_0(\bh)})\in\mathfrak{X}_{\overline{r}}^\Box\times\widehat{T}_L$.\;In particular,\;we write $x:=x_{1}$.\;The goal of this section is to show that $\{x_{w\underline{w}_0}\}_{w\leq w_{\cF}\underline{w}_0}$ are local companion points of $x$,\;i.e.,\;$x_{w\underline{w}_0}\in X_{\mathrm{tri}}(\overline{r}$) for $w\leq w_{\cF}\underline{w}_0$,\;under some mild hypothesis on trianguline variety.\;

For $\rho'_{L}\in \FX_{\overline{r},\cP_{\min}}^{\Box,\bh-\mathrm{st}}$,\;there exists a unique $a_{\rho'_{L}}\in k(\rho'_{L})$ such that $a_{\rho'_{L}},\;\cdots\,\;{a_{\rho'_{L}} q_L^{i-1}},\;\cdots\,\;,{a_{\rho'_{L}} q_L^{n-1}}$ are $\varphi^{f_L}$-eigenvalues of $(\varphi,N)$-module $D_{\mathrm{st}}(\rho'_{L})$.\;Consider the following morphisms of rigid spaces over $E$:
\begin{equation}\label{stdeformationtorighttrivar}
	\begin{aligned}
		\iota_{\bh}:\FX_{\overline{r},\cP_{\min}}^{\Box,\bh-\mathrm{st}}&\rightarrow \mathfrak{X}_{\overline{r}}^\Box\times\widehat{T}_L\\
		\rho'_{L}&\mapsto (\rho'_{L},z^{\bh}\unr(a_{\rho'_{L}})_n).\;
	\end{aligned}
\end{equation}
and 
\begin{equation}\label{sedeformationtotrivar}
	\begin{aligned}
		\iota_{w,\bh}:\FX_{\overline{r},\cP_{\min}}^{\Box,\bh-\mathrm{st}}&\rightarrow \mathfrak{X}_{\overline{r}}^\Box\times\widehat{T}_L\\
		\rho'_{L}&\mapsto (\rho'_{L},z^{w(\bh)}\unr(a_{\rho'_{L}})_n).\;
	\end{aligned}
\end{equation}
By definition,\;we have $\iota_{\bh}=\iota_{1,\bh}$.\;Moreover,\;for $w\in\sW_{n,L}$,\;let  $\FX_{\mathrm{tri},\cP_{\min},w}^{\bh-\mathrm{st}}(\overline{r})$  be the inverse image of  $X_{\mathrm{tri}}(\overline{r})$  via $\iota_{\bh,w}$,\;which is a closed subspace of $\FX_{\overline{r},\cP_{\min}}^{\Box,\bh-\mathrm{st}}$.\;

For any $\rho'_{L}\in \FX_{\overline{r},\cP_{\min}}^{\Box,\bh-\mathrm{st}}$,\;by the argument in Section \ref{Omegafil},\;the unique $(\varphi,N)$-stable complete flag on $D_{\mathrm{st}}(\rho'_{L})$ determine a element $w_{\rho'_{L}}\in \sW^{\emptyset,\emptyset}_{n,L,\max}\cong\sW_{n,L}$.\;Then $\rho'_{L}\in \FX_{\overline{r},\cP_{\min},w}^{\Box,\bh-\mathrm{st}}$  if and only $w_{\rho'_{L}}\underline{w}_0=w$.\;Put $\widetilde{\FX}_{\mathrm{tri},\cP_{\min},w}^{\bh-\mathrm{st}}(\overline{r})=\FX_{\mathrm{tri},\cP_{\min},w}^{\bh-\mathrm{st}}(\overline{r})\cap \FX_{\overline{r},\cP_{\min},w}^{\Box,\bh-\mathrm{st}}$.\;

%In particular,\;we set $\FX_{\mathrm{tri},\cP_{\min}}^{\bh-\mathrm{st}}(\overline{r}):=\FX_{\mathrm{tri},\cP_{\min},1}^{\bh-\mathrm{st}}(\overline{r})$.\;

%For each irreducible component $\FX^{\fp}$ of  $\FX_{\overline{\rho}^{\fp}}^\Box$,\;we also put$\FX_{\mathrm{tri},\cP_{\min}}^{\bh-\mathrm{st},\FX^{\fp}-\mathrm{aut}}(\overline{r}):=\iota_\bh^{-1}(X^{\FX^{\fp}-\mathrm{aut}}_{\mathrm{tri}}(\overline{r}))$.\;

%\begin{lem}We have $\FX_{\mathrm{tri},\cP_{\min},w,}^{\bh-\mathrm{st}}(\overline{r})\subseteq \FX_{\mathrm{tri},\cP_{\min},w}^{\bh-\mathrm{st}}(\overline{r})$ if $w'\leq w$.\;
%\end{lem}
%\begin{proof}Recall the injection $\iota_{w,\bh}:\FX_{\overline{r},\cP_{\min},w}^{\Box,\bh-\mathrm{st}}\rightarrow X_{\mathrm{tri}}(\overline{r})$ of rigid spaces over $L$.\;By Proposition \ref{semistableZarisikeargsecond},\;we see that $\overline{\FX_{\overline{r},\cP_{\min},w}^{\Box,\bh-\mathrm{st}}}\subseteq \iota_{w,\bh}^{-1}(X_{\mathrm{tri}}(\overline{r}))$.\;
%\end{proof}

\begin{hypothesis}\label{appenhypothesis}
For any $w\in\sW_n$,\;$\widetilde{\FX}_{\mathrm{tri},\cP_{\min},w}^{\bh-\mathrm{st}}(\overline{r})$ is equal to $\FX_{\overline{r},\cP_{\min},w}^{\Box,\bh-\mathrm{st}}$.\;
\end{hypothesis}
\begin{rmk}\label{hyoprecatesexplian}
	This hypothesis is predicted by \cite[Conjecture 5.3.13]{emerton2023introduction} and \cite[Remark 5.3.5]{emerton2023introduction}.\;Indeed,\;for the generic crystalline case (see \cite{breuil2019local}),\;such hypothesis hold automatically (since $\iota_{w,\bh}(\rho_L)\in U_{\mathrm{tri}}(\overline{r})$,\;for generic $\rho_L\in \FX_{\overline{r},\cP_{\max},w}^{\Box,\bh-\mathrm{st}}$).\;The author believes that such Hypothesis \ref{appenhypothesis} (i.e.\;Steinberg case) is also true.\;
\end{rmk}

%	This hypothesis is equivalent to say that all the dominant points $(\rho_L',\unr(a_{\rho_L'})_nz^{\bh})\in X_{\mathrm{tri}}(\overline{r})\subset\mathfrak{X}_{\overline{r}}^\Box\times\widehat{T}_L$.\;It is predicted by \cite[Conjecture 5.3.13]{emerton2023introduction},\;by combining the description in \cite[Remark 5.3.5]{emerton2023introduction} and the last assertion in \cite[Proposition 4.12]{parabolinePHD} (or see the proof of \cite[Theorem 6.3.13]{XLROBBA}).\;On the other hand,\;it is also conjectured that the subspace $\FX_{\overline{r},\cP_{\min},\underline{w}_0}^{\Box,\bh-\mathrm{st}}$ of $\FX_{\overline{r},\cP_{\min}}^{\Box,\bh-\mathrm{st}}$ consists of non-critical points is contained in $\FX_{\mathrm{tri},\cP_{\min}}^{\bh-\mathrm{st}}(\overline{r})$.\;Then we deduce Hypothesis \ref{appenhypothesis} from this conjecture by applying then Zariski-closure of $\FX_{\overline{r},\cP_{\min},\underline{w}_0}^{\Box,\bh-\mathrm{st}}$ in $\FX_{\overline{r},\cP_{\min}}^{\Box,\bh-\mathrm{st}}$ (Recall the definition of $\FX_{\overline{r},\cP_{\min},w}^{\Box,\bh-\mathrm{st}}$ before Proposition \ref{pminloucsofsemidefringscP}).\;

%For $w\in\sW_{n,L}$,\;denote by $\widehat{T}_L_{w,\bh}\subset \widehat{T}_L$ the Zariski-closed (reduced) subset of characters $(\eta_1,\cdots,\eta_n)$ defined by the equations:
%\[\wt_\tau(\eta_{w_{\tau}(i)}\eta_i^{-1})=\bh_{\tau,i}-\bh_{\tau,w^{-1}_{\tau}(i)}\]
%for $1\leq i\leq n$ and $\tau\in \Sigma_L$.\;For instance one has $z^{w(\bh)}\unr(a_{\rho_{L}})_n\in \widehat{T}_L_{w,\bh}$.\;
Then  Proposition \ref{pminloucsofsemidefringscP} gives:
\begin{pro}\label{semistableZarisikeargsecond} For any $w'\in\sW_{n,L}$,\;we put $\widetilde{\FX}_{\mathrm{tri},\cP_{\min},(w,w')}^{\bh-\mathrm{st}}(\overline{r})=\FX_{\mathrm{tri},\cP_{\min},w}^{\bh-\mathrm{st}}(\overline{r})\cap \FX_{\overline{r},\cP_{\min},w'}^{\Box,\bh-\mathrm{st}}$.\;
	\begin{itemize}
		\item[(1)] We have $\widetilde{\FX}_{\mathrm{tri},\cP_{\min},(w,w')}^{\bh-\mathrm{st}}(\overline{r})=0$ if $w<w'$.\;
		\item[(2)]	 The Zariski-closure $\overline{\widetilde{\FX}_{\mathrm{tri},\cP_{\min},w}^{\bh-\mathrm{st}}(\overline{r})}\subseteq \FX_{\mathrm{tri},\cP_{\min},w\underline{w}_0}^{\bh-\mathrm{st}}(\overline{r})$.\;We have
		\[\overline{\widetilde{\FX}_{\mathrm{tri},\cP_{\min},w}^{\bh-\mathrm{st}}(\overline{r})}=\coprod_{w'\leq w}\widetilde{\FX}_{\mathrm{tri},\cP_{\min},(w,w')}^{\bh-\mathrm{st}}(\overline{r}),\]
		where the Zariski-closure is taken in $\FX_{\mathrm{tri},\cP_{\min},w}^{\bh-\mathrm{st}}(\overline{r})$.\;
		\item[(3)] Assume Hypothesis \ref{appenhypothesis}.\;Then we have \[\overline{\widetilde{\FX}_{\mathrm{tri},\cP_{\min},w}^{\bh-\mathrm{st}}(\overline{r})}=\coprod_{w'\leq w}\widetilde{\FX}_{\mathrm{tri},\cP_{\min},w'}^{\bh-\mathrm{st}}(\overline{r}),\]
		where the Zariski-closure is taken in $\FX_{\mathrm{tri},\cP_{\min}}^{\bh-\mathrm{st}}(\overline{r})$.\;In particular,\;$\rho'_{L}\in\overline{\widetilde{\FX}_{\mathrm{tri},\cP_{\min},w}^{\bh-\mathrm{st}}(\overline{r})}$ if and only $w_{\rho'_{L}}\underline{w}_0\leq w$.\;		
			\end{itemize}
\end{pro}
\begin{proof}
Part $(1)$ follows from Proposition \ref{localgeomertyonspecial} $(d)$.\;We use the same strategy  as in the proof of Proposition \ref{pminloucsofsemidefringscP} to prove Part $(2)$.\;We get that \[\FY_{\overline{r},\cP_{\min},w}^{\Box,\bh-\mathrm{st}}\times_{\FX_{\overline{r},\cP_{\min},w}^{\Box,\bh-\mathrm{st}}}{\FX}_{\mathrm{tri},\cP_{\min},w}^{\bh-\mathrm{st}}(\overline{r}) \rightarrow \big((\mathrm{Res}_{L/\bQ_p}\GLN_{n,L}/\mathrm{Res}_{L/\bQ_p}\bB)\times_{\bQ_p} E\big)^{\rig}\] is also an open morphism.\;By the similar argument as in the proof of  Proposition \ref{pminloucsofsemidefringscP} (i.e.\;descent along the map $\FY_{\overline{r},\cP_{\min}}^{\Box,\bh-\mathrm{st}}\rightarrow \FX_{\overline{r},\cP_{\min}}^{\Box,\bh-\mathrm{st}}$),\;we deduce from the identity in Part $(2)$.\;Part $(3)$ is a direct consequence of Proposition \ref{pminloucsofsemidefringscP},\;if Hypothesis \ref{appenhypothesis} holds.\;
\end{proof}

%\begin{rmk}Moreover,\;the closed immersion $\iota_{\bh}:\widetilde{\FX}_{\mathrm{tri},\cP_{\min},\underline{w}_0}^{\bh-\mathrm{st}}(\overline{r})\rightarrow X_{\mathrm{tri}}(\overline{r})$ extends to a closed immersion:
%$\iota_{\bh}:\overline{\widetilde{\FX}_{\mathrm{tri},\cP_{\min},\underline{w}_0}^{\bh-\mathrm{st}}(\overline{r})}\rightarrow X_{\mathrm{tri}}(\overline{r}).$ Then we have a closed immersion $\overline{\widetilde{\FX}_{\mathrm{tri},\cP_{\min},\underline{w}_0}^{\bh-\mathrm{st}}(\overline{r})}\hookrightarrow\FX_{\mathrm{tri},\cP_{\min}}^{\bh-\mathrm{st}}(\overline{r})$.\;Put $\widetilde{\FX}_{\mathrm{tri},\cP_{\min},w}^{\bh-\mathrm{st},\#}(\overline{r})=\overline{\widetilde{\FX}_{\mathrm{tri},\cP_{\min},\underline{w}_0}^{\bh-\mathrm{st}}(\overline{r})}\cap \FX_{\overline{r},\cP_{\min},w}^{\Box,\bh-\mathrm{st}}$.\;It may be worth having a method to compare $\overline{\widetilde{\FX}_{\mathrm{tri},\cP_{\min},\underline{w}_0}^{\bh-\mathrm{st}}(\overline{r})}$ (resp.,\;$\widetilde{\FX}_{\mathrm{tri},\cP_{\min},w}^{\bh-\mathrm{st},\#}(\overline{r})$) with $\FX_{\mathrm{tri},\cP_{\min}}^{\bh-\mathrm{st}}(\overline{r})$ and $\FX_{\overline{r},\cP_{\min}}^{\Box,\bh-\mathrm{st}}$ (resp.,\;$\widetilde{\FX}_{\mathrm{tri},\cP_{\min},w}^{\bh-\mathrm{st}}(\overline{r})$).\;
%\end{rmk}

\begin{dfn} A point $(\rho_L'',(\delta''_i)_{1\leq i\leq n})\in X_{\mathrm{tri}}(\overline{r})$ is called a local companion point of  $(\rho_L',(\delta'_i)_{1\leq i\leq n})\in X_{\mathrm{tri}}(\overline{r})$  if $\rho_L''=\rho_L'$ and $\delta''_i/\delta'_i$ is algebraic.\;
\end{dfn}
Therefore,\;if $x'=(\rho_L,(\delta'_i)_{1\leq i \leq n})\in X_{\mathrm{tri}}(\overline{r})$ is a local companion point of  $x$,\;then there exists a $w'\in \sW_{n,L}$ such that $\wt(\delta)=w'(\bh)$.\;Thus each companion point  of $x$ is of the form $x_{w}:=(\rho_L,\unr(\alpha)_nz^{w(\bh)})$ for some $w\in \sW_{n,L}$.\;The following proposition is an analogue of \cite[Theorem 4.2.3]{breuil2019local} (but in our setting).\;
%The Zarisiki closure argument in \cite[Section 4.2]{breuil2019local} is not suitable for our case.\;
%We have following results on local companion points of $x$.\;
\begin{pro}\label{prolocalcomapn1}Assume  Hypothesis \ref{appenhypothesis} (so that $x_{w_{\cF}\underline{w}_0}\in X_{\mathrm{tri}}(\overline{r})$).\;Then $x_{w\underline{w}_0}\in X_{\mathrm{tri}}(\overline{r})$ if and only if $w_{\cF}\leq w$ (in partcular,\;they are local companion points of $x$ and $x_{w_{\cF}\underline{w}_0}$).\;
	%	\begin{itemize}
		%		\item[(1)] If $x_{w\underline{w}_0}$ is a  companion point of $x$,\;then $w\geq w_{\cF}$.\;Therefore,\;there exist $\{w_i^{(x)}\}_{1\leq i\leq d}$ with $w_{\cF}\leq w_{i}^{(x)}$ consists of the minimal companion points of $x$.\;Then all the companion points of $x$ are $\{x_{w\underline{w}_0}\}_{w_{i}^{(x)}\leq w}$.\;
		%		\item[(2)] If $d=1$ and $w_1^{(x)}=w_{\cF}$.\;Then 
		%	\end{itemize}
\end{pro}
\begin{proof}
	By Proposition \ref{semistableZarisikeargsecond},\;we have an injection $\iota_{w\underline{w}_0,\bh}:\widetilde{\FX}_{\mathrm{tri},\cP_{\min},w\underline{w}_0}^{\bh-\mathrm{st}}(\overline{r})\rightarrow X_{\mathrm{tri}}(\overline{r})$ of rigid spaces over $L$.\;By Proposition \ref{semistableZarisikeargsecond},\;we see that $\overline{\widetilde{\FX}_{\mathrm{tri},\cP_{\min},w\underline{w}_0
		}^{\bh-\mathrm{st}}(\overline{r})}\subseteq \iota_{w\underline{w}_0 ,\bh}^{-1}(X_{\mathrm{tri}}(\overline{r}))$.\;Since $x_{w_{\cF}\underline{w}_0}\in X_{\mathrm{tri}}(\overline{r})$ and $w_{\cF}\leq w$,\;the result follows.\;
\end{proof}
\begin{rmk}If $x_{w_{\cF}\underline{w}_0}\in {X}_{\mathrm{tri}}(\overline{r})$.\;Let $\widetilde{X}_{\mathrm{tri}}(\overline{r})\subseteq {X}_{\mathrm{tri}}(\overline{r})$ be the subspace defined before \cite[Conjeture 2.8]{breuil2017smoothness},\;then  \cite[Conjecture 3.23]{breuil2017interpretation} (or \cite[Conjeture 5.6]{breuil2017smoothness},\;modular lifting theorem) shows that there exists a point $z\in \FX_{\overline{\rho}^{\fp}}^\Box\times \BU^g$ such that $(\iota_{\fp}(x_{w\underline{w}_0}),z)\in \iota_{\fp}\big(X_{\mathrm{tri}}(\overline{r})\big)\times
	\FX_{\overline{\rho}^{\fp}}^\Box\times \BU^g$ are in $X_{\fp}(\overline{\rho})(E)$.\;Then \cite[Theorem 5.5]{breuil2017smoothness} gives the existence of all global companion points,\;i.e.,\;$y_{w\underline{w}_0}\in X_{\fp}(\overline{\rho})(E)$ if and only if $w\leq w_{\cF}\underline{w}_0$.\;
\end{rmk}

\subsection{Galois cycles}\label{galoiscycles}

We construct certain cycles on the deformation space $X_{\rho_L}\cong (\widehat{\mathfrak{X}_{\overline{r}}^\Box})_{\rho_L}$.\;We follow the notation in \cite[Section 6.5]{Ding2021}.\;We denote by $Z(\Spec \widehat{\cO}_{\mathfrak{X}_{\overline{r}},\rho_L})$ (resp.,\;$Z^{d}(\Spec \widehat{\cO}_{\mathfrak{X}_{\overline{r}},\rho_L})$,\;resp.,\;$Z^{\leq d}(\Spec \widehat{\cO}_{\mathfrak{X}_{\overline{r}},\rho_L})$) for the free abelian group  generated by the irreducible closed subschemes (resp.,\;the irreducible closed subschemes of codimension $d$,\;resp.,\;the irreducible closed subschemes of codimension  $\leq d$) in $\widehat{\cO}_{\mathfrak{X}_{\overline{r}},\rho_L}$.\;If $A$ is a noetherian complete local ring which is a quotient of $\widehat{\cO}_{\mathfrak{X}_{\overline{r}},\rho_L}$,\;we set
\[[\Spec A]:=\sum_{\fp:\text{minimal\;prime\;of\;A}}m(\fp,A)[\Spec A/\fp]\in Z(\Spec \widehat{\cO}_{\mathfrak{X}_{\overline{r}},\rho_L})\]
the associated cycle in $Z(\Spec \widehat{\cO}_{\mathfrak{X}_{\overline{r}},\rho_L})$.\;

 In the sequel,\;we will use the following abusive notation for local formal schemes.\;Let $\Spf R$ be a local formal scheme.\;Let $\cP$ be a property of scheme,\;we say that $\Spf R$ satisfies $\cP$ if $\Spec R$ satisfies $\cP$.\;Moreover,\;for a given irreducible component $\Spec R/\fa\subseteq \Spec R$,\;we will refer to the formal subscheme $\Spf R/\fa\subseteq \Spf R$ as an irreducible component.\;Similarity,\;we also put $Z(\Spf R):=Z(\Spec R)$ (resp.,\;$Z^d(\Spf R):=Z^d(\Spec R)$,\;resp.,\;$Z^{\leq d}(\Spf R)=Z^{\leq d}(\Spec R)$).\;

Recall that we have fixed a $p$-adic potentially semistable non-crystalline Galois representation $\rho_L:\gal_L\rightarrow \GLN_{n}(E)$ which admits a special $\omepik$-filtration with parameter $(\bx_0,\bmdel)\in \sbanpik\times\rigchl $ (resp.,\;with parameter $(\widetilde{\bx}_{\pi,\bh},\widetilde{\bm{\delta}}_\bh)\in \sbanpik\times\rigch $).\;By inverting $t$,\;$\rho_L$ admits an $\omepik$-filtration $\cM_{\bullet}$ of $\cM$ with parameter $(\bx_0,\bmdel)\in \sbanpik\times\rigchl$.\;We refer to Section \ref{localmodelforgalodef} for the discussion on local models.\;Recall that $y$ (resp.,\;$\widehat{y}$) is the closed point of the $E$-scheme $X_{r,L}$ (resp.,\;$\cX_{r,L}$) corresponding to the triple $(\alpha^{-1}(\cD_{\bullet}),\alpha^{-1}(\fil_{\bW^+_\Dpik,\bullet}),N_{\bW_\Dpik})$ (resp.,\;$(0,y)$) and  $w_{\cF}\in \sW^{\Delta_n^k,\emptyset}_{n,L,\max}$ is the element that encodes the relative localization of two flags.\;

In this section,\;we discuss two partially de-Rham Galois cycles with respect to the parabolic group (see \cite[Section 3.6]{wu2021local}) and the $\Sigma_L$-components (see \cite[Section 5.2]{CompanionpointforGLN2L}) respectively.\;

Firstly,\;we have a commutative diagram of affine formal schemes (and the same for its $\Spec$-version) over $E$:
\begin{equation}
	\xymatrix{
		X^{w}_{\rho_L,\cM_{\bullet}} \ar@{^(->}[d]
		&  X^{\Box,w}_{\rho_L,\cM_{\bullet}} \ar[l]\ar@{^(->}[d] \ar[r] &  \widehat{\cX}^{\flat}_{r,w,\widehat{y}}\ar@{^(->}[d] \\
		X_{\rho_L,\cM_{\bullet}} \ar@{=}[d] & X^{\Box}_{\rho_L,\cM_{\bullet}} \ar[l] \ar[r] \ar@{=}[d] & \widehat{\cX}^{\flat}_{r,L,\widehat{y}}\ar[d]^{\iota^{\flat}} \\
		X_{\rho_L,\cM_{\bullet}} \ar@{^(->}[d] \ar[dr]_{\omega_{\underline{\delta}}}& X^{\Box}_{\rho_L,\cM_{\bullet}} \ar[l] \ar[r]  & \widehat{\cX}_{r,L,\widehat{y}}\ar[d]^{\kappa_1}\\
		X_{\rho_L} & \widehat{(\cZ_{\bL_r,L})}_{\bm{\delta}_{\bh}} \ar[r]^{\;\;\;\;\;\wt-\wt(\underline{\delta})}   & \widehat{\ft_L}.}
\end{equation}
For any $J\subseteq \Sigma_L$,\;let
\begin{equation}
	\begin{aligned}
&\overline{X}^{J}_{r,L}:=\prod_{\tau\in J}\overline{X}_{r,\tau}\times \prod_{\tau\in \Sigma_L\backslash J}{X}_{r,\tau},
\overline{X}^{J}_{r,w}:=\prod_{\tau\in J}\overline{X}_{r,w_\tau}\times \prod_{\tau\in \Sigma_L\backslash J}{X}_{r,w_\tau},\\
&{Z}^{J}_{r,L}:=\prod_{\tau\in J}{Z}_{r,\tau}\times \prod_{\tau\in \Sigma_L\backslash J}{X}_{r,\tau},
{Z}^{J}_{r,w}:=\prod_{\tau\in J}{Z}_{r,w_\tau}\times \prod_{\tau\in \Sigma_L\backslash J}{X}_{r,w_\tau}.\;				
	\end{aligned}
\end{equation}
Put $\widehat{(\overline{\cX}^{J}_{r,w})}_{\widehat{y}}=\widehat{\cX}_{r,w,\widehat{y}}\times_{\widehat{{X}}_{r,w,{y}}}\widehat{(\overline{X}^{J}_{r,w})}_{y}$ and $\widehat{(\overline{\cX}^{J}_{r,L})}_{\widehat{y}}=\widehat{\cX}_{r,L,\widehat{y}}\times_{\widehat{{X}}_{r,L,{y}}}\widehat{(\overline{X}^{J}_{r,w})}_{y}$.\;Taking everywhere (expect for $X_{\rho_L}$) the fiber over $0\in \ft_J(E)$ of the above diagram ($\Spec$-version),\;we obtain the following commutative diagram:
\begin{equation}\label{factorwww2}
	\xymatrix{
		\Spec \overline{R}^{J,w}_{\rho_L,\cM_{\bullet}} \ar@{^(->}[d]
		&  \Spec\overline{R}^{J,\Box,w}_{\rho_L,\cM_{\bullet}} \ar[l]\ar@{^(->}[d] \ar[r] &  \Spec \widehat{\cO}^{\flat}_{\overline{{\cX}}^{J}_{r,w},\widehat{y}}\ar@{^(->}[d] \\
		\Spec\overline{R}^{J}_{\rho_L,\cM_{\bullet}} \ar@{=}[d] & \Spec\overline{R}^{J,\Box}_{\rho_L,\cM_{\bullet}} \ar[l] \ar[r] \ar@{=}[d] & \Spec\widehat{\cO}^{\flat}_{\overline{{\cX}}^{J}_{r,L},\widehat{y}}\ar[d]^{\overline{\iota}^{J,\flat}} \\
		\Spec\overline{R}^{J}_{\rho_L,\cM_{\bullet}} \ar@{^(->}[d] & \Spec \overline{R}^{J,\Box}_{\rho_L,\cM_{\bullet}} \ar[l] \ar[r]^{\overline{\iota}^J}  & \Spec \widehat{\cO}_{\overline{\cX}^{J}_{r,L},\widehat{y}}\\
		\Spec \widehat{\cO}_{\mathfrak{X}_{\overline{r}},\rho_L}  &    & }
\end{equation}
where $\widehat{\cO}^{\flat}_{\overline{{\cX}}^{J}_{r,w},\widehat{y}}\cong \widehat{\cO}^{\flat}_{{{\cX}}_{r,w,\widehat{y}}}\otimes_{\widehat{\cO}_{{\cX}_{r,w},\widehat{y}}}\widehat{\cO}_{\overline{\cX}^{J}_{r,w},\widehat{y}}$
 and $\widehat{\cO}^{\flat}_{\overline{{\cX}}^{J}_{r,L},\widehat{y}}\cong \widehat{\cO}^{\flat}_{{{\cX}}_{r,L,\widehat{y}}}\otimes_{\widehat{\cO}_{{\cX}_{r,L},\widehat{y}}}\widehat{\cO}_{\overline{\cX}^{J}_{r,L},\widehat{y}}$.\;It is clear that all the horizontal morphisms in (\ref{factorwww2}) are formally smooth (except $\overline{\iota}^J$) and where four squares (except the square at the bottom right) are cartesian (as the vertical maps are closed immersions,\;except $\overline{\iota}^{J,\flat}$).\;By the argument before \cite[Lemma 5.4.4]{Ding2021},\;the irreducible components of $\Spec \widehat{\cO}_{\overline{X}_{r,w},\widehat{y}}$ are the union of the irreducible components of  $\Spec \widehat{\cO}_{Z_{r,w'},y}$ for $w'\in \sW_{n,L}$ such that $y\in Z_{r,w'}(E)$.\;
 
Let $\Spec \widehat{\cO}^{\flat}_{\cZ^J_{r,w'},\widehat{y}}$ be the pullback of $\Spec \widehat{\cO}_{Z^J_{r,w'},y}$ via the morphism $\Spec\widehat{\cO}^{\flat}_{\overline{{\cX}}^{J}_{r,L},\widehat{y}}\rightarrow \Spec \widehat{\cO}_{\overline{\cX}_{r,L},\widehat{y}}\rightarrow \widehat{\cO}_{\overline{X}_{r,L},y}$.\;We denote by $$\FZ^{J,\flat}_{r,w}\in Z(\Spec \widehat{\cO}_{\mathfrak{X}_{\overline{r}},\rho_L})$$ the cycle corresponding via the embedding $\Spec\overline{R}^{J,\flat}_{\rho_L,\cM_{\bullet}}\rightarrow \Spec \widehat{\cO}_{\mathfrak{X}_{\overline{r}},\rho_L}$ to the cycle $[\Spec \widehat{\cO}^{J,\flat}_{\cZ_{r,{w}},\widehat{y}}]$.\;For $w=(w_\tau)$ and $w'=(w'_{\tau})$,\;we put $a^J_{w,w'}:=\prod_{\tau\in J}a_{w_\tau,w'_\tau}$ and $b^J_{w,w'}:=\prod_{\tau\in J}b_{w_\tau,w'_\tau}$,\;where $a_{w_\tau,w'_\tau}$ and $b_{w_\tau,w'_\tau}$ are given in \cite[Theorem 5.4.11]{Ding2021}.\;We put
\begin{equation}\label{bzcycle}
	\FC^{J,\flat}_{r,w'}:=\sum_{w''\in \sW_{\Delta_n^k,L}\backslash \sW_{n,L}}a^J_{w',w''}\FZ^{J,\flat}_{r,w'}\in Z^{\frac{n(n+1)}{2}d_L}(\Spec \widehat{\cO}_{\mathfrak{X}_{\overline{r}},\rho_L}).\;
\end{equation}
Then the following statements are equivalent: $(1):\FC^{J,\flat}_{r,w'}\neq 0$,\;$(2):\FZ^{J,\flat}_{r,w'}\neq 0$ and $(3):w^{\max}\geq w_{\cF}$ (similar to the proof of \cite[Lemma 6.5.3]{Ding2021}).\;When $J=\Sigma_L$,\;we drop the superscript $J$ in above notation.\;

\begin{pro}\label{descriptionzw0}We have the following facts.\;
	\begin{itemize}
		\item[(a)] If $J\subsetneq \Sigma_L$,\;then $\Spec \widehat{\cO}^{\flat}_{\cZ^J_{r,w'},\widehat{y}}$ is irreducible.\;Moreover,\;the irreducible components of $\Spec\widehat{\cO}^{\flat}_{\overline{\cX}^{J}_{r,L},\widehat{y}}$ are  $\Spec\widehat{\cO}^{\flat}_{\cZ^J_{r,w'},\widehat{y}}$ such that $y\in Z^J_{r,w'}(E)$.\;
		\item[(b)] Suppose $r=1$ and  $J\subsetneq \Sigma_L$.\;Assume Conjecture \ref{flatnessXconj} holds,\;then $\Spec\widehat{\cO}^{\flat}_{\overline{{\cX}}_{r,L},\widehat{y}}$ and  $\Spec \widehat{\cO}_{\cZ^{\flat}_{w'},\widehat{y}}$ are  equidimensional.\;The irreducible components of $\Spec\widehat{\cO}^{\flat}_{\overline{{\cX}}_{r,L},\widehat{y}}$ are unions of the irreducible components of $\Spec \widehat{\cO}^{\flat}_{\cZ_{r,w'},\widehat{y}}$ ($\Spec \widehat{\cO}_{\cZ^{\flat}_{r,w'},\widehat{y}}$ is reducible in general) such that $y\in Z_{r,w'}(E)$.\;
		\item[(c)] $\Spec \widehat{\cO}_{\cZ^{\flat}_{\underline{w}_0},\widehat{y}}$ is equidimensional.\;We have $\FC^{\flat}_{r,\underline{w}_0}=\FZ^{\flat}_{r,\underline{w}_0}=r_{w_0}[\Spec\widehat{\cO}_{{\FX_{\overline{r}}^{\Box,\tau,\bh}},\rho_L}]$,\;where $r_{w_0}$ is the number of the irreducible components of $\FZ^{\flat}_{r,\underline{w}_0}$.\;
	\end{itemize}
\end{pro}
\begin{proof}
We firstly prove Part $(a)$.\;Since $X^{\Box}_{\rho_L,\cM_{\bullet}}$ (resp.,\;$X^{\Box,w}_{\rho_L,\cM_{\bullet}}$) is formally smooth over $X^{\Box}_{\bW_\Dpik,\bF_{\bullet},J}$ (resp., $X^{\Box,w}_{\bW_\Dpik,\bF_{\bullet},J}$),\;i.e.,\;the ring homomorphism $\widehat{\cO}_{{X}_{r,J},y_J}\rightarrow {R}^{\Box}_{\rho_L,\cM_{\bullet}}$ is formally smooth.\;Note that 
\[\overline{R}^{J,\Box}_{\rho_L,\cM_{\bullet}}={R}^{\Box}_{\rho_L,\cM_{\bullet}}\otimes_{\widehat{\cO}_{{\cX}_{r,L},\widehat{y}}}\widehat{\cO}_{\overline{\cX}^{J}_{r,L},\widehat{y}}=\overline{R}^{\Box}_{\rho_L,\cM_{\bullet}}\otimes_{\widehat{\cO}_{{X}_{r,J},\widehat{y}_J}}\widehat{\cO}_{\overline{X}_{r,J},{y}_J},\]
we see that $\widehat{\cO}_{\overline{X}_{r,J},y_J}\rightarrow \overline{R}^{J,\Box}_{\rho_L,\cM_{\bullet}}$ (resp,\;$\widehat{\cO}_{\overline{X}_{r,J,w},y_J}\rightarrow \overline{R}^{J,\Box,w}_{\rho_L,\cM_{\bullet}}$,\;recall $X_{r,J}=\prod_{\tau\in J}X_{r}$ and $X_{r,J,w}=\prod_{\tau\in J}X_{r,w_\tau}$ in Section \ref{preonschemegeorepn}) is formall smooth by base change.\;By the geometric properties of $\widehat{\cO}_{\overline{X}_{r,J},y_J}$ and $\widehat{\cO}_{\overline{X}_{r,J,w_J},y_J}$,\;we get the results in Part $(a)$.\;Part $(b)$ is a direct consequence of flatness.\;Finally,\;since the assumptions in Part $(b)$ is true for $w=\underline{w}_0$ (by proposition \ref{flatnessX}).\;We say more about the cycles $\FC^{\flat}_{r,\underline{w}_0}=\FZ^{\flat}_{r,\underline{w}_0}\neq 0$.\;Note that the underlying nilpotent operator is identically $0$ on $Z_{r,\underline{w}_0}$,\;any deformation in $X_{\rho_L,\cM_{\bullet}}(A)$ coming from $\widehat{\cO}_{\cZ^{\flat}_{\underline{w}_0},\widehat{y}}(A)$ is de Rham.\;Since $\rho_L$ is semistable non-crystalline,\;we deduce from \cite[Proposition 2.3.4]{AST2009324R10} that any deformation in $\widehat{\cO}_{\cZ^{\flat}_{\underline{w}_0},\widehat{y}}(A)$ is further semistable non-crystalline.\;Let $(\widetilde{\FX}_{\overline{r}}^{\Box,\tau,\bh})^{\wedge}_{\rho_L}$  be the deformations space consists of the semistable deformations $\rho_A$ of $\rho_L$ together with a deformation of the flag $D_{\mathrm{st}}(\cF)$ on $D_{\mathrm{st}}(\rho_L)$ which is stable under the $\varphi_A$ and $N_A$ (on $D_{\mathrm{st}}(\rho_A)$).\;When forgetting the flag in $(\widetilde{\FX}_{\overline{r}}^{\Box,\tau,\bh})^{\wedge}_{\rho_L}$,\;we obtain $({\FX}_{\overline{r}}^{\Box,\tau,\bh})^{\wedge}_{\rho_L}$.\;Note that $\FX_{\overline{r}}^{\Box,\tau,\bh}$ is equi-dimensional of dimension $n^2+d_L\frac{n(n-1)}{2}$.\;By \cite[Theorem 1.2.7]{ALLENPARK},\;$\rho_L$ is a smooth point of $\FX_{\overline{r}}^{\Box,\tau,\bh}$ so that there is a unique irreducible component $Z(\rho_L)$ containing $\rho_L$.\;The inclusion $Z(\rho_L)\hookrightarrow \FX_{\overline{r}}^{\Box,\tau,\bh}$ induces  isomorphisms of irreducible completed local rings
$ \widehat{\cO}_{{\FX_{\overline{r}}^{\Box,\tau,\bh}},\rho_L}\xleftarrow{\simeq}\widehat{\cO}_{Z(\rho_L),\rho_L}.$
Let $(\widetilde{\FX}_{\overline{r}}^{\Box,\tau,\bh})^{\wedge}_{\rho_L}=\Spf \widehat{\cO}_{{\widetilde{\FX}_{\overline{r}}^{\Box,\tau,\bh}},\rho_L}$.\;Then we deduce that
\[\FC^{\flat}_{r,\underline{w}_0}=\FZ^{\flat}_{r,\underline{w}_0}=[\Spec \widehat{\cO}_{{\widetilde{\FX}_{\overline{r}}^{\Box,\tau,\bh}},\rho_L}]=r_{w_0}[\Spec\widehat{\cO}_{{\FX_{\overline{r}}^{\Box,\tau,\bh}},\rho_L}], \]
where $r_{w_0}$ is the number of the irreducible components of $\FZ^{\flat}_{r,\underline{w}_0}$.\;
\end{proof}

Assume that $x:=(\rho_L,\widetilde{\bx}_{\pi,\bh},\widetilde{\bm{\delta}}_\bh)\in \defvar$.\;Let $\cM_{\bullet,x}$ be the unique $\omepik$-filtration on $\cM$ of parameter $(\bx_0,\bmdel)$.\;Recall we have defined two elements $w_{\cF}\in \sW^{\Delta_n^k,\emptyset}_{n,L,\max}$ (see the argument before Proposition \ref{propertyofxrhombullet}) and $w_x\in \sW^{\Delta_n^k,\emptyset}_{n,L}$ (see above Proposition \ref{localgeomertyonspecial}).\;Let $X^\Box_{\omepik,\mathbf{{h}}}(\overline{r})_{\bh}$ be the fiber of $\defvar$ at $\bh$ via the morphism $X^\Box_{\omepik,\mathbf{{h}}}(\overline{r})\rightarrow \rigch\xrightarrow{\wt}\fz_{r,\emptyset,L}$.\;By an easy variation of \cite[Conjecture 5.4.9]{Ding2021},\;we have
\begin{conjecture}\label{localmodelcycle}We have
	\begin{equation}\label{r1Breuil}
		\begin{aligned}
			[\Spec \widehat{\cO}_{X^\Box_{\omepik,\mathbf{{h}}}(\overline{r})_{\bh},x} ]=&\;\sum_{\substack{w\in\sW_{\Delta_n^k,L}\backslash \sW_{n,L}\\w_{\cF}\leq w^{\max}\leq w_x\underline{w}_0}}b_{w_x\underline{w}_0,w}\FC^{\flat}_{r,w}\in Z^{\frac{n(n+1)}{2}d_L}(\Spec \widehat{\cO}_{\mathfrak{X}_{\overline{r}},\rho_L}).
		\end{aligned}
	\end{equation} 
		where $\FC^{\flat}_{w'}$ is defined using the triangulation $\cM_{\bullet,\underline{\delta}}$ of the above discussion.\;The cycle $\FC^{\flat}_{w'}$ a priori depends on $\rho_L,\cM_{\bullet,\underline{\delta}}$ and $w'$.\;
\end{conjecture}
\begin{rmk}This conjecture is true for the case $r=1$ if Conjecture \ref{flatnessXconj} holds.\;
\end{rmk}

From now on,\;we restrict to the case $r=1$.\;To introduce the second partial de rham deformation ring,\;we consider the schemes (as in \cite[Section 4.1]{HHS} and \cite[Section 2.5]{wu2021local})
\begin{equation}\label{dfnXIZI}
	\begin{aligned}
		&X_{\underline{I}}:=\left(\GLN_{n,L}\times^{\bB_{n,L}}(\fz_{\underline{I}}\oplus \fn_{\underline{I}})\right)\times_{\fg_{L}}\widehat{\fg}_L,\\
		&Z_{\underline{I}}:=\left(\left(\GLN_{n,L}\times^{\bB_{n,L}} \fn_{\underline{I}}\right)\times_{\cN_{L}}\widehat{\cN}_L\right)^{\mathrm{red}}.\;
	\end{aligned}
\end{equation}
The irreducible components of scheme $X_{\underline{I}}$ are index by $\{X_{I,w}\}_{w\in \sW_{n,L}^{\underline{I},\emptyset}}$.\;Put $\overline{X}_{\underline{I}}:=X_{\underline{I}}\times_{X_{L}}\overline{X}_{L}$ and 
$\overline{X}_{\underline{I},w}:=X_{\underline{I},w}\times_{X_{L}}\overline{X}_{L}$.\;By \cite[Proposition 2.17]{wu2021local},\;the scheme $Z_{\underline{I}}$ is equidimensional of dimension $n^2-n$ with irreducible components $\{Z_{\underline{w}_{\underline{I},0}w}\}_{w\in \sW^{\underline{I},\emptyset}_{n,L}}$.\;By \cite[Theorem 2.24]{wu2021local},\;we see that $w\underline{w}_0(\bh)$ is strictly $\underline{I}$-dominant if and only if $Z_{\underline{w}_{\underline{I},0}w}$ is contained in $Z_{\underline{I}}$.\;

Fix $J\subseteq \Sigma_L$ and $\underline{I}:=\prod_{\tau\in J}I_{\tau}\subset \Delta_n^{J}$.\;Put $\underline{I}_{J,L}:=\underline{I}\times \prod_{\tau\in \Sigma\backslash J}\emptyset$  We also consider the scheme
\begin{equation}
	\begin{aligned}
		&X_{J,\underline{I}}:=\left(\GLN_{n,L}\times^{\bB_{n,L}}((\fz_{\underline{I}}\oplus \fn_{\underline{I}})\times\fb_{\Sigma\backslash J})\right)\times_{\fg_{L}}\widehat{\fg}_L=X_{\underline{I}_{J,L}},\\
		&Z_{J,\underline{I}}:=\left(\left(\GLN_{n,L}\times^{\bB_{n,L}}( \fn_{\underline{I}}\times\fb_{\Sigma\backslash J})\right)\times_{\cN_{L}}\widehat{\cN}_L\right)^{\mathrm{red}}.\;
	\end{aligned}
\end{equation}
Similarity,\;the irreducible components of scheme $X_{J,\underline{I}}$ are index by $\{X_{J,\underline{I},w}\}_{w\in \sW_{n,L}^{\underline{I}_{J,L},\emptyset}}$.\;Put $\overline{X}_{J,\underline{I}}:=X_{J,\underline{I}}\times_{X_{L}}\overline{X}_{L}$ and 
$\overline{X}_{J,\underline{I},w}:=X_{J,\underline{I},w}\times_{X_{L}}\overline{X}_{L}$.\;By \cite[Proposition 2.17]{wu2021local},\;the scheme $Z_{J,\underline{I}}$ is equidimensional of dimension $n^2-n$ with irreducible components $\{Z_{w_{\underline{I}_{J,L},0}w}\}_{w\in \sW^{\underline{I}_{J,L},\emptyset}_{n,L}}$.\;

For $\tau\in \Sigma_L$,\;write $\bP_{I_\tau}=\GLN_{q_{1,\tau}}\times \cdots\times \GLN_{q_{t_\tau},\tau}$ for some partition $q_{1,\tau}+\cdots+q_{t_\tau,\tau}=n$ (we put $q_{0,\tau}=0$).\;We say $(A,\rho_A,\cM_{A,\bullet},j_A,\alpha)\in X^{\Box}_{\rho_L,\cM_{\bullet}}$ is quasi-$\bP_{{I}_\tau}$-de Rham (resp.,\;$\bP_{{I}_\tau}$-de Rham) (for some $\tau\in \Sigma_L$) if the nilpotent operator $\nu_A$ on $D_{\pdr,\tau}(W_{\dr}(D_{\rig}(\rho_A)[1/t]))$ is scale (resp.,\;vanishes)  when restricted to the graded pieces
\begin{equation}
	\begin{aligned}
	%		&D_{\pdr,\tau}(W_{\dr}(\cM_{A,q_{1,\tau}+\cdots+q_{i,\tau}}))/D_{\pdr,\tau}(W_{\dr}(\cM_{A,q_{1,\tau}+\cdots+q_{i-1,\tau}+1}))\\\text{resp.,\;}
		D_{\pdr,\tau}(W_{\dr}(\cM_{A,q_{1,\tau}+\cdots+q_{i,\tau}}))/D_{\pdr,\tau}(W_{\dr}(\cM_{A,q_{1,\tau}+\cdots+q_{i-1,\tau}})),i=1,\cdots,t_\tau
	\end{aligned}
	\end{equation}
for $\tau\in \Sigma_L$.\;

Similar to \cite[Section 4.2,\;Theorem 4.7]{HHS},\;let $X^{J,\underline{I},\Box}_{\rho_L,\cM_{\bullet}}$ (resp.,\;$X^{J,\underline{I}}_{\rho_L,\cM_{\bullet}}$) be the full subgroupoid of $X^{\Box}_{\rho_L,\cM_{\bullet}}$ (resp.,\;$X_{\rho_L,\cM_{\bullet}}$)  consists of objects $(A,\rho_A,\cM_{A,\bullet},j_A,\alpha_A)\in X^{\Box}_{\rho_L,\cM_{\bullet}}$ such that the triangulation $\cM_{A,\bullet}$ on $D_{\rig}(\rho_A)[1/t]$ is quasi-$\bP_{\underline{I}}$-de Rham,\;i.e.,\;is $\bP_{{I}_\tau}$-de Rham for all $\tau\in J$.\;

Similar to \cite[Lemma 3.11]{wu2021local},\;we also let $\overline{X}^{J,\underline{I},\Box}_{\rho_L,\cM_{\bullet}}$ (resp.,\;$\overline{X}^{J,\underline{I}}_{\rho_L,\cM_{\bullet}}$) be the full subgroupoid of $X^{\Box}_{\rho_L,\cM_{\bullet}}$ (resp.,\;$X_{\rho_L,\cM_{\bullet}}$)  consists of objects $(A,\rho_A,\cM_{A,\bullet},j_A,\alpha_A)\in X^{\Box}_{\rho_L,\cM_{\bullet}}$ such that the triangulation $\cM_{A,\bullet}$ on $D_{\rig}(\rho_A)[1/t]$ is $\bP_{\underline{I}}$-de Rham,\;i.e.,\;is $\bP_{{I}_\tau}$-de Rham for all $\tau\in J$.\;

%By applying the construction in \cite[Section 2.5]{wu2021local} to $\bP_{\underline{I}}$,\;we have a closed subscheme  $Z_{\underline{I}}$ of $Z_J$.\;By \cite[Proposition 2.17]{wu2021local},\;the irreducible components of $Z_{\underline{I}}$ are given by $\{Z_{\underline{w}_{I,0}w}\}_{w\in \sW^{I,\emptyset}_{n,J}} $.\;For any formal schemes $Y$ over $\widehat{\cO}_{{X}_{L},y}$ (resp.,\;$Y_1$ over $\widehat{\cO}_{\overline{X}_{L},y}$),\;we define 
%\begin{equation}
%	Y^{\underline{I}}:=Y\otimes_{\widehat{\cO}_{{X}_{L},y}}\widehat{\cO}_{X_{J,\underline{I}},y},\;\text{resp.,\;}\overline{Y_1}^{\underline{I}}:=Y_1\otimes_{\widehat{\cO}_{\overline{X}_{L},y}}\widehat{\cO}_{\overline{X}_{J,\underline{I}},y}.\;
%\end{equation}
By \cite[Theorem 4.7]{HHS} and \cite[Lemma 3.11]{wu2021local},\;the functor $|{X}^{J,\underline{I},\Box}_{\rho_L,\cM_{\bullet}}|$ (resp.,\;$|\overline{X}^{J,\underline{I},\Box}_{\rho_L,\cM_{\bullet}}|$) is pro-represented by the formal scheme $\Spf {R}^{J,\underline{I},\Box}_{\rho_L,\cM_{\bullet}}$ (resp.,\;$\Spf \overline{R}^{J,\underline{I},\Box}_{\rho_L,\cM_{\bullet}}$),\;and $|X^{J,\underline{I}}_{\rho_L,\cM_{\bullet}}|$ (resp.,\;$|\overline{X}^{J,\underline{I}}_{\rho_L,\cM_{\bullet}}|$) is also pro-represented by some formal scheme $\Spf {R}^{J,\underline{I}}_{\rho_L,\cM_{\bullet}}$ (resp.,\;$\Spf \overline{R}^{J,\underline{I}}_{\rho_L,\cM_{\bullet}}$)  with a formally smooth morphism formally smooth morphism $\Spf {R}^{J,\underline{I},\Box}_{\rho_L,\cM_{\bullet}}\rightarrow\Spf {R}^{J,\underline{I}}_{\rho_L,\cM_{\bullet}}$ (resp.,\;$\Spf \overline{R}^{J,\underline{I},\Box}_{\rho_L,\cM_{\bullet}}\rightarrow\Spf \overline{R}^{J,\underline{I}}_{\rho_L,\cM_{\bullet}}$).\;Note that the natural morohism $\Spec {R}^{J}_{\rho_L,\cM_{\bullet}}\hookrightarrow {R}^{J,\underline{I}}_{\rho_L,\cM_{\bullet}}$ (resp,\;$\Spec \overline{R}^{J}_{\rho_L,\cM_{\bullet}}\hookrightarrow \overline{R}^{J,\underline{I}}_{\rho_L,\cM_{\bullet}}$) is a closed immersion.\;

By definition,\;we also see that 
\begin{lem}\label{strictlyiff}Assume that  $L\neq \bQ_p$ and $J\subsetneq \Sigma_L$.\;Then $\FZ^{J,\flat}_{r,w\underline{w}_0}$ is contained in $\Spec R^{J,\underline{I}}_{\rho_L,\cM_{\bullet}}$ if and only if $w\underline{w}_0(\bh)$ is strictly $\bP_{\underline{I}}$-dominant (or ${\underline{I}}$-dominant).
\end{lem}
\begin{proof}Note that $R^{J,\underline{I},\Box}_{\rho_L,\cM_{\bullet}}$ is fomally smooth over $\widehat{\cO}_{Z_{{\underline{I}}},y_J}$ by base change (since $\widehat{\cO}_{\overline{X}_{J},y_J}\rightarrow \overline{R}^{J,\Box}_{\rho_L,\cM_{\bullet}}$ is formall smooth).\;Then this lemma follows from \cite[Theorem 2.24]{wu2021local}.\;
\end{proof}

Recall the point $x=(\rho_L,\chi)\in X^{\Box}_{\mathrm{tri}}(\overline{r})$.\;Let $\widehat{T}_{L,\wt(\chi)}$ be the fiber of $\widehat{T}_L$ at $\wt(\chi)$ via the weight morphism $\widehat{T}_L\xrightarrow{\wt}\ft_{\Sigma_L}$.\;Taking the fibers over ${\chi}\in \Spec \widehat{\cO}_{\widehat{T}_{L,\wt(\chi)},\chi}(E)$ in the commutative diagram (\ref{factorwww2}) yields a third diagram:
\begin{equation}\label{factorwwwdelta}
	\xymatrix{
		\Spec \overline{\overline{R}}^{w}_{\rho_L,\cM_{\bullet}} \ar@{^(->}[d]
		&  \Spec\overline{\overline{R}}^{\Box,w}_{\rho_L,\cM_{\bullet}} \ar[l]\ar@{^(->}[d] \ar[r] &  \Spec \widehat{\cO}^{\flat}_{\overline{{X}}_{w},{y}}\ar@{^(->}[d] \\
		\Spec\overline{\overline{R}}_{\rho_L,\cM_{\bullet}} \ar@{=}[d] & \Spec\overline{\overline{R}}^{\Box}_{\rho_L,\cM_{\bullet}} \ar[l] \ar[r] \ar@{=}[d] & \Spec\widehat{\cO}^{\flat}_{\overline{{X}}_{L},{y}}\ar[d]^{\overline{\overline{{\iota}}}^{\flat}} \\
		\Spec\overline{\overline{R}}_{\rho_L,\cM_{\bullet}} \ar@{^(->}[d] & \Spec \overline{\overline{R}}^{\Box}_{\rho_L,\cM_{\bullet}} \ar[l] \ar[r]^{\overline{\overline{\iota}}}  & \Spec \widehat{\cO}_{\overline{X}_{L},{y}}\\
		\Spec \widehat{\cO}_{\mathfrak{X}_{\overline{r}},\rho_L}  &    & }
\end{equation}
where $\Spec \widehat{\cO}^{\flat}_{\overline{{X}}_{w},{y}}$ (resp.,\;$\Spec\widehat{\cO}^{\flat}_{\overline{{X}}_{L},{y}}$) is the closed subscheme  $V(\widehat{\ft})$ of $\Spec \widehat{\cO}^{\flat}_{\overline{{\cX}}_{w},\widehat{y}}$ (resp.,\;$\Spec\widehat{\cO}^{\flat}_{\overline{{\cX}}_{L},\widehat{y}})$. It is clear that all the horizontal morphisms are formally smooth (except $\overline{\overline{\iota}}$) and where four squares are cartesian.\;Note that $\overline{\overline{R}}^{\Box,w}_{\rho_L,\cM_{\bullet}}$ (resp.,\;$\overline{\overline{R}}^{\Box}_{\rho_L,\cM_{\bullet}}$) is a formal power series ring over $\overline{\overline{R}}^{w}_{\rho_L,\cM_{\bullet}}$ (resp.,\;$\overline{\overline{R}}_{\rho_L,\cM_{\bullet}}$).\;

Consider
\begin{equation}
	\begin{aligned}
		&\widetilde{\fg}^0_{L}:=\prod_{\tau\in \Sigma_L}\widetilde{\fg}^0_{\tau},\;\widetilde{\fg}^0_{\tau}=\widetilde{\fg}^0:=\{(g\bB,\psi)\in \GLN_n/\bB\times\fg\;|\;\mathrm{Ad}(g^{-1})\psi\in {\fb^0}\}\\
		&\widetilde{\cN}^0_{L}:=\prod_{\tau\in \Sigma_L}\widetilde{\cN}^0_{\tau},\;\widetilde{\cN}^0_{\tau}=\widetilde{\cN}:=\{(g\bB,\psi)\in \GLN_n/\bB\times\fg\;|\;\mathrm{Ad}(g^{-1})\psi\in {\fn^0_{\emptyset}}\}
	\end{aligned}
\end{equation}
where $\fb^0\subseteq \fb$ (resp.,\;$\fu^0_{\emptyset}\subseteq \fu$) is the closed subspace consists of elements $(b_{ij})_{i<j}\in \fb$ (resp.,\;$(u_{ij})_{i<j}\in \fu$) with $b_{12}=b_{23}=\cdots=b_{n-1,n}=0$ (resp.,\;$u_{12}=u_{23}=\cdots=u_{n-1,n}=0$).\;Put \begin{equation}
	X^0_{L}:=\prod_{\tau\in \Sigma_L}X^0_{\tau},\;X^0_{\tau}=X^0:=q^{-1}(\fb)^0\times_{\fg}\widetilde{\fg}_.\;
\end{equation}
and $X^0_{w}:=X_{w}\times_{X_{L}}X^0_{L}$ for $w=(w_\tau)_{\tau\in \Sigma_L}\in \sW_{n,L}$.\;We also put $\overline{X}^0_{L}:=\overline{X}_{L}\times_{X_{L}}X^0_{L}$ and  $\overline{X}^0_{w}:=\overline{X}_{w}\times_{X_{L}}X^0_{L}$,\;$Z^0_{L}:=Z_{L}\times_{X_{L}}X^0_{L}$ and  $Z^0_{w}:=Z_{w}\times_{X_{L}}X^0_{L}$.\;In particular,\;we have
 \begin{equation}
	Z^0_{L}:=\prod_{\tau\in \Sigma_L}Z_{r,\tau},\;Z_{r,\tau}=Z_{r}:=(\widetilde{\cN}^0\times_{\cN}\widetilde{\cN})^{\red}.\;
\end{equation}
and $Z^0_{w}:=Z_{w}\times_{Z_{L}}Z^0_{L}$ for $w=(w_\tau)_{\tau\in \Sigma_L}\in \sW_{n,L}$.\;

\begin{pro}\label{localmodelforover2} The natural morphsim 
	$\Spec\overline{\overline{R}}^{\Box}_{\rho_L,\cM_{\bullet}}\rightarrow \Spec{\widehat{\cO}}_{{{X}}_{L},{y}}$ factor through $\Spec{\widehat{\cO}}_{\overline{X}^0_{L},{y}}\hookrightarrow\Spec{\widehat{\cO}}_{{{X}}_{L},{y}}$,\;and $\Spec\overline{\overline{R}}^{\Box}_{\rho_L,\cM_{\bullet}}\rightarrow \Spec{\widehat{\cO}}_{\overline{X}^0_{L},{y}}$ is formally smooth.\;Moreover,\;we have an
	 isomorphism of schemes $\Spec{\widehat{\cO}}_{\overline{X}^0_{L},{y}}\cong\Spec\widehat{\cO}^{\flat}_{\overline{X}_{L},{y}}$.\;
\end{pro}
\begin{proof}We only need to point out that all the length two subquotient of the deformation $\cM_{\bullet,A}$ lie in the  kernel of the map $W_{\dr}:N_{i}\rightarrow W_{i}$.\;This follows from Lemma \ref{descriptionW2} since all the $Y_\tau$ and $X_{\sigma,\tau}$ vanish.\;For the formally smoothness,\;we also follows the route of \cite[Theorem 3.4.4]{breuil2019local}.\;Let $\overline{\overline{X}}_{\rho_L,M}$ be the corresponding groupoid that  $\Spf\overline{\overline{R}}^{\Box}_{\rho_L,\cM_{\bullet}}$ pro-represents.\;Indeed,\;$\overline{\overline{X}}_{\rho_L,M}\cong X_{\rho_L}\times_{X_{\cM_\Dpik,\cM_{\bullet}}}\overline{\overline{X}}_{\cM_\Dpik,\cM_{\bullet}}$ for some subgroupoid $\overline{\overline{X}}_{\cM_\Dpik,\cM_{\bullet}}$ of $X_{\cM_\Dpik,\cM_{\bullet}}$.\;It remains to prove all the result for $\overline{\overline{X}}_{\cM_\Dpik,\cM_{\bullet}}$.\;
Suppose $y_A:=(\cM_{\Dpik,A},\cM_{A,\bullet},j_A)$ (resp.,\;$y_B:=(\cM_{\Dpik,B},\cM_{B,\bullet},j_B)$) is an object in $\overline{\overline{X}}_{\cM_\Dpik,\cM_{\bullet}}(A)$ (resp.,\;$\overline{\overline{X}}_{\cM_\Dpik,\cM_{\bullet}}(B)$).\;Suppose that $x_A$ is isomorphic to $x_B$ when modulo $K$.\;Then we show in Proposition \ref{Jversionwdfj} the sujectivity of the following map (similar to the sujectivity of \cite[(3.23)]{breuil2019local}):
	\begin{equation}
		\begin{aligned}
			\hH^1_{(\varphi,\Gamma)}({\cM}_{A,i,j})\longrightarrow\; \hH^1(\gal_L,W_{\dr,J}({\cM}_{A,i,j}))\times_{\hH^1(\gal_L,W_{\dr,J}({\cM}_{B,i,j}))}\hH^1_{(\varphi,\Gamma)}({\cM}_{B,i,j}).
		\end{aligned}
	\end{equation}
	for $j\leq i-2$ (by the definition of $\Spec\overline{\overline{R}}^{\flat,\Box}_{\rho_L,\cM_{\bullet}}$,\;the contribution of $\hH^1_{(\varphi,\Gamma)}({\cM}_{A,i,i-1})$ and $\hH^1_{(\varphi,\Gamma)}({\cM}_{B,i,i-1})$ are zero,\;so we do not need to consider them) and thus the surjectivity of 
	\begin{equation}
		\begin{aligned}
			&\hH^1_{(\varphi,\Gamma)}({\cM}_{A,i-2}\tee\Delta_{\pi}^{\vee}\tee \cR_{E,L}({\delta}_{A,i}^{-1})[1/t])\longrightarrow\\ &\hH^1(\gal_L,W_{\dr}({\cM}_{A,i-2}\tee\Delta_{\pi}^{\vee}\tee \cR_{E,L}({\delta}_{A,i}^{-1})[1/t]))\\
			&\times_{\hH^1(\gal_L,W_{\dr}({\cM}_{B,i-2}\tee\Delta_{\pi}^{\vee}\tee \cR_{E,L}({\delta}_{B,i}^{-1})[1/t]))}\hH^1_{(\varphi,\Gamma)}({\cM}_{B,i-2}\tee\Delta_{\pi}^{\vee}\tee \cR_{E,L}({\delta}_{B,i}^{-1})[1/t]).
		\end{aligned}
	\end{equation}
	This is enough to prove the formally smoothness.\;By the construction of local models,\;the last isomorphism $\Spec{\widehat{\cO}}_{{{X}}^0_{L},{y}}\cong\Spec\widehat{\cO}^{\flat}_{\overline{{X}}_{r,L},{y}}$ is obvious.\;
\end{proof}

Similar to the argument of (\ref{factorwww2}),\;we  denote by $\cZ^{\flat}_{w'}\in Z^{\leq
	\frac{n(n+3)}{2}d_L}(\Spec \widehat{\cO}_{\mathfrak{X}_{\overline{r}},\rho_L})$ the cycle corresponding via the embedding $\Spec\overline{\overline{R}}_{\rho_L,\cM_{\bullet}}\rightarrow \Spec \widehat{\cO}_{\mathfrak{X}_{\overline{r}},\rho_L}$ to the cycle $[\Spec \widehat{\cO}^{\flat}_{Z_{w'},y}]$,\;where $\Spec \widehat{\cO}^{\flat}_{Z_{w'},y}:=\Spec \widehat{\cO}_{Z_{w'},y}\times_{\Spec \widehat{\cO}_{\overline{X}_{L,y}}}\Spec\widehat{\cO}^{\flat}_{\overline{{\cX}}_{L},\widehat{y}}\cong \Spec \widehat{\cO}_{Z^0_{w'},y}$ (note that it is not necessary equi-dimensional,\;see Remark \ref{resonfornonflat}).\;We set as in (\ref{bzcycle}):
\begin{equation}
	\cC^{\flat}_{w'}:=\sum_{w''\in \sW_{n,L}}a_{w',w''}\cZ^{\flat}_{w'}\in Z^{\leq\frac{n(n+3)}{2}d_L}(\Spec \widehat{\cO}_{\mathfrak{X}_{\overline{r}},\rho_L}).\;
\end{equation}
%Contrary to the argument in \cite[Section 4.3]{breuil2019local},\;$\cC^{\flat}_{w'}$ is not the right cycle that corresponds to $\cC_{\Pi_{(w'',1)}}$.\;
In this case,\;we have by the same arguments as for (\ref{r1Breuil}):
\begin{cor}Suppose $x=(\rho_L,\chi)\in X^{\Box}_{\mathrm{tri}}(\overline{r})$.\;The following equality holds:
	\begin{equation}\label{r1Breuilsecond}
		\begin{aligned}
			[\Spec \widehat{\cO}_{X^{\Box}_{\mathrm{tri}}(\overline{r})_{\chi},x} ]=&\;\sum_{\substack{w\in \sW_{n,L}\\w_{\cF}\leq w\leq w_x\underline{w}_0}}b_{w_x\underline{w}_0,w}\cC^{\flat}_{w}\in Z^{\leq \frac{n(n+3)}{2}d_L}(\Spec \widehat{\cO}_{\mathfrak{X}_{\overline{r}},\rho_L}).
		\end{aligned}
	\end{equation} 
\end{cor}

\subsection{Locally analytic ``Breuil-Mezard type" conjecture}

%In this case,\;the identity in Conjecture \ref{localmodelcycle} (now a theorem since $r=1$) becomes 
%\begin{equation}\label{r1Breuil}
%	\begin{aligned}
%		[\Spec \widehat{\cO}_{X^{\Box}_{\mathrm{tri}}(\overline{r})_{\wt(\chi)},x} ]=&\;\sum_{\substack{w\in \sW_{n,L}\\w_{\cF}\leq w\leq w_x\underline{w}_0}}b_{w_x\underline{w}_0,w}\FC^{\flat}_{w}\in Z^{\frac{n(n+1)}{2}d_L}.
%	\end{aligned}
%\end{equation} 

Put ${\bm\lambda}_\bh:=(\hpi_{\tau,i}+i-1)_{\tau\in \Sigma_L,1\leq i\leq n}$,\;which is a dominant weight of $\GLN_{n,L}$ with  respect to $\bB_{n,L}$.\;

We recall the Orlik-Strauch's theory \cite[Theorem]{orlik2014jordan}.\;Let $\cO_{\alge}^{\overline{\fp}_{I,L}}$ be the Bernstein-Gelfand-Gelfand (BGG) category (see \cite[Section 2]{breuil2016socle}).\;The Orlik-Strauch functor (see \cite[Theorem]{orlik2015jordan},\;see also \cite[Section 2]{breuil2016socle}) associates,\;to an object $M\in\cO_{\alge}^{\overline{\fp}_{I,L}}$ and a finite length  smooth admissible representation $\pi$ of $\bL_I(L)$,\;a locally $\bQ_p$-analytic representation $\cF^{G}_{\op_{{I}}}(M,\pi_I)$.\;

Recall that if $\underline{{\lambda}}\in X^+_{I}$,\;then $\overline{L}(-\underline{\lambda}')\in \cO^{\overline{{\fp}}_I,L}_{\alge}$.\;Let $I'$ be a subset of $\Delta_n$ containing $I$,\;then $\cO_{\alge}^{\overline{\fp}_{I',L}}$ is a full subcategory of $\cO_{\alge}^{\overline{\fp}_{I,L}}$.\;Therefore,\;for any object $M\in \cO_{\alge}^{\overline{\fp}_{I,L}}$,\;there is a maximal subset $I'\subseteq \Delta_n$ such that $M\in \cO_{\alge}^{\overline{\fp}_{I',L}}$,\;then we say that $I'$ is maximal for $M$.\;In particular,\;for any $w\in \sW_{n,L}$,\;let
$I(w)$ be the subset of $\Delta_n$ which is maximal for the $\overline{L}(-w\cdot{\lambda}_{\bh})$.\;

%This category is defined as the full subcategory of the category of $E$-linear representation $M$ of $\fg$ (or equivalently of $\fg_{\Sigma_L}$) on $E$-vector space such that
%\begin{itemize}
%	\item $M$ is finite type as $U(\fg_{\Sigma_L})$-module;
%	\item $M$ is a direct sum of simple algebraic $U({\fl}_{I,L})$-modules;
%	\item the action of $\fn_{I,L}$ is locally finite,\;i.e.,\;for any $m\in M$,\;the sub-$E$-vector space $U(\fn_{I,L})m$ is of finite dimensional over $E$.\;
%\end{itemize}

%For example,\;$I(w)$ is the subset of $\Delta_n$ such that the parabolic subgroup $\op_{I(w)}$ is maximal for $\overline{L}(-w\cdot\underline{0})$.\;

%Let $M\in \cO_{\alge}^{\overline{\fp}_{I,L}}$ be a simple object.\;Then by \parencite[P. 640]{breuil2016socle} we have
%$M\cong \bigotimes_{\sigma\in \Sigma_L}M_\sigma$ for some simple objects $M_\sigma\in \cO_{\alge}^{\overline{\fp}_{I,\sigma}}$.\;

For $w\in \sW_n$,\;let $\epsilon_w:=|\cdot|^{\frac{1-n}{2}+w(n-1)}\otimes |\cdot|^{\frac{1-n}{2}+w(n-2)}\otimes\cdots\otimes |\cdot|^{\frac{1-n}{2}+w(1)}$ be the smooth character of $\bT(L)$ over $E$ (note that $\epsilon_1$ is the character of $\bT(L)$ associated to the Zelevinsky-segment $\Delta_{[n-1,0]}(|\cdot|^{\frac{1-n}{2}})=[|\cdot|^{\frac{1-n}{2}+(n-1)},|\cdot|^{\frac{1-n}{2}+(n-2)}\cdot ,\;\cdot,\;|\cdot|^{\frac{1-n}{2}+1},\;|\cdot|^{\frac{1-n}{2}}]$).\;For $w\in \sW_{n,L}$ and $w'\in \sW_n$,\;we put
\[I_{(w,w')}:=\big(\ind^G_{\ob(L)}\chi_{w\cdot{\lambda}_{\bh}}\delta_{\ob}^{{1}/{2}}\epsilon_{w'}\big)^{\bQ_p-\ana}.\]
In particular,\;$\delta_{\ob}^{{1}/{2}}\epsilon_1$ is equal to the trivial representation of $\bT(L)$.\;

For $(w'',w')\in \sW_{n,L}\times\sW_n$,\;we put $i^{\infty}_{w'',w'}:=i^{\bL_{I(w'')}(L)}_{\ob(L)\cap \bL_{I(w'')}(L)}\delta_{\ob}^{{1}/{2}}\epsilon_{w'}$.\;It is well known that the irreducible components of $i^{\infty}_{w'',w'}$ are given by the smooth generalized Steinberg representations $\{v^{\infty}_{J,I(w'')}\}_{J\subseteq I(w'')}$ of $\bL_{I(w'')}(L)$ over $E$ (note that $\st^{\infty}_{\bL_{I(w'')}(L)}:=v^{\infty}_{\emptyset,I(w'')}$ is the smooth  Steinberg representation  of $\bL_{I(w'')}(L)$).\;By the Orlik-Strauch construction \cite[Theorem]{orlik2014jordan},\;the irreducible components of $I_{(w,w')}$ are  $\Pi_{(w'',w',J)}:=\cF^G_{\op_{I(w'')}(L)}(\overline{L}(-w''\cdot{\lambda}_{\bh}),v^{\infty}_{J,I(w'')})$ for $w''\in  \sW_{n,L}$ and $J\subseteq I(w'')$ with multiplicity $M_{w,w''}:=[\overline{M}(-w\cdot{\lambda}_{\bh}):\overline{L}(-w''\cdot{\lambda}_{\bh})]$.\;

Put $\underline{\delta}_{\mathrm{sm}}:=\delta_{\ob}^{{1}/{2}}\epsilon_1$.\;We write $K_0({\lambda}_{\bh},\underline{\delta}_{\mathrm{sm}})$ for the free abelian group generated by the irreducible constituents of the locally $\bQ_p$-analytically induced representation $I_{(w,w')}$ for $w\in \sW_{n,L}$ and $w'\in \sW_n$.\;More precisely,\;write 
$$\Pi_{(w'',w')}:=\cF^G_{\op_{I(w'')}(L)}\Big(\overline{L}(-w''\cdot{\lambda}_{\bh}),i^{\infty}_{w'',w'}\Big).$$
Then we have $[\Pi_{(w'',w')}]=\sum_{J\subseteq I(w'')}[\Pi_{(w'',w',J)}]$ in $K_0({\lambda}_{\bh},\underline{\delta}_{\mathrm{sm}})$.\;Denoted by $C(w'',w'):=\Pi_{(w'',w',\emptyset)}$  the unique quotient of $\Pi_{(w'',w')}$.\;Note that $C(w'',1)\cong \cF^G_{\op_{I(w'')}(L)}\Big(\overline{L}(-w''\cdot{\lambda}_{\bh}),\st^{\infty}_{\bL_{I(w'')}(L)}\Big)$.\;It is clear that $\{{\Pi}_{(w'',1,J)}\}_{w''\in \sW_{n,L},J\subseteq I(w'')}$ actually give a basis of $K_0({\lambda}_{\bh},\underline{\delta}_{\mathrm{sm}})$.\;
%,\;where $\overline{W}$ is the unique  $\ind^{L_{w''}(L)}_{\ob(L)}\delta_{\ob}^{{1}/{2}}\epsilon_{w'}$.\;

For $\beta\in E$,\;we denote by $I_{(w,w')}(\beta)$,\;$\Pi_{(w'',w')}(\beta)$,\;${\Pi}_{(w'',w',J)}(\beta)$ and $C(w'',w',\beta)$ the locally $\bQ_p$-analytic representation $I_{(w,w')}\otimes_E\unr(\beta)\circ\det$,\;$\Pi_{(w'',w')}\otimes_E\unr(\beta)\circ\det$,\;${\Pi}_{(w'',w',J)}\otimes_E\unr(\beta)\circ\det$ and $C(w'',w')\otimes_E\unr(\beta)\circ\det$.\;

%Let $\underline{\delta}=(\delta_{i})_{1\leq i\leq
	%	n}\in \widehat{T}_L^{\mathrm{spl}}$ be a special parameter of $\cM_{\bullet}$.\;Let $\widehat{T}_L^{\mathrm{spl}}(\alpha,\bh)$ be the subspace of $\widehat{T}_L^{\mathrm{spl}}$ consist of characters $(\unr(\alpha)_n)\cdot z^{w(\bh)}$ for $w\in\sW_{m,\Sigma_L}$,\;where  $\unr(\alpha)_n:=(\unr({\alpha }),\unr({\alpha q_L^{1}}),\;\cdots,\;\unr({\alpha q_L^{n-1}}))$.\;

%For any $\underline{\delta}=(\delta_{i})_{1\leq i\leq n}\in \widehat{T}_L$,\;we define 
%\[\bI_{\delta}:=\big(\ind_{\ob(L)}^G\delta_1\otimes\delta_2\unr(q_L^{-1})\otimes\cdots\otimes\delta_n\unr(q_L^{-(n-1)})\big)^{\bQ_p-\ana}\]
%for the locally $\bQ_p$-analytic principal series.\;We define $\bI_{\delta}^{\mathrm{ss}}$ its (topological) semi-simplification.\;If  $\Pi$ is an absolutely irreducible locally $\bQ_p$-analytic representation of $G$ over $E$,\;we denoted by $m_{\underline{\delta},\Pi}$ its multiplicity in $\bI_{\delta}^{\mathrm{ss}}$.\;
%In particular,\;if $\delta=(\unr(\alpha)_n)\cdot z^{w(\bh)}$,\;then $\bI_{\delta}=\big(\ind^G_{\ob(L)}\chi_{w\cdot{\lambda}_{\bh}}\big)^{\bQ_p-\ana}$.\;

\begin{conjecture}\label{locanaBreMezeconj}
Put $\widehat{T}_L^{\mathrm{spl}}:=\big\{(\delta_{i})_{1\leq i\leq n}\in \widehat{T}_L:\delta_{i}\delta_{i+1}^{-1}\text{\;is special} \big\}$ (i.e.,\;$\delta_{i}\delta_{i+1}^{-1}:=\unr(q_L^{-1})z^{\bk}$ for some $\bk:=(\bk_{\tau})_{\tau\in \Sigma_L}\in \BZ^{\Sigma_L}$).\;There exists a unique homomorphism
\[\fa'_{{\lambda}_{\bh},\Delta_{n}}:K_0({\lambda}_{\bh},\underline{\delta}_{\mathrm{sm}})\rightarrow Z^{\bullet}(\Spec \widehat{\cO}_{\mathfrak{X}_{\overline{r}},\rho_L}).\]
For any absolutely irreducible constituent $[\Pi]\in  K_0({\lambda}_{\bh},\underline{\delta}_{\mathrm{sm}})$,\;let $\cC_{\Pi}\in Z^{\bullet}(\Spec \widehat{\cO}_{\mathfrak{X}_{\overline{r}},\rho_L})$ be the image of $[\Pi]$ via $\fa'_{{\lambda}_{\bh},\Delta_{n}}$.\;Then this homomorphism $\fa'_{{\lambda}_{\bh},\Delta_{n}}$ is uniquely determined by conditions: 
\begin{equation}\label{identityBreuilMezard}
	[\Spec \widehat{\cO}_{X_{\mathrm{tri}}(\overline{r})_{\underline{\delta}},(\rho_L,\delta)} ]=\sum_{\Pi\in K_0({\lambda}_{\bh},\underline{\delta}_{\mathrm{sm}})}m_{\underline{\delta},\Pi}\cC_{\Pi}
\end{equation}
in $Z^{\bullet}(\Spec \widehat{\cO}_{\mathfrak{X}_{\overline{r}},\rho_L})$ for all $\underline{\delta}\in \widehat{T}_L^{\mathrm{spl}}(E)$.\;
\end{conjecture}
\begin{rmk}
We can also consider the homomorphism
$\fa''_{{\lambda}_{\bh},\Delta_{n}}:K_0({\lambda}_{\bh},\underline{\delta}_{\mathrm{sm}})\rightarrow Z^{\frac{n(n+3)}{2}d_L}(\Spec \widehat{\cO}_{\mathfrak{X}_{\overline{r}},\rho_L})$ by forgetting components of non-maximal dimension.\;
%The homomorphism $\fa'_{{\lambda}_{\bh},\Delta_{n}}$ is not injective in our case.\;
\end{rmk}

\begin{pro}If the cycles $\cC_{\Pi}$ as in Conjecture \ref{locanaBreMezeconj} exist,\;then they are unique.\;
\end{pro}
\begin{proof}We first note that $\underline{\delta}\in \widehat{T}_L^{\mathrm{spl}}(E)$ is necessary condition for the non-vanishing of $[\Spec \widehat{\cO}_{X_{\mathrm{tri}}(\overline{r})_{\underline{\delta}},(\rho_L,\delta)} ]$.\;For the uniqueness,\;it suffices to show that 
	\[\cC_{\Pi_{(w'',1)}}:=\sum_{J\subseteq I(w'')}\cC_{\Pi_{(w'',1,J)}}=\cC_{C(w'',1)}+\sum_{\substack{\Pi'\neq \Pi_{(w'',1,J)},\\ \emptyset\neq J\subseteq I(w'')}}m_{\underline{\delta}',\Pi'}\cC_{\Pi'}\]
	are unique (since $\cC_{\Pi_{(w'',1,J)}}=0$ if $J\neq\emptyset$,\;see the proof of Theorem \ref{detlamedaSOCLE}) for all $w''\in \sW_{n,L}$ and $J\subseteq I(w'')$.\;Replacing $\underline{\delta}$ by the unique locally algebraic $\underline{\delta}'=\chi_{w''\cdot{\lambda}_{\bh}}$,\;we have that
	\begin{equation}
		\begin{aligned}
			[\Spec \widehat{\cO}_{X_{\mathrm{tri}}(\overline{r})_{\underline{\delta}'},(\rho_L,\delta)} ]=&\;\cC_{\Pi_{(w'',1)}}+\sum_{\substack{\Pi'\neq \Pi_{(w'',1,J)},\\J\subseteq I(w'')}}m_{\underline{\delta}',\Pi'}\cC_{\Pi'}
			%	=&\;a_{w''}\cC_{\Pi}+\sum_{\substack{\Pi'\neq \Pi_{(w'',w')}(W),\\W\in \mathrm{JH}\big(\ind^{\bL_{w''}(L)}_{\ob(L)}\delta_{\ob}^{{1}/{2}}\epsilon_{w'}\big)}}m_{\underline{\delta}',\Pi'}\cC_{\Pi'},
		\end{aligned}
	\end{equation}
	%where $a_{w''}=|\mathrm{JH}\big(\ind^{\bL_{w''}(L)}_{\ob(L)}\delta_{\ob}^{{1}/{2}}\epsilon_{w'}\big)|$ is an integer which only depends on $w''$.\;
	If $w''\cdot{\lambda}_{\bh}$ is maximal for the $\uparrow$ (where ``$\uparrow$" means the strongly linked relation,\;see \cite[Section 5.1]{humphreysBGG} for the notion of strongly linked and the BGG theorem),\;then we must have $\cC_{\Pi_{(w'',1)}}=[\Spec \widehat{\cO}_{X_{\mathrm{tri}}(\overline{r})_{\underline{\delta}'},(\rho_L,\delta)} ]$. Otherwise, for any $\Pi'\neq \Pi_{(w'',1,J)}$,\;it has the form $\Pi_{(w''',1,J')}$ for some   $J'\subseteq I(w''')$ and $w''\cdot{\lambda}_{\bh}\uparrow w'''\cdot{\lambda}_{\bh}$ (i.e.,\;$w''\cdot{\lambda}_{\bh}$ is strongly linked to $w'''\cdot{\lambda}_{\bh}$).\;By induction,\;we can assume the cycle $\cC_{\Pi'}$ are known,\;then we must have $$\cC_{\Pi_{(w'',1)}}=\Big([\Spec \widehat{\cO}_{X_{\mathrm{tri}}(\overline{r})_{\underline{\delta}'},(\rho_L,\delta)} ]-\sum_{\Pi'\neq \Pi_{(w'',1,J)}}m_{\underline{\delta}',\Pi'}\cC_{\Pi'}\Big).$$The result follows.\;
\end{proof}

\begin{thm}
	The Conjecture \ref{locanaBreMezeconj} is true.\;
\end{thm}
\begin{proof}
Just set $\cC_{\Pi_{(w',1)}}:=\cC^{\flat}_{w'}$,\;then the result follows from the equality (\ref{r1Breuilsecond}).\;
\end{proof}

\begin{rmk}\label{stacky1}
We end this section with a remark on the ``categorical'' approach of Conjecture \ref{locanaBreMezeconj}.\;We keep the notation after \cite[Theorem 5.3.36]{emerton2023introduction}.\;For $w\in\sW_{n,L}$,\;recall that $\FX_{n,z^{w\cdot{\lambda}_{\bh}}\underline{\delta}_{\mathrm{sm}}-\mathrm{tri}}\subseteq \overline{\FX}_{n}$  is the closed substack of $\FX_n$ consisting of all $(\varphi,\Gamma)$-modules $D$ such that $D$ admits a triangulation of parameters $\chi_{w\cdot{\lambda}_{\bh}}\underline{\delta}_{\mathrm{sm}}$.\;Put
\[\FX_{n,({\lambda}_{\bh},\underline{\delta}_{\mathrm{sm}})-\mathrm{tri}}=\bigcup_{w\in\sW_{n,L}}\FX_{n,z^{w\cdot{\lambda}_{\bh}}\underline{\delta}_{\mathrm{sm}}-\mathrm{tri}}.\;\]
Then $\FX_{n,({\lambda}_{\bh},\underline{\delta}_{\mathrm{sm}})-\mathrm{tri}}$ is the closed substack of $\FX_n$ consisting of all $(\varphi,\Gamma)$-modules $D$ such that $D$ admits a triangulation of parameters $\chi_{w\cdot{\lambda}_{\bh}}\underline{\delta}_{\mathrm{sm}}$ for some $w\in\sW_{n,L}$.\;On the other hand,\;put 
\[\FX_{n,[{\lambda}_{\bh},\underline{\delta}_{\mathrm{sm}}]-\mathrm{tri}}=\bigcup_{w\in\sW_{n,L}}\FX_{n,({\lambda}_{\bh},w\cdot\underline{\delta}_{\mathrm{sm}})-\mathrm{tri}})\]
It is conjectured in \cite[Conjecture 6.2.34]{emerton2023introduction} that there exists a unique injective  group homomorphism:
\[\fa_{{\lambda}_{\bh},\underline{\delta}_{\mathrm{sm}}}:K_0({\lambda}_{\bh},\underline{\delta}_{\mathrm{sm}})\rightarrow K_0(\mathrm{Coh}(\FX_{n,[{\lambda}_{\bh},\underline{\delta}_{\mathrm{sm}}]-\mathrm{tri}})\] 
which are determined by some similar conditions,\;where $K_0(\mathrm{Coh}(\FX_{d,[{\lambda}_{\bh},\underline{\delta}_{\mathrm{sm}}]-\mathrm{tri}}))$ is the Grothendieck group of coherent sheaves on the stack $\FX_{d,({\lambda}_{\bh},\underline{\delta}_{\mathrm{sm}})-\mathrm{tri}}$.\;Now our homomorphism $\fa'_{{\lambda}_{\bh},\Delta_{n}}$ in Conjecture \ref{locanaBreMezeconj} is obtained by compositing $\fa_{{\lambda}_{\bh},\underline{\delta}_{\mathrm{sm}}}$ with the morphism mapping a coheret sheaf to its support (resp.,\;and forgetting components of non-maximal dimension),\;together with a completion at our Steinberg point.\;
\end{rmk}

\section{Global Applications}\label{conjsection}

Under the Taylor-Wiles hypothesis,\;we show several global results on $p$-adic automorphic representations including a classicality result and the existence of all expected companion constituents.\;In this section,\;we assume $L\neq \bQ_p$.\;

\subsection{Patched  eigenvariety}
\subsubsection{Patching argument and global setup}\label{preforpatching}

We follow the notation of \cite[Section 2]{PATCHING2016} and  \cite[Section 4.1.1]{2019DINGSimple} (a brief summary of \cite[Section 2]{PATCHING2016}).\;Suppose that $p\nmid 2n$, and let $\overline{r}: \Gal_L \ra \GL_n(k_E)$ be a continuous representation such that $\overline{r}$ admits a potentially crystalline lift $r_{\mathrm{pot.diag}}: \Gal_L \ra \GL_n(E)$  of regular weight $\xi$ which is potentially diagonalisable.\;We can find a triple $(F,F^+, \overline{\rho})$,\;where $F$ is an imaginary CM field with maximal totally real subfield $F^+$,\;and $\overline{\rho}:\Gal_{F^+} \ra \cG_n(k_E)$ is a \emph{suitable globalisation} (cf. \cite[Section 2.1]{PATCHING2016}) of $\overline{r}$.\;Let $S_p$ be the set of places of $F^+$ above $p$.\;For any $v|p$ of  $F^+$,\;$v$ splits in $F$, and $F^+_v\cong L$.\;

We use the setting of \cite[Section 2.1]{PATCHING2016},\;and can find the following objects
\begin{equation*}
	\{\widetilde{G}, v_1, \fp\in S_p, \{U_m\}_{m\in \BZ}\},
\end{equation*}
where $\widetilde{G}$ is a certain definite unitary group over $F^+$,\;$v_1$ is a certain finite place of $F^+$ prime to $p$,\;and $\{U_m=\prod_{v} U_{m,v}\}_{m\in \BZ_{\geq 0}}$ is a tower of certain compact open subgroups of $\widetilde{G}(\BA_{F^+}^{\infty})$ (see also \cite[Section 4.1.1]{2019DINGSimple} for a precise description).\;Write $U_{m}^{\fp}=\prod_{v\in S_p\backslash\fp}U_{m,v}$.\;

Let $\xi$ (resp.,\;$\tau$) be the inertial type (resp.,\;weight) of $r_{\mathrm{pot.diag}}$.\;By \cite[Section 2.3]{PATCHING2016},\;we can attach a finite free $\cO_E$-module $L_{\xi,\tau}$,\;which is a locally algebraic representation of $\GLN_{n}(\cO_L)$.\;Put $\BW_{\xi,\tau}:=\bigotimes_{v\in S_p\backslash\fp}L_{\xi,\tau}$, which is equipped with an action of $U_{m}^{\fp}$ by the construction.\;Put $W_{\xi,\tau}:=\BW_{\xi,\tau}\otimes_{\cO_E}E$.\;

Let $\widehat{S}_{\xi,\tau}(U_m, \cO_E/\varpi_E^k)$ be locally smooth functions  $\widetilde{G}(F^+)\backslash \widetilde{G}(\bA_{F^+}^{\infty})/U^{\fp} \ra \cO_E/\varpi_E^k$ such that $f(gg_p^\fp)=(g_p^\fp)^{-1} f(g)$ for $g\in \widetilde{G}(\bA_{F^+}^{\infty})$, $g_p^\fp\in U_m^\fp$.\;

Let $\Sigma$ be the set of primes $v$ of $F^+$ such that $v\not\in S_p\cup\{v_1\}$,\;and $v$ is totally split in $F$.\;Then the  $\co_E/\varpi_E^k$-module $\widehat{S}_{\xi,\tau}(U_m, \cO_E/\varpi_E^k)$ is equipped with a natural action of the spherical Hecke operators
\begin{equation*}
	T_w^{(j)}=\bigg[U_{v} i_w^{-1}\bigg(\begin{pmatrix} \varpi_{F_w} 1_{r,J} & 0 \\ 0 & 1_{n-1}\end{pmatrix}\bigg)U_{v}\bigg]
\end{equation*}
where $w$ is a place of $F$ lying over a place $v\in \Sigma$ of $F^+$ which splits in $F$, $\varpi_{F_w}$ is a uniformizer of $F_w$ and $j\in \{1, \cdots, n\}$. We denote by $\bT^{S_p,\univ}$ the (commutative) $\co_E$-polynomial algebra generated by such $T_w^{(j)}$ and the formal variables $T_{\widetilde{v_1}}^{(j)}$.\;By \cite[Section 2.3]{PATCHING2016},\;we can associate to $\overline{\rho}$ a maximal ideal $\fm_{\overline{\rho}}$ of $\bT^{S_p,\univ}$.\;Let ${S}_{\xi,\tau}(U_m,\cO_E/\varpi_E^k)_{{\overline{\rho}}}$ be the 
localization of ${S}_{\xi,\tau}(U_m,\cO_E/\varpi_E^k)$  at $\fm_{\overline{\rho}}$.\;We put
\begin{equation}
	\begin{aligned}
		&\widehat{S}_{\xi,\tau}(U^{\fp}, \co_E)_{\ast}:=\varprojlim_m\varprojlim_k{S}_{\xi,\tau}(U_m,\cO_E/\varpi_E^k)_\ast\\
		&\widehat{S}_{\xi,\tau}(U^{\fp},E):=\widehat{S}_{\xi,\tau}(U^{\fp}, \co_E)_{\ast}\otimes_{\cO_E}E
	\end{aligned}
\end{equation}
for $\ast\in\{{\overline{\rho}},\emptyset\}$ (roughly speaking,\;the space of $p$-adic algebraic automorphic forms of fixed type $\sigma(\tau)$ (see \cite[Theorem 3.7]{PATCHING2016},\;the ``inertial local Langlands correspondence") at the place $S_p\setminus\{\fp\}$,\;full level at $p$,\;and whose weight is $0$ at places above $\fp$,\;and given by the regular weight $\xi$ at each of the places in $S_p\setminus\{\fp\}$).\;Note that $\widehat{S}_{\xi,\tau}(U^{\fp},E)$ is a Banach space for the supermum norm and is equipped with a continuous (unitary) action of $\GL_n(L)$ (by right translation on functions).\;Therefore $\widehat{S}_{\xi,\tau}(U^{\fp},E)_{*}$ with $*\in \{{\overline{\rho}}, \emptyset\}$ are admissible unitary Banach representation of $\GL_n(L)$ with invariant lattice $\widehat{S}_{\xi,\tau}(U^{\fp}, \co_E)_*$.\;Then the action of $\bT^{S_p,\univ}$ on the localization $\widehat{S}_{\xi,\tau}(U^{\fp}, \co_E)_{{\overline{\rho}}}$ factors through certain  Hecke algebra
\[\bT_{\xi,\tau}^{S_p}(U^{\fp}, \co_E)_{{\overline{\rho}}}:=\varprojlim_m\varprojlim_k\bT_{\xi,\tau}(U^{\fp}, \co_E/\varpi_E^k)_{{\overline{\rho}}},\]
where $\bT_{\xi,\tau}(U^{\fp}, \co_E/\varpi_E^k)_{{\overline{\rho}}}$ denotes the 
$\co_E/\varpi_E^k$-subalgebra of $\EndO_{\co_E/\varpi_E^k}({S}_{\xi,\tau}(U_m,\cO_E/\varpi_E^k)_{{\overline{\rho}}})$ generated by the operators in $\bT^{S_p,\univ}$.\;

For $v\in S_p$,\;we denote by $R_{\widetilde{v}}^{\square}$ the maximal reduced and $p$-torsion free quotient of the universal $\co_E$-lifting ring of $\overline{\rho}_{\widetilde{v}}:=\overline{\rho}|_{\Gal_{F_{\widetilde{v}}}}$ ($\cong \overline{r}$,\;and therefore $R_{\widetilde{v}}^{\square}\cong \defvarring$). For $v\in S_p\backslash\{\fp\}$, we denote by $R_{\widetilde{v}}^{\square, \xi,\tau}$ for the reduced and $p$-torsion free quotient of $R_{\widetilde{v}}^{\square}$ corresponding to potentially crystalline lifts of weight $\xi$ and inertial type $\tau$.\;Consider the following global deformation problem (in the terminology \cite{clozel2008automorphy})
\begin{equation*}
	\begin{aligned}
		\cS&=\bigg\{{F}/{F}^+,T^+,T,\cO_E,\overline{\rho},\chi_{\mathrm{cyc}}^{1-n}\delta_{{F}/{F}^+}^n,\{R_{\widetilde{v}_1}^{\square}\}\cup
		\{R_{\fp}^{\square}\}\cup \{R_{\widetilde{v}}^{\square, \xi,\tau}\}_{v\in S_p\backslash \{\fp\}}\bigg\}
	\end{aligned}
\end{equation*}
They by  \cite[Proposition 2.2,9]{clozel2008automorphy},\;this deformation problem is represented by a universal deformation ring $R_{\cS}^{\univ}$.\;Note that we have a natural morphism $R_{\cS}^{\univ}\rightarrow \bT_{\xi,\tau}^{S_p}(U^{\fp}, \co_E)_{{\overline{\rho}}}$.\;

Following \cite[Section 2.8]{PATCHING2016}  (or \cite[Section 4.1.1]{2019DINGSimple}) we put
\begin{eqnarray*}
	R^{\loc}:=R_{\widetilde{\fp}}^{\square} \widehat{\otimes} \Big(\widehat{\otimes}_{S_p\backslash\{\fp\}}R_{\widetilde{v}}^{\square, \xi,\tau}\Big)\widehat{\otimes} R_{\widetilde{v_1}}^{\square} ,
\end{eqnarray*}
where all completed tensor products are taken over $\cO_E$.\;We put $g:=q-[F^+:\BQ]\frac{n(n-1)}{2}$,\;where $q$ is a certain integer as in \cite[Section 2.8]{PATCHING2016}  (or \cite[Section 4.1.1]{2019DINGSimple}).\;We now put 
\begin{eqnarray*}
	R_{\infty}&:=&R^{\loc}\llbracket x_1,\cdots, x_g\rrbracket,\\
	S_{\infty}&:=&\co_E\llbracket z_1,\cdots, z_{n^2(|S_p|+1)}, y_1,\cdots, y_q\rrbracket,
\end{eqnarray*}
where $x_{i}$,\;$y_{i}$,\;$z_{i}$ are formal variables.\;By \cite[Section 2.8]{PATCHING2016}  (or \cite[Section 4.1.1]{2019DINGSimple}),\;we get the following objects:
\begin{enumerate}
	\item[(1)] a continuous $R_{\infty}$-admissible unitary representation $\Pi_{\infty}$ of $G=\GL_n(L)$ over $E$ together with a $G$-stable and $R_{\infty}$-stable unit ball $\Pi_{\infty}^o\subset \Pi_{\infty}$;
	\item[(2)] a morphism of local $\co_E$-algebras $S_{\infty}\ra R_{\infty}$ such that $M_{\infty}:= \Hom_{\co_L}(\Pi_{\infty}^o, \co_E)$ is finite projective as $S_{\infty}\llbracket \GL_n(\co_L)\rrbracket$-module;
	\item[(3)] a closed ideal $\fa$ of $R_{\infty}$, a surjection $R_{\infty}/\fa R_{\infty}\twoheadrightarrow R_{\cS}^{\univ}$ and a  $G   \times R_{\infty}/\fa R_{\infty}$-invariant isomorphism $\Pi_{\infty}[\fa]\cong \widehat{S}_{\xi,\tau}(U^{\fp},E)_{\fm_{\overline{\rho}}}$, where $R_{\infty}$ acts on $\widehat{S}_{\xi,\tau}(U^{\fp},E)_{\fm_{\overline{\rho}}}$ via $R_{\infty}/\fa R_{\infty}\twoheadrightarrow  R_{\cS}^{\univ}$.
\end{enumerate}
\subsubsection{Patched eigenvariety and Hecke eigenvariety}\label{intropatcheigenvariety}

We briefly recall the Hecke eigenvariety and Patched eigenvariety of \cite{breuil2019local}.\;Indeed,\;our input as in previous section is slightly different from that in \cite{breuil2019local},\;but it is clear that all of the arguments in \cite{breuil2019local} apply in our case.\;

Put $\FX_{\overline{\rho},U^{\fp}}=\big(\spf\;R_{\cS}^{\univ}\big)^{\rig}$ and $\FT_{\overline{\rho},U^{\fp}}:=\big(\spf\;\bT_{\xi,\tau}^{S_p}(U^{\fp}, \co_E)_{{\overline{\rho}}}\big)^{\rig}$.\;Then the natural surjective morphism $R_{\cS}^{\univ}\twoheadrightarrow \bT_{\xi,\tau}^{S_p}(U^{\fp}, \co_E)_{{\overline{\rho}}}$ shows that $\FT_{\overline{\rho},U^{\fp}}$ is a closed subspace of $\FX_{\overline{\rho},U^{\fp}}$.\;Let Hecke eigenvariety $Y(U^{\fp},\overline{\rho})$ be the schematic support of the coherent $\cO_{\FT_{\overline{\rho},U^{\fp}}\times\widehat{T}_L}$-module (equivalently,\;$\cO_{\FX_{\overline{\rho},U^{\fp}}\times\widehat{T}_L}$-module) $\big(J_{\bB}(\widehat{S}(U^{\fp}, W^\fp)^{\ana}_{\overline{\rho}})\big)^\vee$ on $\FT_{\overline{\rho},U^{\fp}}\times \widehat{T}_L$ (equivalently,\;$\FX_{\overline{\rho},U^{\fp}}\times \widehat{T}_L$),\;where $J_{\bB}$ is the Jacquet-Emerton functor (see \cite{emerton2006jacquet}) with respect to $\bB$.\;This is a reduced rigid analytic variety over $E$ of dimension $n[F^+:\bQ]$,\;which admits (or factors through) an injections of rigid spaces over $E$:
\[Y(U^{\fp},\overline{\rho})\hookrightarrow \FT_{\overline{\rho},U^{\fp}}\times\widehat{T}_L\hookrightarrow \FX_{\overline{\rho},U^{\fp}}\times \widehat{T}_L.\]
For $x=(\rho,{\delta})\in \FX_{\overline{\rho},U^{\fp}}\times \widehat{T}_L$,\;it belongs to $Y(U^{\fp},\overline{\rho})$ if and only if 
\[\homo_{T(L)}(\delta,J_{\bB}(\widehat{S}(U^{\fp}, W^\fp)^{\ana}_{\overline{\rho}}[\fm_{\rho}]\otimes_{k(\rho)}k(x)\neq 0,\]
where $\fm_\rho\subset R_{\cS}^{\univ}[1/p]$ denotes the maximal ideal corresponding to the point $\rho\in \FX_{\overline{\rho},U^{\fp}}$.\;

We next briefly recall the following version of patched eigenvariety,\;given in \cite[Section 4.1.1]{2019DINGSimple}.\;Let $R^{\fp}=\Big(\widehat{\otimes}_{S_p\backslash\{\fp\}}R_{\widetilde{v}}^{\square, \xi,\tau}\Big)\widehat{\otimes} R_{\widetilde{v_1}}^{\square}$ and $R_{\infty}^{\fp}:=R^{\fp}\llbracket x_1,\cdots, x_g\rrbracket$.\;Then we have  $R^{\loc}=R^{\fp}\widehat{\otimes} \defvarring$ (recall that $R_{\widetilde{v}}^{\square}\cong \defvarring$) and $R_{\infty}=R_{\infty}^{\fp}\widehat{\otimes} \defvarring$.\;Let $\BU$ be the open unit ball in $\BA^1$.\;We put $\FX_{\overline{\rho}^{\fp}}^\Box:=(\Spf R^{\fp})^{\rig}$ and  $\mathfrak{X}_{\overline{r}}^\Box=(\spf\;\defvarring)^{\rig}$.\;Then $(\Spf R^{\fp}_{\infty})^{\rig}=\FX_{\overline{\rho}^{\fp}}^\Box\times \BU^g$.\;We have thus
$\FX_{\infty}:=(\Spf R_{\infty})^{\rig}\cong (\Spf R^{\fp}_{\infty})^{\rig}\times \mathfrak{X}_{\overline{r}}^\Box\cong \FX_{\overline{\rho}^{\fp}}^\Box\times \BU^g\times \mathfrak{X}_{\overline{r}}^\Box$.\;By \cite[Section 4.1.2]{2019DINGSimple},\;we see that $J_{\bB}(\Pi_\infty^{R_\infty-\ana})^\vee$ is a coadmissible module over $\cO(\FX_{\infty}\times\widehat{T}_L)$, which corresponds to a coherent sheaf $\cM_{\infty}$ over $\FX_{\infty}\times\widehat{T}_L$ such that
\[\Gamma\Big(\FX_{\infty}\times\widehat{T}_L,\cM_{\infty}\Big)\cong J_{\bB}(\Pi_\infty^{R_\infty-\ana})^\vee.\]
Let $X_{\fp}(\overline{\rho})\hookrightarrow\FX_{\infty}\times\widehat{T}_L$ be the Zariski-closed support of $\cM_{\infty}$.\;We call $X_{\fp}(\overline{\rho})$  the patched  eigenvariety.\;By \cite[Theorem 4.1]{2019DINGSimple},\;we have
\begin{pro}\label{basicpropertyeigen}\hspace{20pt}
	\begin{itemize}
		\item[(1)]  For $x=(\fm_x,\chi_x)\in \FX_{\infty}\times\widehat{T}_L$,\;$x\in X_{\fp}(\overline{\rho})$ if and only if
		$J_{\bB}(\Pi_\infty^{R_\infty-\ana} )[\fm_y,\bT(L)=\chi_x]\neq0$.
		\item[(2)] The rigid space $X_{\fp}(\overline{\rho})$ is reduced and equidimensional of dimension
		\[g+nd_L+n^2(|S_p|+1)+[F^+:\BQ]\frac{n(n-1)}{2}.\]
		\item[(3)] The coherent sheaf $\cM_{\infty}$ is Cohen-Macaulay over $X_{\fp}(\overline{\rho})$.\;
		\item[(4)] The set of very classical non-critical generic points is Zarisiki-dense in $X_{\fp}(\overline{\rho})$ and is an accumulation set.\;The set of very classical non-critical generic points  accumulates at point $x=(\fm_x, \chi_x)$ with $\chi_x$ locally algebraic.\;
		\item[(5)] The Hecke variety $Y(U^{\fp},\overline{\rho})$ is the reduced Zariski-closed subspace of $X_{\fp}(\overline{\rho})$ underlying the vanishing locus of $\fa\Gamma(\FX_{\infty},\cO_{\FX_{\infty}})$.\;
	\end{itemize}
\end{pro}

The Hecke eigenvariety $Y(U^{\fp},\overline{\rho})$ and patched  eigenvariety $X_{\fp}(\overline{\rho})$ are related to the trianguline variety $X_{\mathrm{tri}}(\overline{r})$ as follows.\;Let $\iota_{\fp}:\widehat{T}_L\rightarrow \widehat{T}_L$ be the automorphism defined by
\[\iota_{\fp}(\delta_1,\cdots,\delta_n):=\delta_{B}\cdot(\delta_1,\delta_2\unr(q_L^{-1}),\cdots,\delta_n\unr(q_L^{-(n-1)})).\]
Note that $\iota_{\fp}(\delta_1,\cdots,\delta_n)=(\delta_1,\cdots,\delta_n)\cdot \zeta$,\;where
\[\zeta:=\Big(\unr(q_L^{1-n}),\cdots,\unr(q_L^{i-n})\prod_{\tau\in\Sigma_L}\tau^{i-1},\cdots,\;\prod_{\tau\in\Sigma_L}\tau^{n-1}\Big).\]
Then $\mathrm{id}\times \iota_{\fp}$ induces an isomorphism of rigid spaces $\mathrm{id}\times \iota_{\fp}:\mathfrak{X}_{\overline{\rho}_{\fp}}^\Box\times\widehat{T}_L
\xrightarrow{\sim} \mathfrak{X}_{\overline{\rho}_{\fp}}^\Box\times\widehat{T}_L$.\;Let $\iota_{\fp}\big(X_{\mathrm{tri}}(\overline{r})\big)$ be the image of $X_{\mathrm{tri}}(\overline{r})$ via this automorphism.\;Then the natural embedding 
\[X_{\fp}(\overline{\rho})\hookrightarrow\FX_{\infty}\times\widehat{T}_L\cong (\Spf R_{\infty})^{\rig}\times\widehat{T}_L\cong \FX_{\overline{\rho}^{\fp}}^\Box\times \BU^g\times \mathfrak{X}_{\overline{r}}^\Box\times\widehat{T}_L\]
factors through
\begin{equation}\label{injpatchtotri}
	Y(U^{\fp},\overline{\rho})\hookrightarrow X_{\fp}(\overline{\rho})\hookrightarrow\FX_{\overline{\rho}^{\fp}}^\Box\times \BU^g\times \iota_{\fp}\big(X_{\mathrm{tri}}(\overline{r})\big).
\end{equation}
Therefore,\;$\iota_\fp$ induces morphisms
\begin{equation}\label{injpatchtotri1}
	\iota_\fp^{-1}:Y(U^{\fp},\overline{\rho})\hookrightarrow X_{\fp}(\overline{\rho})\rightarrow X_{\mathrm{tri}}(\overline{r}).
\end{equation}
For each irreducible component $\FX^{\fp}$ of  $\FX_{\overline{\rho}^{\fp}}^\Box$,\;there is a (possibly empty) union $X^{\FX^{\fp}-\mathrm{aut}}_{\mathrm{tri}}(\overline{r})$ of irreducible components of $X_{\mathrm{tri}}(\overline{r})$ such that we have an isomorphism of closed analytic subsets of $\FX_{\infty}\times\widehat{T}_L$:
\begin{equation}\label{comparepatchedandtrivar}
	X_{\fp}(\overline{\rho})\cong\bigcup_{\FX^{\fp}}\FX^{\fp}\times \iota_{\fp}\big(X^{\FX^{\fp}-\mathrm{aut}}_{\mathrm{tri}}(\overline{\rho}_{\fp})\big)\times \BU^g.
\end{equation}

%Note that the composition 
%\[X_{\fp}(\overline{\rho})\hookrightarrow\FX_{\overline{\rho}^{\fp}}^\Box\times \iota_{\fp}\big(X_{\mathrm{tri}}(\overline{r})\big)\times \BU^g\twoheadrightarrow \iota_{\fp}\big(X_{\mathrm{tri}}(\overline{r})\big)\xrightarrow{\iota_{\fp}^{-1}}X_{\mathrm{tri}}(\overline{r})\]
%coincides with the map (\ref{heckeeigenrelatitrivar}).\;

%Then the morphism or rigid space:
%\begin{equation}
%	\begin{aligned}
	%		\FX_{\overline{\rho},\U^{\fp}}\times \widehat{T}_L\rightarrow  \spf\;\cR_{\overline{\rho}_\fp}\times\widehat{T}_L\\
	%		(\rho,(\delta_1,\cdots,\delta_n))\mapsto (\rho_L,\iota_{\fp}^{-1}(\delta_1,\cdots,\delta_n))
	%	\end{aligned}
%\end{equation}
%induces a morphism of reduced rigid spaces over $E$:
%\begin{equation}\label{heckeeigenrelatitrivar}
%	Y(U^{\fp},\overline{\rho})\longrightarrow  X_{\mathrm{tri}}(\overline{r}).
%\end{equation}

%\subsection{Automorphic cycles}

%See \cite[Section 5.2]{breuil2019local} for some preliminaries on locally analytic representations.\;

\subsection{Classicality}

We fix a Galois representation $\rho\in \FX_{\overline{\rho},U^{\fp}}$.\;We make the following Hypothesis.\;
\begin{hypothesis}\hspace{20pt}\label{hyongaloisrep1}
	\begin{itemize}
		\item[(a)] $\rho$ comes from a point $y\in Y(U^{\fp},\overline{\rho})$ of the form $y=(\rho,\underline{\delta})$.\;
		\item[(b)] $\rho_L:=\rho_{\fp}$ is a semistable non-crystalline $p$-adic Galois representation with full monodromy rank.,\;i.e., the monodromy operator $N$ on $D_{\mathrm{st}}(\rho_{L})$ satisfies $N^{n-1}\neq 0$.
		\item[(c)] Let $\bh:=(\hpi_{\tau,1}>\hpi_{\tau,2}>\cdots>\hpi_{\tau,n} )_{\tau\in \Sigma_L}$ be the distinct Hodge-Tate weights of $\rho_{L}$.\;Let ${\alpha}\in E$ such that ${\alpha},{\alpha q_L^{1}},\;\cdots,\;{\alpha q_L^{n-1}}$ are  $\varphi^{f_L}$-eigenvalues of $D_{\mathrm{st}}(\rho_{L})$.\;Then $D_{\rig}(\rho_{L})$ admits a triangulation $\cF$ with parameters $(\unr(\alpha)_n)\cdot z^{w_{\cF}\underline{w}_0(\bh)}$ for $w_{\cF}\in \sW^{\emptyset,\emptyset}_{n,L,\max}\cong \sW_{n,L}$,\;where $\unr(\alpha)_n:=(\unr({\alpha}),\unr({\alpha q_L^{1}}),\;\cdots,\;\unr({\alpha q_L^{n-1}}))$.\;
		%		\item[(c)] The image of $y$ in $\FX_{\overline{\rho}^{\fp}}^\Box$ lies in some irreducible component $\FX^{\fp}$.\;This is equivalent to say that
		%		\[y\in X_{\fp}(\overline{\rho})^{\FX^{\fp}}:=\FX^{\fp}\times \iota_{\fp}\big(X^{\FX^{\fp}-\mathrm{aut}}_{\mathrm{tri}}(\overline{\rho}_{\fp})\big)\times \BU^g.\]
		%	It is clear that $X_{\fp}(\overline{\rho})^{\FU^{\fp}}$ is a Zariski-open subset of $X_{\fp}(\overline{\rho})$.\;
	\end{itemize}
\end{hypothesis}
Recall that $\hpi_{i}=(\hpi_{\tau,i})_{\tau\in \Sigma_L}$ for $1\leq i\leq n$.\;For $w\in\sW_{n,L}$,\;we put $y_{w\underline{w}_0}:=(\rho,\delta_B\chi_{w\underline{w}_0\cdot{\lambda_{\bh}}}\unr(\alpha)\circ\det)\in \FX_{\overline{\rho},U^{\fp}}\times \widehat{T}_L$ (so $y_{1}=y_{\underline{w}_0\underline{w}_0}$).\;It is easy to see that  $x_1:=(\rho_{L},(\unr(\alpha)_n)\cdot z^{\bh})$ (resp.,\;$x_{w\underline{w}_0}:=(\rho_{L},(\unr(\alpha)_n)\cdot z^{w\underline{w}_0(\bh)})$) is the image of $y_1$ (resp.,\;$y_{w\underline{w}_0}$) via the morphism  (\ref{injpatchtotri1}).\;Then $y=y_{w_y\underline{w}_0}$ for some $w_y\in\sW_{n,L}$.\;

\begin{thm}\label{Classicality}(Classicality) Assume Hypothesis \ref{hyongaloisrep1}.\;Suppose $x_{w_{\cF}w_0}\in X_{\mathrm{tri}}(\overline{r})$,\;then $\widehat{S}_{\xi,\tau}(U^{\fp},E)^{\lalg}_{\overline{\rho}}[\fm_{{\rho}}]\neq 0$,\;i.e., $\rho$ is associated to a classical automorphic representation of $\widetilde{G}(\BA_{F^+}^{\infty})$.\;
\end{thm}
\begin{proof}
By applying Proposition \ref{localgeomertyonspecial} to $X_{\mathrm{tri}}(\overline{r})$ and $x_{w_y\underline{w}_0(\bh)}$,\;there is a unique irreducible component $Z$ of $X_{\mathrm{tri}}(\overline{r})$ passing through $x_{w_y\underline{w}_0(\bh)}$.\;From (\ref{comparepatchedandtrivar}) we thus necessary have $x\subseteq \iota_{\fp}\big(Z\big)\times \BU^g\subseteq  \FX^{\fp}\times \iota_{\fp}\big(X^{\FX^{\fp}-\mathrm{aut}}_{\mathrm{tri}}(\overline{r})\big)\times \BU^g$ for some irreducible component $\FX^{\fp}$ of $\FX_{\overline{\rho}^{\fp}}^\Box$.\;In particular,\;for $V\subseteq X_{\mathrm{tri}}(\overline{r})$ a sufficiently small open neighbourhood of $x_{w_y\underline{w}_0(\bh)}$ in $X_{\mathrm{tri}}(\overline{r})$,\;we have $V\subseteq Z\subseteq X^{\FX^{\fp}-\mathrm{aut}}_{\mathrm{tri}}(\overline{r})$.\;We modify the proof of \cite[Theorem 3.9]{breuil2017smoothness} to our case.\;

Keep the argument and notation in the \cite[Section 5]{BSCONJ},\;the $R_{\infty}\otimes_{R_{\widetilde{\fp}}^{\square}}R_{\overline{r}}^{\Box,\bh-\mathrm{st}}$-module $(\Pi_{\infty}(\sigma^{\circ}_{\min})[1/p])^{\vee}$ is supported on a union of irreducible components of $\iota_{\fp}\times \mathfrak{X}_{\overline{r}}^\Box\times \BU^g$,\;and we have to prove that $r_y$ ia a point on one of these irreducible components.\;Recall that $Z(\rho_L)$ is  the unique irreducible component of $\mathfrak{X}_{\overline{r}}^\Box$ containing $\rho_L$.\;It is enough to prove that $\iota_{\fp}\times Z(\rho_L)\times \BU^g$ is one of the irreducible components in the support of $(\Pi_{\infty}(\sigma^{\circ}_{\min})[1/p])^{\vee}$,\;or equivalently that $\iota_{\fp}\times Z(\rho_L)\times \BU^g$ contains at least one point which is in the support of $(\Pi_{\infty}(\sigma^{\circ}_{\min})[1/p])^{\vee}$.\;By Hypothesis \ref{appenhypothesis},\;we have a closed immersion $	\iota_{\bh}:	\FX_{\overline{r},\cP_{\min}}^{\Box,\bh-\mathrm{st}}\hookrightarrow X_{\mathrm{tri}}(\overline{r})$.\;For any point  $x'=(r_{x'},\delta')\in \iota_{\bh}(Z(\rho_L))\cap V\subseteq X^{\FX^{\fp}-\mathrm{aut}}_{\mathrm{tri}}(\overline{r})$,\;by the choice of $V$,\;we may choose point $x'$ in $\iota_{\bh}(Z(\rho_L))\cap V$ such that the associated semistable non-crystalline Galois representation $r_{x',\fp}$ belongs to subspace $\FX_{\overline{r},\cP_{\min},\underline{w}_0}^{\Box,\bh-\mathrm{st}}$ (non-critical points in $\FX_{\overline{r},\cP_{\min}}^{\Box,\bh-\mathrm{st}}$).\;Since non-critical point are classical (by the global triangulation theory,\;see \cite[Proposition 4.7]{2019DINGSimple}),\;we see that $r_{x'}$ is in the support of $(\Pi_{\infty}(\sigma^{\circ}_{\min})[1/p])^{\vee}$.\;This completes the proof.\;
\end{proof}
\begin{rmk}
Apply the above proof to $L=\bQ_p$ case,\;together with the results in \cite[Theroem 5.6.5]{MSW},\;the classicality theorem still holds for $L=\bQ_p$ case.\;%See Section \ref{proofsLQp} for more details.\;
\end{rmk}

\subsection{Patching functor}

\subsubsection{Recollections of variants of the BGG-category}\label{deformedverma}

Given $I\subseteq \Delta_n$,\;we apply the argument in \cite[Section 2]{HHS} to $\GLN_{n,L}$,\;and get the Bernstein-Gelfand-Gelfand (BGG) category and its variants (deformations of BGG category) \[\cO_{\alge}^{\overline{\fp}_{I,L}}\subseteq\cO_{\alge}^{\overline{\fp}_{I,L},\infty}\subseteq \widetilde{\cO}_{\alge}^{\overline{\fp}_{I,L}}.\]
We modify the notation of  \cite[Section 2]{HHS} by replacing the superscript $I$ in \cite[Section 2]{HHS} by $\overline{\fp}_{I,L}$.\;

For $J\subseteq\Sigma_L$,\;let $\fm_{I,J}$ (resp.,\;$\fm_{J}$) be the augmentation ideal of the envelope algebra $U(\fz_{I,J})$ (resp.,\;$U(\ft_{J})$). We set $A_{I,J}:=U(\fz_{I,J})_{\fm_{I,J}}$ and  $A_{\emptyset,J}=U(\ft_{J})_{\fm_{J}}$ the localizations.\;The canonical Lie algebra decomposition $\fz_{I,L}=\fz_{I,J}\oplus \fz_{I,L\backslash J}$ defines a surjective morphisms $U(\fz_{I,L})\twoheadrightarrow  U(\fz_{I,J})$ and $A_{I,\Sigma_L} \twoheadrightarrow A_{I,J}$ of $A_{I,L}$-modules.\;

For $\mu\in X_{I}^+$,\;we define the $J$-deformed generalized Verma module of weight $\mu$ as
\[\widetilde{M}^{J}_I(\mu):=U(\fg_{\Sigma_L})\otimes_{U(\fp_{I,L})}(L_I(\mu)\otimes_EA_{I,J})\in \widetilde{\cO}_{\alge}^{\overline{\fp}_{I,L}},\]
where $U(\fp_{I,L})$ acts on $A_{I,J}$ via the composition $U(\fp_{I,L})\rightarrow U(\fl_{I,L})\rightarrow U(\fz_{I,L})\rightarrow A_{I,L}\twoheadrightarrow A_{I,J}$.\;If $J=\Sigma_L$,\;we drop the superscript $J$ in $\widetilde{M}^{J}_I(\mu)$.\;Moreover,\;by \cite[Section 2.3]{HHS},\;we have an isomorphism of $U(\fg_{\Sigma_L})_{A_{I,L}}$-modules $\widetilde{M}^{J}_I(\mu)\otimes_{A_{I,J}}A_{I,J}/\fm_{I,J}={M}_I(\mu)$.\;The surjective morphism  $A_{I,L} \twoheadrightarrow A_{I,J}$ induces a surjective morphism $\widetilde{M}_I(\mu)\twoheadrightarrow\widetilde{M}^{J}_I(\mu)$.

\subsubsection{Locally analytic patching functor}

For $y\in X_{\fp}(\overline{\rho})\hookrightarrow\FX_{\infty}\times\widehat{T}_L\cong \FX_{\overline{\rho}^{\fp}}^\Box\times \BU^g\times \mathfrak{X}_{\overline{r}}^\Box\times\widehat{T}_L$,\;denoted by $r_{y}$ (resp.,\;$\fm_{r_y}$) its image in $\FX_{\infty}=\FX_{\overline{\rho}^{\fp}}^\Box\times \BU^g\times \mathfrak{X}_{\overline{r}}^\Box$ (resp.,\;the corresponding maximal ideal of $R_{\infty}[1/p]$).\;Denoted by $r_{y,\fp}$ (resp.,\;$r_y^{\fp}$,\;resp.,\;$z$) its image in $\mathfrak{X}_{\overline{r}}^\Box=(\spf\;\defvarring)^{\rig}$ (resp.,\;$\FX_{\overline{\rho}^{\fp}}^\Box$,\;resp.,\;$\BU^g$),\;and by $\underline{\epsilon}_y$ its image in $\widehat{T}_L$ (so $y=(r_y^{\fp},r_{y,\fp},z,\underline{\epsilon}_y)=(r_y,\underline{\epsilon}_y)$).\;

%We introduce some closed subspaces of $X_{\fp}(\overline{\rho})$,\;which are closely related to the companion points and companion constituents.\;

We follow the argument and notation in \cite[Section 6.1]{HHS}.\;For $M\in \cO^{\infty}_{\mathrm{alg}}$,\;we have the functor
\begin{equation}
	\cM_{\infty,y}(M):=\homo_{U(\fg)}\Big(M,(\Pi_\infty^{R_\infty-\ana})^{U_0}\Big)_{\mathrm{fs}}[\fm_{r_y}^\infty][\fm^\infty_{\underline{\epsilon}_{\mathrm{sm}}}]^\vee,\;
\end{equation}
where $U_0$ is a compact open subgroup of $\bN(L)$.\;In \cite[Section 6.1]{HHS},\;the authors extend $\cM_{\infty,y}$ to larger category $\widetilde{\cO}^{\infty}_{\mathrm{alg}}$ by putting
\begin{equation}
	\cM_{\infty,y}(M):=\varprojlim_n\cM_{\infty,y}(M/\fm_I^n).\;
\end{equation}
for $M\in \widetilde{\cO}^{\infty}_{\mathrm{alg}}$.\;By \cite[Proposition 6.5]{HHS},\;the functor $\cM_{\infty,y}$ is exact on $\widetilde{\cO}^{\infty}_{\mathrm{alg}}$.\;For each $M\in \widetilde{\cO}^{\infty}_{\mathrm{alg}}$,\;$\cM_{\infty,y}(M)$ is a finitely generated and Cohen-Macaulay $\widehat{R}_{\infty,r_y}$-module with dimension $t+\dim_L\fz_I$,\;where $t=g+\dim R^{\mathrm{loc}}$ (see Section \ref{preforpatching}).\;

Let $\mu$ be a dominant weight.\;Assume that $\underline{\epsilon}_y=\underline{\epsilon}_{\mathrm{sm}}z^{w\cdot\mu}$ with $\underline{\epsilon}_{\mathrm{sm}}$ a smooth character of $\bT(L)$ and $w\in\sW_{n,L}$.\;For any $w\in \sW_{n,L}$,\;denote by $X_{\fp}(\overline{\rho})_{w\cdot \mu}$  the fiber at $w\cdot \mu\in \ft^{\rig}(E)$ of the composition $X_{\fp}(\overline{\rho})\rightarrow \widehat{T}_L\xrightarrow{\wt}\ft^{\rig}$,\;where $\ft^{\rig}$ denotes the rigid space associated with $\homo_{E}(\ft,E)$ and the map $\wt$ sends a character of $\widehat{T}_L$ to its weight.\;Consider
\[\cM_{w\cdot \mu}:=\cM_{\infty}\otimes_{\cO_{X_{\fp}(\overline{\rho})}}\cO_{X_{\fp}(\overline{\rho})_{w\cdot \mu}}.\;\]
By the argument in \cite[Section 4.4, (4.3)]{wu2021local},\;we see that 
the vector space of compact type $\Gamma(X_{\fp}(\overline{\rho})_{w\cdot \mu},\cM_{w\cdot \mu})^\vee$ is topologically isomorphic to the following vector spaces of compact type:
\[\homo_{U(\fg)}\big(M(w\cdot \mu),(\Pi_\infty^{R_\infty-\ana})^{U_0}\big)_{\mathrm{fs}}\cong \homo_{U(\ft)}\big(w\cdot \mu,J_{\bB}(\Pi_\infty^{R_\infty-\ana})\big),\]
where $(-)_{\mathrm{fs}}$ denotes Emerton's finite slope part functor \cite[Definition 3.2.1]{emerton2006jacquet}.\;On the other hand, the quotient $M(w\cdot \mu)\twoheadrightarrow L(w\cdot \mu)$ induces a closed immersion 
\[\homo_{U(\fg)}\big(L(w\cdot \mu),(\Pi_\infty^{R_\infty-\ana})^{U_0}\big)_{\mathrm{fs}}\hookrightarrow \homo_{U(\ft)}\big(w\cdot \mu,J_{\bB}(\Pi_\infty^{R_\infty-\ana})\big).\]
Then the continuous dual $\homo_{U(\fg)}\big(L(w\cdot \mu),(\Pi_\infty^{R_\infty-\ana})^{U_0}\big)_{\mathrm{fs}}^\vee$ corresponds to a coherent sheaf $\cL_{w\cdot \mu}$ on $X_{\fp}(\overline{\rho})_{w\cdot \mu}$ (so $\Gamma(X_{\fp}(\overline{\rho})_{w\cdot \mu},\cL_{w\cdot \mu}^\vee)\cong \homo_{U(\fg)}\big(L(w\cdot \mu),(\Pi_\infty^{R_\infty-\ana})^{U_0}\big)_{\mathrm{fs}}$).\;The schematic support of $\cL_{w\cdot \mu}$ defines a Zariski-closed rigid subspace $Y_{\fp}(\overline{\rho})_{w\cdot \mu}$ in $X_{\fp}(\overline{\rho})_{w\cdot \mu}$.\;Let $Y_{\fp}(\overline{\rho})_{w\cdot \mu}^{\red}$ be the underlying reduced analytic subvariety of $Y_{\fp}(\overline{\rho})_{w\cdot \mu}$.\;Then we have 
$\cM_{\infty,y}(L(w\cdot \mu))\neq 0$ if and only if $y\in Y_{\fp}(\overline{\rho})_{w\cdot \mu}$,\;by \cite[(4.4)]{wu2021local} or \cite[(5.16),\;(5.18)]{breuil2019local}.\;If non-zero,\;the schematic support of $\cM_{\infty,y}(L(w\cdot \mu))$ is supported in $\widehat{Y_{\fp}(\overline{\rho})}_{w\cdot \mu,y}$ and 
\begin{equation}
	\begin{aligned}
			&\cM_{\infty,y}(M(w\cdot \mu))=\cM_{\infty}\otimes_{\cO_{X_{\fp}(\overline{\rho})}}\cO_{X_{\fp}(\overline{\rho})_{w\cdot \mu},y},\\
			&\cM_{\infty,y}(L(w\cdot \mu))=\cL_{w\cdot \mu} \otimes_{\cO_{X_{\fp}(\overline{\rho})_{w\cdot \mu}}}\cO_{X_{\fp}(\overline{\rho})_{w\cdot \mu},y}\cong \cL_{w\cdot\mu} \otimes_{\cO_{Y_{\fp}(\overline{\rho})_{w\cdot \mu}}}\cO_{Y_{\fp}(\overline{\rho})_{w\cdot \mu},y}.
		\end{aligned}
\end{equation}

\subsection{Global companion points and companion constituents}\label{companionpointandconsti}
%\begin{rmk}The dimension of $\FX_{\overline{r},\cP_1,\cP_2}^{\Box,\tau,\bh}$ is $n^2+d_L\frac{n(n-1)}{2}-(k_{\cP_1}-k_{\cP_2})$.\;
%\end{rmk}

%If $Y$ is a subspace of $\mathfrak{X}_{\overline{\rho}_{\fp}}^\Box$,\;we put $X_{\mathrm{tri}}(Y,\overline{\rho}_{\fp}):=X_{\mathrm{tri}}(\overline{r})\times_{\mathfrak{X}_{\overline{\rho}_{\fp}}^\Box}Y$.\;

%We can identify $\FX_{\mathrm{tri},\cP_{\min},w}^{\bh-\mathrm{st},\triangle}(\overline{\rho}_{\fp})$ with its image in $X_{\mathrm{tri}}(\overline{r})$ via $\iota_{w,\bh}$.\;
%We consider the space $X_{\mathrm{tri}}(\FX_{\overline{\rho}_{\fp}}^{\Box,\bh-\mathrm{st}},\overline{\rho}_{\fp})\times_{\widehat{T}_L}\widehat{T}_L_{w,\bh}$.\;Then the morphism $\iota_{w,\bh}$ induces an isomorphism between $\FX_{\mathrm{tri},\cP_{\min},w}^{\bh-\mathrm{st}}(\overline{\rho}_{\fp})$ and $X_{\mathrm{tri}}(\FX_{\overline{\rho}_{\fp}}^{\Box,\bh-\mathrm{st}},\overline{\rho}_{\fp})\times_{\widehat{T}_L}\widehat{T}_L_{w,\bh}$.\;

We now state the main results on the appearance of companion constituents in the completed cohomology and the existences of global companion points.\;See Proposition \ref{prolocalcomapn1},\;Proposition \ref{prolocalcomapn},\;Proposition \ref{equicompanpointconsti} and Theorem \ref{globalcompanion}.\;We prove our main theorems by using the same strategy as in the proof of \cite[Proposition 4.7, Proposition 4.9, Theorem 4.10,\;Theorem 4.12]{wu2021local} and \cite[Proposition 7.2]{HHS}.\;
%.\;They are proved by descending induction on some induction hypothesis $\cH_l$,\;which is modified by the Step $8$ and Step $9$ in the proof of \cite[Theorem 5.3.3]{breuil2019local},\;and some Zariski-closure argument on semistable deformation rings.\;Therefore,\;some assumption on the smoothness of $\FX^{\fp}\times \BU^g$ at certain point in the proof of \cite[Theorem 5.3.3]{breuil2019local} is expected to be non-necessary.\;

%We denote by $Z_{\fp}(\overline{\rho})_{w\cdot \mu}\subset \big(Y_{\fp}(\overline{\rho})_{w\cdot \mu}\big)^{\mathrm{red}}$ the union of its irreducible components of dimension $\dim X_{\fp}(\overline{\rho})_{w\cdot \mu}$.\;

\begin{hypothesis}\label{hyongaloisrep} Assume $\widehat{S}_{\xi,\tau}(U^{\fp},E)^{\lalg}_{\overline{\rho}}[\fm_{{\rho}}]\neq 0$ and $(b),(c)$ in Hypothesis \ref{hyongaloisrep1}.\;
\end{hypothesis}

%the smooth locus of the reduced rigid variety $\FX_{\overline{\rho}^{\fp}}^\Box$ and We write $\FU^{\fp}=\FX^{\fp}\cap \FX_{\overline{\rho}^{\fp}}^{\Box,\mathrm{sm}}$,\;where 
%$\FX_{\overline{\rho}^{\fp}}^{\Box,\mathrm{sm}}$ is the smooth locus of $ \FX_{\overline{\rho}^{\fp}}^\Box$.\;
%Under hypothesis \ref{hyongaloisrep},\;we have $\underline{\delta}=\delta_B\chi_{{\lambda_{\bh}}}\unr(\alpha)\circ\det$.\;

The first goal of this section is to show that $\{y_{w\underline{w}_0}\}_{w\leq w_{\cF}\underline{w}_0}$ are  global companion points of $y$ ,\;i.e.,\;$y_{w\underline{w}_0}\in Y(U^{\fp},\overline{\rho})$  for $w\leq w_{\cF}\underline{w}_0$.\;Put ${\bm\lambda}_\bh:=(\hpi_{\tau,i}+i-1)_{\tau\in \Sigma_L,1\leq i\leq n}$, which is a dominant weight of $\GLN_{n,L}$ with  respect to $\bB_{n,L}$.\;Recall that  $r_{y_1}:=(r_{y_1}^{\fp},\rho_L,z)\in {{\FX}}_{\infty}=\FX_{\overline{\rho}^{\fp}}^\Box\times \BU^g\times \mathfrak{X}_{\overline{r}}^\Box$ is the image of $y_1$ in ${{\FX}}_{\infty}$.\;We have
\[\widehat{{\FX}}_{\infty,r_{y_1}}=\widehat{(\FX_{\overline{\rho}^{\fp}}^\Box)}_{r_{y_1}^{\fp}}\times 
\widehat{(\BU^g)}_{z}\times \widehat{(\mathfrak{X}_{\overline{r}}^\Box)}_{\rho_L}\cong \widehat{(\FX_{\overline{\rho}^{\fp}}^\Box)}_{r_{y_1}^{\fp}}\times 
\widehat{(\BU^g)}_{z}\times X_{\rho_L}.\]

In the sequel,\;let $y$ be a possible companion point of $y_1$ (so that $r_{y}=r_{y_1}$).\;We rewrite  ${\FX}_{\infty,y}^{\fp}:=\widehat{(\FX_{\overline{\rho}^{\fp}}^\Box)}_{r_{y_1}^{\fp}}\times 
\widehat{(\BU^g)}_{z}$ for simplicity (so that $\widehat{{\FX}}_{\infty,r_{y}}={\FX}_{\infty,y}^{\fp}\times X_{\rho_L}$).\;We define 
\[{{\FX}}_{\infty,y}={\FX}_{\infty,y}^{\fp}\times X_{\rho_L,\cM_{\bullet}}\subset \widehat{{\FX}}_{\infty,r_{y}}.\]
We define some subspaces of ${\FX}_{\infty,y}$ (using partially de-Rham Galois cycles),\;which are related to the supports of the patching functor $\cM_{\infty,y}(-)$.\;For $J\subseteq \Sigma_L$,\;${I}\subset \Delta_n$ and $w\in \sW_{n,L}$,\;we put
\begin{equation}
	\begin{aligned}
		&{\FX}^{\underline{I}^J}_{\infty,y}={\FX}_{\infty,y}^{\fp}\times X^{J,\underline{I}^J}_{\rho_L,\cM_{\bullet}},\;\text{and\;}{{\FX}}^{\underline{I}^J,w}_{\infty,y}={\FX}_{\infty,y}^{\fp}\times X^{J,\underline{I}^J,w}_{\rho_L,\cM_{\bullet}},\\	&{\overline{\FX}}^{\underline{I}^J}_{\infty,y}={\FX}_{\infty,y}^{\fp}\times \overline{X}^{J,\underline{I}^J}_{\rho_L,\cM_{\bullet}},\;\text{and\;}{\overline{\FX}}^{\underline{I}^J,w}_{\infty,y}={\FX}_{\infty,y}^{\fp}\times \overline{X}^{J,\underline{I}^J,w}_{\rho_L,\cM_{\bullet}},
	\end{aligned}
\end{equation}
where $\underline{I}^J=\prod_{\tau\in J}I$,\;and we drop the superscript $J$ if $J=\Sigma_L$.\;Let $Y$ be a subspace of  $\widehat{\cX}^{\flat}_{L,\widehat{y}}$.\;Let $X_{\rho_L,\cM_{\bullet}}(Y)$ be the essential image of $X^{\Box}_{\rho_L,\cM_{\bullet}}\times_{\widehat{\cX}^{\flat}_{L,\widehat{y}}}Y$  via the formally smooth morphism $X^{\Box}_{\rho_L,\cM_{\bullet}}\rightarrow X_{\rho_L,\cM_{\bullet}}$.\;We put (recall the diagram (\ref{factorwwwdelta}))
\begin{equation}
	\begin{aligned}
		&\overline{\overline{\FX}}_{\infty,y}={\FX}_{\infty,y}^{\fp}\times\Spec \overline{\overline{R}}_{\rho_L,\cM_{\bullet}}\cong {\FX}_{\infty,y}^{\fp}\times X_{\rho_L,\cM_{\bullet}}(\Spf \widehat{\cO}^{\flat}_{\overline{{X}}_{L,{y}}}),\\
		&\overline{\overline{\FX}}^{w}_{\infty,y}={\FX}_{\infty,y}^{\fp}\times\Spec \overline{\overline{R}}^{w}_{\rho_L,\cM_{\bullet}}\cong {\FX}_{\infty,y}^{\fp}\times X_{\rho_L,\cM_{\bullet}}(\Spf \widehat{\cO}^{\flat}_{\overline{{X}}_{w,{y}}}).
	\end{aligned}
\end{equation}
We further put
\begin{equation}
	\begin{aligned}
		{{\FZ}}^{w}_{\infty,y}={\FX}_{\infty,y}^{\fp}\times X_{\rho_L,\cM_{\bullet}}(\Spf \widehat{\cO}^{\flat}_{\cZ_{w},\widehat{y}}),\;\text{resp.,\;}
		\overline{{\FZ}}^{w}_{\infty,y}={\FX}_{\infty,y}^{\fp}\times  X_{\rho_L,\cM_{\bullet}}(\Spf \widehat{\cO}^{\flat}_{Z_{w},y}),
	\end{aligned}
\end{equation} 
which is a closed subspace of ${\overline{\FX}}_{\infty,y}$ (resp.,\;$\overline{\overline{\FX}}_{\infty,y}$).\;

\subsubsection{General discussion on coherent sheaves and cycles}\label{Analysissheafcycles}

%	\item[(c)] Assume $L\neq \bQ_p$.\;For $w^{\min}\underline{w}_0\geq w_{\cF}$,\;the schematic supports of the sheaves $\cM_{\infty,y}(\widetilde{M}^{\Sigma_L\backslash J}_{{I}}(w\cdot {\bm\lambda}_\bh))$ and $\cM_{\infty,y}(\widetilde{M}^{\Sigma_L\backslash J}_{{I}}(w\cdot {\bm\lambda}_\bh)^{\vee})$ are either ${\overline{\FX}}_{\infty,y}^{\underline{I}^J,w^{\min}\underline{w}_0}$ or empty.\;If non empty,\;their schematic support is  ${\overline{\FX}}_{\infty,y}^{\underline{I}^J,w^{\min}\underline{w}_0}$.\;These sheaves are generically free over its support ${\overline{\FX}}_{\infty,y}^{\underline{I}^J,w^{\min}\underline{w}_0}$.\;Thus we have $[\cM_{\infty,y}(\widetilde{M}^{\Sigma_L\backslash J}_{{I}}(w\cdot {\bm\lambda}_\bh))]=r[{\overline{\FX}}_{\infty,y}^{\underline{I}^J,w^{\min}\underline{w}_0}]\in Z^0({\overline{\FX}}_{\infty,y}^{\underline{I}^J})$.\;

Assume $L\neq \bQ_p$.\;We begin with a computation of  the schematic supports $\cM_{\infty,y}(\widetilde{M}^{\Sigma_L\backslash J}_{{I}}(w\cdot {\bm\lambda}_\bh))$ and $\cM_{\infty,y}(\widetilde{M}^{\Sigma_L\backslash J}_{{I}}(w\cdot {\bm\lambda}_\bh)^{\vee})$ when $J\subsetneq \Sigma_L$.\;

\begin{pro}\label{schemesupport}Fix ${I}\subset \Delta_n$,\;$w\in \sW_{n,L}$ and $J\subsetneq \Sigma_L$ or $w=1$ and $J=\Sigma_L$,\;we have
\begin{itemize}
	\item[(a)] For $w^{\min}\underline{w}_0\geq w_{\cF}$,\;the formal scheme ${{\FX}}_{\infty,y}^{\underline{I},w^{\min}\underline{w}_0}$ is reduced and is an irreducible component of ${{\FX}}_{\infty,y}^{\underline{I}}$.\;
	\item[(b)] The schematic supports of $\cM_{\infty,y}(\widetilde{M}_{{I}}(w\cdot {\bm\lambda}_\bh))$ and $\cM_{\infty,y}(\widetilde{M}_{{I}}(w\cdot {\bm\lambda}_\bh)^{\vee})$ are generically free of rank $m_{y,I,w}$ over their support (it equal to ${{\FX}}_{\infty,y}^{\underline{I},w^{\min}\underline{w}_0}$ or empty) for some integer $m_{y,I,w}$.\;If no empty,\;then $w^{\min}\underline{w}_0\geq w_{\cF}$.\;We obmit the subscript $I$ in $m_{y,I,w}$ if $I=\emptyset$.\;
	\item[(c)] For $w\underline{w}_0\geq w_{\cF}$,\;the schematic supports of the sheaves $\cM_{\infty,y}(\widetilde{M}^{\Sigma_L\backslash J}(w\cdot {\bm\lambda}_\bh))$ and $\cM_{\infty,y}(\widetilde{M}^{\Sigma_L\backslash J}(w\cdot {\bm\lambda}_\bh)^{\vee})$ are either ${\overline{\FX}}_{\infty,y}^{w\underline{w}_0}$ or empty.\;If non empty,\;their schematic support is  ${\overline{\FX}}_{\infty,y}^{w\underline{w}_0}$.\;These sheaves are generically free over its support ${\overline{\FX}}_{\infty,y}^{w\underline{w}_0}$.\;Thus we have $[\cM_{\infty,y}(\widetilde{M}^{\Sigma_L\backslash J}(w\cdot {\bm\lambda}_\bh))]=m_{y,w\underline{w}_0}[{\overline{\FX}}_{\infty,y}^{w\underline{w}_0}]\in Z^0({\overline{\FX}}_{\infty,y})$.\;
\end{itemize}
\end{pro}
\begin{proof} Part $(a)$ is obvious.\;The proof of Part $(b)$ is the same as \cite[Theorem 6.16]{HHS} and \cite[Proposition 7.2,(1)]{HHS}.\;As the sheaves,\;$\cM_{\infty,y}(\widetilde{M}_{{I}}(w\cdot {\bm\lambda}_\bh))$ and $\cM_{\infty,y}(\widetilde{M}_{{I}}(w\cdot {\bm\lambda}_\bh)^{\vee})$ are Cohen-Macaulay of the same dimension as ${{\FX}}_{\infty,r_y}^{\underline{I},w^{\min}\underline{w}_0}$ (which is generically smooth and irreducible),\;the last assertion in Part $(b)$ follows.\;For general $J\subsetneq \Sigma_L$,\;Part $(c)$ follows by quoting the regular sequence generating the maximal ideal of $U(t_{J})_{\fm_{J}}$.\;Note that
this sequence is regular for $\cM_{\infty,y}(\widetilde{M}(w\cdot {\bm\lambda}_\bh))$ and $\cM_{\infty,y}(\widetilde{M}(w\cdot {\bm\lambda}_\bh)^{\vee})$ by \cite[Proposition 6.5]{HHS},\;and is also regular for ${{\FX}}_{\infty,y}^{w^{\min}\underline{w}_0}$ (since $X^{\Box}_{\rho_L,\cM_{\bullet}}$ is formally smooth over $X^{\Box}_{\bW_\Dpik,\bF_{\bullet},J}$).\;For $J=\Sigma_L$ and $w=1$,\;this sequence is regular for ${{\FX}}_{\infty,y}^{w^{\min}\underline{w}_0}$ by Proposition \ref{flatnessX}.\;As coherent sheaves,\;$\cM_{\infty,y}(\widetilde{M}^{\Sigma_L\backslash J}(w\cdot {\bm\lambda}_\bh))$ and $\cM_{\infty,y}(\widetilde{M}^{\Sigma_L\backslash J}(w\cdot {\bm\lambda}_\bh)^{\vee})$ are Cohen-Macaulay of the same dimension as ${\overline{\FX}}_{\infty,r_y}^{w\underline{w}_0}$ (which is generically smooth),\;the last assertion in Part $(c)$ also follows.\;
\end{proof}
Recall in the proof of \cite[Theorem 2.4.7 (iii)]{breuil2019local},\;$(a_{w_{\tau},w'_{\tau}})_{(w,w')\in \sW_n\times\sW_n }$ is the product of two upper triangular matrices $B_{\tau}P_{\tau}$ with $1$ on the diagonal.\;Let $M_{\tau}$ be the diagonal matrix $(m_{y,w})_{(w,w)}$.\;Write  $(a'_{w_{\tau},w'_{\tau}})_{(w,w')\in \sW_n\times\sW_n }=B_{\tau}M_{\tau}P_{\tau}$ and put $a'_{w,w'}:=\prod_{\tau\in \Sigma_L}a_{w_\tau,w'_\tau}$.\;

\begin{cor}\label{corforcoherent}Assume Conjecture \ref{flatnessXconj} holds.\;Then the Part $(c)$ in Proposition \ref{schemesupport} also holds for every $w\in \sW_{n,L}$ and $J=\Sigma_L$.\;In particular,\;we have equality $[\cM_{\infty,y}({M}(w\cdot {\bm\lambda}_\bh))]=[\cM_{\infty,y}({M}(w\cdot {\bm\lambda}_\bh)^{\vee})]=m_{y,w}[{\overline{\FX}}_{\infty,y}^{w^{\min}\underline{w}_0}]\in Z^0({\overline{\FX}}_{\infty,y})$.\;Moreover,\;by the exactness of the patching functor $\cM_{\infty,y}(-)$,\;we have \[[\cM_{\infty,y}({L}(w\cdot {\bm\lambda}_\bh))]=\sum_{w'\geq w}a'_{w,w'}[{{\FZ}}_{\infty,y}^{w'\underline{w}_0}]=m_{y,w}[{{\FZ}}_{\infty,y}^{w\underline{w}_0}]+\sum_{w'> w}a'_{w,w'}[{{\FZ}}_{\infty,y}^{w'\underline{w}_0}]\in Z^0({\overline{\FX}}_{\infty,y}).\]  
\end{cor}
\begin{proof}
The first assertion follows by quoting the regular sequence generating the maximal ideal of $U(t_{L})_{\fm_{L}}$. Note that
this sequence is regular for $\cM_{\infty,y}(\widetilde{M}(w\cdot {\bm\lambda}_\bh))$ and $\cM_{\infty,y}(\widetilde{M}(w\cdot {\bm\lambda}_\bh)^{\vee})$ by \cite[Proposition 6.5]{HHS},\;and is also regular for ${{\FX}}_{\infty,y}^{w^{\min}\underline{w}_0}$ if  Conjecture \ref{flatnessXconj} holds.\;The remainder of this corollary is now obvious.\;
\end{proof}

The following proposition compare $m_{y,I,w}$ with $m_{y,w}$ and study the sheaves $\cM_{\infty,y}(\widetilde{M}_{{I}}(w\cdot {\bm\lambda}_\bh))$ and $\cM_{\infty,y}(\widetilde{M}_{{I}}(w\cdot {\bm\lambda}_\bh)^{\vee})$,\;if Conjecture \ref{flatnessXconj} holds.\;
\begin{pro}\label{multiconj}Suppose that Conjecture \ref{flatnessXconj} holds.\;For ${I}\subset \Delta_n$ and $w\in \sW_{n,L}$,\;the sheaves $\cM_{\infty,y}(\widetilde{M}_{{I}}(w\cdot {\bm\lambda}_\bh))$ and $\cM_{\infty,y}(\widetilde{M}_{{I}}(w\cdot {\bm\lambda}_\bh)^{\vee})$ are generally free of rank $m_{y,w}$ over their support.\;In particular,\;the integer $m_{y,I,w}$ appeared in Proposition \ref{schemesupport} is equal to $m_{y,w}$.\;
\end{pro}
\begin{proof}
	As ${{\FX}}^{\underline{I},w}_{\infty,y}$ is generically smooth for any $w$,\;we see that $\cM_{\infty,y}(M)$ is generically free of rank $r$ over its support,\;where $M$ equals to $\widetilde{M}_{{I}}(w\cdot {\bm\lambda}_\bh)$ or $\widetilde{M}_{{I}}(w\cdot {\bm\lambda}_\bh)^{\vee}$.\;As in the proof of \cite[Proposition 7.2,(7)]{HHS},\;there exists an  open subset $U$ in the regular locus of ${{\FX}}^{\underline{I},w}_{\infty,y}$ (as a scheme) such that
	the intersection of $U$ and the support of $\cM_{\infty,y}({L}(w\cdot {\bm\lambda}_\bh))$ is non-empty.\;Then $\cM_{\infty,y}(\widetilde{M}_{{I}}(w\cdot {\bm\lambda}_\bh))$ and thus $\cM_{\infty,y}({M}_{{I}}(w\cdot {\bm\lambda}_\bh))$ is locally free of rank $r$ over its support intersected with $U$.\;By Corollary \ref{corforcoherent},\;$\cM_{\infty,y}({L}(w'\cdot {\bm\lambda}_\bh))$ is not supported at every generic points of ${{\FZ}}_{\infty,y}^{w\underline{w}_0}$ if $w'>w$,\;and $\cM_{\infty,y}({L}(w\cdot {\bm\lambda}_\bh))$ has length $m_{y,w}$ at each generic point of ${{\FZ}}_{\infty,y}^{w\underline{w}_0}$.\;Since ${L}(w\cdot {\bm\lambda}_\bh)$ has multiplicity one in ${M}_{{I}}(w\cdot {\bm\lambda}_\bh)$,\;we have $m_{y,I,w}=m_{y,w}$.\;
\end{proof}

We let $m_y=\dim_{k(y)} \cM_{\infty,y}(L( {\bm\lambda}_\bh))\otimes k(y)=m_{y,1}$.\;Since we have an injection $M(w\cdot{\bm\lambda}_\bh)\hookrightarrow M({\bm\lambda}_\bh)$,\;we have $m_{y,w}\leq m_y$ for $w\in \sW_{n,L}$.\;We will show that $m_{y,w}=m_y$ in Proposition \ref{determinemulti},\;if Conjecture \ref{flatnessXconj} holds.\;

%\begin{rmk}
%	We may ask if $r=m_y$.\;Indeed,\;by the proof of Theorem \ref{Classicality},\;a sufficiently small open neighbourhood of $y$ contains a non-critical point $y'$ of Steinberg type.\;Similar to the proof of \cite[Proposition 7.2]{HHS},\;we see that $r=m_{y'}$.\;It remains to compare $m_{y'}$ and $m_{y}$.\;On the other hand,\;the regular sequence generating the maximal ideal of $U(t_{\Sigma_L})_{\fm_{\Sigma_L}}$ is regular for $\cM_{\infty,y}(\widetilde{M}_{{I}}(w\cdot {\bm\lambda}_\bh))$ and $\cM_{\infty,y}(\widetilde{M}_{{I}}(w\cdot {\bm\lambda}_\bh)^{\vee})$ by \cite[Proposition 6.5]{HHS},\;but it seems that not regular for ${{\FX}}_{\infty,y}^{w^{\min}\underline{w}_0}$.\;
%\end{rmk}

%Recall that we have an isomorphism of closed analytic subsets of $\FX_{\infty}\times\widehat{T}_L$:
%\[X_{\fp}(\overline{\rho})\cong\bigcup_{\FX^{\fp}}\FX^{\fp}\times \iota_{\fp}\big(X^{\FX^{\fp}-\mathrm{aut}}_{\mathrm{tri}}(\overline{\rho}_{\fp})\big)\times \BU^g.\]

\begin{pro}\label{caseforw0coherentsheaf}If $w_{\cF}=\underline{w}_0$ (i.e.,\;$y$ is non-critical).\;Then $\cM_{\infty,y}({M}(w\underline{w}_0\cdot {\bm\lambda}_\bh))\neq 0$ if and only if $w=\underline{w}_0$,\;and $\cM_{\infty,y}({M}({\bm\lambda}_\bh))=\cM_{\infty,y}(L({\bm\lambda}_\bh))=m_yr_{w_0}[{\overline{\FX}}^{\bh-\mathrm{st}}_{\infty,y}]\in Z^0({\overline{\FX}}_{\infty,r_y})$,\;where ${\overline{\FX}}^{\bh-\mathrm{st}}_{\infty,y}:=\FX_{\overline{\rho}^{\fp}}^\Box\times \BU^g\times \widehat{(\FX_{\overline{r}}^{\Box,\bh-\mathrm{st}})}_{\rho_L}$.\; 
\end{pro}
\begin{proof}By \cite[Corollary 4.4,\;Lemma 4.6]{2019DINGSimple},\;$y=y_{w_{\cF}\underline{w}_0}=y_1$ does not have any other companion point (so that $\cM_{\infty,y}({M}(w\underline{w}_0\cdot {\bm\lambda}_\bh))\neq 0$ if and only if $w=\underline{w}_0$) and is a smooth point on  $X_{\fp}(\overline{\rho})$.\;Therefore,\;the sheaf $\cM_{\infty,y}({M}({\bm\lambda}_\bh))=\cM_{\infty,y}(L({\bm\lambda}_\bh))$ is free over its support of rank $m_y$.\;By  \cite[Section 5]{BSCONJ},\;the schematic support of $\cM_{\infty,y}(L({\bm\lambda}_\bh))$ is contained in the ${\overline{\FX}}^{\bh-\mathrm{st}}_{\infty,y}\subseteq {\overline{\FX}}_{\infty,y}$,\;which is smooth and irreducible of the same dimension as the support of $\cM_{\infty,y}(L({\bm\lambda}_\bh))$.\;Thus we get the last equality by Proposition \ref{descriptionzw0} $(3)$.\;This completes the proof.\;
\end{proof}

Recall the Bezrukavnikov's functor recollected in \cite[Section 7.2]{HHS},\;which
associates,\;to an object $M$ in BGG category $\cO_{\chi_{\lambda_{\bh}}}$,\;a Cohen-Macaulay coherent sheaf $\cB(M)$ on $X_L^{\wedge}$,\;where ${X}_L^{\wedge}$ is the completion of $X_L$ along the preimage of $(0,0)\in \cT_{L}:=\ft_{L}\times_{\ft_{L}/\sW_{n,L}}\ft_{L}$ in $X_L$.\;It satisfies the following properties:
\begin{itemize}
	\item[(a)] For all $w\in \sW_{n,L}$,\;there is an isomorphism $\cB({M}(w\cdot{\bm\lambda}_\bh)^{\vee})\cong \cO_{\overline{X}_w}$;
	\item[(b)] For all $w\in \sW_{n,L}$,\;there is an isomorphism $\cB({M}(w\cdot{\bm\lambda}_\bh))\cong \omega_{\overline{X}_w}$,\;where $\omega_{\overline{X}_w}$ is the dualizing sheaf of $\overline{X}_w$;
	\item[(c)] Recall $P(\underline{w}_0\cdot{\bm\lambda}_\bh)$ is the anti-dominant projective envelope.\;Then $\cB(P(\underline{w}_0\cdot{\bm\lambda}_\bh))$ is equal to the structure sheaf $\cO_{\overline{X}}$.\;
\end{itemize}
As in the proof of \cite[Corollary 7.5]{HHS},\;we have the diagram
\begin{equation}\label{Bezrukavnikov1}
	\xymatrix{
		& {{\FX}}^{\Box}_{\infty,y}
		\ar[dl]_{\pi} \ar[dr]_{W^{\flat}} \ar[drr]^{W} &   &  \\
		{{\FX}}_{\infty,y}	&  & \widehat{\cX}^{\flat}_{L,\widehat{y}} \ar[r]_{\iota^{\flat}} & \widehat{X}_{L,\widehat{y}}.}
\end{equation} 
By the proof of \cite[Corollary 7.5]{HHS},\;the Bezrukavnikov's functor $\cB$ induces a functor
\[\cB_y:\cO_{\chi_{\lambda_{\bh}}}\rightarrow \mathrm{Coh}({{\FX}}_{\infty,y})\]
such that for all $M\in \cO_{\chi_{\lambda_{\bh}}}$,\;the sheaf $\cB_y(M)$ is a Cohen-Macaulay sheaf.\;In precise,\;for any $M\in \cO_{\chi_{\lambda_{\bh}}}$,\;the sheaf $\cB(M)$ gives rise to a $G$-equivariant sheaf $\cB(M)^{\widehat{}}_{\widehat{y}}$ on $ \widehat{X}_{L,\widehat{y}}$,\;the pullback of $\cB(M)^{\widehat{}}_{\widehat{y}}$ along $W$ is a $G$-equivariant sheaf on ${{\FX}}^{\Box}_{\infty,y}$,\;and hence descends to a coherent sheaf $\cB_y(M)\in \mathrm{Coh}({{\FX}}_{\infty,y})$.\;In particular,\;we have 
\begin{lem}\label{lemBezrukavnikovvalue}For all $w\in \sW_{n,L}$,\;we have isomorphisms $\cB({M}(w\cdot{\bm\lambda}_\bh)^{\vee})\cong \widehat{\cO}^{\flat}_{\cX_{L,w},\widehat{y}}$ and $\cB(P(\underline{w}_0\cdot{\bm\lambda}_\bh))\cong \widehat{\cO}^{\flat}_{\cX_{L},\widehat{y}}$.\;
\end{lem}

\begin{rmk}
This functor is not exact in our case since $W$ is no longer flat.\;Note that
$W^{\ast}\cB(M)^{\widehat{}}_{x_{\pdr}}=W^{\flat,\ast}(\iota^{\flat})^{\ast}\cB(M)^{\widehat{}}_{\widehat{y}}$.\;Since the local model map $W^{\flat}$ is formally smooth,\;we just need to treat the $G$-equivariant sheaf $(\iota^{\flat})^{\ast}\cB(M)^{\widehat{}}_{\widehat{y}}$ on $\widehat{\cX}^{\flat}_{L,\widehat{y}}$.\;
\end{rmk}

\begin{rmk}\label{stackynotforfixsmooth}
	In general,\;the right functor $R\cB_y:\cO_{\chi_{\lambda_{\bh}}}\rightarrow D^b_{\mathrm{Coh}}({{\FX}}_{\infty,y})$ should
	given by \[R\FB_{y}({M})=RW^{!}\cB(M)^{\widehat{}}_{\widehat{y}}=W^{\flat,\ast}R(\iota^{\flat})^{!}\cB(M)^{\widehat{}}_{\widehat{y}},\;\]where $D^b_{\mathrm{Coh}}({{\FX}}_{\infty,y})$ is the (bounded) derived category of coherent sheaves on ${{\FX}}_{\infty,y}$,\;and $(-)^{!}$ is the upper shriek functor (see \cite[Tag 0A9Y]{Stack}).\;
\end{rmk}

For $M\in \cO^{\infty}_{\mathrm{alg}}$,\;we also consider the following functor (i.e.,\;we fix the smooth part ${\underline{\epsilon}_{\mathrm{sm}}}$ of $\underline{\epsilon}$)
\begin{equation}
	\cM_{\infty,y,\underline{\epsilon}_{\mathrm{sm}}}(M):=\homo_{U(\fg)}\Big(M,(\Pi_\infty^{R_\infty-\ana})^{U_0}\Big)[\fm_{r_y}^\infty][\fm_{\underline{\epsilon}_{\mathrm{sm}}}]^\vee,\;
\end{equation}
where $U_0$ is a compact open subgroup of $\bN(L)$.\;For $M\in \widetilde{\cO}^{\infty}_{\mathrm{alg}}$,\;we define
\begin{equation}
	\cM_{\infty,y,\underline{\epsilon}_{\mathrm{sm}}}(M):=\varprojlim_n\cM_{\infty,y}(M/\fm_I^n).\;
\end{equation}
In this case,\;the functor $\cM_{\infty,y,\underline{\epsilon}_{\mathrm{sm}}}(-)$ is only left exact on $\widetilde{\cO}^{\infty}_{\mathrm{alg}}$.\;The support of the coherent sheaf $\cM_{\infty,y,\underline{\epsilon}_{\mathrm{sm}}}(-)$ is contained in $\overline{\overline{\FX}}_{\infty,y}$.\; 
%We next study the module structure of $\cM_{\infty,y,\underline{\epsilon}_{\mathrm{sm}}}({M}(w\cdot {\bm\lambda}_\bh))$ over its support.\;
\begin{rmk}\label{stackyforfixsmooth}
Let $g:\Spec \widehat{\cO}^{\flat}_{\overline{X}_L,y}\cong \Spec{\widehat{\cO}}_{\overline{X}^0_{L},{y}}\rightarrow \Spec \widehat{\cO}_{\overline{X}_L,y}$ be the natural closed embedding given in Proposition \ref{localmodelforover2}.\;Define $R\overline{\cB}_y:\cO_{\chi_{\lambda_{\bh}}}\rightarrow D^b_{\mathrm{Coh}}({{\FX}}_{\infty,y})$  by $R\overline{\cB}_{y}({M})=R(g\circ W)^{!}\cB(M)^{\widehat{}}_{y}$.\;In spirit of the categorical approach $p$-adic Langlands correspondence,\;the functor $R\cM_{\infty,y,\underline{\epsilon}_{\mathrm{sm}}}(-)$ given by  $R\cM_{\infty,y,\underline{\epsilon}_{\mathrm{sm}}}(M):=R\homo_{U(\fg)}\Big(M,(\Pi_\infty^{R_\infty-\ana})^{U_0}\Big)[\fm_{r_y}^\infty][\fm_{\underline{\epsilon}_{\mathrm{sm}}}]^\vee$ should closely related to the ``local" functor $R\overline{\cB}_y(-)$.\;
\end{rmk}

\subsubsection{Proofs of the existences of companion points (constituents) for $L\neq \bQ_p$}\label{proofsL}

In this section,\;we will show that $\cM_{\infty,y}(\widetilde{M}(w\cdot {\bm\lambda}_\bh))\neq 0$,\;equivalently,\;$\cM_{\infty,y}(\widetilde{M}^{\Sigma_L\backslash J}_{{I}}(w\cdot {\bm\lambda}_\bh))\neq 0$  for any $J\subseteq \Sigma_L$ (in particular,\;$\cM_{\infty,y}({M}(w\cdot {\bm\lambda}_\bh))\neq 0$) if and only if $w\underline{w}_0\geq w_{\cF}$.\;

The key step is the following proposition,\;which is an analogue of \cite[Proposition 4.7]{wu2021local} (but in our setting).\;
\begin{pro}\label{prolocalcomapn}Assume that $x_{w\underline{w}_0}\in X_{\mathrm{tri}}(\overline{r})$ for any  $w_{\cF}<w$ and $w_{\cF}\neq \underline{w}_0$.\;Suppose that there exists a point $z\in \FX_{\overline{\rho}^{\fp}}^\Box\times \BU^g$ such that $(\iota_{\fp}(x_{w\underline{w}_0}),z)\in \iota_{\fp}\big(X_{\mathrm{tri}}(\overline{r})\big)\times
	\FX_{\overline{\rho}^{\fp}}^\Box\times \BU^g$ are in $X_{\fp}(\overline{\rho})(E)$.\;Then $(\iota_{\fp}(x_{w_{\cF}\underline{w}_0}),z)\in X_{\fp}(\overline{\rho})(E)$.\;
\end{pro}
\begin{proof}
We adapt the proof of \cite[Proposition 4.7]{wu2021local} to our case.\;For $w$ such that $w_{\cF}\leq w$,\;we write $z_{w\underline{w}_0}:=(\iota_{\fp}(x_{w\underline{w}_0}),z)$.\;By \cite[Lemma 2.26]{wu2021local},\;there exists a simple root $\alpha_\tau$ (for one $\tau\in\Sigma_L$) of $\GLN_{n,L}$ and a standard parabolic subgroup $\bP_{L}$ of $\GLN_{n,L}$  containing $\bB_{L}$ such that $w\underline{w}_0(\bh)$ is strictly $\bP_{L}$-dominant (more precisely,\;strictly $\bP_{\tau}$-dominant) and $w_{\cF}\underline{w}_0(\bh)$ is not strictly $\bP_{L}$-dominant (not strictly $\bP_{\tau}$-dominant),\;where $w:=s_{\alpha_\tau}w_{\cF}$ and $\lg(w)=\lg(w_{\cF})+1$.\;By assumption,\;we have 
$\cM_{\infty,y}(M(w\underline{w}_0\cdot \lambda_{\bh}))\neq 0$ (and equivalently,\;$\cM_{\infty,y}(\widetilde{M}^{\Sigma_L\backslash\tau}(w\underline{w}_0\cdot \lambda_{\bh}))\neq 0$).\;Similar to the proof of \cite[Proposition 4.7]{wu2021local},\;we get an exact sequence:
\begin{equation}
	\begin{aligned}
		0\rightarrow\cM_{\infty,y}(\widetilde{M}^{\Sigma_L\backslash\tau}(w_{\Sigma_L\backslash\tau}\underline{w}_{0,\Sigma_L\backslash\tau}&\cdot \lambda_{\bh})\otimes_EL(w_{\cF,\tau}{w}_0\cdot \lambda_{\bh}))
		\rightarrow\cM_{\infty,y}(\widetilde{M}^{\Sigma_L\backslash\tau}(w\underline{w}_0\cdot \lambda_{\bh}))\\ \rightarrow&\cM_{\infty,y}(\widetilde{M}^{\Sigma_L\backslash\tau}(w_{\Sigma_L\backslash\tau}\underline{w}_{0,\Sigma_L\backslash\tau}\cdot \lambda_{\bh})\otimes_EL(w_\tau{w}_0\cdot \lambda_{\bh}))\rightarrow 0.
	\end{aligned}
\end{equation}	
Therefore,\;to show that $\cM_{\infty,y}(L(w_{\cF}\underline{w}_0\cdot \lambda_{\bh}))\neq 0$,\;it suffices to show that the natural surjection \[\cM_{\infty,y}(\widetilde{M}^{\Sigma_L\backslash\tau}(w\underline{w}_0\cdot \lambda_{\bh})) \twoheadrightarrow\cM_{\infty,y}(\widetilde{M}^{\Sigma_L\backslash\tau}(w_{\Sigma_L\backslash\tau}\underline{w}_{\Sigma_L\backslash\tau}\cdot \lambda_{\bh})\otimes_EL(w_\tau{w}_0\cdot \lambda_{\bh}))\]
(equivalently,\;$\cM_{\infty,y}(L(w\underline{w}_0\cdot \lambda_{\bh}))\twoheadrightarrow
\cM_{\infty,y}(M(w\underline{w}_0\cdot \lambda_{\bh}))$) is not an isomorphism.\;We prove it by contradiction.\;Assume that the above surjection is an isomorphism.\;Then this isomorphism shows that (the same as in the proof of \cite[Proposition 4.7]{wu2021local}) 
\begin{equation}\label{twospeccontain}
{\overline{\FX}}_{\infty,r_y}^{\{\tau\},w\underline{w}_0}=\Supp\cM_{\infty,y}(\widetilde{M}^{\Sigma_L\backslash\tau}(w_{\Sigma_L\backslash\tau}\underline{w}_{0,\Sigma_L\backslash\tau}\cdot \lambda_{\bh})\otimes_EL(w_\tau{w}_0\cdot \lambda_{\bh}))\subseteq {\overline{\FX}}_{\infty,r_y}^{I_{\tau}}\;(\text{viewed as schemes}),
\end{equation}
where $\bP_\tau=\bP_{I_\tau}$ for some $I_\tau\subseteq \Delta_n$.\;We show that this is impossible.\;The underlying topological space of ${\overline{\FX}}_{\infty,y}^{\{\tau\},w\underline{w}_0}$ is equal to  the union of non-empty cycles denoted by $\FZ^{\{\tau\},w\underline{w}_0}_{\infty,y}$ and $\FZ^{\{\tau\},w_{\cF}\underline{w}_0}_{\infty,y}$.\;But by Lemma \ref{strictlyiff},\;$\FZ^{\{\tau\},w_{\cF}\underline{w}_0}_{\infty,y}$ is not contained in ${\overline{\FX}}_{\infty,r_y}^{I_{\tau}}$ since  $w_{\cF}\underline{w}_0(\bh)$ is not strictly $\bP_{L}$-dominant (not strictly $\bP_{\tau}$-dominant),\;which lead a contradiction to (\ref{twospeccontain}).\;
\end{proof}

%The proof of this assertion is not changed until we need to use the .\;We need to prove a analogue of \cite[Theorem 4.4]{wu2021local},\;but in the our setting.\;By checking the proof of \cite[Theorem 4.4]{wu2021local} carefully,\;the main modification is the analogue of Proposition 5.13 and Theorem 5.15 (also in the our setting) in \cite{wu2021local}.\;This is the Proposition \ref{PartiallydeRham} below.\;

We are ready to prove the main theorem on global companion points.\;From now on,\;we set $\beta:=\alpha q_L^{\frac{n-1}{2}}$.\;

\begin{pro}\label{equicompanpointconsti}Assume $x_{w_{\cF}\underline{w}_0}\in X_{\mathrm{tri}}(\overline{r})$.\;If $y_{w\underline{w}_0}\in X_{\fp}(\overline{\rho})(E)$ for some $w\leq w_{\cF}\underline{w}_0$,\;then  $$\homo_{G}\Big(\Pi_{(w\underline{w}_0,1)}(\beta),\Pi_\infty^{R_\infty-\ana}[\fm_{r_y}^\infty]\Big)\neq 0.$$
\end{pro}
By \cite[(4.1)]{wu2021local},\;we see that
\begin{equation}
	\begin{aligned}
		\homo_{G}(\Pi_{(w\underline{w}_0,1)}(\beta),\Pi_\infty^{R_\infty-\ana}[\fm_{r_y}^\infty])\neq 0\Leftrightarrow	\cM_{\infty,y}(L(w\cdot \lambda_{\bh}))\neq 0\Leftrightarrow y_{w\underline{w}_0}\in Y_{\fp}(\overline{\rho})_{w\underline{w}_0\cdot \lambda_{\bh}}.\;
	\end{aligned}
\end{equation}
% \Leftrightarrow \big(Y_{\fp}(\overline{\rho})_{w\cdot \lambda_{\bh}}\big)^{\mathrm{red}}$ and $[\cL({w\cdot \lambda_{\bh}})]\neq 0\Leftrightarrow y_{w\underline{w}_0}\in Z_{\fp}(\overline{\rho})_{w\cdot \lambda_{\bh}} \Leftrightarrow Z_{\fp}(\overline{\rho})_{w\cdot \lambda_{\bh}}^{\FU^{\fp}}:=Z_{\fp}(\overline{\rho})_{w\cdot \lambda_{\bh}}\cap X_{\fp}(\overline{\rho})^{\FU^{\fp}}$.\;
Therefore Proposition \ref{equicompanpointconsti} is equivalent to
\begin{pro}\label{equicompanpointconsti123}Assume  $x_{w_{\cF}\underline{w}_0}\in X_{\mathrm{tri}}(\overline{r})$.\;If $y_{w\underline{w}_0}\in X_{\fp}(\overline{\rho})(E)$,\;then $\cM_{\infty,y}(L(w\cdot \lambda_{\bh}))\neq 0$.\;
\end{pro}
\begin{rmk}
	Note that $\cM_{\infty,y}(L(w\cdot \lambda_{\bh}))\neq 0$ implies $y_{w\underline{w}_0}\in  X_{\fp}(\overline{\rho})$ and $x_{w\underline{w}_0}\in X_{\mathrm{tri}}(\overline{r})$.\;Therefore this proposition is stronger than predicting the set of companion points.\;But in our Steinberg case,\;$\cM_{\infty,y}(L(w\cdot \lambda_{\bh}))\neq 0$ is still slightly weaker than the locally analytic scole conejecture (see Theorem \ref{detlamedaSOCLE},\;which is not far from Proposition \ref{equicompanpointconsti}).\;
\end{rmk}
\begin{proof}[Proof of Proposition \ref{equicompanpointconsti123}]
We follow the route of the proof in \cite[Proposition 4.9]{wu2021local}.\;This proposition holds clearly in the case when $w=w_{\cF}$.\;Suppose that $y_{w\underline{w}_0}\in \FX^{\fp}_{w\underline{w}_0}\times \FX_{\mathrm{tri},\cP_{\min}}^{\bh-\mathrm{st}}(\overline{r})\times \BU^g$ for some irreducible component $\FX^{\fp}_{w\underline{w}_0}\subset \FX_{\overline{\rho}^{\fp}}^\Box$.\;Recall that we have a closed immersion:
\[\iota_{\bh,w\underline{w}_0}:\overline{\widetilde{\FX}_{\mathrm{tri},\cP_{\min},w}^{\bh-\mathrm{st}}(\overline{r})}\hookrightarrow X_{\mathrm{tri}}(\overline{r}).\]
Then the point $y_{w\underline{w}_0}$ is in the image of $\iota_{\bh,w\underline{w}_0}$ since $w\geq w_{\cF}$.\;We can take an affinoid neighbourhood $U$ of $y_{w\underline{w}_0}$ in $X$.\;Let $V$ be a sufficiently mall open affinoid $V\subset \iota_{\bh,w\underline{w}_0}^{-1}(U)$ such that $y_{w\underline{w}_0}\in \iota_{\bh,w\underline{w}_0}(V)$.\;Note that $V\cap \widetilde{\FX}_{\mathrm{tri},\cP_{\min},w\underline{w}_0}^{\bh-\mathrm{st}}(\overline{r})$ is Zariski open dense in $V$.\;Since the any point $z$ in $\big(\mathrm{id}\times(\iota_{\fp}\circ\iota_{\bh,w\underline{w}_0})\times\mathrm{id}\big)(\FU^{\fp}\times\widetilde{\FX}_{\mathrm{tri},\cP_{\min},w}^{\bh-\mathrm{st}}(\overline{r})\times \BU^g)$ satisfies that $w_z=w_{\cF}$,\;we deduce that 
\begin{equation}
	\begin{aligned}
		\FU^{\fp}\times\widetilde{\FX}_{\mathrm{tri},\cP_{\min},w}^{\bh-\mathrm{st}}(\overline{r})\times \BU^g\subset \big(\mathrm{id}\times(\iota_{\fp}\circ\iota_{\bh,w\underline{w}_0})\times\mathrm{id}\big)^{-1}\big(Y_{\fp}(\overline{\rho})_{w\underline{w}_0\cdot \lambda_{\bh}}\big) .\;
	\end{aligned}
\end{equation}
for any $l\leq \lg(w)$.\;Therefore,\;we deduce:
\[\FU^{\fp}\times \overline{\widetilde{\FX}_{\mathrm{tri},\cP_{\min},w}^{\bh-\mathrm{st}}(\overline{r})}\times \BU^g \subset \big(\mathrm{id}\times(\iota_{\fp}\circ\iota_{\bh,w\underline{w}_0})\times\mathrm{id}\big)^{-1}\big( Y_{\fp}(\overline{\rho})_{w\underline{w}_0\cdot \lambda_{\bh}}\big).\;\]
This shows that the companion point $y_{w\underline{w}_0}$ is in $ Y_{\fp}(\overline{\rho})_{w\underline{w}_0\cdot \lambda_{\bh}}$.\;
\end{proof}

\begin{thm}\label{globalcompanion}
Assume  Hypothesis \ref{TWhypo},\;Hypothesis \ref{hyongaloisrep} and {Hypothesis} \ref{appenhypothesis} and $x_{w_{\cF}\underline{w}_0}\in X_{\mathrm{tri}}(\overline{r})$.\;Then $y_{w\underline{w}_0}\in X_{\fp}(\overline{\rho})(E)$ if and only if $w\geq w_{\cF}$.\;
\end{thm}
\begin{proof}
The "only if" part follows from Proposition \ref{prolocalcomapn1}.\;We prove "if" part by descending induction on the integer integer $l\leq \lg(\underline{w}_0)$ for the following hypothesis $\cH_l$:\;if $y_{w_0}\in X_{\fp}(\overline{\rho})(E)$ is a Steinberg point which is $w_{\cF}$-critical,\;then for any $w_{\cF}\leq w$ and $l\leq \lg(w)$,\;$y_{w\underline{w}_0}\in X_{\fp}(\overline{\rho})(E)$.\;For $l=\lg(\underline{w}_0)$,\;there is nothing to prove.\;It suffices to prove $\cH_{l-1}$ when $\cH_l$ holds.\;

If  $\lg(w_{\cF}\underline{w}_0)\geq l-1$,\;then the Hypothesis $\cH_l$ and  Proposition \ref{prolocalcomapn} imply $y_{w_{\cF}\underline{w}_0}\in  X_{\fp}(\overline{\rho})(E)$,\;and hence $y_{w_{\cF}\underline{w}_0}\in   Y_{\fp}(\overline{\rho})_{w_{\cF}\underline{w}_0\cdot \lambda_{\bh}}$. Now we assume that $\lg(w_{\cF}\underline{w}_0)< l-1$.\;We need to prove that for any $w$ such that $w\leq w_{\cF}\underline{w}_0$ and $\lg(w\underline{w}_0)=l-1$,\;we have $y_{w\underline{w}_0}\in X_{\fp}(\overline{\rho})(E)$.\;The point $y=y_1=y_{\underline{w}_0\underline{w}_0}$ is in the image of $\iota_{\bh,1}(\overline{\widetilde{\FX}_{\mathrm{tri},\cP_{\min},w}^{\bh-\mathrm{st}}(\overline{r})})$ since $w\geq w_{\cF}$.\;We can take an affinoid neighbourhood $U$ of $y$ in $X$.\;Then $V:=\iota_{\bh,1}^{-1}(U)\cap \widetilde{\FX}_{\mathrm{tri},\cP_{\min},w}^{\bh-\mathrm{st}}(\overline{r})$,\;which is Zariski open dense in the affinoid $\overline{V}:=\iota_{\bh,1}^{-1}(U)\cap \overline{\widetilde{\FX}_{\mathrm{tri},\cP_{\min},w}^{\bh-\mathrm{st}}(\overline{r})}$.\;Since any point $z$ in $\big(\mathrm{id}\times(\iota_{\fp}\circ\iota_{\bh,1})\times\mathrm{id}\big)(\FU^{\fp}\times\widetilde{\FX}_{\mathrm{tri},\cP_{\min},w}^{\bh-\mathrm{st}}(\overline{r})\times \BU^g)$ satisfies the condition in $\cH_{l-1}$,\;and $w_z=w_{\cF}$,\;$\lg(w_z)=l-1$.\;Hence their companion points are contained in $X_{\fp}(\overline{\rho})$.\;We hence deduce that 
\begin{equation}
	\begin{aligned}
		\FU^{\fp}\times\widetilde{\FX}_{\mathrm{tri},\cP_{\min},w}^{\bh-\mathrm{st}}(\overline{r})\times \BU^g\subset \big(\mathrm{id}\times(\iota_{\fp}\circ\iota_{\bh,w})\times\mathrm{id}\big)^{-1}\big(X_{\fp}(\overline{\rho})\big) .\;
	\end{aligned}
\end{equation}
for any $l\leq \lg(w)$.\;Therefore,\;we deduce:
\[\FU^{\fp}\times \overline{\widetilde{\FX}_{\mathrm{tri},\cP_{\min},w}^{\bh-\mathrm{st}}(\overline{r})}\times \BU^g \subset \big(\mathrm{id}\times(\iota_{\fp}\circ\iota_{\bh,w\underline{w}_0})\times\mathrm{id}\big)^{-1}\big( X_{\fp}(\overline{\rho})\big).\;\]
This show that the companion point $y_{w\underline{w}_0}$ is in $ X_{\fp}(\overline{\rho})$.\;This completes the proof of $\cH_{l-1}$.\;
\end{proof}

\begin{pro}\label{determinemulti}
Assume Conjecture \ref{flatnessXconj},\;then $m_{y,w}=m_y$ for any $w\in\sW_{n,L}$.\;Furhermore,\;we have $[\cM_{\infty,y}({M}(w\cdot {\bm\lambda}_\bh)^{\vee})]=m_y[\FB_{y}({M}(w\cdot {\bm\lambda}_\bh))^{\vee}]=m_y[\overline{{\FX}}_{\infty,y}^{w{w}_0}]$ for all $w\in\sW_n$.\;
\end{pro}
\begin{proof}
The proof is just a copy of the Step $10$ in the proof of \cite[Theorem 5.3.3]{breuil2019local}.\;We prove this Proposition by descending induction on the integer integer $l\leq \lg(\underline{w}_0)$ for the following hypothesis $\cH_l'$:\;for any $l\leq \lg(w)$,\;we have $m_{y,w}=m_y$.\;For $l=\lg(\underline{w}_0)$,\;there is nothing to prove.\;It suffices to prove $\cH'_{l-1}$ when $\cH'_l$ holds.\;Choose $w_1,w_2,w_3$ as in \cite[Lemma 5.2.7]{breuil2019local} applied to $w$.\;Then we have the equality:
\[[\cM_{\infty,y}(M(w_i\underline{w}_0\cdot \lambda_{\bh}))]=m_y[{\overline{\FX}}_{\infty,y}^{w_i\underline{w}_0}]\in Z^0({\overline{\FX}}_{\infty,y}),\;i\in\{1,2,3\}.\]
In particular,\;we have the equalities in $Z^0({\overline{\FX}}_{\infty,y})$:
\begin{equation}\label{multiequality1}
	\begin{aligned}
		&[{\overline{\FX}}_{\infty,y}^{w_i\underline{w}_0}]=[{{\FZ}}_{\infty,y}^{w_i}]+[{{\FZ}}_{\infty,y}^{w}],\;i\in\{1,2\}\\
		&[{\overline{\FX}}_{\infty,y}^{w_3\underline{w}_0}]=[{{\FZ}}_{\infty,y}^{w_3}]+[{{\FZ}}_{\infty,y}^{w_1}]+[{{\FZ}}_{\infty,y}^{w_1}]+[{{\FZ}}_{\infty,y}^{w}].\;
	\end{aligned}
\end{equation}
On the other hand,\;we have the following equalities in $Z^0({\overline{\FX}}_{\infty,y})$:
\begin{equation}\label{multiequality2}
	\begin{aligned}
 [\cM_{\infty,y}(M(w_i\underline{w}_0\cdot \lambda_{\bh}))]=&[\cM_{\infty,y}(L(w_i\underline{w}_0\cdot \lambda_{\bh}))]+[\cM_{\infty,y}(L(w\underline{w}_0\cdot \lambda_{\bh}))],\;i\in\{1,2\}\\
 [\cM_{\infty,y}(M(w_3\underline{w}_0\cdot \lambda_{\bh}))]=&[\cM_{\infty,y}(L(w_3\underline{w}_0\cdot \lambda_{\bh}))]+[\cM_{\infty,y}(L(w_1\underline{w}_0\cdot \lambda_{\bh}))]\\&+[\cM_{\infty,y}(L(w_2\underline{w}_0\cdot \lambda_{\bh}))]+[\cM_{\infty,y}(L(w\underline{w}_0\cdot \lambda_{\bh}))]
	\end{aligned}
\end{equation}		
By the argument below \cite[(5.37)]{breuil2019local} and Corollary \ref{corforcoherent},\;we have $[\cM_{\infty,y}(L(w_i\underline{w}_0\cdot \lambda_{\bh}))]=m_y[{{\FZ}}_{\infty,y}^{w_i}]$ for $i\in\{1,2,3\}$.\;Compare the equations in (\ref{multiequality2}) and (\ref{multiequality2}),\;we see that $[\cM_{\infty,y}(L(w\underline{w}_0\cdot \lambda_{\bh}))]=m_y[{{\FZ}}_{\infty,y}^{w}]$ and thus $m_{y,w}=m_y$ (using  Corollary \ref{corforcoherent}).\;This completes the proof of the first assertion.\;By Lemma \ref{lemBezrukavnikovvalue},\;we have $\FB_{y}({M}(w\cdot {\bm\lambda}_\bh)^{\vee})=W^{\flat,\ast}\widehat{\cO}_{{\overline{\cX}_{w}},x_{\pdr}}$.\;This implies $[\FB_{y}({M}(w\cdot {\bm\lambda}_\bh))]=[\overline{{\FX}}_{\infty,y}^{w{w}_0}]$ by definition.\;The last assertion then follows from Corollary \ref{corforcoherent}.\;
\end{proof}

%We end this section by explaining how to relax the assumption $x_{w_{\cF}\underline{w}_0}\in X_{\mathrm{tri}}(\overline{r})$.\;

\subsection{Locally analytic socle conjecture}

For any $w\in\sW_{n,L}$,\;recall that the locally $\bQ_p$-analytic irreducible admissible representation $C(w\underline{w}_0,1)\cong \cF^G_{\op_{w\underline{w}_0}(L)}\Big(\overline{L}(-w\underline{w}_0\cdot{\lambda}_{\bh}),\st^{\infty}_{\bL_{w\underline{w}_0}(L)}\Big)$ is the unique quotient of $\Pi_{(w\underline{w}_0,1)}$.\;
\begin{thm}\label{detlamedaSOCLE}Assume Hypothesis \ref{TWhypo} ,\;Hypothesis \ref{hyongaloisrep} and {Hypothesis} \ref{appenhypothesis} and $x_{w_{\cF}\underline{w}_0}\in X_{\mathrm{tri}}(\overline{r})$.\;Then $C(w\underline{w}_0,1,\beta)$ is a subrepresentation of $\widehat{S}_{\xi,\tau}(U^{\fp},E)^{\ana}_{\overline{\rho}}[\fm_{r_y}]$ if and only if $w\geq w_{\cF}$.\;
\end{thm}
\begin{proof}It suffices to show that $\Pi_{(w\underline{w}_0,1,J)}(\beta)$ for any $\emptyset\neq J \subseteq I(w\underline{w}_0)$ cannot be embedded into the space $\widehat{S}(U^{\fp},W^{\fp})_{\overline{\rho}}^{\ana}[\fm_{r_y}]$. Indeed,\;there exists an $w'\neq 1$ such that $\Pi_{(w\underline{w}_0,1,J)}(\beta)$ becomes the unique quotient of ${\Pi}_{(w\underline{w}_0,w')}(\beta)$.\;If $\Pi_{(w\underline{w}_0,1)}(W)\hookrightarrow \widehat{S}(U^{\fp},W^{\fp})_{\overline{\rho}}^{\ana}[\fm_{r_y}]$.\;Then we see that
	\[\homo_G\Big({\Pi}_{(w\underline{w}_0,w')}(\beta),\widehat{S}_{\xi,\tau}(U^{\fp},E)^{\ana}_{\overline{\rho}}[\fm_{r_y}]\Big)\neq 0.\]
This gives a companion point $(\rho, z^{w\underline{w}_0(\bh)}\eta)\in X_{\fp}(\overline{\rho})$ of $y$ with $\eta$ a $w'$-twist of $\unr(\alpha)_n$,\;thus this point is not equal to $y_{w\underline{w}_0}$ for any $w\geq w_{\cF}$,\;a contradiction.\;
\end{proof}
%\begin{rmk}
%The coherent sheaf $\cM_{\infty}$ on $X_{\fp}(\overline{\rho})$ is free of finite rank in the Zariski-open dense irreducible smooth locus of a small enough affinoid neighbourhood of $y$ in $X_{\fp}(\overline{\rho})^{\FU^{\fp}}$.\;Denoted by $m_y\leq 1$ this rank of $\cM_{\infty}$.\;Recall in the the proof of \cite[Theorem 5.3.3]{breuil2019local},\;the author consider the following induction hypothesis $\cH_l$ for integer $l\leq \lg(\underline{w}_0)$:\;for $y\in  X_{\fp}(\overline{\rho})^{\FU^{\fp}}$ as above with $l\leq \lg(w_{\cF})$,\;then $[\cL(w\cdot \lambda_{\bh})]\neq 0$ for all $w\geq w_{\cF}$,\;and the rank of $\cM_{\infty}$ in the smooth locus of a small enough affinoid neighbourhood of $y_{w\underline{w}_0}$ in  $X_{\fp}(\overline{\rho})^{\FU^{\fp}}$ is still $\fm_y$.\;We may first explore if $\cH_l$ holds for $l\geq \lg(\underline{w}_0)-1$.\;But in this case,\;the point $y_{w\underline{w}_0}$ is not necessary smooth on $X_{\fp}(\overline{\rho})$ (this is true for generic crystalline case in \cite[P. 400,\;Step 7]{breuil2019local}).\;The original proof cannot be applied to our case,\;so we cannot prove the induction basis by the same strategy.\;
%\end{rmk}

\section{Appendix:\;another approach to socle conjecture for $\GLN_2(L)$-case}\label{appGL2(L)}

We combine the methods in \cite{CompanionpointforGLN2L} and \cite{2015Ding}.\;We recall and keep the notation in \cite{CompanionpointforGLN2L}.\;

For $\underline{h}=(h_{\tau,1},h_{\tau,2})_{\tau\in \Sigma_L}$,\;let $J\subseteq \Sigma_L$ and $\underline{h}_J=(h_{\tau,1},h_{\tau,2})_{\tau\in J}$.\;Let $\widehat{T}_L(\underline{h}_J)$ be the reduced closed subspace of $\widehat{T}_L$  such that $\widehat{T}_L(\underline{h}_J)(\overline{E})=\{\delta=\delta_1\otimes\delta_2|\wt_\tau(\delta_i)=h_{i,\tau},\tau\in J\}$.\;Put $X_{\mathrm{tri}}(\overline{r},\underline{h}_J)=X_{\mathrm{tri}}(\overline{r})\times_{\widehat{T}_L}\widehat{T}_L(\underline{h}_J)$.\;Let $X_{\mathrm{tri},J-\dR}(\overline{r},\underline{h}_J)$ be the reduced closed subspace of $X_{\mathrm{tri}}(\overline{r},\underline{h}_J)$ with the $A$-point are $(r_A,\delta_A)$ such the $r_A$ is $J$-de Rham of Hodge-Tate weights $\underline{h}_J$.\;For $J'\subset J$,\;we have natural morphism $	\widehat{T}_L(\underline{h}_J)\rightarrow \widehat{T}_L(\underline{h}_{J'})$.\;Put $X_{\mathrm{tri},J'-\dR}(\overline{r},\underline{h}_J):=X_{\mathrm{tri},J'-\dR}(\overline{r},\underline{h}_{J'})\times_{\widehat{T}_L(\underline{h}_{J'})}\widehat{T}_L(\underline{h}_J)$.\;

by \cite[(7)]{CompanionpointforGLN2L},\;such spaces fall into the following commutative diagram:
\begin{equation}
	\xymatrix@C=2ex{ X_{\mathrm{tri},J-\dR}(\overline{r},\underline{h}_J) \ar[d] \ar[r]&  X_{\mathrm{tri},J'-\dR}(\overline{r},\underline{h}_J) \ar[r] \ar[d]&
		X_{\mathrm{tri},J'-\dR}(\overline{r},\underline{h}_{J'})\ar[r]\ar[d] & X_{\mathrm{tri}}(\overline{r},\underline{h}_{J'}) \ar[d] \ar[r]  & X_{\mathrm{tri}}(\overline{r}) \ar[d]\\
		\widehat{T}_L(\underline{h}_J)\ar[r]&	\widehat{T}_L(\underline{h}_J) \ar[r] &\widehat{T}_L(\underline{h}_{J'}) \ar[r] & \widehat{T}_L(\underline{h}_{J'}) \ar[r] & \widehat{T}_L,}
\end{equation}
where the horizontal maps are closed embedding,\;and the second and fourth square are cartesian.\;For a closed subspace $X\subset X_{\mathrm{tri}}(\overline{r})$,\;put $X(\underline{h}_J):=X\times_{X_{\mathrm{tri}}(\overline{r})}X_{\mathrm{tri}}(\overline{r},\underline{h}_{J})$,\; $X_{J-\dR}(\underline{h}_J):=X\times_{X_{\mathrm{tri}}(\overline{r})}X_{\mathrm{tri},J-\dR}(\overline{r},\underline{h}_{J})$ and $X_{J'-\dR}(\underline{h}_J):=X\times_{X_{\mathrm{tri}}(\overline{r})}X_{\mathrm{tri},J'-\dR}(\overline{r},\underline{h}_{J})$.\;

Keep the assumption and notation in Section \ref{companionpointandconsti}.\;Let $\rho:\gal_F\rightarrow \GLN_2(E)$ be a continuous representation such that $\rho\otimes\epsilon\cong \rho^{\vee}\circ c$ and $\rho$ is unramified outside $S$.\;Firstly,\;we assume  that:

\begin{hypothesis}\label{hyongaloisrepgl2l} Assume \begin{itemize}
		\item[(1)] $\rho$ comes from a classical point $y\in Y(U^{\fp},\overline{\rho})$ of the form $y=(\rho,\chi)$ (thus $\widehat{S}_{\xi,\tau}(U^{\fp},E)_{\overline{\rho}}^{\lalg}[\fm_{\rho}]\neq 0$); 
		\item[(2)] $\rho_{L}:=\rho|_{\gal_{F_{\widetilde{v}}}}$ is semistable non-crystalline of Hodge-Tate weights ${\bh}$,\;and $\{\alpha,q_L\alpha\}$ the eigenvalues of $\varphi^{f_L}$ on $D_{\mathrm{st}}(\rho_{L})$;
		%	\item[(3)] $D_{\mathrm{rig}}(\rho_{L})$ admits a triangulation of parameters $(\delta_1,\delta_2)$,\;where $\eta_1:=\unr(\alpha)z^{\bh_{1,\Sigma_L\backslash\Sigma(\rho_{L})}}z^{\bh_{2,\Sigma(\rho_{L})}}$ and $\delta_2:=\unr(\alpha q_L)z^{\bh_{1,\Sigma(x)}}z^{\bh_{2,\Sigma_L\backslash\Sigma}}$ for some subset $\Sigma(\rho_{L})\subset \Sigma_L$.\;
	\end{itemize}
\end{hypothesis}

By assumptions, the classical point $y=(\rho,\chi)\in X_{\fp}(\overline{{\rho}})$ satisfies $\chi=\chi_1\otimes\chi_2$ with $\chi_1=\unr(\alpha q_L^{-1})z^{\bh_{1}}$ and $\chi_2=\unr(\alpha q_L)z^{\bh_{2}+1}$ for some $\alpha\in E^{\times}$.\;

Suppose that $x=(\rho_{L},\delta=\delta_1\otimes\delta_2)$ is a closed point in
$X_{\mathrm{tri}}(\overline{r})$.\;Let $\Sigma^+(\delta)=\{\tau\in\Sigma_L:\wt_\tau(\delta_1)>\wt_\tau(\delta_2)\}$ and $\Sigma^-(\delta):=\Sigma_L\backslash\Sigma^+(\delta)$.\;Then we have
\begin{equation}
	\left\{
	\begin{array}{ll}
		\delta_1:=\unr(\alpha q_L^{-1})z^{\bh_{1,\Sigma^+(\delta)}}z^{\bh_{2,\Sigma^-(\delta)}}\\
		\delta_2:=\unr(\alpha q_L)z^{\bh_{1,\Sigma^-(\delta)}}z^{\bh_{2,\Sigma^+(\delta)}}.
	\end{array}
	\right.
\end{equation}
Then by \cite[Theorem 4.15]{2015Ding},\;there exists $\Sigma(x)\subseteq \Sigma^+(\delta)$ such that $\rho_{L}$ admits a trianguline of parameter 
\begin{equation}
	\left\{
	\begin{array}{ll}
		\delta'_1:=\unr(\alpha)z^{\bh_{1,\Sigma^+(\delta)\backslash\Sigma(x)}}z^{\bh_{2,\Sigma^-(\delta)\cup \Sigma(x)}} \\
		\delta'_2:=\unr(\alpha q_L)z^{\bh_{1,\Sigma^-(\delta)\cup\Sigma(x)}}z^{\bh_{2,\Sigma^+(\delta)\backslash\Sigma(x)}}.
	\end{array}
	\right.
\end{equation}
From now on,\;we assume that $\Sigma^-(\delta)\cup \Sigma(x)\neq \emptyset$ (i.e.,\;$\rho_L$ admits a critical special triangulation).\;For $J\subset \Sigma^+(\delta)$,\;we see that $x$ is also a closed point of the spaces $X_{\mathrm{tri},J-\dR}(\overline{r},\underline{\bh}_J)\hookrightarrow X_{\mathrm{tri}}(\overline{r},\underline{\bh}_J)\hookrightarrow X_{\mathrm{tri}}(\overline{r})$ and $X_{\mathrm{tri},J'-\dR}(\overline{r},\underline{\bh}_J)\hookrightarrow X_{\mathrm{tri}}(\overline{r},\underline{\bh}_{J'})$.\;Let $X$ be a union of irreducible components of an open subset of $X_{\mathrm{tri}}(\overline{r})$ such that $X$ satisfies the accumulation property at $x$ (see \cite[Definition 2.11]{breuil2017smoothness}).\;Then we have the following results on tangent space of $X$ at $x$.\;
\begin{thm}Keep the above situation.\;Let $J'\subset J$ and $J'\cap \Sigma(x)\neq \Sigma(x)$.\;
\begin{itemize}
	\item[(1)] $\dim_ET_{X,x}=4+3d_L$;
	\item[(2)] $\dim_ET_{X(\underline{\bh}_J),x}=4+3d_L-2|J\cap(\Sigma_L\backslash\Sigma(x))|-|J\cap\Sigma(x)|$;
	\item[(3)] $\dim_ET_{X_{J-\dR}(\underline{\bh}_J),x}=4+3d_L-2|J|$;
	\item[(4)] $\dim_ET_{X_{J'-\dR}(\underline{\bh}_J),x}=4+3d_L-2|J'|-2|(J\backslash J')
	\cap (\Sigma_L\backslash\Sigma(x))|-|(J\backslash J')
	\cap \Sigma(x)|$;
\end{itemize}
\end{thm}
\begin{rmk}The fist two results are analogue of \cite[Theorem 2.2]{CompanionpointforGLN2L},\;and $(3),(4)$ are analogue of \cite[Theorem 2.4]{CompanionpointforGLN2L}.\;
\end{rmk}
\begin{proof}Let $W:=\{(d_{1,\tau},d_{2,\tau})|\;d_{1,\tau}=d_{2,\tau},\tau\in\Sigma(x)\}$ and $W_J:=\{(d_{1,\tau},d_{2,\tau})|\;d_{1,\tau}=d_{2,\tau}=0,\tau\in J\}$.\;As in \cite[(9)]{CompanionpointforGLN2L},\;one has an exact sequence 
\begin{equation}\label{exactTangentspace}
0\rightarrow K(\rho_{L})\cap T_{X,x}\rightarrow T_{X,x}\xrightarrow{f}\ext^1_{\gal_L}(\rho_{L},\rho_{L}).\;
\end{equation}
We need to control the $\mathrm{Im}(f)$.\;For $t\in T_{X,x}:\Spec E[
\epsilon]/\epsilon^2\rightarrow X_{\mathrm{tri}}(\overline{r})$,\;we get the following map.\;The composition $\Spec E[
\epsilon]/\epsilon^2\rightarrow X_{\mathrm{tri}}(\overline{r})\rightarrow \mathfrak{X}_{\overline{r}}^\Box$ gives a continuous representation $\widetilde{\rho_{L}}$ (we view it as an element in $\ext^1_{\gal_L}(\rho_{L},\rho_{L})$).\;We define a $E$-linear map $\nabla:\ext^1_{\gal_L}(\rho_{L},\rho_{L})\rightarrow E^{2d_L}$ by sending $\widetilde{\rho_{L}}$ to $(d_{1,\tau},d_{2,\tau})_{\tau\in\Sigma_L}$ such that $(\wt_\tau(\delta_1)+\epsilon d_{1,\tau},\wt_\tau(\delta_1)+\epsilon d_{2,\tau})_{\tau\in\Sigma_L}$ equals to the Sen's weights of $\widetilde{\rho_{L}}$.\;Secondly,\;the composition $\Spec E[
\epsilon]/\epsilon^2\rightarrow X_{\mathrm{tri}}(\overline{r})\rightarrow \widehat{T}_L$ also gives a deformation $\widetilde{\delta}=\widetilde{\delta}_1\otimes\widetilde{\delta}_2$ of ${\delta}={\delta}_1\otimes{\delta}_2$.\;The two properties in \cite[(12)]{CompanionpointforGLN2L} also hold in semistable non-crystalline case.\;Let $V_1$ be the kernel of the composition:
\begin{equation}
	\begin{aligned}
		\ext^1_{(\varphi,\Gamma)}(D,D)\rightarrow\;&\ext^1_{(\varphi,\Gamma)}(\cR_{E,L}(\delta'_1),D)\\&\rightarrow\ext^1_{(\varphi,\Gamma)}(\cR_{E,L}(\delta'_1),\cR_{E,L}(\delta'_2))\rightarrow\ext^1_{(\varphi,\Gamma)}(\cR_{E,L}(\delta_1),\cR_{E,L}(\delta'_2)).\;
	\end{aligned}
\end{equation}
Since $\Sigma^-(\delta)\cup \Sigma(x)\neq \emptyset$,\;the first two maps are still surjective,\;so the conclusions in \cite[Lemma 2.6,\;Lemma 2.7]{CompanionpointforGLN2L} also hold.\;We thus  have $\mathrm{Im}(f)=V_1\cap \nabla^{-1}(W)$ and $K(\rho_{L})\cap T_{X,x}=K(\rho_{L})$.\;These prove $(1)$ and $(2)$,\;by applying the same arguments as in \cite[Page 62]{CompanionpointforGLN2L}.\;We need more argument in the proof of $(3)$ and $(4)$.\;We need the following exact sequences:
\begin{equation}
	\begin{aligned}
	&0\rightarrow K(\rho_{L})\rightarrow T_{X_{J-\dR}(\underline{\bh}_J),x}\xrightarrow{f}\ext^1_{\gal_L,g,J}(\rho_{L},\rho_{L})\cap V_1\cap \nabla^{-1}(W)\rightarrow 0,\\
	&0\rightarrow K(\rho_{L})\rightarrow T_{X_{J'-\dR}(\underline{\bh}_J),x}\xrightarrow{f}\ext^1_{\gal_L,g,J'}(\rho_{L},\rho_{L})\cap V_1\cap \nabla^{-1}(W\cap W_J)\rightarrow 0.\\
	\end{aligned}
\end{equation}
We claim that:
\begin{itemize}
	\item[(a)] $\dim_E\ext^1_{\gal_L,g,J}(\rho_{L},\rho_{L})\cap V_1=\dim_E\ext^1_{\gal_L}(\rho_{L},\rho_{L})-3|J|-(d_L-|\Sigma(x)|-|J\cap(\Sigma^+(\delta)\backslash\Sigma(x))|)$.\;
%	\item[(b)] The induced map $\nabla:\ext^1_{\gal_L,g,J }(\rho_{L},\rho_{L})\cap V_1 \rightarrow W_{J\cup\Sigma^{-}(\delta)}$ is surjective.\;
\end{itemize}
We instead of considering the cohomology of the corresponding $E$-$B$-pairs:
 \begin{equation}\label{Jcomposition}
	\begin{aligned}
		\hH^1_{g,J}(\gal_L,W(\rho_{L})\otimes W(\rho_{L})^{\vee})\xrightarrow{j_1}\;&\hH^1_{g,J}(\gal_L,B_E(\delta'_2)\otimes W(\rho_{L})^{\vee})\\&\xrightarrow{j_2} \hH^1_{g,J}(\gal_L,B_E(\delta'_2(\delta'_1)^{-1}))\xrightarrow{j_3}\hH^1_{g,J}(\gal_L,B_E(\delta'_2(\delta_1)^{-1})).\;
	\end{aligned}
\end{equation}
Denote $\delta_0=\delta'_2(\delta_1)^{-1}$ and $\delta'_0=\delta'_2(\delta'_1)^{-1}$.\;Put ${\mathbf{n}}=\bh_{1}-\bh_2$.\;Then $\delta_0=\unr(q_L)z^{-\mathbf{n}_{\Sigma^+(\delta)\backslash\Sigma(x)}}z^{\mathbf{n}_{\Sigma^-(\delta)}}$,\;$\delta'_0=\unr(q_L)z^{-\mathbf{n}_{\Sigma^+(\delta)\backslash\Sigma(x)}}z^{\mathbf{n}_{\Sigma^-(\delta)\cup \Sigma(x)}}$ and $\delta'_0=\delta_0z^{\mathbf{n}_{\Sigma(x)}}$.\;

Since $\widetilde{\hH}^2_J(\gal_L,B_E(\delta'_2(\delta'_2)^{-1}))=0$,\;we deduce from \cite[Proposition A.5]{CompanionpointforGLN2L} that $j_2$ is surjective.\;For the map $j_3$,\;the same strategy as in the proof of \cite[Page 65]{CompanionpointforGLN2L} shows that $j_3$ restrict to a surjective map $\hH^1_{g,J}(\gal_L,B_E(\delta'_0))\twoheadrightarrow \hH^1_{g,J\cup \Sigma(x)}(\gal_L,B_E(\delta_0))$.\;

By \cite[Proposition A.5]{CompanionpointforGLN2L},\;the surjectivity of $j_1$ are related to the cohomology group $\widetilde{\hH}^2_J(\gal_L,B_E(\delta'_1)\otimes W(\rho_{L})^{\vee})$.\;We distinguish two cases $J\cap \Sigma(x)\neq \Sigma(x)$ and $\Sigma(x)\subseteq J$.\;

Suppose $J\cap \Sigma(x)\neq \Sigma(x)$.\;Then $\widetilde{\hH}^2_J(\gal_L,B_E(\delta'_1)\otimes W(\rho_{L})^{\vee})=0$ and the composition (\ref{Jcomposition}) induces a surjection $\hH^1_{g,J}(\gal_L,W(\rho_{L})\otimes W(\rho_{L})^{\vee})\twoheadrightarrow \hH^1_{g,J\cup \Sigma(x)}(\gal_L,B_E(\delta_0))$.\;Note that $B_E(\delta'_1)\otimes W(\rho_{L})^{\vee}$ is an extension of $B_E(1)$ by $B_E((\delta'_0)^{-1})$.\;Then the surjectivity follows from $\widetilde{\hH}^2_J(B_E(1))=0$ and \[\widetilde{\hH}^2_J(B_E((\delta'_0)^{-1}))={\hH}^2(B_E(\unr(q_L^{-1})z^{\mathbf{n}_{\Sigma^+(\delta)\backslash\Sigma(x)}}z^{-\mathbf{n}_{\Sigma^-(\delta)\cup \Sigma(x)}}z^{1+\mathbf{n}_{(\Sigma^-(\delta)\cup \Sigma(x))\cap J}}))=0.\]
We then compute $\hH^1_{g,J}(\gal_L,W(\rho_{L})\otimes W(\rho_{L})^{\vee})$
and $\hH^1_{g,J}(\gal_L,B_E(\delta_0))$.\;We use \cite[Proposition A.3]{CompanionpointforGLN2L}.\;It remains to compute  $\dim_E\widetilde{\hH}^2_{J}(\gal_L,W(\rho_{L})\otimes W(\rho_{L})^{\vee})$
and $\dim_E\widetilde{\hH}^2_{J\cup\Sigma(x)}(\gal_L,B_E(\delta_0))$.\;First,\;it is clear that $\dim_E\widetilde{\hH}^2_{g,J\cup\Sigma(x)}(\gal_L,B_E(\delta_0))=0$.\;Then $\dim_E\hH^1_{g,J}(\gal_L,B_E(\delta_0))=d_L-|\Sigma(x)|-|J\cap(\Sigma^+(\delta)\backslash\Sigma(x))|$. On the other hand,\;since the gradded pieces of $W(\rho_{L})\otimes W(\rho_{L})^{\vee}$ are $B_E((\delta'_0)^{-1})$,\;$B_E(1)$ and $B_E(\delta'_0)$.\;Then the $\widetilde{\hH}^2_J(\gal_L,-)$ of these objects are all zero,\;thus $\dim_E\widetilde{\hH}^2_{J}(\gal_L,W(\rho_{L})\otimes W(\rho_{L})^{\vee})=0$.\;We deduce from \cite[Proposition A.3]{CompanionpointforGLN2L} that $\dim_E\hH^1_{g,J}(\gal_L,W(\rho_{L})\otimes W(\rho_{L})^{\vee})=\dim_E\ext^1_{\gal_L}(\rho_{L},\rho_{L})-3|J|$.\;In this case,\;we get that $\dim_E\ext^1_{\gal_L,g,J}(\rho_{L},\rho_{L})\cap V_1=\dim_E\ext^1_{\gal_L}(\rho_{L},\rho_{L})-3|J|-(d_L-|\Sigma(x)|-|J\cap(\Sigma^+(\delta)\backslash\Sigma(x))|)$ (the same as in the proof of \cite[Page 66]{CompanionpointforGLN2L}).\;Note that the argument in \cite[Page 66,\;Proof of Lemma 2.10]{CompanionpointforGLN2L} also holds in our case,\;i.e.,\;the induced map $\nabla:\ext^1_{\gal_L,g,J }(\rho_{L},\rho_{L})\cap V_1 \rightarrow W_{J}$ is surjective.\;Therefore,\;the calculations of $(3)$ and $(4)$ for the case  $J\cap \Sigma(x)\neq \Sigma(x)$ are the same as \cite[Theorem 2.4]{CompanionpointforGLN2L}.\;

Now suppose $\Sigma(x)\subseteq J$.\;Then $\widetilde{\hH}^2_J(\gal_L,B_E(\delta'_1)\otimes W(\rho_{L})^{\vee})\neq 0$ but $j_3$ is surjective.\;In this case,\;we can prove that $\dim_E\hH^1_{g,J}(\gal_L,B_E(\delta_0))=1+d_L-|J|$ and  $\dim_E\hH^1_{g,J}(\gal_L,W(\rho_{L})\otimes W(\rho_{L})^{\vee})=1+4d_L-3|J|$.\;The map $j_1$ lies in the following exact sequence:
\begin{equation}
	\begin{aligned}
		0\rightarrow\hH^0(\gal_L,W(\rho_{L})\otimes W(\rho_{L})^{\vee})\rightarrow&\;\hH^0(\gal_L,B_E(\delta'_2)\otimes W(\rho_{L})^{\vee})\rightarrow \hH^1_{g,J}(\gal_L,B_E(\delta'_1)\otimes W(\rho_{L})^{\vee})\\ &\rightarrow \hH^1_{g,J}(\gal_L,W(\rho_{L})\otimes W(\rho_{L})^{\vee})\xrightarrow{j_1}\hH^1_{g,J}(\gal_L,B_E(\delta'_2)\otimes W(\rho_{L})^{\vee})
	\end{aligned}
\end{equation}
The first map induces an isomorphism.\;Thus we get that $\ker(j_1)=\hH^1_{g,J}(\gal_L,B_E(\delta'_1)\otimes W(\rho_{L})^{\vee})$ and $\dim_E\ker(j_1)=\dim_E\hH^1_{g,J}(\gal_L,B_E(\delta'_1)\otimes W(\rho_{L})^{\vee})=1+2d_L-\dim_E\hH^0(\gal_L,(B_E(\delta'_1)\otimes W(\rho_{L})^{\vee})^+_{\dr})=1+2d_L-(|J|+|J\cap(\Sigma^-(\delta)\cup \Sigma(x)|)$.\;Put $\eta''_J= z^{\mathbf{n}_{(\Sigma^+(\delta)\backslash\Sigma(x))\cap J}+1}\prod_{\tau\in \Sigma(x)}\tau$,\;$\eta^{\ast}=\prod_{\tau\in \Sigma(x)}\tau$ and $\eta^{\#}_J=z^{\mathbf{n}_{(\Sigma^+(\delta)\backslash\Sigma(x))\cap J}+1}$. Then $\eta^{\#}_J=\eta''_J\eta^{\ast}_J$.\;We have an injection $B_E(\eta^{\#}_J)\hookrightarrow B_E(\eta''_J)$ and a commutative diagram:
	\[\xymatrix@C=3.5ex{
	\hH^1(\gal_L,W(\rho_{L})\otimes W(\rho_{L})^{\vee}\otimes B_E(\eta^{\#}_Jz^{\mathbf{n}_{\Sigma(x)\cup \Sigma^{-}(\delta)}+1})) \ar[r]^{\hspace{25pt}f'}\ar[d]_{\fj_1}  & \hH^1(\gal_L,B_E(\delta'_2)\otimes W(\rho_{L})^{\vee}\otimes B_E(\eta^{\#}_J)) \ar[d]_{\fj_2}  \\
	\hH^1(\gal_L,W(\rho_{L})\otimes W(\rho_{L})^{\vee}) \ar[r] & \hH^1(\gal_L,B_E(\delta'_2)\otimes W(\rho_{L})^{\vee})}.\]
\[\xymatrix@C=3.5ex{
\ar[r]	& \hH^1(\gal_L,B_E(\delta'_2(\delta'_1)^{-1}\eta^{\#}_J)) \ar[r]\ar[d]_{\fj_3} & \hH^1(\gal_L,B_E(\delta'_2(\delta_1)^{-1}\eta''_J)) \ar[d]_{\fj_4}  \\
\ar[r]	&\hH^1(\gal_L,B_E(\delta'_2(\delta'_1)^{-1})) \ar[r] & \hH^1(\gal_L,B_E(\delta'_2(\delta_1)^{-1}))}.\]
We then have $\mathrm{Im}\fj_1=\hH^1_{g,J}(\gal_L,W(\rho_{L})\otimes W(\rho_{L})^{\vee})$,\;$\mathrm{Im}\fj_2=\hH^1_{g,J}(\gal_L,B_E(\delta'_2)\otimes W(\rho_{L})^{\vee})$,\;and $\mathrm{Im}\fj_3=\hH^1_{g,J}(\gal_L,B_E(\delta'_2(\delta'_1)^{-1}))$,\;$\mathrm{Im}\fj_4=\hH^1_{g,J}(\gal_L,B_E(\delta'_2(\delta_1)^{-1}))$.\;By \cite[(1.7)]{2015Ding},\;we have $$\mathrm{Im}(f')=\hH^1_{g,\Sigma^{-}(\delta)\cup \Sigma(x)}(\gal_L,B_E(\delta'_2)\otimes W(\rho_{L})^{\vee}\otimes B_E(\eta^{\#}_J)).\;$$
Therefore,\;we deduce that  the composition (\ref{Jcomposition}) restricts to a  surjective map:
\[\hH^1_{g,J}(\gal_L,W(\rho_{L})\otimes W(\rho_{L})^{\vee})\rightarrow\hH^1_{g,J\cup\Sigma^{-}(\delta)}(\gal_L,B_E(\delta'_2(\delta_1)^{-1})).\]
Note that $\dim_E\hH^1_{g,J\cup\Sigma^{-}(\delta)}(\gal_L,B_E(\delta'_2(\delta_1)^{-1}))=1+d_L-|J|$ (the $\Sigma^{-}(\delta)$-component has no effect).\;In this case,\;we get that $\dim_E\ext^1_{\gal_L,g,J}(\rho_{L},\rho_{L})\cap V_1=\dim_E\ext^1_{\gal_L}(\rho_{L},\rho_{L})-(2|J|+d_L)$.\;We complete the proof of the claim.\;Note that $W_{J}\subset W$ since $\Sigma(x)\subseteq J$.\;The same argument as in the proof of \cite[Proposition 2.8]{CompanionpointforGLN2L} show $(3)$ and $(4)$.\;Indeed,\;we have 
\begin{equation}
	\begin{aligned}
			\ext^1_{\gal_L,g,J}(\rho_{L},\rho_{L})\cap V_1\cap \nabla^{-1}(W)=\;&\dim_E\ext^1_{\gal_L,g,J}(\rho_{L},\rho_{L})\cap V_1\\
			=\;&\dim_E\ext^1_{\gal_L}(\rho_{L},\rho_{L})-(2|J|+d_L).\;
		\end{aligned}
\end{equation}
Applying the above formula to $J'$.\;Since $J'\cap \Sigma(x)\neq \Sigma(x)$,\;then $\ext^1_{\gal_L,g,J'}(\rho_{L},\rho_{L})\cap V_1\cap \nabla^{-1}(W\cap W_J)$ is the preimage of $W\cap W_{J}$ via the surjective map $\ext^1_{\gal_L,g,J'}(\rho_{L},\rho_{L})\cap V_1\cap \nabla^{-1}(W)\rightarrow W\cap W_{J'}$.\;This implies $(4)$.\;
\end{proof}
\begin{rmk}If $\Sigma(x)\subseteq J'$,\;then this computation is not clear since we do not know the image of $\nabla:\ext^1_{\gal_L,g,J' }(\rho_{L},\rho_{L})\cap V_1\rightarrow W$ explicitly.\;Note that $W_{J}\subset W$ since $\Sigma(x)\subseteq J$.\;We claim that $W_{J\cup\Sigma^{-}(\delta)}\subset \mathrm{Im}(\nabla)$.\;The Colmez-Greenberg-Stevens formula \cite[Theorem 2.1]{2015Ding} describes the obstructions of liftings to $E[\epsilon]/\epsilon^2$.\;We show that the induced map $\nabla:\ext^1_{\gal_L,g,J}(\rho_{L},\rho_{L})\cap V_1 \rightarrow W_{J\cup \Sigma^{-}(\delta)}$ is surjective.\;Let $\widetilde{\delta_1'}:L^\times\rightarrow E[\epsilon]/\epsilon^2$ be a continuous character with $\widetilde{\delta_1'}\equiv {\delta_1'}\Modo\;\epsilon$ and $\wt_\tau(\widetilde{\delta_1'})=\wt_\tau({\delta_1'})$ for $\tau\in J\cup \Sigma^{-}(\delta)$.\;Then  \cite[Theorem 2.1]{2015Ding} and its proof show that there exists a deformation $\widetilde{W}$ of $W(\rho_{L})$ over $E[\epsilon]/\epsilon^2$ and a deformation $\widetilde{\delta_2'}$ of ${\delta_2'}$ over $E[\epsilon]/\epsilon^2$ such that $\widetilde{W}$ comes from $\ext^1_{\gal_L,g,J\cup \Sigma^{-}(\delta)}(\rho_{L},\rho_{L})\cap V_1$ with parameter $(\widetilde{\delta_1'},\widetilde{\delta_2'})$ if and only if $(\widetilde{\delta_2'}\widetilde{\delta_2'}^{-1}(p)-1)/\epsilon+\sum_{\tau\in  \Sigma^{+}(\delta)\backslash\Sigma(x)}\sL_{\tau}(d_{1,\tau}-d_{2,\tau})=0$ (so the value $\widetilde{\delta_2'}(p)$ is uniquely determined),\;where $\sL_{\tau}$ are the $\sL$-invariants defined in \cite[Definition 1.20]{2015Ding},\;and $(\wt_\tau(\delta_1)+\epsilon d_{1,\tau},\wt_\tau(\delta_1)+\epsilon d_{2,\tau})_{\tau\in\Sigma_L}$ equals to the Sen's weights of $\widetilde{W}$.\;Thus,\;by choosing the value of $\widetilde{\delta_2'}\widetilde{\delta_2'}^{-1}(p)$ carefully,\;the above assertion follows.\;
\end{rmk}
As a corollary,\;we immediately get:
\begin{cor}\label{corofdimTangent}If Let $J'\subset J$ such that $J'\cap \Sigma(x)\neq \Sigma(x)$ and $(J\backslash J')\cap \Sigma(x)\neq \emptyset$,\;then $X_{J-\dR}({\bh}_J)$ is a proper closed subspace of $X_{J'-\dR}({\bh}_J)$.\;
\end{cor}

Put $\lambda_{\bh}=(\bh_1,\bh_2+1)$.\;In \cite[Section 3.3.2]{CompanionpointforGLN2L},\;the author introduces some stratifications on patched eigenvariety.\;For any $J\subset \Sigma_L$,\;we get a reduced closed subspace  $X_{\fp}(\overline{{\rho}},\lambda_{\bh,J})$ of $X_{\fp}(\overline{{\rho}})$ such that $(\rho,\delta')\in  X_{\fp}(\overline{{\rho}},\lambda_{\bh,J})$ if and only if the eigenspace $J_{\bB}(\Pi_\infty^{R_\infty-\ana}(\lambda_{\bh,J}))[\fm_\rho,\bT(L)=\delta']\neq0$,\;where $\Pi_\infty^{R_\infty-\ana}(\lambda_{\bh,J}):=(\Pi_\infty^{R_\infty-\ana}\otimes_EL(\lambda_{\bh,J}^{\vee}))^{\Sigma_L\backslash J}\otimes L(\lambda_{\bh,J})$.\;We have a natural morphism $	X_{\fp}(\overline{{\rho}},\lambda_{\bh,J})\rightarrow \widehat{T}_L(\lambda_{\bh,J})$.\;For $J'\subseteq J$,\;we put $X_{\fp}(\overline{{\rho}},\lambda_{\bh,J},J'):=X_{\fp}(\overline{{\rho}},\lambda_{\bh,J'}){\times}_{\widehat{T}_L(\lambda_{\bh,J'})}\widehat{T}_L(\lambda_{\bh,J})$.\;Such subspaces fit into the following commutative diagram (by \cite[(22)]{CompanionpointforGLN2L}):
\begin{equation}
	\xymatrix@C=2ex{ X_{\fp}(\overline{{\rho}},\lambda_{\bh,J}) \ar[d] \ar[r]&  X_{\fp}(\overline{{\rho}},\lambda_{\bh,J},J') \ar[r] \ar[d]&
		X_{\fp}(\overline{{\rho}},\lambda_{\bh,J'})\ar[r]\ar[d] & X_{\fp}(\overline{{\rho}},\lambda_{\bh,J'})' \ar[d] \ar[r]  & X_{\fp}(\overline{{\rho}}) \ar[d]\\
		\widehat{T}_L(\lambda_{\bh,J})\ar[r]&	\widehat{T}_L(\lambda_{\bh,J}) \ar[r] &\widehat{T}_L(\lambda_{\bh,J'}) \ar[r] & \widehat{T}_L(\lambda_{\bh,J'}) \ar[r] & \widehat{T}_L,}
\end{equation}
where the horizontal maps are closed embedding,\;and the second and fourth square are cartesian.\;The injection (\ref{injpatchtotri}) induces a closed embedding (by \cite[(34)]{CompanionpointforGLN2L}):
\begin{equation}\label{injpatchtotriJver}
	X_{\fp}(\overline{\rho},\lambda_{\bh,J})^{\mathrm{red}}\hookrightarrow\FX_{\overline{\rho}^{\fp}}^\Box\times \BU^g\times \iota_{\fp}\big(X_{J-\dR}({\bh}_J)).\;
\end{equation}

By the same argument as in the proof of \cite[Theorem 3.21]{CompanionpointforGLN2L},\;we show that
\begin{thm}\label{appeGL2LR=T}(Infinitesimal ``$R=T$" results)
Let $y=(r_y,\delta)\in X_{\fp}(\overline{{\rho}},\lambda_{\bh,J})(E)$ such that $r_{\fp}$ is isomorphic to $\rho_{L}$.\;Suppose $x$ is spherical (i.e.,\;$\delta$ is locally algebraic and $\delta z^{-\wt(\delta)}$ is unramified).\;Then $X_{\fp}(\overline{{\rho}},\lambda_{\bh,J})$ is smooth at $x$,\;and we have a natural isomorphism of complete regular noetherian local $E$-algebras:
\[\widehat{\cO}_{X_{\fp}(\overline{\rho},\lambda_{\bh,J}),{x}}\cong \widehat{\cO}_{\FX_{\overline{\rho}^{\fp}}^\Box\times \iota_{\fp}\big(X_{J-\dr}(\underline{\bh}_J)\big)\times \BU^g,x}.\]
\end{thm}

By Theorem \ref{appeGL2LR=T} and Corollary \ref{corofdimTangent},\;we get

\begin{cor}\label{R=Tcor}Keep the situation of Theorem \ref{appeGL2LR=T}.\;Let $J'\subset J$ and $J'\cap \Sigma(x)\neq \Sigma(x)$.\;The following statements are equivalent:
\begin{itemize}
	\item[(i)] The natural projection $\widehat{\cO}_{X_{\fp}(\overline{\rho},\lambda_{\bh,J},J'),{x}}\twoheadrightarrow\widehat{\cO}_{X_{\fp}(\overline{\rho},\lambda_{\bh,J}),{x}}$ is an isomorphism;
	\item[(ii)] $X_{\fp}(\overline{\rho},\lambda_{\bh,J},J')$ is smooth at $x$;
	\item[(iii)] $(J\backslash J')\cap \Sigma(x)=\emptyset$.\;
\end{itemize} 
\end{cor}

We now state the locally analytic socle conjecture  for $\GLN_2(L)$ case.\;For $J\subset \Sigma_L$,\;we denote
\[I^c_S(\alpha,\bh):=C(1,s_J)=\cF^{\GLN_2(L)}_{\ob(L)}(\underline{L}(-s_J\cdot\lambda_{\bh)},1).\;\]
\begin{conjecture}\label{appenconj-1}
Assume Hypothesis \ref{hyongaloisrepgl2l}.\;Then $I^c_S(\alpha,\bh)\hookrightarrow \widehat{S}(U^{\fp},E)_{\overline{\rho}}^{\ana}[\fm_{\rho}]$ if and only if $S\subseteq\Sigma(y)$.\;
\end{conjecture}

In $\GLN_2(L)$-case,\;this conjecture is in fact equivalent to the following conjecture on companion points on $X_{\fp}(\overline{{\rho}})$.\;For $J\subset \Sigma_L$,\;we put $\chi_J^c=(\unr(\alpha q_L^{-1})z^{\bh_{1,\Sigma_L\backslash J}}z^{\bh_{2,J}})\otimes(\unr(\alpha q_L^{-1})z^{\bh_{1,\Sigma_L\backslash J}+1}z^{\bh_{2,J}+1})$.\;
\begin{conjecture}\label{appenconj-2}
Assume Hypothesis \ref{hyongaloisrepgl2l}.\;Then $(r_y,\chi)\in X_{\fp}(\overline{{\rho}})$ if and only if $\chi=\chi_J^c$ for some $J\subseteq \Sigma(x)$.\;In particular,\;the point $y_J:=(r_y,\chi_J^c)$ lie in $X_{\fp}(\overline{{\rho}},\lambda_{\bh,J})$.\;
\end{conjecture}

\begin{thm}
Let $y=(r_y,\delta)\in X_{\fp}(\overline{{\rho}},\lambda_{\bh,J})(E)$ such that $r_{\fp}$ is isomorphic to $\rho_{L}$.\;Suppose $y$ is spherical  and $r_{y,\widetilde{v}}$ is generic for $v\in \Sigma(U^p)\backslash S_p$.\;Let $x$ be the associated point in $X_{\mathrm{tri}}(\overline{r},{\bh}_J)$.\;Suppose $\Sigma(x)\neq \emptyset$.\;Then for all $\tau\in \Sigma(x)$,\;$y_{\tau}^c=(r_y,\delta_{\tau}^c)\in X_{\fp}(\overline{{\rho}},\lambda_{\bh,\Sigma^{+}(\delta)\backslash \tau})(E)$.\;
\end{thm}
\begin{proof}We put $J=\Sigma^{+}(\delta)$ and $J'=\Sigma^{+}(\delta)\backslash \tau$.\;Then we have $J'\cap \Sigma(x)
\neq \Sigma(x)$.\;Then the claim,\;i.e.,\;$(37)$ in the proof of 
\cite[Theorem 4.4]{CompanionpointforGLN2L} also holds in our case,\;instead by using our Theorem \ref{appeGL2LR=T} and Corollary \ref{R=Tcor}.\;Thus this theorem follows from the Breuil's adjunction formula,\;as in the proof of \cite[Theorem 4.4,\;Page 94]{CompanionpointforGLN2L}.\;
\end{proof}

Finally,\;similar to the proof of \cite[Corollary 46]{CompanionpointforGLN2L},\;we actually get
\begin{thm}The Conjecture \ref{appenconj-2} (and hence Conjecture \ref{appenconj-1}) is true.\;
\end{thm}
\begin{rmk}The above theorem gives a complete version of \cite[Theorem 4.22]{2015Ding},\;which only proves Conjecture \ref{appenconj-2} (and hence Conjecture \ref{appenconj-1}) if $|S|=1$.\;
\end{rmk}

%\subsection{Appendix C.\;Explicit computations on local models}

%\subsection{Appendix B.\;Some argument on Hypothesis \ref{appenhypothesis}}\label{Appendixsemiastable}{\color{red}{(......)}}

%\subsection{Locally analytic socle conjecture-II}

\bibliographystyle{plain}

\printindex
\end{document}